\documentclass{article}

\usepackage{amsfonts}
\usepackage{amsmath}
\usepackage{amssymb}
\usepackage{amsthm}
	
	\newtheorem{thm}{Theorem}[section]
	\newtheorem{cor}[thm]{Corollary}
	\newtheorem{lem}[thm]{Lemma}
	\newtheorem{prop}[thm]{Proposition}
	\newtheorem{defn}[thm]{Definition}
\usepackage{authblk}
\usepackage[margin=1in]{geometry}
\usepackage{graphicx}
\usepackage{hyperref}
\usepackage{makecell}
\usepackage{multicol}
\usepackage{tikz}
	\usetikzlibrary{calc, arrows, decorations.markings, decorations.pathmorphing, positioning, decorations.pathreplacing}

\newcommand{\To}{\longrightarrow}
\newcommand{\C}{\mathbb{C}}
\newcommand{\N}{\mathbb{N}}
\newcommand{\Z}{\mathbb{Z}}
\newcommand{\R}{\mathbb{R}}
\newcommand{\M}{\mathcal{M}}
\newcommand{\Q}{\mathbb{Q}}
\newcommand{\CP}{\mathbb{CP}}
\DeclareMathOperator*{\Res}{Res}
\DeclareMathOperator{\Sgn}{Sgn}

\newcolumntype{x}[1]{>{\centering\arraybackslash}p{#1}}
\newcommand\diag[4]{%
  \multicolumn{1}{p{#2}|}{\hskip-\tabcolsep
  $\vcenter{\begin{tikzpicture}[baseline=0,anchor=south west,inner sep=#1]
  \path[use as bounding box] (0,0) rectangle (#2+2\tabcolsep,\baselineskip);
  \node[minimum width={#2+2\tabcolsep},minimum height=\baselineskip+\extrarowheight] (box) {};
  \draw (box.north west) -- (box.south east);
  \node[anchor=south west] at (box.south west) {#3};
  \node[anchor=north east] at (box.north east) {#4};
 \end{tikzpicture}}$\hskip-\tabcolsep}}

\begin{document}

\title{Counting curves on surfaces}

\author{Norman Do}
\affil{School of Mathematical Sciences,
Monash University,
VIC 3800, Australia
\texttt{norm.do@monash.edu}}

\author{Musashi A. Koyama}
\affil{School of Mathematical Sciences,
Monash University,
VIC 3800, Australia
\texttt{koyama.musashi@gmail.com}} 

\author{Daniel V. Mathews}
\affil{School of Mathematical Sciences,
Monash University,
VIC 3800, Australia
\texttt{Daniel.Mathews@monash.edu}}

\maketitle

\begin{abstract} 
In this paper we consider an elementary, and largely unexplored, combinatorial problem in low-dimensional topology. Consider a real 2-dimensional compact surface $S$, and fix a number of points $F$ on its boundary. We ask: how many configurations of disjoint arcs are there on $S$ whose boundary is $F$?

We find that this enumerative problem, counting curves on surfaces, has a rich structure. For instance, we show that the curve counts obey an effective recursion, in the general framework of topological recursion. Moreover, they exhibit quasi-polynomial behaviour. 

This ``elementary curve-counting" is in fact related to a more advanced notion of ``curve-counting" from algebraic geometry or symplectic geometry. The asymptotics of this enumerative problem are closely related to the asymptotics of volumes of moduli spaces of curves, and the quasi-polynomials governing the enumerative problem encode intersection numbers on moduli spaces.

Furthermore, among several other results, we show that generating functions and differential forms for these curve counts exhibit structure that is reminiscent of the mathematical physics of free energies, partition functions, topological recursion, and quantum curves.
\end{abstract}

\tableofcontents

\section{Introduction}

\subsection{Summary and motivation}

``Curve-counting" plays an important part of several areas of contemporary mathematics. For instance, moduli spaces of curves are central to Gromov--Witten theory, and zero-dimensional moduli spaces consist of a finite number of curves, which can be counted. Such curve counts are used to define boundary operators in Floer homology theories.

In this paper we count curves of a much simpler type: we count arrangements of curves on surfaces. Consider a real 2-dimensional compact surface $S$ with boundary; fixing some boundary conditions, we count collections of curves on that surface --- that is, embedded 1-manifolds --- with those boundary conditions, up to some notion of equivalence. In this paper we present several results about the numbers of such curves --- including how they are related to ``curve-counting" of the more advanced type. 

Fixing the genus $g$ and number of boundary components $n$ of the surface $S$, and a number of points $b_1, \ldots, b_n$ on each boundary component, we define numbers $G_{g,n}(b_1, \ldots, b_n)$ and $N_{g,n}(b_1, \ldots, b_n)$ (and various refined versions thereof), counting collections of curves of various types on this surface. Roughly, our main results say the following.
\begin{itemize}
\item
If we fix $g$ and $n$, these curve counts exhibit behaviour that is ``essentially" polynomial. (More precisely, \emph{quasi-polynomial} behaviour.)
\item
The curve counts on a surface $S$ can be given recursively in terms of curve counts for surfaces of simpler topology.
\item
The degrees of these polynomials, and their top-degree coefficients, are closely related to \emph{moduli spaces of curves} and in fact recover the intersection numbers of $\psi$-classes.
\item
The counts can be encoded in generating functions and differential forms and in fact different types of counts can be obtained by expanding the same differential form in different coordinates.
\item
Various generating functions encoding these curve counts obey differential equations reminiscent of the mathematical physics of free energies and partition functions. 
\end{itemize}

These results are similar in spirit to a wide range of results on the \emph{topological recursion} of Chekhov, Eynard, and Orantin~\cite{ChekhovEynard06, EynardOrantin07, EynardOrantin09}. There has been a great deal of recent work demonstrating that many enumerative problems formulated in terms of surfaces display similar phenomena: polynomiality, recursion, and differential forms and generating functions obeying physically suggestive equations. Such problems arise, for instance, in matrix models~\cite{ChekhovEynard06}, the theory of Hurwitz numbers~\cite{BHLM14, BouchardMarino, DoLeighNorbury, EynardMulaseSafnuk11}, moduli spaces of curves~\cite{Do-Norbury11, Mulase_Penkava12, Norbury10_counting_lattice_points, Norbury13_string}, Gromov-Witten theory~\cite{BKMP09, DOSS14, EynardOrantin15, FangLiuZong, Norbury-Scott14} and combinatorics~\cite{DoManescu, Dumitrescu-Mulase-Safnuk-Sorkin13, DOPS14, Mulase_Sulkowski12}.

We also note that the enumeration of isotopy classes of contact structures near a convex surface in a contact 3-manifold essentially reduces to a similar question, counting of arrangements of \emph{dividing sets} on the surface (see e.g. \cite{Gi91, Hon00I, Me09Paper}). While dividing sets are a more specific notion than the arc diagrams we count here, a similar analysis may be possible.

Nonetheless, the counting question we consider is an elementary one. On a disc, it leads immediately to the Catalan numbers. Our curve counts are thus an elementary generalisation of the Catalan numbers from discs to surfaces of general topology. (Other generalisations also exist, see e.g. \cite{Dumitrescu-Mulase15, Dumitrescu-Mulase-Safnuk-Sorkin13, Mulase13_laplace, Mulase_Sulkowski12}.)

Despite being a straightforward combinatorial question that could have been asked well over a century ago, we have not found many results about these curve counts in the literature, beyond discs and annuli. Recently, Drube--Pongtanapaisan in \cite{Drube-Pongtanapaisan15} counted a slightly different notion of curves on annuli, and Kim in \cite{Kim12} counted noncrossing matchings and permutations on annuli.

In this introduction we present an outline of the results in this paper.

\subsection{Counts of curves on surfaces}
\label{sec:intro_curve_counts}

Let $G_{g,n}(b_1, \ldots, b_n)$ be the number of collections of curves on a surface of genus $g$, with $n$ boundary components, with $b_1, \ldots, b_n$ endpoints respectively on  the boundary components. We give a precise definition of these \emph{arc diagrams} in section \ref{sec:arc_diagrams}, along with a discussion of various possible types of curves to count. In the case of a disc, $G_{0,1}(2m)$ is the $m$'th Catalan number.

For fixed $(g,n)$, we give several explicit formulae for these numbers. The
formulae depend on the parity of the $b_i$, and so we write $b_i = 2m_i$ or $2m_i + 1$, with $m_i$ a non-negative integer, accordingly.
\begin{thm}
\label{thm:formulas}
For any integers $m_1, m_2, m_3 \geq 0$,
\begin{align}
G_{0,1}(2m) &= C_m = \frac{1}{m+1} \binom{2m}{m}, \quad \text{the $m$'th Catalan number} \label{eqn:G01} \\
G_{0,2}(2m_1, 2m_2) &= \frac{m_1 + m_2 + m_1 m_2}{m_1 + m_2} \binom{2m_1}{m_1} \binom{2m_2}{m_2} \label{eqn:G02ee}\\
G_{0,2}(2m_1 + 1, 2m_2 + 1) &= \frac{(2m_1 + 1)(2m_2 + 1)}{m_1 + m_2 + 1} \binom{2m_1}{m_1} \binom{2m_2}{m_2} \label{eqn:G02oo} \\
G_{0,3}(2m_1, 2m_2, 2m_3) &= (m_1 + 1)(m_2 + 1)(m_3 + 1) \binom{2m_1}{m_1} \binom{2m_2}{m_2} \binom{2m_3}{m_3} \label{eqn:G03eee} \\
G_{0,3}(2m_1 + 1, 2m_2 + 1, 2m_3) &= (2m_1 + 1)(2m_2 + 1)(m_3 + 1) \binom{2m_1}{m_1} \binom{2m_2}{m_2} \binom{2m_3}{m_3} \label{eqn:G03ooe} \\
G_{1,1} (2m) &= \left( \frac{m^2}{12} + \frac{5m}{12} + 1 \right) \binom{2m}{m} \label{eqn:G11}
\end{align}
\end{thm}
The result for $G_{0,1}(2m)$ is general knowledge. The special case $G_{0,2}(2n,0) = \binom{2n}{n}$ appears in a paper of Przytycki \cite{Przytycki99_fundamentals}; possibly it was known earlier but we are unable to find it elsewhere in the literature. The result for $G_{0,2}$ was found by Kim \cite[thm. 6.2]{Kim12}; we were informed of this result after posting the initial version of this paper. The other formulae, so far as we know, are new. We prove the statements for annuli by direct combinatorial arguments, which we develop in section \ref{sec:counting_on_annuli}. Similar formulas can be obtained for other values of $(g,n)$ using the results of sections~\ref{sec:top_recursion} and~\ref{sec:polynomiality}.

In each case above, $G_{g,n}(b_1, \ldots, b_n)$ is given by a product of combinatorial factors of the form $\binom{2m}{m}$, multiplied by a symmetric rational function in the $b_i$ (or equivalently $m_i$); these factors and rational functions depend on the parity of the $b_i$. We show that the $G_{g,n}$ have a similar structure for all $(g,n)$. In fact, the cases $(g,n) = (0,1)$ and $(0,2)$ are exceptional: for any other $(g,n)$, we obtain \emph{polynomials} rather than rational functions. 
\begin{thm}
\label{thm:G_polynomiality}
For $(g,n) \neq (0,1), (0,2)$, $G_{g,n}(b_1, \ldots, b_n)$ is the product of
\begin{enumerate}
\item
a combinatorial factor $\binom{2m_i}{m_i}$ for each $i = 1, \ldots, n$, where $b_i = 2m_i$ if $b_i$ is even and $b_i = 2m_i + 1$ if $b_i$ is odd; and
\item
a quasi-polynomial $P_{g,n}(b_1, \ldots, b_n)$, symmetric in the variables $b_1, \ldots, b_n$, depending on the parity of $b_1, \ldots, b_n$, with rational coefficients, of degree $3g-3+2n$.
\end{enumerate}
\end{thm}
(A \emph{quasi-polynomial} $f(x_1, \ldots, x_n)$ is a family of polynomial functions depending on some congruence classes of the integers $x_1, \ldots, x_n$.)

Thus, for instance, if $(g,n) \neq (0,1), (0,2)$ and all $b_i$ are even, $b_i = 2m_i$, then
\[
G_{g,n} (2m_1, \ldots, 2m_n) = \binom{2m_1}{m_1} \cdots \binom{2m_n}{m_n} P_{g,n}(2m_1, \ldots, 2m_n)
\]
where $P_{g,n}$ is a polynomial of degree $3g-3+2n$ with rational coefficients; there will be a similar expression (but with a different polynomial), for $G_{g,n}(2m_1+1, 2m_2+1, 2m_3, \ldots, 2m_n)$; and so on.

The proof of theorem \ref{thm:G_polynomiality} is effective: it provides a method by which such formulae can be calculated for any $(g,n)$.

The $G_{g,n}$ also satisfy a recursion, expressing the counts on a surface in terms of counts on surfaces with simpler topology.
\begin{thm}
\label{thm:Ggn_recursion}
\begin{align*}
G_{g,n}(b_1, \ldots, b_n) &= \sum_{i+j = b_1 - 2} G_{g-1,n+1}(i,j,b_2, \ldots, b_n) \\
& \quad + \sum_{k=2}^n b_k G_{g,n-1}(b_1 + b_k - 2, b_2, \ldots, \widehat{b_k}, \ldots, b_n) \\
& \quad + \sum_{i+j = b_1 - 2} \sum_{\substack{g_1 + g_2 = g \\ I \sqcup J = \{2, \ldots, n\}}} G_{g_1, |I|+1} (i, b_I) \; G_{g_2, |J|+1} (j, b_J).
\end{align*}
\end{thm} 
A precise statement is given in theorem \ref{thm:G_recursion}. The notation $\widehat{b}_k$ means that $b_k$ is omitted from the list $b_2, \ldots, b_n$. We will discuss the details of this theorem, including the notation, in section \ref{sec:recursion_for_curve_counts}. 

This recursion is not new: an identical recursion is satisfied by the \emph{generalised Catalan numbers} studied by Dumitrescu, Mulase, Safnuk, Sorkin and Su{\l}kowski~\cite{Dumitrescu-Mulase15, Dumitrescu-Mulase-Safnuk-Sorkin13, Mulase13_laplace, Mulase_Sulkowski12}.

\subsection{Counts of non-boundary-parallel curves}

It turns out to be natural also to count collections of curves satisfying an additional condition: that no curve be \emph{boundary-parallel}. In other words, we require that no curve cut off a disc. Let $N_{g,n}(b_1, \ldots, b_n)$ be the number of such collections of curves. The relationship between $G_{g,n}$ and $N_{g,n}$ is analogous to the relationship between Hurwitz numbers and pruned Hurwitz numbers \cite{Do-Norbury13}. 

As with the $G_{g,n}$, we give some explicit formulae for the $N_{g,n}$.
\begin{thm}
\label{thm:N_formulas}
For any integers $b_1, b_2, b_3, b_4 \geq 0$,
\begin{align}
N_{0,1}(b_1) &= \delta_{b_1, 0} \label{eqn:N01} \\
N_{0,2}(b_1, b_2) &= \bar{b}_1 \delta_{b_1, b_2} \label{eqn:N02} \\
N_{0,3}(b_1, b_2, b_3) &= \bar{b}_1 \bar{b}_2 \bar{b}_3 \quad \text{provided $b_1+b_2+b_3$ is even; $0$ otherwise.} \label{eqn:N03} \\
N_{0,4}(b_1, b_2, b_3, b_4) &= \bar{b}_1 \bar{b}_2 \bar{b}_3 \bar{b}_4 \;
 \widehat{N}_{0,4} (b_1, b_2, b_3, b_4) \quad \text{(provided not all $b_i = 0$)} \label{eqn:N04} \\
N_{1,1}(b_1) &= \bar{b}_1 \left( \frac{b_1^2}{48} + \frac{5}{12} \right) \quad \text{(provided $b_1 \neq 0$ even)}. \label{eqn:N11}
\end{align}
where $\widehat{N}_{0,4}(b_1, b_2, b_3, b_4)$ is the quasi-polynomial
\[
\widehat{N}_{0,4} (b_1, b_2, b_3, b_4) = \left\{ \begin{array}{ll}
\frac{1}{4} (b_1^2 + b_2^2 + b_3^2 + b_4^2) + 2 & \text{all $b_i$ even} \\
\frac{1}{4} (b_1^2 + b_2^2 + b_3^2 + b_4^2) + \frac{1}{2} & \text{two $b_i$ even, two $b_i$ odd} \\
\frac{1}{4} (b_1^2 + b_2^2 + b_3^2 + b_4^2) + 2 & \text{all $b_i$ odd} \\
0 & \text{otherwise} \end{array} \right.
\]
\end{thm}
Here $\bar{n}$ is a convenient notation: we define $\bar{n} = n$ for $n$ positive, and $\bar{0} = 1$. 

The pattern in the structure of $N_{g,n}$ continues, the cases $(g,n) = (0,1)$ and $(0,2)$ again being exceptional. We again obtain symmetric quasi-polynomials; in fact, they are all \emph{even} symmetric polynomials. 
\begin{thm}
\label{thm:intro_N_polynomiality}
For $(g,n) \neq (0,1), (0,2)$ and $(b_1, \ldots, b_n) \neq (0, \ldots, 0)$,
\[
N_{g,n}(b_1, \ldots, b_n) = \bar{b}_1 \cdots \bar{b}_n \; \widehat{N}_{g,n}(b_1, \ldots, b_n),
\]
where $\widehat{N}_{g,n}(b_1, \ldots, b_n)$ is a symmetric quasi-polynomial over $\Q$ in $b_1^2, \ldots, b_n^2$ of degree $3g-3+n$, depending on the parity of $b_1, \ldots, b_n$.
\end{thm}
The proof is again effective: in principle we can calculate the quasi-polynomials for any $N_{g,n}$.

The general curve count $G_{g,n}$ and the non-boundary-parallel curve count $N_{g,n}$ are related by the following result, for which we give a direct combinatorial proof in section \ref{sec:decomposing_arc_diagrams}.
\begin{thm}
\label{thm:G_in_terms_of_N}
\label{thm:G_and_N}
For $(g,n) \neq (0,1)$,
\[
G_{g,n} (b_1, \ldots, b_n) = \sum_{a_1 = 0}^{b_1} \cdots \sum_{a_n = 0}^{b_n} \binom{b_1}{ \frac{b_1 - a_1}{2} } \cdots \binom{b_n}{ \frac{b_n - a_n}{2}} N_{g,n} (a_1, \ldots, a_n).
\]
\end{thm}
Here we consider the binomial coefficient $\binom{M}{N}$ to be zero if $N$ is not integral.

The degree $3g-3+n$ of the quasi-polynomials $\widehat{N}_{g,n}$ is familiar as the (complex) dimension of the moduli space of curves $\M_{g,n}$. We will show, in fact, that the top-degree terms of these polynomials encode \emph{intersection numbers} in the compactified moduli space $\overline{\M}_{g,n}$.
\begin{thm}
\label{thm:intro_intersection_numbers}
For $(g,n) \neq (0,1), (0,2)$, the nonzero polynomials representing the quasi-polynomial $\widehat{N}_{g,n}(b_1, \ldots, b_n)$ agree in their top-degree terms. For non-negative integers $d_1, \ldots, d_n$ such that $d_1 + \cdots + d_n = 3g-3+n$, the coefficient $c_{d_1, \ldots, d_n}$ of $b_1^{d_1} \cdots b_n^{d_n}$ satisfies
\[
c_{d_1, \ldots, d_n} = \frac{1}{2^{5g-6+2n} \; d_1 ! \cdots d_n ! }
\langle \psi_1^{d_1} \cdots \psi_n^{d_n}, \; \overline{\M}_{g,n} \rangle.
\]
\end{thm}
Here $\psi_i \in H^2 (\overline{\M}_{g,n}; \Q)$ is the Chern class of the vector bundle over $\overline{\M}_{g,n}$ given by pulling back the cotangent bundle at the $i$'th marked point. We could also write
\[
c_{d_1, \ldots, d_n} = \frac{1}{2^{5g-6+2n} \; d_1 ! \cdots d_n ! } \; \int_{\overline{\M}_{g,n}} \psi_1^{d_1} \cdots \psi_n^{d_n}.
\]
For more information on moduli spaces of curves and their intersection theory, see the book of Harris and Morrison~\cite{harris-morrison}.

The top-degree coefficients $c_{d_1, \ldots, d_n}$ in fact agree exactly with the \emph{lattice count polynomials} of Norbury~\cite{Norbury10_counting_lattice_points} and agree up to simple normalisation constants with the \emph{volume polynomials} of Kontsevich \cite{Kontsevich_Intersection} and the \emph{Weil--Petersson volume polynomials} calculated by Mirzakhani~\cite{mirzakhani}. Hence the asymptotics of the polynomials $\widehat{N}_{g,n}$ are equivalent to the asymptotics of volumes of moduli spaces of curves ${\mathcal M}_{g,n}$. Details are given in section \ref{sec:volume_moduli}, in particular theorem \ref{thm:volume_polynomials}. 

Thus, such a naive enterprise as counting curves on surfaces leads naturally to the topology of moduli spaces. 

The $N_{g,n}$ and $\widehat{N}_{g,n}$ also obey a recursion, of a similar nature as for the $G_{g,n}$. The recursion for $N_{g,n}$ is given in theorem \ref{thm:Ngn_recursion}, and for $\widehat{N}_{g,n}$ in corollary \ref{cor:HatNgn_recursion}.

\subsection{Curve-counting refined by regions}

While counting curves on surfaces, we can also keep track of the number of \emph{regions} into which they cut the surface. We define $G_{g,n,r} (b_1, \ldots, b_n)$ to be the number of collections of curves cutting the surface into $r$ regions; a precise definition is given in section \ref{sec:refined_counts}. Similarly we can define $N_{g,n,r} (b_1, \ldots, b_n)$. 

In this way, we \emph{refine} counts of curves on a surface $S$ by the number of regions into which they cut $S$. It turns out that these refined counts obey many similar properties as the overall unrefined curve counts.

For instance, the refined counts $G_{g,n,r}$ obey a similar recursion as the unrefined counts $G_{g,n}$. Cutting along an arc preserves the number of complementary regions.
\begin{thm}
\label{thm:G_refined_recursion_intro}
\begin{align*}
G_{g,n,r}(b_1, \ldots, b_n)
&=
\sum_{\substack{i,j \geq 0 \\ i+j = b_1 - 2}} G_{g-1,n+1,r} (i,j,b_2, \ldots, b_n) \\
&\quad + \sum_{k=2}^n b_k G_{g,n-1,r} (b_1 + b_k - 2, b_2, \ldots, \widehat{b}_k, \ldots, b_n) \\
&\quad + \sum_{\substack{g_1 + g_2 = g \\ I_1 \sqcup I_2 = \{2, \ldots, n\} }}
\sum_{\substack{i,j \geq 0 \\ i+j = b_1 - 2}} \sum_{r_1 + r_2 = r}
G_{g_1, |I_1|+1, r_1} (i, b_{I_1}) G_{g_2, |I_2| + 1, r_2} (j, b_{I_2}).
\end{align*} 
\end{thm} 
A precise statement is given in theorem \ref{thm:G_refined_recursion}, and a recursion of a similar nature is given for $N_{g,n,r}$ in theorem \ref{thm:Ngnr_recursion}; however, $r$ is not so well preserved in this recursion.

Once $g,n, b_1, \ldots, b_n$ are fixed, the number of regions $r$ into which $S$ can be cut by a collection of arcs is clearly bounded. We prove various inequalities between these parameters in section \ref{sec:inequalities_on_regions}. In the process, we find that it is useful to introduce an alternative parameter to keep track of regions, which we call $t$. (Explicitly, $t = r - \chi(S) - \frac{1}{2} \sum b_i$.) As such, we have a second way of refining the curve counts, which we denote $G_{g,n}^t$ and $N_{g,n}^t$. These also obey recursions: a recursion for $N_{g,n}^t$ is given in corollary \ref{cor:Ngnt_recursion}.

The $G_{g,n}^t$ and $N_{g,n}^t$ obey polynomiality properties similar to, but more complicated than, $G_{g,n}$ and $N_{g,n}$. One result is the following.
\begin{thm} 
\label{thm:N_polynomiality_k_t_both_0}
For $(g,n) \neq (0,1), (0,2)$, positive integers $b_1, \ldots, b_n$, and setting $t=0$,
\[
N_{g,n}^0 = \bar{b}_1 \cdots \bar{b}_n \widehat{N}_{g,n}^0 (b_1, \ldots, b_n),
\]
where $\widehat{N}_{g,n}^0$ is a symmetric quasi-polynomial over $\Q$ in $b_1^2, \ldots, b_n^2$, of degree $3g-3+n$, depending on the parity of $b_1, \ldots, b_n$.
\end{thm}
Precise and more detailed statements are given in theorems \ref{thm:Nhatgnt_polynomiality} (proving polynomiality) and \ref{thm:Ngntk_Ngn_agreement} (proving the degree). A precise statement of polynomiality for the $G_{g,n}^t$ is theorem \ref{thm:Ggnt_polynomiality}.

We compute several examples of refined counts explicitly in section \ref{sec:refined_discs_annuli}. We give formulae for $G_{0,1,r}$ and $G_{0,1}^t$ in lemma \ref{lem:G01r}; for $G_{0,2,r}$ and $G_{0,2}^t$ in lemma \ref{lem:G02r}; and for $N_{0,1,r}$, $N_{0,1}^t$, $N_{0,2,r}$ and $N_{0,2}^t$ in lemma \ref{lem:N01t_N02t_computations}. 

These refined polynomials also recover intersection numbers on moduli spaces.
\begin{thm}
\label{thm:intersection_numbers_refined}
For $(g,n) \neq (0,1), (0,2)$, positive integers $b_1, \ldots b_n$ and $t=0$, the nonzero polynomials representing the quasi-polynomial $\widehat{N}_{g,n}^0 (b_1, \ldots, b_n)$ agree in their top-degree terms, and they agree with the top-degree terms of $\widehat{N}_{g,n} (b_1, \ldots, b_n)$. That is, the coefficient $c_{d_1, \ldots, d_n}$ of $b_1^{d_1} \cdots b_n^{d_n}$ is given by
\[
c_{d_1, \ldots, d_n} = \frac{1}{2^{5g-6+2n} \; d_1 ! \cdots d_n ! } \; \int_{\overline{\M}_{g,n}} \psi_1^{d_1} \cdots \psi_n^{d_n}.
\]
\end{thm}
A more general statement is proved in theorems \ref{thm:Ngntk_Ngn_agreement} and \ref{thm:intersection_numbers_k}.

The $\widehat{N}_{g,n}^t$ have similar properties for other values of $t$ (not just $t=0$). We can also set some of the variables $b_i$ to zero. For different choices of $t$ and choices of variables set to zero, we obtain \emph{different} polynomials. It is as if $\widehat{N}_{g,n}(b_1, \ldots, b_n)$ is a quasi-polynomial depending on the ``parity" of $b_1, \ldots, b_n$, where there are three possible ``parities": even, odd, and zero.

For each choice of $k$, the number of variables set to zero, and $t$, in an appropriate range, we obtain a separate quasi-polynomial in the $b_i^2$. We show that $\widehat{N}_{g,n}^t$ has degree at most $3g-3+n-t+k$ in general (theorem \ref{thm:Nhatgnt_degree_upper_bound}); and if $k=t$, the degree is exactly $3g-3+n$, with top-degree coefficients agreeing with $\widehat{N}_{g,n}$ (theorem \ref{thm:Ngntk_Ngn_agreement}).

In a certain sense, given a collection of curves, $t$ is a measure of ``how separating" the curves are. Fixing $k$, the minimal value of $t$ is $k$. (Section \ref{sec:inequalities_on_regions} makes these notions precise.) Thus, the above theorems say that it is sufficient to consider curves which are ``as non-separating as possible" in order to recover the geometry of moduli spaces.

\subsection{Differential forms and free energies}

The curve counts $G_{g,n}$ and $N_{g,n}$ fit, at least to some extent, into the framework of Eynard--Orantin topological recursion, with its connections to enumerative geometry and mathematical physics.

Following this framework (e.g. \cite{Dumitrescu-Mulase15, Mulase13_laplace, Mulase_Penkava12, Mulase_Sulkowski12}), we define several generating functions based on the $N_{g,n}$ and $G_{g,n}$. Chief among these are multidifferentials in $n$ variables $x_1, \ldots, x_n$ on $\CP^1$ defined by
\[ 
\omega_{g,n} (x_1, \ldots, x_n) = \sum_{\mu_1 \geq 0} \cdots \sum_{\mu_n \geq 0} G_{g,n} (\mu_1, \ldots, \mu_n) \; x_1^{-\mu_1 - 1} \cdots x_n^{-\mu_n - 1} \; dx_1 \cdots dx_n.
\]
A full definition and discussion of $\omega_{g,n}$ is given in section \ref{sec:defns}. Although defined as a formal power series, we show $\omega_{g,n}$ is in fact meromorphic (proposition \ref{prop:omega1_meromorphic}).

It turns out, if we rewrite $\omega_{g,n}$ with respect to new variables $z_1, \ldots, z_n$ defined by $x_i = z_i + \frac{1}{z_i}$, then the coefficients switch from the curve counts $G_{g,n}$, to the non-boundary-parallel curve counts $N_{g,n}$. A similar phenomenon occurs with pruned Hurwitz numbers \cite{Do-Norbury13}. Strictly speaking we pull back the forms; a full discussion is given in section \ref{sec:change_of_coords}.
 \begin{thm}
\label{thm:change_of_coords}
For $(g,n) \neq (0,1)$,
\[
\omega_{g,n} = \sum_{\nu_1 \geq 0} \cdots \sum_{\nu_n \geq 0} N_{g,n} ( \nu_1, \ldots, \nu_n ) \; z_1^{\nu_1 - 1} \cdots z_n^{\nu_n - 1} \; dz_1 \cdots dz_n.
\]
\end{thm}
(In the case $(g,n) = (0,1)$, we have two distinct forms, which we denote $\omega_{0,1}^G$ and $\omega_{0,1}^N$.) We compute several $\omega_{g,n}$ explicitly.
\begin{thm}
\label{thm:small_omegas}
\begin{align*}
\omega_{0,1}^N (z_1) &=
z_1^{-1} dz_1  \\
\omega_{0,1}^G (x_1) &=  
\frac{x_1-\sqrt{x_1^2-4}}{2} \; dx_1 = z_1 \; dx_1 = (z_1 - z_1^{-1}) \; dz_1  \\
\omega_{0,2}  (z_1, z_2) &= \left( \frac{1}{z_1 z_2} + \frac{1}{(1-z_1 z_2)^2} \right) \; dz_1 dz_2 \\
\omega_{0,3} (z_1, z_2, z_3) &= \frac{1 + z_1^4 z_2^4 z_3^4 + 
\sum_{\text{cyc}} (z_1^4 + z_1 z_2 + z_1^3 z_2^3 + z_1^4 z_2^4) 
+ \sum_{\text{sym}} (z_1^3 z_2 + z_1^4 z_2^3 z_3 + z_1^4 z_2 z_3)}
{z_1 z_2 z_3 (1-z_1^2)^2 (1-z_2^2)^2 (1-z_3^2)^2} \; dz_1 dz_2 dz_3 
\end{align*}
\end{thm}

We can then obtain \emph{free energies} $F_{g,n}$ by integrating the $\omega_{g,n}$. (A precise definition \ref{defn:free_energy} and discussion is given in section \ref{sec:free_energies}.) We compute some explicitly.
\begin{thm}
\label{thm:free_energy_examples}
The following functions are free energy functions.
\begin{align*}
F_{0,1}^N (z_1) &= \log z_1 \\
F_{0,1}^G (z_1) &= \frac{1}{2} z_1^2 - \log z_1 \\
F_{0,2} (z_1, z_2) &= \log z_1 \log z_2 - \log(1 - z_1 z_2) \\
F_{0,3}(z_1, z_2, z_3) &= \log z_1 \log z_2 \log z_3 
+ \frac{z_1 z_2 + z_2 z_3 + z_3 z_1 + 1}{(1-z_1^2)(1-z_2^2)(1-z_3^2)}
+ \sum_{\text{cyc}} \left( 
\frac{\log z_1 \log z_2}{1-z_3^2} 
+ \frac{(z_1 z_2 + 1) \log z_3}{(1-z_1^2)(1-z_2^2)} \right),
\end{align*}
\end{thm}

In section \ref{sec:refining_omega} we discuss how the differential forms $\omega_{g,n}$ can be refined according to number of regions. For each value of the number of regions $r$, and the related parameter $t$, we obtain meromorphic forms $\omega_{g,n,r}$ (proposition \ref{prop:omegagnr_meromorphic}) and $\omega_{g,n}^t$ (proposition \ref{prop:omegagnt_meromorphic}). We also show (theorem \ref{thm:omegas_equal_refined}) that changing coordinates from $z_i$ to $x_i$, changes $\omega_{g,n}^t$ from a generating function for the $N_{g,n}^t$, into a generating function for the $G_{g,n}^t$. (However, such a statement does not hold for $\omega_{g,n,r}$.) In other words, theorem \ref{thm:change_of_coords} can be refined with respect to $t$.

Thus, there are natural refinements $\omega_{g,n}^t$ of the differential forms $\omega_{g,n}$. Moreover, for given $g,n$, there are only finitely many possible values of $t$, so $\omega_{g,n}$ splits as a finite sum of $\omega_{g,n}^t$. We compute some $\omega_{g,n}^t$ explicitly in section \ref{sec:small_cases_refined}: for discs (proposition \ref{prop:f01Nt}), annuli (proposition \ref{prop:f02Nt}) and pants (proposition \ref{prop:f03Nt}). 

We can similarly refine free energies $F_{g,n}$ into a finite sum of $F_{g,n}^t$. We give some explicit computations, which can be compared with theorem \ref{thm:free_energy_examples}.

\begin{thm}
\label{thm:refined_free_energies}
The following functions are free energy functions.
\begin{align*}
F_{0,1}^{N, 0} (z_1) &= \log z_1 \\
F_{0,1}^{G,0} (x_1) &= \frac{1}{2} z_1^2 - \log z_1 \\
F_{0,2}^0 (z_1, z_2) &= -\log(1-z_1 z_2)  \\
F_{0,2}^1 (z_1, z_2) &= \log z_1 \log z_2 \\
F_{0,3}^0 (z_1, z_2, z_3) &= \frac{z_1 z_2 + z_2 z_3 + z_3 z_1 + 1}{(1-z_1^2)(1-z_2^2)(1-z_3^2)} \\
F_{0,3}^1 (z_1, z_2, z_3) &= \frac{(z_2 z_3  + 1) \log z_1}{(1-z_2^2)(1-z_3^2)}
+ \frac{(z_3 z_1 + 1) \log z_2}{(1-z_3^2)(1-z_1^2)}
+ \frac{(z_1 z_2 + 1) \log z_3}{(1-z_1^2)(1-z_2^2)} \\
F_{0,3}^2 (z_1, z_2, z_3) &= \log z_1 \log z_2 \log z_3 +
\frac{\log z_1 \log z_2}{1-z_3^2} + \frac{\log z_2 \log z_3}{1-z_1^2} + \frac{\log z_3 \log z_1}{1-z_2^2} \end{align*}
\end{thm}

\subsection{Differential equations and partition function}

The recursions on the curve counts $G_{g,n}$ and $N_{g,n}$ (and also their refined versions) translate into recursive differential equations on their generating functions.

The differential forms $\omega_{g,n}$ can be written as $f_{g,n} (x_1, \ldots, x_n) \; dx_1 \cdots dx_n$, where
\[
f_{g,n} (x_1, \ldots, x_n) = \sum_{\mu_1, \ldots, \mu_n \geq 0} G_{g,n}(\mu_1, \ldots, \mu_n) x_1^{-\mu_1 - 1} \cdots x_n^{-\mu_n - 1}.
\]
is a function of $n$ variables. (Precise definitions are given in section \ref{sec:defns}.) To form a recursive differential equation on the $f_{g,n}$, we take the recursion in theorem \ref{thm:Ggn_recursion}, multiply by $x_1^{-\mu_1 - 1} \cdots x_n^{-\mu_n - 1}$, and sum over $\mu_1, \ldots, \mu_n$. However, when we formulate the recursion precisely (theorem \ref{thm:G_recursion}), we note it does not apply when $b_1 = 0$, so certain terms are missing, corresponding to the initial conditions in the recursion. In other words, the obstacle to obtaining a recursive differential equation in the $f_{g,n}$ is not the recursion, but the initial conditions. 

One way to deal with this issue is to ``differentiate out" the initial terms; doing so, we obtain a differential equation given in theorem \ref{thm:first_diff_eq}.

A better way to deal with this issue is to use the \emph{refined} counts of curves, keeping track of the number of regions. With refined counts, there is a simple way to express $G_{g,n,r}(0, b_2, \ldots, b_n)$ in terms of $G_{g,n-1,r} (b_2, \ldots, b_n)$ (proposition \ref{prop:b1_equals_zero}). This is something like a ``dilaton equation" for curve-counting.

Therefore, we define generating functions which keep track of the number of regions $r$, using a new variable $r$. We can define a generating function
\[
\mathfrak{f}_{g,n} (x_1, \ldots, x_n; \alpha) 
= \sum_{r \geq 1} \sum_{\mu_1, \ldots, \mu_n \geq 0} G_{g,n,r} (\mu_1, \ldots, \mu_n) \; x_1^{-\mu_1 - 1} \cdots x_n^{-\mu_n - 1} \; \alpha^r
\]
See definition \ref{defn:mathfrakf} for details (we call this function $\mathfrak{f}_{g,n}^G$ there). In fact in section \ref{sec:putting_together} we consider various generating functions and differential forms, which use the various refined counts $G_{g,n,r}$, $G_{g,n}^t$, $N_{g,n,r}$ and $N_{g,n}^t$. We find relations between them (proposition \ref{prop:relations_between_fs}) and show they are all meromorphic (propositions \ref{prop:bffgn_meromorphic} and \ref{prop:mathfrakf_meromorphic}). We also compute them in various small cases (discs in proposition \ref{prop:f01G_f01N_etc}, annuli and pants in proposition \ref{prop:bff0203}). For instance, for discs we find
\[
\mathfrak{f}_{0,1} (x_1; \alpha) = \frac{x - \sqrt{x^2 - 4\alpha}}{2},
\]
which reduces to the equation $f_{0,1}(x_1) = \frac{x-\sqrt{x^2 - 4}}{2}$ of theorem \ref{thm:small_omegas} (recall $\omega_{0,1} (x_1) = f_{0,1} (x_1) \; dx_1$) upon setting $\alpha = 1$.

We can then obtain a recursive differential equation in the $\mathfrak{f}_{g,n}$ (section \ref{sec:refined_diff_eqns}).
\begin{thm}
\label{thm:diff_eqn_gen_fns}
For any $g \geq 0$ and $n \geq 1$,
\begin{align*}
x_1 \; \mathfrak{f}_{g,n} (x_1, \ldots, x_n; \alpha)
&=
\mathfrak{f}_{g-1,n+1} (x_1, x_1, x_2, \ldots, x_n; \alpha) \\
&\quad +
\sum_{k=2}^n 
\frac{\partial}{\partial x_k}
\frac{1}{x_k - x_1}
\left(
\mathfrak{f}_{g,n-1} (x_2, \ldots, x_n; \alpha) - \mathfrak{f}_{g,n-1} (x_1, x_2, \ldots, \widehat{x}_k, \ldots, x_n; \alpha)
\right) \\
&\quad +
\sum_{\substack{g_1 + g_2 = g \\ I_1 \sqcup I_2 = \{2, \ldots, n\} }}
\mathfrak{f}_{g_1, |I_1|+1} (x_1, x_{I_1}; \alpha) \mathfrak{f}_{g_2, |I_2| + 1} (x_1, x_{I_2}; \alpha) \\
&\quad +
\alpha \frac{\partial}{\partial \alpha} \mathfrak{f}_{g,n-1} (x_2, \ldots, x_n; \alpha).
\end{align*}
\end{thm}

From this, we find a differential equation on free energies $\mathfrak{F}_{g,n} (x_1, \ldots, x_n; \alpha)$ defined by integrating the $\mathfrak{f}_{g,n}$; precise definitions and statements can be found in section \ref{sec:diff_eqns} and theorem \ref{thm:free_energy_recursion}.
\begin{thm}
\label{thm:free_energy_recursion_intro}
\begin{align*}
x_1 \frac{\partial}{\partial x_1} \mathfrak{F}_{g,n} (x_1, \ldots, x_n; \alpha)
&=
\frac{\partial^2}{\partial u \partial v} \mathfrak{F}_{g-1,n+1} (u,v,x_2, \ldots, x_n; \alpha) \Big|_{u=v=x_1} \\
&\quad + \sum_{k=2}^n \frac{1}{x_k - x_1} 
\left( 
\frac{\partial}{\partial x_k} \mathfrak{F}_{g,n-1} (x_2, \ldots, x_n; \alpha)
- \frac{\partial}{\partial x_1} \mathfrak{F}_{g,n-1} (x_1, \ldots, x_n; \alpha)
\right) \\
&\quad + \sum_{\substack{g_1 + g_2 = g \\ I_1 \sqcup I_2 = \{2, \ldots, n\}}}
\frac{\partial}{\partial x_1} \mathfrak{F}_{g_1, |I_1|+1} (x_1, x_{I_1}; \alpha) \;
\frac{\partial}{\partial x_1} \mathfrak{F}_{g_2, |I_2|+1} (x_1, x_{I_2}; \alpha) \\
&\quad + \alpha \frac{\partial}{\partial \alpha} \mathfrak{F}_{g,n-1} (x_2, \ldots, x_n; \alpha).
\end{align*}
\end{thm}
This differential recursion on the free energies $\mathfrak{F}_{g,n}$ resembles the recursion on free energies of Mulase--Su{\l}kowsi's ``generalised Catalan numbers" \cite{Mulase_Sulkowski12}. An identical recursion applies in that case, but the harder initial conditions here require our recursion to have an extra term.

Combining the free energies into a so-called \emph{partition function} 
\[
{\bf Z} = \exp \left[ \sum_{m=0}^\infty \hbar^{m-1} \sum_{2g+n-1=m} \frac{1}{n!} \mathfrak{F}_{g,n} (x, \ldots, x; \alpha) \right],
\]
we obtain a differential equation satisfied by ${\bf Z}$.
\begin{thm}
\label{thm:quantum_curve_intro}
\[
\left( \hbar^2 \frac{\partial^2}{\partial x^2} - \hbar x \frac{\partial}{\partial x} + \hbar^2 \alpha \frac{\partial}{\partial \alpha} + \alpha \right) {\bf Z} = 0.
\]
\end{thm}
This differential equation provides something like a \emph{``quantum curve"} result for the curve counts $G_{g,n}$, although the extra parameter $\alpha$ appears nonstandard. There is a resemblance to the equation $x^2 - xz + \alpha = 0$, which is obtained from setting $z = \mathfrak{f}_{0,1}^G (x_1; \alpha) = \frac{x_1 - \sqrt{x_1^2 - 4\alpha}}{2}$.

\subsection{Structure of paper}

This paper is organised as follows. In section \ref{sec:Which_curves} we set up our framework for counting curves, and discuss which curves we count. We define $G_{g,n}(b_1, \ldots, b_n)$, $N_{g,n}(b_1, \ldots, b_n)$ and related concepts, and make some elementary observations.

In section \ref{sec:counting_on_annuli} we count curves on discs and annuli, giving formulae for curve counts by elementary combinatorial arguments.

We then turn to the relationship between the curve counts $G_{g,n}$ and $N_{g,n}$. In section \ref{sec:decomposing_arc_diagrams} we show that any collection of curves on a surface can be decomposed in an essentially unique way into a part ``local to the boundary", and a ``core" (section \ref{sec:canonical_decomp}). We use this ``local decomposition" to count arc diagrams (section \ref{sec:counting_local_decomposition}) and express $G_{g,n}$ in terms of $N_{g,n}$ (section \ref{sec:counting_local_annuli}).

We are then able to count curves on pants in section \ref{sec:pants}. After establishing some terminology (section \ref{sec:pants_approach}), we directly compute $N_{0,3}$ (section \ref{sec:non_boundary_parallel_pants}), and then compute $G_{0,3}$ (section \ref{sec:general_pants}).

In section \ref{sec:top_recursion} we turn to recursion. We establish recursions for $G_{g,n}$ (section \ref{sec:recursion_for_curve_counts}) and $N_{g,n}$ (section \ref{sec:non-parallel_recursion}), and use these to make some computations including $N_{1,1}$ (section \ref{sec:applying_small_recursion}).

We then turn to polynomiality. After some preliminary work (section~\ref{sec:useful_sums}), we establish polynomiality of the $N_{g,n}$ (section \ref{sec:Ngn_polynomiality}). Reflecting on this proof establishes the agreement of top-degree terms with Norbury's lattice count (section \ref{sec:comparison_with_lattice}), giving us results about moduli spaces and intersection numbers (section \ref{sec:volume_moduli}). We can then prove polynomiality for the $G_{g,n}$ (section \ref{sec:polynomiality_general}).

Next, we consider generating functions and differential forms. After defining (section \ref{sec:defns}) and computing some small cases (section \ref{sec:small_gen_fns}) of these generating functions, we show they are meromorphic (section \ref{sec:meromorphicity}). We can then show that the expansion of $\omega_{g,n}$ in $x$ and $z$ coordinates yields the $G_{g,n}$ and $N_{g,n}$ (section \ref{sec:change_of_coords}). Free energies can then be defined and computed in small cases (section \ref{sec:free_energies}), and we can make some initial observations about recursions (section \ref{sec:recursion_generating_functions}) and differential equations for the generating functions (section \ref{sec:diff_eqn_on_generating_fns}).

In section \ref{sec:regions} we introduce the refinement of counts by regions. After making definitions (section \ref{sec:refined_counts}), we prove a sort of ``dilaton equation" (section \ref{sec:punctures}) and then compute refined counts on discs and annuli (section \ref{sec:refined_discs_annuli}). We discuss how the concept of local decomposition (section \ref{sec:refining_local}) can be refined, and then use it to compute refined counts on pants (section \ref{sec:refined_pants}). We consider  bounds on the number of regions (sections \ref{sec:inequalities_on_regions} and \ref{sec:existence}), refine the recursion (section \ref{sec:refining_recursion}) and then use these results to prove polynomiality for general refined curve counts (sections \ref{sec:polynomiality_small_cases} to \ref{sec:polynomiality_refined}). Along the way, we obtain relations between the refined polynomials and intersection numbers on moduli spaces of curves (section~\ref{sec:relations_refined}).

Section~\ref{sec:dequations} is devoted to refining the results obtained in section \ref{sec:differential_forms_functions} according to the number of complementary regions (sections \ref{sec:refining_omega} to \ref{sec:refined_diff_eqns}). Finally, we obtain differential equations for the free energies and use this to determine an equation satisfied by the partition function that is reminiscent of the notion of quantum curve (section \ref{sec:diff_eqns}).

\subsection{Acknowledgments}

This paper grew out of a summer vacation research scholarship completed by the second author and supervised by the third author, funded by the Australian Mathematical Sciences Institute. The first author was partly supported by the Australian Research Council grant DE130100650.

\section{Which curves to count?}
\label{sec:Which_curves}

\subsection{Arc diagrams and equivalence}
\label{sec:arc_diagrams}

Throughout, we assume all surfaces are compact, connected, and oriented, unless specified otherwise. We will write $S_{g,n}$ to denote a surface of genus $g$ with $n$ boundary components; when $g$ and $n$ are clear we simply write $S$.

Consider a finite set of points on $\partial S_{g,n}$, called \emph{marked points}. Label the boundary components $B_1, \ldots, B_n$, and let $b_i$ be the number of points on boundary component $B_i$. We may write the $b_i$ as a vector ${\bf b} = (b_1, \ldots, b_n)$. We allow $b_i = 0$ and indeed we allow ${\bf b} = {\bf 0}$. We will write $F(b_1, \ldots, b_n) = F({\bf b})$ to denote such a finite set, and when ${\bf b}$ is clear we simply write $F$. We wish to count curves on $S$ with boundary conditions specified by $F$.
\begin{defn}
An \emph{arc diagram} on $(S,F)$ is a properly embedded collection of arcs $C \subset S$ with boundary $F$.
\end{defn}
Such a $C$ contains finitely many unoriented arcs connecting points of $F$. Proper embedding requires that precisely one arc of $C$ emanate from each point of $F$, and that arcs never cross. By prohibiting crossings, the topology of $S$ restricts the possible arrangements of curves.

(Strictly speaking, one should distinguish between the embedding of a disjoint union of intervals into $S$, and the image of this map. Pre-composing such an embedding with a self-homeomorphism of these intervals gives an equivalent embedding. In practice, we abuse notation and conflate the embedding with its image, regarding $C$ as a subset of $S$. It should not cause any confusion.)

In our arguments, we will often need to consider the regions into which $S$ is cut by a curve diagram.
\begin{defn}
\label{def:comp_region}
A \emph{complementary region} of an arc diagram $C$ on $(S,F)$ is a connected component of $S \setminus C$. The number of complementary components is denoted $r$.
\end{defn}

We note in passing that several other reasonable definitions of collections of curves on $(S,F)$ are possible: for instance, one might allow certain closed curves, require curves to be oriented, or consider dividing sets or sutures. 

Of course, the ``number" of arc diagrams on a given $(S,F)$ is infinite. 
We will thus define a notion of equivalence of arc diagrams, and count the equivalence classes.
\begin{defn}
Two arc diagrams $C_1$ and $C_2$ on $(S, F)$ are \emph{equivalent} if there is a homeomorphism $\phi : S \to S$, such that $\phi|_{\partial S}$ is the identity, and $\phi(C_1) = C_2$.
\end{defn}

This notion of equivalence is stronger than isotopy.  
However, arc diagrams in general have infinitely many isotopy classes. For instance, for an arc diagram $C$ which essentially intersects a homologically nontrivial simple closed curve $\gamma$, applying Dehn twists about $\gamma$ to $C$ yields infinitely many non-isotopic arc diagrams. Arguably, then, the simplest way to count curves on surfaces is to count equivalence classes of arc diagrams as we have defined them.

As our notion of equivalence involves homeomorphisms fixing the boundary pointwise, the labels $1, \ldots, n$ on the boundary components, and the numbers $b_1, \ldots, b_n$ are fixed (they are not permuted) as we count curves. The number of equivalence classes only depends on the numbers $g,n,b_1, \ldots, b_n$. Hence the following definition makes sense. 
\begin{defn}
The set of equivalence classes of arc diagrams on $(S_{g,n}, F({\bf b}))$ is denoted $\mathcal{G}_{g,n}({\bf b})$. The number of such equivalence classes is denoted $G_{g,n} ({\bf b})$.
\end{defn}
Thus $G_{g,n}({\bf b}) = | \mathcal{G}_{g,n}({\bf b}) |$. Our notion of arc diagram includes the empty arc diagram. Thus for all $g$ and $n$, $G_{g,n}({\bf 0}) = 1$.

At this stage it may not be clear that $G_{g,n}({\bf b})$ is finite; however, it is in fact finite, as we observe in section \ref{sec:recursion_for_curve_counts}.

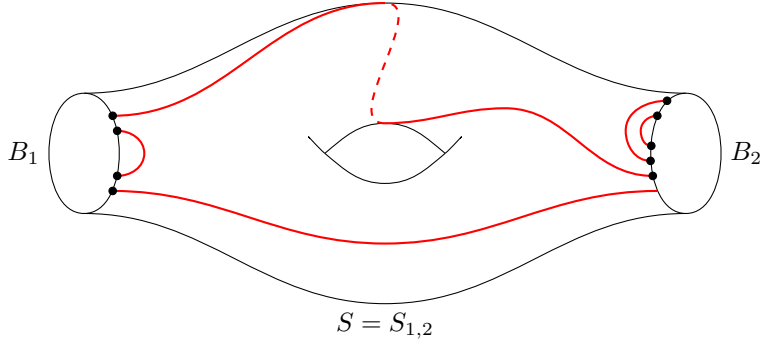
\begin{figure}
\begin{center}
\begin{tikzpicture}
\def\oradius{0.5mm}

\draw (-40mm,8mm) to[out=0,in=0] (-40mm,-8mm);
\draw (-40mm,8mm) to[out=180,in=180] node[left] {$B_1$} (-40mm,-8mm);

\draw (40mm,8mm) to[out=0,in=0] node[right] {$B_2$} (40mm,-8mm);
\draw (40mm,8mm) to[out=180,in=180] (40mm,-8mm);

\draw (-40mm,8mm) to[out=0,in=180] (0mm,20mm) to[out=0,in=180] (40mm,8mm);
\draw (-40mm,-8mm) to[out=0,in=180] (0mm,-20mm) to[out=0,in=180] (40mm,-8mm);
\node at (0mm,-23mm) {$S = S_{1,2}$};

\draw (-8mm,0mm) to[out=40,in=180] (0mm,4mm) to[out=0,in=140] (8mm,0mm);
\draw (-8mm,0mm) to[out=-40,in=180] (0mm,-4mm) to[out=0,in=-140] (8mm,0mm);
\draw (-8mm,0mm) to[out=140,in=130] (-10mm,2mm);
\draw (8mm,0mm) to[out=40,in=50] (10mm,2mm);

\draw[red,thick] (-36.2mm,-5mm) to[out=0,in=180] (0mm,-12mm) to[in=180,out=0] (36.2mm,-5mm);
\draw[red,thick] (35.6mm,-3mm) to[out=180,in=0] (16mm,6mm) to[out=180,in=0] (0mm,4mm);
\draw[red,thick,dashed] (0mm,4mm) to[out=180,in=0] (0mm,20mm);
\draw[red,thick] (0mm,20mm) to[out=180,in=0] (-36.2mm,5mm);

\draw[red,thick] (-35.6mm,3mm) to[out=0,in=90] (-32mm,0mm) to[out=-90,in=0] (-35.6mm,-3mm);

\draw[red,thick] (37.5mm,7mm) to[out=180,in=90] (32mm,3mm) to[out=-90,in=180] (35.3mm,-1mm);
\draw[red,thick] (36.2mm,5mm) to[out=180,in=90] (34mm,3mm) to[out=-90,in=180] (35.4mm,1mm);

\draw[fill=black] (-36.2mm,-5mm) circle (\oradius);
\draw[fill=black] (35.6mm,-3mm) circle (\oradius);
\draw[fill=black] (37.5mm,7mm) circle (\oradius);
\draw[fill=black] (36.2mm,5mm) circle (\oradius);
\draw[fill=black] (-35.6mm,-3mm) circle (\oradius);
\draw[fill=black] (35.3mm,-1mm) circle (\oradius);
\draw[fill=black] (35.4mm,1mm) circle (\oradius);
\draw[fill=black] (-36.2mm,5mm) circle (\oradius);
\draw[fill=black] (-35.6mm,3mm) circle (\oradius);
\end{tikzpicture}
\end{center}
\caption{An arc diagram on $(S, F)$, with $S = S_{1,2}$, $F = F(4, 6)$ and four complementary regions.}
\label{fig:intro}
\end{figure}

If $C_1, C_2$ are equivalent arc diagrams, then any equivalence $\phi$ between them is a self-homeomorphism of $S$ fixing $\partial S$ pointwise. Hence $\phi$ takes arcs of $C_1$ to arcs of $C_2$ in a canonical fashion, and we may refer to an \emph{arc} of the equivalence class without ambiguity. Similarly, $\phi$ takes the complementary regions of $C_1$ to $C_2$ in a canonical fashion, so a \emph{complementary region} of an equivalence class is well-defined.  In practice, we simply represent an equivalence class of arc diagrams by drawing an arc diagram, and refer to its arcs and complementary regions as the arcs and complementary regions of the equivalence class. In a similar fashion, we often drop the phrase ``equivalence classes of" for convenience, and refer only to counting arc diagrams; we hope that the meaning is clear.

\subsection{Non-boundary-parallel arc diagrams}
\label{sec:non_boundary_parallel}

As discussed in the introduction, it is useful to consider arc diagrams without boundary-parallel arcs. An embedded arc in $S$ is boundary-parallel if it is homotopic (relative to endpoints) to an arc lying entirely in $\partial S$.

\begin{defn}
The set of equivalence classes of arc diagrams on $(S_{g,n}, F({\bf b}))$ without boundary-parallel arcs is denoted $\mathcal{N}_{g,n}({\bf b})$. The number of such equivalence classes is denoted $N_{g,n}({\bf b})$.
\end{defn}

\begin{defn}
\label{defn:bar}
For an integer $n \geq 0$ we define
\[
\bar{n} = 
n + \delta_{n,0} =
\begin{cases}
n & n > 0, \\
1 & n = 0.
\end{cases}
\] 
\end{defn}

\begin{defn}
For $g \geq 0$, $n \geq 1$ and $b_1, \ldots, b_n \geq 0$ we define
\[
\widehat{N}_{g,n}({\bf b}) = \frac{N_{g,n}({\bf b})}{\bar{b}_1 \cdots \bar{b}_n}.
\]
\end{defn}

\subsection{First considerations}

Some initial observations about $G_{g,n}({\bf b})$ are clear.

\begin{lem}
\label{lem:even_odd}
For any $g \geq 0$ and $n \geq 1$, if $b_1 + \cdots + b_n$ is odd then $G_{g,n}({\bf b}) = 0$.
\end{lem}

\begin{proof}
Every arc in an arc diagram has two endpoints, so the number of endpoints $b_1 + \cdots + b_n$ in an arc diagram is even.
\end{proof}

We may regard $G_{g,n}$ as a function $\N_0^n \To \N_0$, where $\N_0 = \{0, 1, 2, \ldots\}$. That is, $G_{g,n}$ takes an $n$-tuple of non-negative integers $(b_1, \ldots, b_n)$ and returns a non-negative integer.
\begin{lem}
The function $G_{g,n}(b_1, \ldots, b_n)$ is a symmetric function of $b_1, \ldots, b_n$.
\end{lem}

\begin{proof}
For any permutation $\sigma \in S_n$, there is a homeomorphism $\phi: S \To S$ permuting the boundary components according to $\sigma$, $\phi(B_i) = B_{\sigma(i)}$. So $G_{g,n}(b_1, \ldots, b_n) = G_{g,n}(b_{\sigma(1)}, \ldots, b_{\sigma(n)})$.
\end{proof}

\section{Counting curves on annuli and discs}
\label{sec:counting_on_annuli}

\subsection{Definitions and statements}

We now turn to counting (equivalence classes of) arc diagrams on some simple surfaces. We begin with annuli. As it turns out, along the way we will be able to count arc diagrams on discs.

Throughout this section $S=S_{0,2}$ denotes an annulus, and $F = F(b_1, b_2)$ a set of boundary points. Let $b_i = 2m_i$ or $2m_i + 1$ accordingly as $b_i$ is even or odd. By lemma \ref{lem:even_odd}, if an arc diagram exists on $(S,F)$, then $b_1, b_2$ are either both even or both odd.

For definiteness we can consider $S$ as the region between two concentric circles in the plane (and we will draw annuli in this standard way). We can naturally then speak of ``clockwise" and ``anticlockwise" orientations on boundary components.

Observe that a properly embedded arc $\gamma$ on $S$ is boundary-parallel if and only if both its endpoints lie on the same boundary component of $S$. We may then make the following definition.
\begin{defn}
\label{def:traversing_insular}
A properly embedded arc on an annulus is \emph{traversing} if its endpoints lie on distinct boundary components, and \emph{insular} if its endpoints lie on the same boundary component.
\end{defn}
Thus, insular is synonymous with boundary-parallel, and traversing with non-boundary-parallel. 

\begin{defn}
An arc diagram on an annulus is \emph{traversing} if it contains a traversing arc, and \emph{insular} if all its arcs are insular.

On $(S_{0,2}, F(b_1, b_2))$, the number of equivalence classes of traversing arc diagrams  is denoted $T(b_1, b_2)$, and the number of equivalence classes of insular arc diagrams is denoted $I(b_1, b_2)$.
\end{defn}
So in an insular arc diagram, all the arcs stay close to their home boundary component. In a traversing arc diagram, there is a brave arc traversing the annulus from one side to the other. (The empty arc diagram is vacuously insular.) We will give $I(b_1, b_2)$ and $T(b_1, b_2)$ explicitly.
\begin{prop}
\label{prop:insular_count}
For integers $m_1, m_2 \geq 0$, 
\begin{align*}
I(2m_1, 2m_2) &= \binom{2m_1}{m_1} \binom{2m_2}{m_2} \\
I(2m_1 + 1, 2m_2 + 1) &= 0
\end{align*}
\end{prop}

\begin{prop}
\label{prop:traversing_count}
For integers $m_1, m_2 \geq 0$,
\begin{align*}
T(2m_1, 2m_2) &= \frac{m_1 m_2}{m_1 + m_2} \binom{2m_1}{m_1} \binom{2m_2}{m_2} \\
T(2m_1 + 1, 2m_2 + 1) &= \frac{(2m_1 + 1)(2m_2 + 1)}{m_1 + m_2 + 1} \binom{2m_1}{m_1} \binom{2m_2}{m_2}
\end{align*}
\end{prop}
Clearly $G_{0,2}(b_1, b_2) = I(b_1, b_2) + T(b_1, b_2)$, so theorem \ref{thm:formulas}(\ref{eqn:G02ee})--(\ref{eqn:G02oo}) follows from these two propositions.

We prove both propositions by bijective combinatorial arguments. They are proved in sections \ref{sec:orientations_insular} and \ref{sec:traversing_diagrams} respectively. Proposition \ref{prop:traversing_count} is identical in content to theorem 6.2 of \cite{Kim12}, which in fact contains a more general result with cyclic sieving (lemma 3.3); as mentioned in the introduction, we were informed of this result after posting the initial version of this paper.

\subsection{Number of complementary regions}

First, however, we characterise $r$, the number of complementary regions of an arc diagram, in terms of whether the diagram is insular or traversing.
\begin{lem}
\label{lem:traversing_complementary_regions}
\label{lem:annulus_complementary_regions}
Let $C$ be an arc diagram on an annulus.
\begin{enumerate}
\item
If $C$ is insular then $r = \frac{1}{2}(b_1 + b_2) + 1$. One complementary region is an annulus, and the rest are discs.
\item
If $C$ is traversing then $r = \frac{1}{2} (b_1 + b_2)$. All complementary regions are discs.
\end{enumerate}
\end{lem}

\begin{proof}
First note that $C$ has precisely $\frac{1}{2}(b_1 + b_2)$ arcs. 

If $\gamma$ is a traversing arc, then cutting along $\gamma$ cuts $S$ into a disc. Cutting further along the other $\frac{1}{2}(b_1+b_2) -1$ arcs of $C$, each cut slices off an extra disc. So $S \setminus C$ consists of $\frac{1}{2}(b_1+b_2)$ discs.

If $C$ is insular then, successively cutting along outermost arcs, each cut slices off a disc. At the end we have $\frac{1}{2}(b_1 + b_2)$ discs and an annulus.
\end{proof}

\subsection{Insular diagrams}
\label{sec:orientations_insular}
\label{sec:arrow_diagrams}

An oriented insular arc $\gamma$ on an annulus (between two concentric circles in the plane) is isotopic to a properly embedded arc consisting of a radial arc, followed by an ``angular" arc at constant radius from the centre, followed by another radial arc. We say $\gamma$ is clockwise or anticlockwise according to the direction of the angular arc. 

Given an insular arc diagram $C$ on $(S,F)$, we may orient each arc anticlockwise. Each arc then points into $S$ at one endpoint, and out at the other end. The points of $F$ can be labelled \emph{in} and \emph{out} accordingly. We can represent the ins and outs pictorially by drawing arrows into or out of $S$ at each point. This leads us to the following definition.
\begin{defn}
\label{def:arrow_diagram}
An \emph{arrow diagram} on $(S,F)$ is a labelling of points of $F$ either ``in" or ``out" so that, on each boundary component, exactly half the points are labelled ``in" and half are labelled ``out". 

The set of all arrow diagrams on $(S,F)$ is denoted $A(b_1, b_2)$.
\end{defn}
If an arrow diagram exists on $(S,F)$ then $b_1, b_2$ must both be even, $(b_1, b_2) = (2m_1, 2m_2)$. Choosing the labels of ``in" and ``out" on each boundary component, we have
\[
|A(2m_1, 2m_2)| = \binom{2m_1}{m_1} \binom{2m_2}{m_2}.
\]

We first turn our attention to the case where one boundary component has no marked points, $(b_1, b_2) = (b_1, 0) = (2m, 0)$; in this case any arc diagram is necessarily insular. Let $\Phi: \mathcal{G}_{0,2}(2m, 0) \To A(2m, 0)$ be the function which takes an arc diagram (necessarily insular) to the arrow diagram obtained by orienting each arc anticlockwise.
\begin{prop}
\label{prop:m_0_construction_bijective}
The map $\Phi$ is a bijection. Hence
\[
G_{0,2}(2m,0) = |\mathcal{G}_{0,2}(2m,0)| = |A(2m,0)| = \binom{2m}{m}.
\]
\end{prop}
That is, given an arrow diagram $a \in A(2m,0)$, there is a unique (equivalence class of) arc diagram $C \in \mathcal{G}_{0,2}(2m,0)$ such that $\Phi (C) = a$.

The idea is that $C$ can be constructed from $a$ by starting at a point of $F$ (any point will do) and proceeding anticlockwise around the annulus. Each time we arrive at a point of $F$ labelled ``in", we start drawing a new arc, proceeding anticlockwise around the annulus. Each time we arrive at a point of $F$ labelled ``out", we end an arc there (if possible). This process produces a unique arc diagram. See figure \ref{fig:x6}.

\begin{figure}
\begin{multicols}{2}
\begin{center}
\begin{tikzpicture}[scale=0.7]
\def\pi{3.14159}
\def\oradius{30mm}
\def\iradius{6mm}
\def\ratio{0.85}

\clip(-35mm,-35mm) rectangle (35mm,35mm);

\coordinate (O) at (0,0);
\draw[fill=lightgray!10] (O) circle (\oradius);
\draw[fill=white] (O) circle (\iradius);

\draw[->,very thick,red] (0:\oradius) -- (0:\ratio*\oradius);
\draw[<-,very thick,red] (30:\oradius) -- (30:\ratio*\oradius);
\draw[<-,very thick,red] (60:\oradius) -- (60:\ratio*\oradius);
\draw[->,very thick,red] (90:\oradius) -- (90:\ratio*\oradius);
\draw[->,very thick,red] (120:\oradius) -- (120:\ratio*\oradius);
\draw[<-,very thick,red] (150:\oradius) -- (150:\ratio*\oradius);
\draw[->,very thick,red] (180:\oradius) -- (180:\ratio*\oradius);
\draw[->,very thick,red] (210:\oradius) -- (210:\ratio*\oradius);
\draw[->,very thick,red] (240:\oradius) -- (240:\ratio*\oradius);
\draw[<-,very thick,red] (270:\oradius) -- (270:\ratio*\oradius);
\draw[<-,very thick,red] (300:\oradius) -- (300:\ratio*\oradius);
\draw[<-,very thick,red] (330:\oradius) -- (330:\ratio*\oradius);
\end{tikzpicture}

\begin{tikzpicture}[scale=0.7]
\def\pi{3.14159}
\def\oradius{30mm}
\def\iradius{6mm}
\def\ratio{0.85}

\clip(-35mm,-35mm) rectangle (35mm,35mm);

\coordinate (O) at (0,0);
\draw[fill=lightgray!10] (O) circle (\oradius);
\draw[fill=white] (O) circle (\iradius);

\draw[->,very thick,red] (0:\oradius) -- (0:\ratio*\oradius);
\draw[<-,very thick,red] (30:\oradius) -- (30:\ratio*\oradius);
\draw[<-,very thick,red] (60:\oradius) -- (60:\ratio*\oradius);
\draw[->,very thick,red] (90:\oradius) -- (90:\ratio*\oradius);
\draw[->,very thick,red] (120:\oradius) -- (120:\ratio*\oradius);
\draw[<-,very thick,red] (150:\oradius) -- (150:\ratio*\oradius);
\draw[->,very thick,red] (180:\oradius) -- (180:\ratio*\oradius);
\draw[->,very thick,red] (210:\oradius) -- (210:\ratio*\oradius);
\draw[->,very thick,red] (240:\oradius) -- (240:\ratio*\oradius);
\draw[<-,very thick,red] (270:\oradius) -- (270:\ratio*\oradius);
\draw[<-,very thick,red] (300:\oradius) -- (300:\ratio*\oradius);
\draw[<-,very thick,red] (330:\oradius) -- (330:\ratio*\oradius);

\draw[red,thick] (120:\ratio*\oradius) to[out=300,in=330] (150:\ratio*\oradius);
\draw[red,thick] (180:\ratio*\oradius) to[out=0,in=165] (255:0.35*\oradius+0.65*\iradius);
\draw[red,thick] (90:\ratio*\oradius) to[out=270,in=75] (165:0.2*\oradius+0.8*\iradius) to[out=255,in=165] (255:0.2*\oradius+0.8*\iradius);
\end{tikzpicture}

\begin{tikzpicture}[scale=0.7]
\def\pi{3.14159}
\def\oradius{30mm}
\def\iradius{6mm}
\def\ratio{0.85}

\clip(-35mm,-35mm) rectangle (35mm,35mm);

\coordinate (O) at (0,0);
\draw[fill=lightgray!10] (O) circle (\oradius);
\draw[fill=white] (O) circle (\iradius);

\draw[->,very thick,red] (0:\oradius) -- (0:\ratio*\oradius);
\draw[<-,very thick,red] (30:\oradius) -- (30:\ratio*\oradius);
\draw[<-,very thick,red] (60:\oradius) -- (60:\ratio*\oradius);
\draw[->,very thick,red] (90:\oradius) -- (90:\ratio*\oradius);
\draw[->,very thick,red] (120:\oradius) -- (120:\ratio*\oradius);
\draw[<-,very thick,red] (150:\oradius) -- (150:\ratio*\oradius);
\draw[->,very thick,red] (180:\oradius) -- (180:\ratio*\oradius);
\draw[->,very thick,red] (210:\oradius) -- (210:\ratio*\oradius);
\draw[->,very thick,red] (240:\oradius) -- (240:\ratio*\oradius);
\draw[<-,very thick,red] (270:\oradius) -- (270:\ratio*\oradius);
\draw[<-,very thick,red] (300:\oradius) -- (300:\ratio*\oradius);
\draw[<-,very thick,red] (330:\oradius) -- (330:\ratio*\oradius);

\draw[red,thick] (120:\ratio*\oradius) to[out=300,in=330] (150:\ratio*\oradius);
\draw[red,thick] (240:\ratio*\oradius) to[out=60,in=165] (255:0.65*\oradius+0.35*\iradius) to[out=-15,in=90] (270:\ratio*\oradius);
\draw[red,thick] (210:\ratio*\oradius) to[out=30,in=165] (255:0.5*\oradius+0.5*\iradius) to[out=-15,in=120] (300:\ratio*\oradius);
\draw[red,thick] (180:\ratio*\oradius) to[out=0,in=165] (255:0.35*\oradius+0.65*\iradius);
\draw[red,thick] (90:\ratio*\oradius) to[out=270,in=75] (165:0.2*\oradius+0.8*\iradius) to[out=255,in=165] (255:0.2*\oradius+0.8*\iradius);
\end{tikzpicture}

\begin{tikzpicture}[scale=0.7]
\def\pi{3.14159}
\def\oradius{30mm}
\def\iradius{6mm}
\def\ratio{0.85}

\clip(-35mm,-35mm) rectangle (35mm,35mm);

\coordinate (O) at (0,0);
\draw[fill=lightgray!10] (O) circle (\oradius);
\draw[fill=white] (O) circle (\iradius);

\draw[->,very thick,red] (0:\oradius) -- (0:\ratio*\oradius);
\draw[<-,very thick,red] (30:\oradius) -- (30:\ratio*\oradius);
\draw[<-,very thick,red] (60:\oradius) -- (60:\ratio*\oradius);
\draw[->,very thick,red] (90:\oradius) -- (90:\ratio*\oradius);
\draw[->,very thick,red] (120:\oradius) -- (120:\ratio*\oradius);
\draw[<-,very thick,red] (150:\oradius) -- (150:\ratio*\oradius);
\draw[->,very thick,red] (180:\oradius) -- (180:\ratio*\oradius);
\draw[->,very thick,red] (210:\oradius) -- (210:\ratio*\oradius);
\draw[->,very thick,red] (240:\oradius) -- (240:\ratio*\oradius);
\draw[<-,very thick,red] (270:\oradius) -- (270:\ratio*\oradius);
\draw[<-,very thick,red] (300:\oradius) -- (300:\ratio*\oradius);
\draw[<-,very thick,red] (330:\oradius) -- (330:\ratio*\oradius);

\draw[red,thick] (120:\ratio*\oradius) to[out=300,in=330] (150:\ratio*\oradius);
\draw[red,thick] (240:\ratio*\oradius) to[out=60,in=165] (255:0.65*\oradius+0.35*\iradius) to[out=-15,in=90] (270:\ratio*\oradius);
\draw[red,thick] (210:\ratio*\oradius) to[out=30,in=165] (255:0.5*\oradius+0.5*\iradius) to[out=-15,in=120] (300:\ratio*\oradius);
\draw[red,thick] (180:\ratio*\oradius) to[out=0,in=165] (255:0.35*\oradius+0.65*\iradius) to[out=-15,in=150] (330:\ratio*\oradius);
\draw[red,thick] (90:\ratio*\oradius) to[out=270,in=75] (165:0.2*\oradius+0.8*\iradius) to[out=255,in=165] (255:0.2*\oradius+0.8*\iradius) to[out=-15,in=255] (-15:0.2*\oradius+0.8*\iradius) to[out=75,in=240] (60:\ratio*\oradius);
\draw[red,thick] (0:\ratio*\oradius) to[out=180,in=210] (30:\ratio*\oradius);
\end{tikzpicture}
\end{center}
\end{multicols}
\caption{Constructing an arc diagram from an arrow diagram.}
\label{fig:x6}
\end{figure}
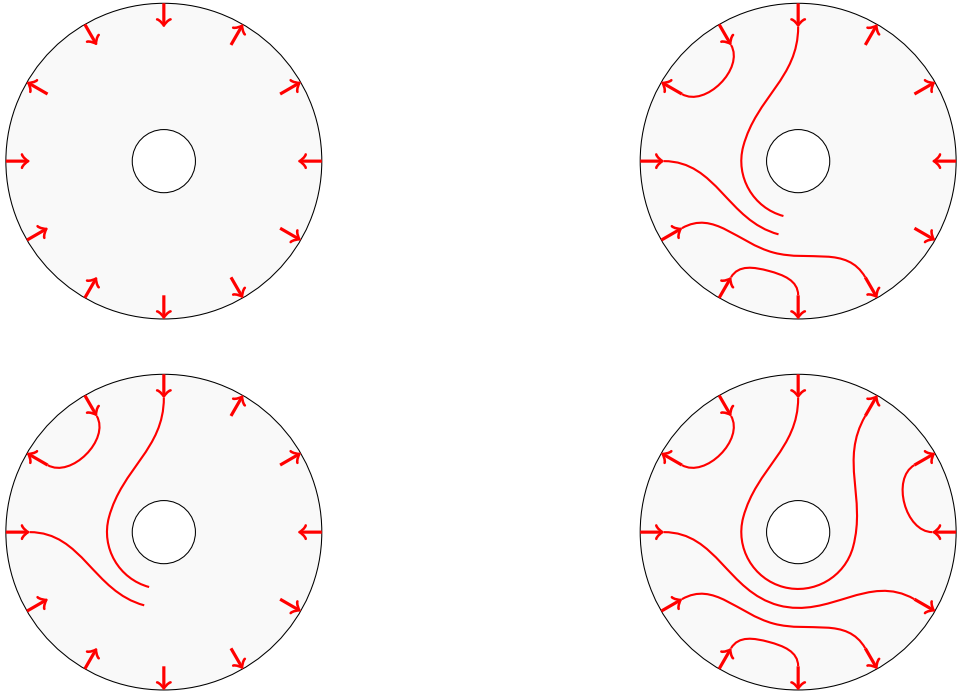

\begin{proof}
Proof by induction on $m$. When $m=0$ there is nothing to prove. When $m=1$ the construction is clear: draw an arc anticlockwise from the ``in" to the ``out" point of $F$. This is clearly unique up to equivalence of arc diagrams.

For general $m$, note that in the arrow diagram $a$, as we proceed anticlockwise around the $2m$ boundary points of $F$, there must be at least one point $f_{in}$ labelled ``in" followed immediately by another point $f_{out}$ labelled ``out". Any (equivalence class of) arc diagram $C$ such that $\Phi(C) = a$ must contain a ``short" boundary-parallel arc $\gamma$ anticlockwise from $f_{in}$ to $f_{out}$. The remaining $2m-2$ points of $a$ form an arrow diagram $a' \in A(2m-2,0)$. By induction there exists a unique arc diagram $C'$ with $\Phi(C')=a'$; taking $C'$ together with $\gamma$ gives an arc diagram $C$ with $\Phi(c)=a$. Moreover, since $\gamma$ must be included, the uniqueness of $C'$ implies uniqueness of $C$, and the result holds.
\end{proof}

The idea of the proof above appears in Przytycki \cite{Przytycki99_fundamentals} and is due to him, so far as we know. This argument is then used to give a formula for the Catalan numbers, as we show now.

\begin{prop}[Przytycki]
\label{prop:Prz_Catalan}
For any integer $m \geq 0$,
\[
G_{0,1}(2m) = \frac{1}{m+1} \binom{2m}{m} = C_m.
\]
\end{prop}

\begin{proof}
Consider the annulus $(S,F) = (S_{0,2}, F(2m,0))$ and the disc $(D, F(2m))$. Given an arc diagram on $(S, F)$, we can glue a disc to the boundary component $B_2$ to obtain an arc diagram on $(D, F(2m))$. This gives a map $\mathcal{G}_{0,2}(2m,0) \To \mathcal{G}_{0,1}(2m)$.

Conversely, given an arc diagram $C$ on $(D, F(2m))$, we can remove a small disc $D'$ from the interior of $D$, not intersecting any arcs, and obtain an arc diagram on $(S, F(2m,0))$. There are $m+1$ complementary regions of the $m$ arcs of $C$, and removing $D'$ from these distinct regions produces $m+1$ distinct 
arc diagrams on $(S, F(2m,0))$. These are precisely the arc diagrams on annuli for which filling in a disc gives $C$. Hence the map $\mathcal{G}_{0,2}(2m,0) \To \mathcal{G}_{0,1}(2m)$ is $(m+1)$-to-$1$, and
\[
G_{0,1}(m) = | \mathcal{G}_{0,1}(m) | = \frac{|\mathcal{G}_{0,2}(m,0)|}{m+1} = \frac{G_{0,2}(m,0)}{m+1} = \frac{1}{m+1} \binom{2m}{m} = C_m. \qedhere
\]
\end{proof}
This also proves theorem \ref{thm:formulas}(\ref{eqn:G01}).

\begin{proof}[Proof of \ref{prop:insular_count}]
Clearly insular diagrams must have an even number of boundary points on each boundary component, so $I(2m_1 + 1, 2m_2 + 1) = 0$.

Now an insular diagram in $\mathcal{G}_{0,2}(2m_1, 2m_2)$ can be cut along a core curve into two insular arc diagrams, on $(S, F(2m_1, 0))$ and $(S, F(2m_2, 0))$ respectively. This gives a bijection between insular arc diagrams on $(S_{0,2}, F(2m_1, 2m_2))$ and $\mathcal{G}_{0,2}(2m_1,0) \times \mathcal{G}_{0,2}(2m_2, 0)$, so 
\[
I (2m_1, 2m_2) = G_{0,2}(2m_1, 0) G_{0,2}(2m_2, 0) = \binom{2m_1}{m_1} \binom{2m_2}{m_2}. \qedhere
\]
\end{proof}

\subsection{Traversing diagrams}
\label{sec:traversing_diagrams}

We now turn to traversing diagrams. 
To fix notation, henceforth we draw annuli $(S_{0,2}, F(b_1, b_2))$ in the plane with $B_1$ as ``outer" and $B_2$ as ``inner" boundary. 

Note, if an arc diagram $C$ exists on $(S,F)$ then $b_1 \equiv b_2 \pmod{2}$, and each insular arc connects two points on the same boundary component, so after all insular arcs are drawn, the number of points remaining on boundary component $B_i$, not yet connected to other points by arcs, has the same parity as $b_i$. 
Hence we have the following.
\begin{lem}
\label{lem:traversing_arcs_unoriented}
Let $C$ be a traversing arc diagram on $(S,F(b_1,b_2))$. Then the number of traversing arcs in $C$ has the same parity as $b_1$ and $b_2$. 
\qed
\end{lem}

Our computation of $T(b_1, b_2)$ involves bijections between certain combinatorial sets and certain decorations on arc diagrams.
\begin{defn}
A \emph{decorated arc diagram} on $(S, F)$ is a pair $(C,R)$ where $C$ is an arc diagram on $(S,F)$, and $R$ is a complementary region of $C$.

The set of equivalence classes of decorated traversing arc diagrams on $(S,F) = (S_{0,2}, F(b_1, b_2))$ is denoted $DT(b_1, b_2)$.
\end{defn}

Now 
by lemma \ref{lem:traversing_complementary_regions} above, a traversing arc diagram has $\frac{1}{2}(b_1 + b_2)$ complementary regions. It follows that
\[
\left| DT(b_1, b_2) \right| = \frac{1}{2} (b_1 + b_2) \; T(b_1, b_2),
\]
so to find $T(b_1, b_2)$ it suffices to count $DT(b_1, b_2)$.

We will count $DT(b_1, b_2)$ by finding a bijection with certain sets of arrow diagrams with extra decorations, which we call \emph{special arrow diagrams}. 

We will deal with the even case $(b_1, b_2) = (2m_1, 2m_2)$ and the odd case $(b_1, b_2) = (2m_1 + 1, 2m_2 + 1)$ separately. First we consider the even case. 
\begin{defn}
A \emph{special arrow diagram} on $(S, F(2m_1, 2m_2))$ 
 is an arrow diagram  together with a choice of one inward arrow on $B_2$ (the \emph{special inward arrow}), and an outward arrow on $B_1$ (the \emph{special outward arrow}). The set of such arrow diagrams is denoted $SA(m_1, m_2)$.
\end{defn}
We have seen there are $\binom{2m_1}{m_1} \binom{2m_2}{m_2}$ arrow diagrams on $(S,F)$; hence, with the additional choices of special arrows, we have $|SA(2m_1, 2m_2)| = m_1 m_2 \binom{2m_1}{m_1} \binom{2m_2}{m_2}$.

Next, we define a map $\Psi: SA(2m_1, 2m_2) \To DT(2m_1, 2m_2)$ as follows. Let $a$ be a special arrow diagram with special inward arrow $i$ on the inner boundary $B_2$ and special outward arrow $o$ on the outer boundary $B_1$. 
\begin{enumerate}
\item
Join special arrows $i$ to $o$ by an oriented arc $\gamma$. (There are many choices for $\gamma$, but they are all related by homeomorphisms of $S$ fixing the boundary, hence lead to equivalent arc diagrams.) 
\item
Cut $S$ along $\gamma$ to obtain a disc $D$. The boundary of $D$, starting from $o$ and proceeding anticlockwise around $B_1$, consists of $B_1$ (traversed anticlockwise), followed by $\gamma$ (traversed backwards), followed by $B_2$ (traversed clockwise), followed by $\gamma$ (traversed forwards). The remaining (unconnected, non-special) arrows on $(S,F)$  provide $D$ with $2(m_1 + m_2 - 1)$ marked boundary points, half labelled ``in" and half labelled ``out".
\item 
Choose a point $p$ in the interior of $D$, and remove a small neighbourhood of $p$. We then have an annulus with $2(m_1 + m_2 - 1)$ marked boundary points on one boundary component, half ``in" and half ``out", and no points on the other boundary component. That is, we have an arrow diagram in $A(2m_1 + 2m_2 - 2, 0)$.
\item
By the bijection $\Phi$ of section \ref{sec:arrow_diagrams}, we obtain a unique (equivalence class of) arc diagram $\widetilde{C}$ on this annulus. This $\widetilde{C}$ has the property that if its arcs are oriented anticlockwise around the annulus, then the orientations agree with the arrows at boundary points.
\item
Gluing back the neighbourhood of $p$ which was previously removed gives the disc $D$, which now has an arc diagram $\overline{C}$. The point $p$ and its removed-and-returned neighbourhood now lie in a complementary region $\overline{R}$ of $\overline{C}$.
\item
Recall that $\partial D$ contains two copies of the oriented arc $\gamma$, along which we originally cut. We now glue these two copies of $\gamma$ back together, reconstructing the original annulus $S$. Combining $\overline{C}$ and $\gamma$ gives an (equivalence class of) arc diagram $C$ on $(S,F)$, and the complementary region $\overline{R}$ of $\overline{C}$ becomes a complementary region $R$ of $C$. We define $\Psi(a) = (C,R)$.
\end{enumerate}
At each step, all choices are unique up to homeomorphisms of the surface fixing the boundary. Thus the equivalence class of the arc diagram $C$, and the region $R$, are well-defined; so $\Psi$ is well-defined.

This construction is perhaps more natural than it seems. The special arrows show us how to cut the annulus into a disc. The remaining arrows around a disc show us how to draw the remaining arcs. But arrows around the boundary are equivalent to an arc diagram on an \emph{annulus} obtained by removing a small sub-disc. This is equivalent to an arc diagram on the annulus together with a choice of region.

Suppose $a$ is special arrow diagram, and $\Psi(a) = (C,R)$. The arc diagram $C$ consists of $m_1 + m_2$ arcs, which are all given an orientation in the construction, agreeing with the arrows in $a$. From lemma \ref{lem:traversing_arcs_unoriented} the number of traversing arcs is even; let this number be $2k$.

We claim that $k$ traversing arcs run ``inward" from $B_1$ to $B_2$, and $k$ traversing arcs run ``outward" from $B_2$ to $B_1$. To see this, note that the arrow diagram $a$ consists of $m_1$ inward and $m_1$ outward arrows on $B_1$, and insular arcs connect some of the inward to outward arrows in pairs. Thus, the remaining $2k$ arrows, which are the endpoints on $B_1$ of traversing arcs, contain the same number $k$ of inward and outward arrows.

Let $C_t$ denote the oriented arc diagram $C$ with insular arcs removed. So $C_t$ consists of $k$ inward and $k$ outward traversing arcs. These arcs cut the annulus $S$ into $2k$ complementary disc regions, and we may speak of proceeding clockwise or anticlockwise around the annulus from one traversing arc to the next, or from one region to the next. One of the traversing arcs is the arc $\gamma$ connecting the special arrows; by construction $\gamma$ points outward. And one of the regions $\tilde{R}$ of $C_t$ contains the region $R$ and the point $p$ in the construction. (As $C_t$ is obtained from $C$ by removing arcs, the complementary regions of $C$ are subsets of the complementary regions of $C_t$.) See figure \ref{fig:x7}.

\begin{lem}
\label{lem:traversing_arc_arrangement_even_case}
Starting from $\gamma$ and proceeding anticlockwise, the first $k$ traversing arcs of $C_t$ (including $\gamma$) are oriented outward; then we pass through the region $\tilde{R}$; and the final $k$ traversing arcs of $C_t$ are oriented inward.
\end{lem}

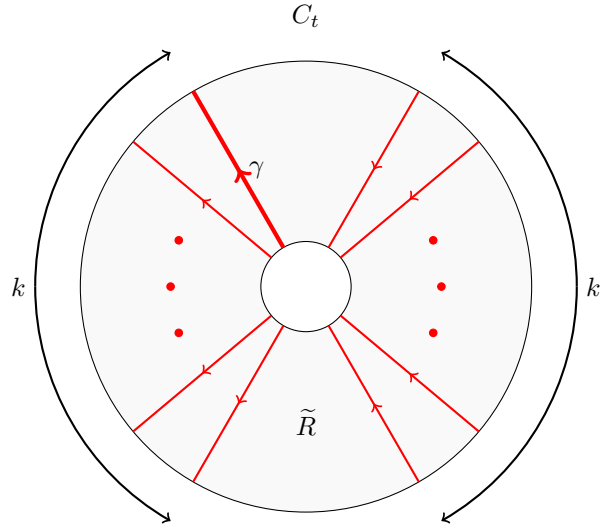
\begin{figure}
\begin{center}
\begin{tikzpicture}
\def\oradius{30mm}
\def\iradius{6mm}

\clip(-40mm,-40mm) rectangle (40mm,40mm);

\coordinate (O) at (0,0);
\draw[fill=lightgray!10] (O) circle (\oradius);
\draw[fill=white] (O) circle (\iradius);

\begin{scope}[thick,decoration={markings,mark=at position 0.5 with {\arrow{>}}}] 
	\draw[postaction={decorate},red] (60:\oradius) -- (60:\iradius);
	\draw[postaction={decorate},red] (40:\oradius) -- (40:\iradius);
	\draw[postaction={decorate},red] (-60:\oradius) -- (-60:\iradius);
	\draw[postaction={decorate},red] (-40:\oradius) -- (-40:\iradius);
	\draw[postaction={decorate},ultra thick,red] (120:\iradius) -- node[right,black] {$\gamma$} (120:\oradius);
	\draw[postaction={decorate},red] (140:\iradius) -- (140:\oradius);
	\draw[postaction={decorate},red] (220:\iradius) -- (220:\oradius);
	\draw[postaction={decorate},red] (240:\iradius) -- (240:\oradius);
	\draw[->,thick] (1.2*\oradius,0) arc (0:60:1.2*\oradius);
	\draw[->,thick] (1.2*\oradius,0) arc (0:-60:1.2*\oradius);
	\draw[->,thick] (-1.2*\oradius,0) arc (180:240:1.2*\oradius);
	\draw[->,thick] (-1.2*\oradius,0) arc (180:120:1.2*\oradius);
\end{scope}
\draw[red,fill=red] (20:0.5*\oradius+0.5*\iradius) circle (0.5mm);
\draw[red,fill=red] (0:0.5*\oradius+0.5*\iradius) circle (0.5mm);
\draw[red,fill=red] (-20:0.5*\oradius+0.5*\iradius) circle (0.5mm);
\draw[red,fill=red] (160:0.5*\oradius+0.5*\iradius) circle (0.5mm);
\draw[red,fill=red] (180:0.5*\oradius+0.5*\iradius) circle (0.5mm);
\draw[red,fill=red] (200:0.5*\oradius+0.5*\iradius) circle (0.5mm);
\node at (90:1.2*\oradius) {$C_t$};
\node at (0,-0.5*\oradius-0.5*\iradius) {$\widetilde{R}$};
\node[right] at (1.2*\oradius,0) {$k$};
\node[left] at (-1.2*\oradius,0) {$k$};
\end{tikzpicture}
\caption{The arc diagram $C_t$ in the case $(b_1, b_2) = (2m_1, 2m_2)$.} 
\label{fig:x7}
\end{center}
\end{figure}

\begin{proof}
Recall that in the construction of $C$, we first draw $\gamma$, oriented outward; then we cut along $\gamma$, remove a neighbourhood of $p$, and construct an oriented arc diagram on the resulting annulus. These arcs are oriented so as to run anticlockwise around this annulus; that is, they run anticlockwise around the point $p$. Hence, once the traversing arcs are constructed, we see that those traversing arcs that lie anticlockwise of $\tilde{R}$ and clockwise of $\gamma$ must be oriented inward. Similarly, the traversing arcs that lie clockwise of $\tilde{R}$ and anticlockwise of $\gamma$ must be oriented outward. Since there are $k$ inward and $k$ outward traversing arcs, they must be arranged as claimed.
\end{proof}

In particular, proceeding clockwise through $C_t$ from $\tilde{R}$, the arc $\gamma$ is the $k$'th traversing arc encountered. (Similarly, proceeding anticlockwise through $C_t$ from $\tilde{R}$, the arc $\gamma$ is the $(k+1)$'th traversing arc encountered.)

\begin{prop}
\label{prop:traversing_bijection}
Given $(C,R) \in DT(2m_1, 2m_2)$, there is a unique special arrow diagram $a \in SA(2m_1, 2m_2)$ such that $\Psi(a) = (C,R)$.
\end{prop}

\begin{proof}
Recall that $C$ is an (equivalence class of) traversing arc diagram on $(S,F)$, and $R$ is a complementary region of $C$.

Let $C_t$ be the arc diagram obtained by removing all insular arcs from $C$; as $C$ is traversing, $C_t$ is nonempty, with $2k>0$ arcs. The complementary regions of $C$ are subsets of the complementary regions of $C_t$; denote the complementary region of $C_t$ containing $R$ as $\widetilde{R}$.

Proceed clockwise through $C_t$ from $\widetilde{R}$; denote the $k$'th traversing arc encountered as $\gamma$, orient it outward, and draw a special inward and outward arrow at its endpoints. By the preceding remark, if $a$ is an arrow diagram such that $\Psi(a) = (C,R)$, then the special arrows must be located at these points.

Now return to the original diagram $C$, cut along $\gamma$, and remove a small neighbourhood of some point $p \in R$. Then we have an annulus $(S', F(2m_1 + 2m_2 - 2, 0))$, containing an arc diagram $C'$. By proposition \ref{prop:m_0_construction_bijective}, there exists a unique arrow diagram $a'$ on $(S',F(2m_1+2m_2-2,0))$ such that $\Phi(a') = C'$.

We now take the special arrow diagram $a$ to consist of the arrows of $a'$, together with the special arrows constructed above. By construction, applying $\Psi$ to $a$ first reconstructs $\gamma$; then cuts along $\gamma$ and removes a point $p$; then reconstructs the arc diagram $C'$ on $(S', F(2m_1 + 2m_2 - 2,0))$; and finally fills in the hole and selects the region containing the filled-in hole. This region is $R$, since the construction removes a point from $R$ to create the annulus $S'$. Thus $\Psi(a) = (C,R)$. 

To show uniqueness, suppose we have a special arrow diagram $\widetilde{a}$ satisfying $\Psi(\widetilde{a}) = (C,R)$; we will show $\widetilde{a}= a$. This $\widetilde{a}$ must first contain the special arrows of $a$, as remarked above. Cutting along the arc $\gamma$ between them, and removing a neighbourhood of a point, we obtain an arrow diagram $\widetilde{a}'$ on $(S', F(2m_1 + 2m_2 - 2,0))$. Now applying $\Phi$ to $\widetilde{a}'$ produces an arc diagram on $S'$ such that, after filling in the hole and labelling the region $R$ and re-gluing along $\gamma$, we obtain $(C,R)$. Thus, $\Phi(\widetilde{a}') = \Phi(a')$. By proposition \ref{prop:m_0_construction_bijective}, $\Phi$ is bijective, so $\widetilde{a}' = a'$, and combined with the special arrows, which agree, we have $\widetilde{a} = a$.
\end{proof}

We have now shown $\Psi$ is bijective, so
\[
|DT(2m_1, 2m_2)| = |SA(2m_1, 2m_2)| = m_1 m_2 \binom{2m_1}{m_2} \binom{2m_2}{m_2}
\]
and since above we showed $|DT(b_1, b_2)| = \frac{b_1 + b_2}{2} T(b_1, b_2)$, we have
\[
T(2m_1, 2m_2) = \frac{|DT(2m_1, 2m_2)| }{m_1 + m_2} = \frac{m_1 m_2}{m_1 + m_2} \binom{2m_1}{m_1} \binom{2m_2}{m_2}
\]
proving proposition \ref{prop:traversing_count} in the even case.

The argument in the odd case is similar, but with slightly different diagrams. 
\begin{defn}
A \emph{special arrow diagram} on $(S, F(2m_1 + 1, 2m_2 + 1))$ 
 consists of a triple $(f_1, f_2, a)$, where $f_i \in F \cap B_i$ is an \emph{exceptional marked point} on each boundary component, and $a$ is an arrow diagram on $(S, F \setminus \{f_1, f_2\})$.

The set of special arrow diagrams on $(S,F)$ is denoted $SA(2m_1 + 1, 2m_2 + 1)$.
\end{defn}
After removing $f_1, f_2$, each boundary component contains an even number of marked points, so that an arrow diagram exists. There are $2m_1 + 1$ choices for $f_1$, $2m_2 + 1$ choices for $f_2$, and $\binom{2m_i}{m_i}$ choices for the arrows on boundary component $B_i$. Thus
\[
\left| SA(2m_1 + 1, 2m_2 + 1) \right| = (2m_1 + 1) (2m_2 + 1) \binom{2m_1}{m_1} \binom{2m_2}{m_2}.
\]

We now define a map $SA(2m_1 + 1, 2m_2 + 1) \To DT(2m_1 + 1, 2m_2 + 1)$, which we will show to be a bijection. The definition is similar to the map $SA(2m_1,. 2m_2) \To DT(2m_1, 2m_2)$. We call both maps $\Psi$, so that we will have bijections $\Psi: SA(b_1, b_2) \To DT(b_1,b_2)$ for all $b_1, b_2$. Let $(f_1, f_2, a) \in SA(2m_1 + 1, 2m_2 + 1)$.
\begin{enumerate}
\item
Join the exceptional points $f_1$ and $f_2$ by a traversing arc $\gamma$. (There are many choices for $\gamma$, but they are all related by homeomorphisms of $S$ fixing the boundary.)
\item
Cut along $\gamma$ to obtain a disc $D$. The arrow diagram $a$ gives an arrow diagram on $D$ with $2m_1 + 2m_2$ arrows, half in and half out.
\item
Choose a point $p$ in the interior of $D$ and remove a small neighbourhood of $p$. We then have an arrow diagram in 
 $A(2m_1 + 2m_2, 0)$.
\item
Using the bijection $\Phi$ of section \ref{sec:arrow_diagrams}, we obtain a unique arc diagram $\widetilde{C}$ on this annulus; if the arcs of $\widetilde{C}$ are oriented anticlockwise around the annulus, then the orientations agree with the arrows.
\item
Glue back the neighbourhood of $p$, which now lies in a complementary region $\widetilde{R}$ of the arc diagram $\overline{C}$ on $D$.
\item
Glue the two copies of $\gamma$ on $\partial D$ back together to reconstruct the original annulus. Combining $\overline{C}$ and $\gamma$ gives an arc diagram $C$ on $(S,F)$, and the complementary region $\widetilde{R}$ of $\overline{C}$ becomes a complementary region $R$ of $C$. We define $\Psi(f_1, f_2, a) = (C,R)$.
\end{enumerate}
As in the even case, the construction at each stage is unique up to equivalence, so $\Psi$ is well defined.

The arc diagram $\widetilde{C}$ in this construction can be regarded as an oriented arc diagram; the arcs are oriented so as to agree with the arrows. Hence, the arc diagram $C$ resulting from the construction can be regarded as having one ``exceptional" arc $\gamma$, and all other arcs oriented.

Now consider the traversing arcs of $C$ on $(S,F)$. By lemma \ref{lem:traversing_arcs_unoriented}, the number of traversing arcs is odd; let the  number be $2k+1$. One of these is the exceptional arc $\gamma$, which we leave unoriented; the other $2k$ traversing arcs are oriented. As in the even case, exactly $k$ of the oriented traversing arcs run inward, and $k$ run outward. 

Considering $C_t$, the arc diagram with insular arcs removed, we have $2k+1$ traversing arcs, consisting of the exceptional arc $\gamma$, together with $k$ inward arcs and $k$ outward arcs. These cut the annulus $S$ into $2k+1$ complementary regions, which are naturally in a cyclic order, so we can proceed clockwise or anticlockwise through them. One of these regions $\widetilde{R}$ contains the decorated region $R$ from the construction.

Again, just as in the even case, starting from $\gamma$ and proceeding anticlockwise, the first $k$ traversing arcs of $C_t$ after $\gamma$ are oriented outward; then we pass through the region $\widetilde{R}$; then the final $k$ traversing arcs of $C_t$ are oriented inward. For in the construction of $C$, after cutting along $\gamma$ and removing a neighbourhood of $p$, we construct an oriented arc diagram on the resulting annulus where arcs run anticlockwise. So those traversing arcs in $S$ which are anticlockwise of $\widetilde{R}$ and clockwise of $\gamma$ are oriented inward, and those clockwise of $\widetilde{R}$ and anticlockwise of $\gamma$ are oriented outward. See figure \ref{fig:x10}.

\begin{figure}
\begin{center}
\begin{tikzpicture}
\def\oradius{30mm}
\def\iradius{6mm}

\clip(-40mm,-40mm) rectangle (40mm,40mm);

\coordinate (O) at (0,0);
\draw[fill=lightgray!10] (O) circle (\oradius);
\draw[fill=white] (O) circle (\iradius);

\begin{scope}[thick,decoration={markings,mark=at position 0.5 with {\arrow{>}}}] 
	\draw[postaction={decorate},red] (60:\oradius) -- (60:\iradius);
	\draw[postaction={decorate},red] (40:\oradius) -- (40:\iradius);
	\draw[postaction={decorate},red] (-60:\oradius) -- (-60:\iradius);
	\draw[postaction={decorate},red] (-40:\oradius) -- (-40:\iradius);
	\draw[postaction={decorate},red] (120:\iradius) -- (120:\oradius);
	\draw[postaction={decorate},red] (140:\iradius) -- (140:\oradius);
	\draw[postaction={decorate},red] (220:\iradius) -- (220:\oradius);
	\draw[postaction={decorate},red] (240:\iradius) -- (240:\oradius);
	\draw[->,thick] (1.2*\oradius,0) arc (0:60:1.2*\oradius);
	\draw[->,thick] (1.2*\oradius,0) arc (0:-60:1.2*\oradius);
	\draw[->,thick] (-1.2*\oradius,0) arc (180:240:1.2*\oradius);
	\draw[->,thick] (-1.2*\oradius,0) arc (180:120:1.2*\oradius);
\end{scope}
\draw[red,ultra thick] (90:\iradius) -- node[right,black] {$\gamma$} (90:\oradius);
\draw[red,fill=red] (20:0.5*\oradius+0.5*\iradius) circle (0.5mm);
\draw[red,fill=red] (0:0.5*\oradius+0.5*\iradius) circle (0.5mm);
\draw[red,fill=red] (-20:0.5*\oradius+0.5*\iradius) circle (0.5mm);
\draw[red,fill=red] (160:0.5*\oradius+0.5*\iradius) circle (0.5mm);
\draw[red,fill=red] (180:0.5*\oradius+0.5*\iradius) circle (0.5mm);
\draw[red,fill=red] (200:0.5*\oradius+0.5*\iradius) circle (0.5mm);
\node at (90:1.2*\oradius) {$C_t$};
\node at (0,-0.5*\oradius-0.5*\iradius) {$\widetilde{R}$};
\node[right] at (1.2*\oradius,0) {$k$};
\node[left] at (-1.2*\oradius,0) {$k$};
\end{tikzpicture}
\caption{The arc diagram $C_t$ in the case $(b_1, b_2) = (2m_1+1, 2m_2+1)$.} \label{fig:x10}
\end{center}
\end{figure}
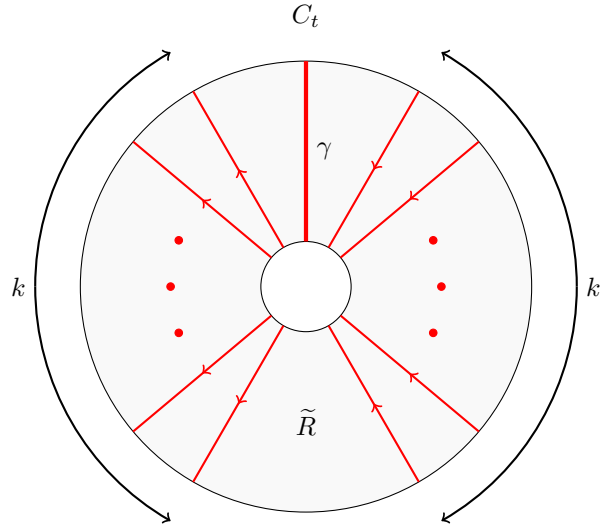

Hence, proceeding clockwise through $C_t$ from $\widetilde{R}$, the arc $\gamma$ is the $(k+1)$'th traversing arc encountered. (Similarly, proceeding anticlockwise through $C_t$ from $\widetilde{R}$, the arc $\gamma$ is the $(k+1)$'th traversing arc encountered.)

We can now prove $\Psi$ is bijective; again the proof is similar to the even case.
\begin{prop}
Given a pair $(C,R) \in DT(2m_1 + 1, 2m_2 + 1)$, there is a unique special arrow diagram $(f_1, f_2, a) \in SA(2m_1 + 1, 2m_2 + 1)$ such that $\Psi(f_1, f_2, a) = (C,R)$.
\end{prop}

\begin{proof}
First consider $C_t$, the arc diagram obtained from $C$ by removing insular arcs. Let $C_t$ contain $2k+1>0$ arcs, and let the complementary region of $C_t$ containing $R$ be $\widetilde{R}$. Proceed clockwise through $C_t$ from $\widetilde{R}$; let the $(k+1)$'th traversing arc encountered be $\gamma$. Let its endpoints on $B_1$ and $B_2$ be $f_1$ and $f_2$ respectively.

Now return to the original diagram $C$, cut along $\gamma$, and remove a small neighbourhood of a point $p \in R$. Then we have an annulus $S'$ with boundary conditions $F(2m_1 + 2m_2, 0)$ and an arc diagram $C'$. By proposition \ref{prop:m_0_construction_bijective}, there exists a unique arrow diagram $a'$ on $(S', F(2m_1 + 2m_2, 0))$ such that $\Phi(a') = C'$.

After gluing back along $\gamma$, the arrow diagram $a'$ gives an arrow diagram $a$ on the original annulus with the points $f_1, f_2$ removed. We take $(f_1, f_2, a)$ as our special arrow diagram.

We claim that $\Psi(f_1, f_2, a) = (C,R)$. First we connect $f_1$ to $f_2$, constructing $\gamma$, up to equivalence. Then we cut along $\gamma$ and remove a point $p$; then we  reconstruct $C'$ on $(S',F(2m_1 + 2m_2,0))$; and finally we fill the hole, select the region containing the filled-in hole, and glue back together along $\gamma$. The diagram obtained is $C$, and the region is $R$, since the construction removes a point from $R$ to create the annulus $S'$. Thus $\Psi(f_1, f_2, a) = (C,R)$.

To show uniqueness, suppose we have an exceptional arrow diagram $(\widetilde{f}_1, \widetilde{f}_2, \widetilde{a})$ satisfying $\Psi(\widetilde{f}_1, \widetilde{f}_2, \widetilde{a}) = (C,R)$. This $\widetilde{a}$ must first have the same exceptional points as constructed, and the same arc $\gamma$ (up to equivalence). Cutting along $\gamma$ and removing a neighbourhood of a point, we obtain an arrow diagram $\widetilde{a}'$ on $(S', F(2m_1 + 2m_2, 0))$; applying $\Phi$ to $\widetilde{a}'$ produces the same arc diagram as $a'$, so by bijectivity of $\Phi$ we have $\widetilde{a}' = a'$, and hence $\widetilde{a} = a$.
\end{proof}

We have now shown $\Psi$ is a bijection $SA(2m_1 + 1, 2m_2 + 1) \To DT(2m_1 + 1, 2m_2 + 1)$. Comparing the sizes of these sets, we have $(2m_1 + 1) (2m_2 + 1) \binom{2m_1}{m_1} \binom{2m_2}{m_2} = (m_1 + m_2 + 1) T(2m_1 + 1, 2m_2 + 1)$, and we conclude
\[
T(2m_1 + 1, 2m_2 + 1) = \frac{(2m_1 + 1)(2m_2 + 1)}{m_1 + m_2 + 1} \binom{2m_1}{m_1} \binom{2m_2}{m_2}.
\]
This proves the second and final half of proposition \ref{prop:traversing_count}.

Putting together our counts of insular and traversing arc diagrams, we can compute $G_{0,2}(b_1, b_2)$ as $I(b_1, b_2) + T(b_1, b_2)$. We have now proved theorem \ref{thm:formulas}(\ref{eqn:G02ee})--(\ref{eqn:G02oo}).

\subsection{Non-boundary-parallel diagrams}
\label{sec:non_boundary_parallel_small_cases}

We now compute $N_{0,1}(b_1)$ and $N_{0,2}(b_1, b_2)$.

On a disc, every arc in an arc diagram is boundary-parallel, hence the only arc diagram without boundary-parallel curves is the empty arc diagram. Thus $N_{0,1}(0) = 1$, and all other $N_{0,1}(b) = 0$.

On an annulus, if there are no boundary-parallel arcs then every arc must be traversing. It follows that $b_1 = b_2= b$, and once one arc is drawn the others are determined up to equivalence. If $b > 0$ then this gives $b$ equivalence classes of arc diagrams; if $b=0$ then there is one equivalence class, namely that of the empty diagram.

Thus, we obtain the following lemma. Recall the notation $\bar{b}$ from definition \ref{defn:bar}.
\begin{lem}
\label{lem:Ngn_small_cases}
For any integer $b \geq 0$,
\begin{align*}
N_{0,1}(0) &= 1 \\
N_{0,2}(b, b) &= \bar{b}.
\end{align*}
All other $N_{0,1}(b_1)$ and $N_{0,2}(b_1, b_2)$ are zero.
\qed
\end{lem}

This establishes equations (\ref{eqn:N01})--(\ref{eqn:N02}) in theorem  \ref{thm:N_formulas}.

\section{Decomposing arc diagrams}
\label{sec:decomposing_arc_diagrams}

\subsection{Canonical decomposition}
\label{sec:canonical_decomp}

We now show how to decompose an arc diagram $C$ on $S=S_{g,n}$ into arc diagrams on annular neighbourhoods $A_1, \ldots, A_n$ of the boundary components $B_1, \ldots, B_n$ (``local" to the boundary components), together with an arc diagram on the remaining surface $S' = S \setminus \left( \bigcup_{i=1}^n A_i \right)$ (the ``core"). The annuli $A_i$ will contain all curves boundary-parallel to $B_i$ (``local" to $B_i$). The arc diagram on the ``core" $S'$ contains no boundary-parallel arcs. We denote the boundary component of $A_i$ other than $B_i$ as $B'_i$, and the arc diagram on $A_i$ by $C_i$. Each arc of $C_i$ will be boundary-parallel to $B_i$, or traversing; we call the traversing arcs ``legs". In particular, all arcs intersecting $B'_i$ are traversing. This leads to the following definition.
\begin{defn}
\label{defn:local_diagram}
Let $S_{0,2}$ be an annulus with boundary components $B,B'$. An arc diagram $C$ on $S_{0,2}$ is \emph{$B$-local} if every arc of $C$  intersecting $B'$ is traversing.

If $F=F(b,b')$ consists of $b$ points on $B$ and $b'$ points on $B'$, then a $B$-local arc diagram on $(S_{0,2},F)$ is called \emph{$b$-local with $b'$ legs}.
\end{defn}

Note that in a $b$-local arc diagram with $b'$ legs, we must have $b' \leq b$ and $b \equiv b' \pmod{2}$.
\begin{defn}
The set of equivalence classes of $b$-local arc diagrams with $b'$ legs on $(S_{0,2}, F(b,b'))$ is denoted $L(b,b')$.
\end{defn}

\begin{defn}
\label{defn:local_decomposition}
Let $S=S_{g,n}$ have boundary components $B_1, \ldots, B_n$ and let $C$ be an arc diagram on $(S, F(b_1, \ldots, b_n))$. A \emph{local decomposition} of $C$ consists of a set of simple closed curves $B'_1, \ldots, B'_n$ on $S$, such that the following conditions hold.
\begin{enumerate}
\item
Cutting $S$ along $\bigcup_{i=1}^n B'_i$ produces a collection of annuli $A_1, \ldots, A_n$, where each annulus $A_i$ has boundary $\partial A_i = B_i \cup B'_i$, and a surface $S'$ (the \emph{core}) homeomorphic to $S$.
\item
The restriction of the arc diagram $C$ to each annulus $A_i$ is $B_i$-local.
\item
The restriction of the arc diagram $C$ to $S'$ contains no boundary-parallel arcs.
\end{enumerate}
\end{defn}

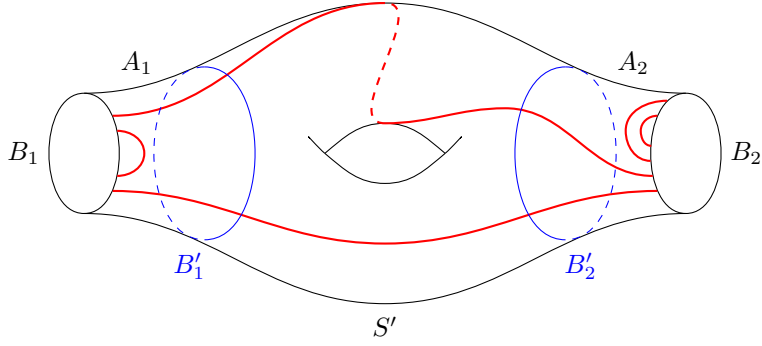
\begin{figure}
\begin{center}
\begin{tikzpicture}
\draw (-40mm,8mm) to[out=0,in=0] (-40mm,-8mm);
\draw (-40mm,8mm) to[out=180,in=180] node[left] {$B_1$} (-40mm,-8mm);

\draw (40mm,8mm) to[out=0,in=0] node[right] {$B_2$} (40mm,-8mm);
\draw (40mm,8mm) to[out=180,in=180] (40mm,-8mm);

\draw (-40mm,8mm) to[out=0,in=180] (0mm,20mm) to[out=0,in=180] (40mm,8mm);
\draw (-40mm,-8mm) to[out=0,in=180] (0mm,-20mm) to[out=0,in=180] (40mm,-8mm);
\node at (0mm,-23mm) {$S'$};
\node at (-33mm,12mm) {$A_1$};
\node at (33mm,12mm) {$A_2$};

\draw (-8mm,0mm) to[out=40,in=180] (0mm,4mm) to[out=0,in=140] (8mm,0mm);
\draw (-8mm,0mm) to[out=-40,in=180] (0mm,-4mm) to[out=0,in=-140] (8mm,0mm);
\draw (-8mm,0mm) to[out=140,in=130] (-10mm,2mm);
\draw (8mm,0mm) to[out=40,in=50] (10mm,2mm);

\draw[red,thick] (-36.2mm,-5mm) to[out=0,in=180] (0mm,-12mm) to[in=180,out=0] (36.2mm,-5mm);
\draw[red,thick] (35.6mm,-3mm) to[out=180,in=0] (16mm,6mm) to[out=180,in=0] (0mm,4mm);
\draw[red,thick,dashed] (0mm,4mm) to[out=180,in=0] (0mm,20mm);
\draw[red,thick] (0mm,20mm) to[out=180,in=0] (-36.2mm,5mm);

\draw[red,thick] (-35.6mm,3mm) to[out=0,in=90] (-32mm,0mm) to[out=-90,in=0] (-35.6mm,-3mm);

\draw[blue] (-24mm,11.5mm) to[out=0,in=0] (-24mm,-11.5mm);
\draw[blue,dashed] (-24mm,11.5mm) to[out=180,in=180] (-24mm,-11.5mm);
\node[blue] at (-26mm,-15mm) {$B_1'$};

\draw[blue,dashed] (24mm,11.5mm) to[out=0,in=0] (24mm,-11.5mm);
\draw[blue] (24mm,11.5mm) to[out=180,in=180] (24mm,-11.5mm);
\node[blue] at (26mm,-15mm) {$B_2'$};

\draw[red,thick] (37.5mm,7mm) to[out=180,in=90] (32mm,3mm) to[out=-90,in=180] (35.3mm,-1mm);
\draw[red,thick] (36.2mm,5mm) to[out=180,in=90] (34mm,3mm) to[out=-90,in=180] (35.4mm,1mm);
\end{tikzpicture}
\end{center}
\caption{Local decomposition of an arc diagram.}
\label{fig:y1}
\end{figure}

See figure \ref{fig:y1}. We will show that a local decomposition of an arc diagram exists and is unique up to a natural form of equivalence.
\begin{prop}
\label{prop:local_decomposition}
Let $S$ be an oriented connected compact surface with boundary other than a disc.
Any arc diagram $C$ on $S$ has a local decomposition $B'_1, \ldots, B'_n$. If $B'_1, \ldots, B'_n$ and $B''_1, \ldots, B''_n$ are two local decompositions of $C$, then there is a homeomorphism $\phi: S \to S$, fixing $\partial S$ pointwise, such that $\phi(B_i) = \phi(B'_i)$ and $\phi(C)=C$.
\end{prop}
(On a disc, a local decomposition is obtained by drawing $B'_1$ inside a single complementary region; drawing $B'_1$ in distinct complementary regions leads to inequivalent local decompositions.)

Note that the homeomorphism $\phi$ of the proposition takes each annulus $A_i$ of the first decomposition to the corresponding annulus $A''_i$ of the second decomposition, while fixing their common boundary $B_i$ pointwise, so that the arc diagrams on $A_i$ and $A''_i$ are homeomorphic. The fact that $A_i, A''_i$ are $B_i$-local then implies that $\phi$ identifies the points of $B'_i \cap C$ and $B''_i \cap C$ in a canonical way. The core $S'$ of the first decomposition is taken to the core $S''$ of the second decomposition, with boundary points identified, so that the arc diagrams on $S'$ and $S''$ are homeomorphic.

\begin{proof}
First we show a local decomposition exists. Consider an annulus $A_i$ obtained by taking a small collar neighbourhood of the boundary component $B_i$, enlarged to contain neighbourhoods of each arc of $C$ parallel to $B_i$. We can take such $A_i$ to be disjoint. Let the boundary components of $A_i$ be $B_i$ and $B'_i$, and let $S' = S \setminus \bigcup_{i=1}^n A_i$. The restriction of $C$ to $A_i$ consists of arcs parallel to $B_i$, and traversing arcs, so is $B_i$-local. The restriction $C'$ of $C$ to $S'$ contains no boundary-parallel arcs: if $\gamma'$ were such an arc, then $\gamma'$ would lie in a boundary-parallel arc $\gamma$ of $C$, so would be contained in $A_i$ and hence not in $S'$.

To demonstrate uniqueness, essentially we show any local decomposition must look like the one just described. Consider a local decomposition $B'_1, \ldots, B'_n$ of $C$, and an arc $\gamma$ of $C$ with an endpoint on the boundary component $B_i$ of $S$. Either $\gamma$ is boundary-parallel to $B_i$, or $\gamma$ is not boundary-parallel. 

If $\gamma$ is boundary-parallel to $B_i$, then in any local decomposition, the annulus containing $B_i$ must contain $\gamma$: if $\gamma$ took any other route, then it would create a boundary-parallel arc in $S'$, or an arc in some $A_j$ boundary-parallel to $B'_j$, violating the definition of local decomposition.

Similarly, if $\gamma$ is not boundary-parallel, let $\gamma$ have endpoints on $B_i$ and $B_j$ (possibly $i=j$). Then in any local decomposition, $\gamma$ must proceed from $B_i$ across annulus $A_i$ via a traversing arc, across the core $S'$ to the annulus $A_j$, and then across $A_j$ via a traversing arc to $B_j$. If $\gamma$ took any other route, then it would create a boundary-parallel arc in $S'$ or some $A_j$ violating the local decomposition.

Thus, in any local decomposition of $C$, each $A_i$ contains precisely the arcs of $C$ boundary-parallel to $B_i$, and traversing arcs from the remaining points of $F \cap B_i$. Hence there is a homeomorphism taking the local annuli of any decomposition to the local annuli of any other decomposition, fixing $\partial S$ pointwise and preserving $C$; this homeomorphism then extends across the core, preserving $C$, giving the desired equivalence.
\end{proof}

\subsection{Counting arc diagrams via local decomposition}
\label{sec:counting_local_decomposition}

We now take advantage of local decomposition to count arc diagrams. 

Let $C$ be an arc diagram on $(S=S_{g,n}, F(b_1, \ldots, b_n))$, with a local decomposition $B'_1, \ldots, B'_n$, local annuli $A_i$ and core $S'$. Let $|C \cap B'_i| = a_i$. Then on each $A_i$ we have a $B_i$-local arc diagram which lies in $L(b_i, a_i)$. The integer $a_i$ must satisfy $0 \leq a_i \leq b_i$ and $a_i \equiv b_i \pmod{2}$. The arc diagram on the core $S'$ has no boundary-parallel arcs, hence lies in $\mathcal{N}_{g,n}(a_1, \ldots, a_n)$.

In a similar vein, elements of $L(b_i, a_i)$ and $\mathcal{N}_{g,n}(a_1, \ldots, a_n)$ can be glued together to construct an arc diagram on $S$ in locally-decomposed form. Hence there is a map
\[
L(b_1, a_1) \times L(b_2, a_2) \times \cdots L(b_n, a_n) \times \mathcal{N}_{g,n}(a_1, \ldots, a_n) \To \mathcal{G}_{g,n}(b_1, \ldots, b_n).
\]
However, in defining an element of $L(b_i, a_i)$ or $\mathcal{N}_{g,n}(a_1, \ldots, a_n)$ we need to label the marked points; and in the curves $B'_i$ of the local decomposition, points could be labelled in several distinct ways. So this map is not injective: relabelling the marked points of $B'_i$ starting from a distinct basepoint will produce different elements in $L(b_i, a_i)$ and $\mathcal{N}_{g,n}(a_1, \ldots, a_n)$, but the same element of $\mathcal{G}_{g,n}(b_1, \ldots, b_n)$. If $a_i > 0$ then there are precisely $a_i$ ways to choose each basepoint; indeed there is a $\Z_{a_i}$ action on $L(b_i, a_i)$ and $\mathcal{N}_{g,n}(a_1, \ldots, a_n)$. If $a_i = 0$ then there is no basepoint to choose; effectively there is precisely one choice.

Thus, there is a $\Z_{\overline{a}_i}$ action on each $L(b_i, a_i)$ and $\mathcal{N}_{g,n}(a_1, \ldots, a_n)$, cyclically relabelling the points on $B'_i$. In fact, the product $\Z_{\overline{a}_1} \times \cdots \times \Z_{\overline{a}_n}$ acts on $L(b_1, a_1) \times \cdots \times L(b_n, a_n) \times \mathcal{N}_{g,n}({\bf a})$ and the orbits correspond precisely to the equivalence classes of arc diagrams obtained in $\mathcal{G}_{g,n}({\bf b})$. That is, the map
\[
\frac{ L(b_1, a_1) \times \cdots \times L(b_n, a_n) \times \mathcal{N}_{g,n}(a_1, \ldots, a_n) }{ \Z_{\overline{a}_1} \times \cdots \times \Z_{\overline{a}_n} } \To \mathcal{G}_{g,n}(b_1, \ldots, b_n)
\]
obtained by gluing together 
arc diagrams along labelled boundary components is well-defined and injective. Taken over all $(a_1, \ldots, a_n)$ where each $a_i$ satisfies $0 \leq a_i \leq b_i$ and $a_i \equiv b_i \pmod{2}$, we obtain a bijection.

This bijection provides a correspondence between an arc diagram in $\mathcal{G}_{g,n}({\bf b})$, and its local decomposition. 
We call this map
\[
\Delta : \mathcal{G}_{g,n}(b_1, \ldots, b_n) \To \bigsqcup_{\substack{0 \leq a_i \leq b_i \\ a_i \equiv b_i  \!\!\!\! \pmod{2}}} 
\frac{ L(b_1, a_1) \times \cdots \times L(b_n, a_n) \times \mathcal{N}_{g,n}(a_1, \ldots, a_n)}{ \Z_{\overline{a}_1} \times \cdots \times \Z_{\overline{a}_n} }.
\]
Now, the action of $\Z_{\overline{a}_1} \times \cdots \times \Z_{\overline{a}_n}$ on $L(b_1, a_1) \times \cdots \times L(b_n, a_n) \times \mathcal{N}_{g,n}({\bf a})$ is faithful; indeed, the stabiliser of each element of $L(b_i, a_i)$ under the action of $\Z_{\overline{a}_i}$ is trivial. Thus, counting the two sets in bijection we obtain the following proposition.

\begin{prop}
\label{prop:G_and_L}
For any $g \geq 0$, $n \geq 1$ other than $(g,n) = (0,1)$, and any $b_1, \ldots, b_n \geq 0$, we have
\[
G_{g,n}(b_1, \ldots, b_n) = \sum_{\substack{0 \leq a_i \leq b_i \\ a_i \equiv b_i \!\!\!\! \pmod{2}}} \frac{|L(b_1, a_1)| \; \cdots \; |L(b_n, a_n)|}{\bar{a}_1 \bar{a}_2 \cdots \bar{a}_n} N_{g,n}(a_1, \ldots, a_n).
\]
\qed
\end{prop}

In the next section we give an expression for $L(b, a)$.

\subsection{Counting local annuli}
\label{sec:counting_local_annuli}

For the rest of this section $(S,F) = (S_{0,2}, F(b,b'))$ denotes an annulus with boundary components $B, B'$. We regard $B$ as the ``outer" and $B'$ as the ``inner" boundary. We suppose $b' \leq b$ and $b \equiv b' \pmod{2}$, and consider $B$-local arc diagrams. 

In any such arc diagram, there are $b'$ traversing arcs, so $b-b'$ points of $F \cap B$ are endpoints of boundary-parallel arcs. The number of boundary-parallel arcs is therefore $\frac{1}{2}(b-b')$.

We enumerate $L(b,b')$ by a slight generalisation of the argument in section \ref{sec:arrow_diagrams}, which demonstrated a bijection between arrow diagrams and insular arc diagrams.
\begin{defn}
A \emph{local arrow diagram} on $(S,F) $ is a labelling of $\frac{1}{2}(b-b')$ points of $F \cap B$ as ``in"; other points of $F$ remain unlabelled.
\end{defn}

From the data of a local arrow diagram, we can almost reconstruct a unique $b$-local arc diagram with $b'$ legs. 
Start at a marked point on $B$ and proceed anticlockwise around $B$. Each time we arrive at a point of $F$ labelled ``in", we start drawing a new arc anticlockwise. Each time we arrive at a point of $F$ that is unlabelled, we end an arc there if possible. This process produces a 
partial arc diagram on the annulus, consisting only of anticlockwise-oriented insular arcs, with $b'$ remaining unlabelled points on each boundary component which are not yet endpoints of arcs. We connect these remaining points by traversing arcs. If $b' > 0$ then these remaining points can be connected in $b'$ ways: the first point on $B'$ can be connected to any remaining point on $B$, and then the remaining points can only be connected by traversing arcs in one way. If $b' = 0$ however all points are connected; there is one way to connect up the remaining arcs, which is to leave them as they are! This is the idea of the following proposition.

\begin{prop}
\label{prop:Lbb}
For any integers $0 \leq b' \leq b$ of the same parity,
\[
|L(b,b')| = \binom{b}{\frac{1}{2}(b-b')} \bar{b'}.
\]
\end{prop}

\begin{proof}
We first show that a local arrow diagram uniquely determines the boundary-parallel arcs of a $b$-local arc diagram by the process described above. We use induction on the number $p = \frac{1}{2}(b-b')$ of boundary-parallel arcs or equivalently, the number of arrows. If $p=0$ then $b=b'$, a local arrow diagram contains no arrows, and there are no boundary-parallel arcs. If $p>0$ then, as we proceed anticlockwise around $B$, there is at least one point $f_{in}$ labelled ``in", followed immediately by another unlabelled point $f$. Any $b$-local arc diagram compatible with this labelling must have an outermost arc from $f_{in}$ to $f$. Removing this arc and its endpoints produces an arrow diagram with $p-1$ arrows, which by induction uniquely determines the boundary-parallel arcs of a local arc diagram.

Conversely, the boundary-parallel arcs of a local arc diagram immediately provide a local arrow diagram. So specifying the boundary-parallel arcs is equivalent to specifying a local arrow diagram.

Once boundary-parallel arcs are drawn, it remains to draw the $b'$ traversing arcs. Up to equivalence, there are $\bar{b'}$ ways to draw them.
\end{proof}

Proposition \ref{prop:G_and_L} now immediately simplifies to give $G_{g,n}({\bf b})$ in terms of the $N_{g,n}({\bf a})$. For any $g \geq 0$ and $n \geq 1$ other than $(g,n) = (0,1)$, and any $b_1, \ldots, b_n \geq 0$, we now have
\[
G_{g,n}(b_1, \ldots, b_n) = \sum_{\substack{0 \leq a_i \leq b_i \\ a_i \equiv b_i  \!\!\!\! \pmod{2}}}
\binom{b_1}{\frac{b_1 - a_1}{2}} \binom{b_2}{\frac{b_2 - a_2}{2}} \cdots \binom{b_n}{\frac{b_n - a_n}{2}} N_{g,n}(a_1, \ldots, a_n).
\]
Theorem \ref{thm:G_in_terms_of_N} is thus proved.

This result holds even when some or all of the $b_i$ are zero. In fact, when $b_i$ is negative, we can regard $G_{g,n}(b_1, \ldots, b_n) = 0$, and all the $\binom{b_i}{\frac{b_i - a_i}{2}} = 0$, so that the equation still holds.

Moreover, when $\frac{b_i-a_i}{2}$ is negative, we can regard $\binom{b_i}{\frac{b_i - a_i}{2}} = 0$. And when $a_i$ is a negative integer, we can regard $N_{g,n}(a_1, \ldots, a_n) = 0$. So we can regard the sums as being over all integers $a_i$. Further, when $b_i = 0$. We then have the following slightly stronger result.
\begin{prop}
\label{prop:stronger_G_N}
For any integers $b_1, \ldots, b_n$,
\[
G_{g,n}(b_1, \ldots, b_n) 
= \sum_{a_1, \ldots, a_n \in \Z}
\binom{b_1}{\frac{b_1 - a_1}{2}} \cdots \binom{b_n}{\frac{b_n - a_n}{2}}
N_{g,n}(a_1, \ldots, a_n).
\]
\qed
\end{prop}

\section{Counting curves on pants}
\label{sec:pants}

\subsection{Approach}
\label{sec:pants_approach}

We now turn our attention to pairs of pants. Throughout this section let $(S,F) = (S_{0,3}, F(b_1, b_2, b_3))$. We will first compute $N_{0,3}(b_1, b_2, b_3)$, then use local decomposition to compute $G_{0,3}(b_1, b_2, b_3)$.

We set some conventions. We draw pants as twice-punctured discs in the plane, with one outer boundary $B_1$ and two inner boundaries, $B_2$ (on the left) and $B_3$ (on the right). The orientation on the plane induces an orientation on the pants, hence on boundary components: $B_1$ is oriented anticlockwise, and $B_2, B_3$ are oriented clockwise. See figure \ref{fig:x2}.

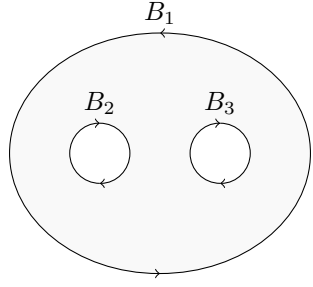
\begin{figure}
\begin{center}
\begin{tikzpicture}
\draw[white,fill=lightgray!10] (0,0) ellipse (20mm and 16mm);
\draw[white,fill=white] (8mm,0) circle (4mm);
\draw[white,fill=white] (-8mm,0) circle (4mm);
\node[above] at (0,16mm) {$B_1$};
\node[above] at (-8mm,4mm) {$B_2$};
\node[above] at (8mm,4mm) {$B_3$};
\draw[->] (8mm,-4mm) arc (270:90:4mm);
\draw[->] (8mm,4mm) arc (90:-90:4mm);
\draw[->] (-8mm,-4mm) arc (270:90:4mm);
\draw[->] (-8mm,4mm) arc (90:-90:4mm);
\draw[->] ((0,-16mm) arc (-90:90:20mm and 16mm);
\draw[->] ((0,16mm) arc (90:270:20mm and 16mm);
\end{tikzpicture}
\caption{Orientations on boundary components of pants.}
\label{fig:x2}
\end{center}
\end{figure}

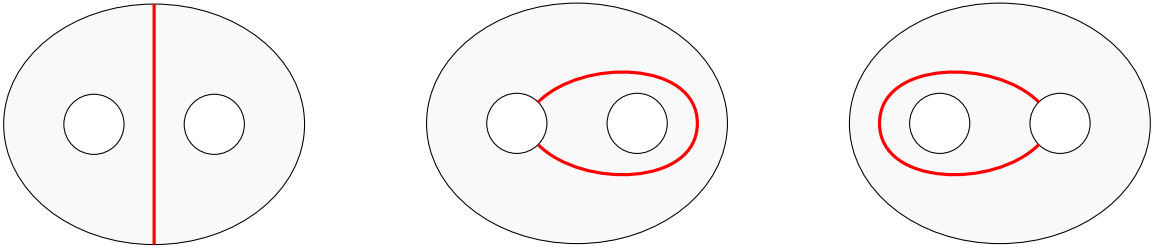
\begin{figure}
\begin{center}
\begin{multicols}{3}
\begin{center}
\begin{tikzpicture}
\draw[fill=lightgray!10] (0,0) ellipse (20mm and 16mm);
\draw[fill=white] (8mm,0) circle (4mm);
\draw[fill=white] (-8mm,0) circle (4mm);
\draw[very thick,red] (0,16mm) -- (0,-16mm);
\end{tikzpicture}

\begin{tikzpicture}
\draw[fill=lightgray!10] (0,0) ellipse (20mm and 16mm);
\draw[fill=white] (8mm,0) circle (4mm);
\draw[fill=white] (-8mm,0) circle (4mm);
\draw[very thick,red] ($(-8mm,0) + (45:4mm)$) to[out=45,in=90] (16mm,0) to[out=-90,in=-45] ($(-8mm,0) + (-45:4mm)$);
\end{tikzpicture}

\begin{tikzpicture}
\draw[fill=lightgray!10] (0,0) ellipse (20mm and 16mm);
\draw[fill=white] (8mm,0) circle (4mm);
\draw[fill=white] (-8mm,0) circle (4mm);
\draw[very thick,red] ($(8mm,0) + (135:4mm)$) to[out=135,in=90] (-16mm,0) to[out=-90,in=-135] ($(8mm,0) + (-135:4mm)$);
\end{tikzpicture}
\end{center}
\end{multicols}
\caption{Three prodigal arcs.}
\label{fig:x1}
\end{center}
\end{figure}

We also establish some terminology, extending terminology from the annulus case. See figure \ref{fig:x1}.
\begin{defn}
An arc on a pair of pants is
\begin{enumerate}
\item 
\emph{traversing} if its endpoints lie on distinct boundary components;
\item
\emph{prodigal} if its endpoints lie on the same boundary component, but it is not boundary-parallel;
\item
\emph{insular} if it is boundary-parallel.
\end{enumerate}
\end{defn}
Thus, a prodigal arc travels extravagantly but eventually returns home; an insular arc never goes far from home. In a local decomposition, insular arcs are contained in local annuli, while prodigal and traversing arcs pass through the core. 

It will be useful to keep track of the number of arcs of certain types.
\begin{defn}
\label{def:pants_notation}
In an arc diagram on a pair of pants, let the number of 
\begin{enumerate}
\item
prodigal arcs with endpoints on $B_j$ be $p_j$;
\item
traversing arcs with endpoints on $B_i$ and $B_j$ be $t_{ij}$
\end{enumerate}
\end{defn}

\subsection{Non-boundary-parallel arc diagrams}
\label{sec:non_boundary_parallel_pants}

We now compute $N_{0,3}$; it is remarkably simple.

\begin{prop}
\label{prop:N03}
For any integers $b_1, b_2, b_3 \geq 0$ such that $b_1 + b_2 + b_3$ is even,
\[
N_{0,3}(b_1, b_2, b_3) = \bar{b}_1 \bar{b}_2 \bar{b}_3.
\]
If $b_1 + b_2 + b_3$ is odd, then $N_{0,3}(b_1, b_2, b_3) = 0$.
\end{prop}

The case $b_1 + b_2 + b_3$ odd is clear (lemma \ref{lem:even_odd}), so we assume $b_1 + b_2 + b_3$ is even.

Now an arc diagram in $\mathcal{N}_{0,3}(b_1, b_2, b_3)$ contains no insular, only prodigal and traversing arcs. A prodigal arc cuts the pants into two annuli. If $p_1 > 0$, then a prodigal arc with endpoints on $B_1$ separates $B_2$ from $B_3$, so that there cannot be any traversing arc from $B_2$ to $B_3$, nor any prodigal arcs from these components; hence $p_2 = p_3 = t_{23} = 0$. Similarly, if $p_2 > 0$ then $p_3 = p_1 = t_{31} = 0$; and if $p_3 > 0$ then $p_1 = p_2 = t_{12} = 0$. In fact, such conditions are also sufficient to be able to draw an arc diagram. We can state this precisely.
\begin{lem}
\label{lem:conditions_for_pants_diagram}
There exists an arc diagram without boundary-parallel arcs on a pair of pants if and only if $t_{12}, t_{23}, t_{31}, p_1, p_2, p_3$ satisfy the following conditions:
\begin{enumerate}
\item If $p_1 > 0$ then $p_2 = p_3 = t_{23} = 0$.
\item If $p_2 > 0$ then $p_3 = p_1 = t_{31} = 0$.
\item If $p_3 > 0$ then $p_1 = p_2 = t_{12} = 0$.
\end{enumerate}
\end{lem}
(Note that if $p_1 = p_2 = p_3 = 0$, these conditions are all satisfied.)

\begin{proof}
The discussion above shows that the conditions are necessary. Now suppose we have $p_i$ and $t_{ij}$ satisfying these conditions. If all $p_i = 0$ then the only possible nonzero parameters are $t_{12}, t_{23}, t_{31}$ and such traversing arcs can easily be drawn. If some $p_i$ is nonzero, say $p_1$, then the only possible nonzero parameters are $t_{12}$ and $t_{31}$. After drawing $p_1$ parallel prodigal arcs with endpoints on $B_1$, there remain two complementary annuli on which any number of traversing arcs from $B_1$ to $B_2$, and from $B_3$ to $B_1$, can be drawn.
\end{proof}

In an arc diagram without boundary-parallel arcs, the parameters $p_i, t_{ij}$ determine the number of boundary marked points $b_1, b_2, b_3$. Each prodigal arc with endpoints on $B_j$ contributes two points to $b_j$; each traversing arc between $B_i$ and $B_j$ contributes one point to $b_i$ and one to $b_j$. Thus
\[
b_1 = 2 p_1 + t_{12} + t_{31}, \quad
b_2 = 2 p_2 + t_{23} + t_{12}, \quad
b_3 = 2 p_3 + t_{31} + t_{12}.
\]
The converse turns out also to be true: the $b_i$ determine the $p_i$ and $t_{ij}$, as in the following lemma.

\begin{prop}
\label{prop:finding_parameters}
\label{prop:pants_numbers_of_curves}

Let $b_1, b_2, b_3 \geq 0$ be integers such that $b_1 + b_2 + b_3$ is even. Then there are unique non-negative integers $t_{12}, t_{23}, t_{31}, p_1, p_2, p_3$ satisfying the following conditions:
\begin{enumerate}
\item
\begin{enumerate}
\item $b_1 = t_{12} + t_{31} + 2 p_1$
\item $b_2 = t_{23} + t_{12} + 2 p_2$
\item $b_3 = t_{31} + t_{23} + 2 p_3$
\end{enumerate}
\item
\begin{enumerate}
\item If $p_1 > 0$ then $p_2 = p_3 = t_{23} = 0$.
\item If $p_2 > 0$ then $p_3 = p_1 = t_{31} = 0$.
\item If $p_3 > 0$ then $p_1 = p_2 = t_{12} = 0$.
\end{enumerate}
\end{enumerate}
Explicitly, such $t_{12}, t_{23}, t_{31}, p_1, p_2, p_3$ are given as follows. Let $\{i,j,k\} = \{1,2,3\}$ such that $b_i \leq b_j \leq b_k$.
\begin{enumerate}
\item
If $b_i + b_j \geq b_k$ then $p_1 = p_2 = p_3 = 0$ and
\[
t_{12} = \frac{1}{2}(b_1 + b_2 - b_3), \quad
t_{23} = \frac{1}{2}(b_2 + b_3 - b_1), \quad
t_{31} = \frac{1}{2}(b_3 + b_1 - b_2).
\]
\item
If $b_i + b_j < b_k$ then $p_i = p_j = t_{ij} = 0$ and
\[
p_k = \frac{1}{2}(b_k - b_i - b_j), \quad
t_{ik} = b_i, \quad
t_{jk} = b_j.
\]
\end{enumerate}
\end{prop}

The two cases above correspond to whether or not $b_1, b_2, b_3$ obey the \emph{triangle inequality} -- that is, when any two of the $b_i$ sum to at least the third. When the triangle inequality is satisfied, the proposition says that the $t_{ij}$ are given by $t_{ij} = \frac{1}{2} (b_i + b_j - b_k)$. These are the lengths of the tangents from the vertices of the Euclidean triangle to its incircle!

\begin{proof}
First we note that the triangle inequality is satisfied if and only if all $p_i = 0$. For if some $p_i$, say $p_1$, is positive, then $p_2 = p_3 = t_{23} = 0$ so $b_1 = 2p_1 + t_{12} + t_{31} > t_{12} + t_{31} = b_2 + b_3$ and the triangle inequality is violated. And if all $p_i = 0$ then we have $b_1 = t_{12} + t_{31}$, $b_2 = t_{23} + t_{12}$ and $b_3 = t_{31} + t_{12}$ so, for instance, $b_1 + b_2 = 2t_{12} + t_{23} + t_{31} \geq t_{23} + t_{31} = b_3$ and the triangle inequality holds.

Now if the triangle inequality holds, then all $p_i = 0$ so the $b_i$ are given by $b_1 = t_{12} + t_{31}$, $b_2 = t_{23} + t_{12}$ and $b_3 = t_{31} + t_{23}$. This system of linear equations can be inverted to give the unique solution claimed for $t_{12}, t_{23}, t_{31}$, which are all non-negative by the triangle inequality.

If the triangle inequality fails, then some $p_i > 0$, say $p_1 > 0$, so $p_2 = p_3 = t_{23} = 0$ and we have $b_1 = t_{12} + t_{31} + 2p_1$, $b_2 = t_{12}$ and $b_3 = t_{31}$. So $b_2, b_3$ are as claimed and we immediately obtain $p_1 = \frac{1}{2} (b_1 - b_2 - b_3)$.
\end{proof}

\begin{proof}[Proof of proposition \ref{prop:N03}]
Given $b_1, b_2, b_3 \geq 0$ with even sum, proposition \ref{prop:finding_parameters} shows that there exist unique $t_{ij}$ and $p_i$ which satisfy the conditions of lemma \ref{lem:conditions_for_pants_diagram}, and hence give the numbers of traversing and prodigal arcs in any arc diagram in $\mathcal{N}_{0,3}(b_1, b_2, b_3)$. With the numbers of each type of arc determined, the arc diagram is uniquely determined, up to labelling of points on the boundary. There are $\bar{b}_i$ ways to choose a basepoint from the $b_i$ points on the boundary component $B_i$, which determines the arc diagram up to equivalence. Hence $N_{0,3}(b_1, b_2, b_3) = \bar{b}_1 \bar{b}_2 \bar{b}_3$ as claimed.
\end{proof}

We have now proved equation (\ref{eqn:N03}) in theorem \ref{thm:N_formulas}.

\subsection{General arc diagrams}
\label{sec:general_pants}

We now have $N_{0,3}$, so from theorem \ref{thm:G_and_N} we can express $G_{0,3}$ in terms of $N_{0,3}$:
\begin{equation}
G_{0,3}(b_1, b_2, b_3) = \sum_{\substack{0 \leq a_i \leq b_i \\ a_i \equiv b_i  \!\!\!\! \pmod{2}}} \binom{b_1}{\frac{b_1 - a_1}{2}} \binom{b_2}{\frac{b_2 - a_2}{2}} \binom{b_3}{\frac{b_3 - a_3}{2}} \bar{a}_1 \bar{a}_2 \bar{a}_3,
\label{eqn:G03nearlythere}
\end{equation}
so it remains to calculate the sum
\[
\sum_{\substack{0 \leq a \leq b \\ a \equiv b  \!\!\!\! \pmod{2}}} \binom{b}{\frac{b-a}{2}} \bar{a}
= \sum_{\substack{0 \leq a \leq b \\ a \equiv b  \!\!\!\! \pmod{2}}} L(b,a).
\]
In fact, we will calculate some more general sums, which will prove useful in the sequel, namely
\begin{equation}
\label{eqn:sums_to_add}
\sum_{\substack{0 \leq a \leq b \\ a \equiv b  \!\!\!\! \pmod{2}}} \binom{b}{\frac{b-a}{2}} \overline{a} \; a^{2\alpha}
\quad \text{ and }
\sum_{\substack{0 \leq a \leq b \\ a \equiv b   \!\!\!\! \pmod{2}}} \binom{b}{\frac{b-a}{2}} a^{2\alpha + 1},
\end{equation}
where $\alpha$ is a non-negative integer. We apply ideas from the work of Norbury--Scott \cite{Norbury-Scott14}. Several of the following definitions come from that paper.
\begin{defn}
\label{defn:ps_and_qs}
For an integer $\alpha \geq 0$, define the functions $\tilde{p}_\alpha (n), \tilde{q}_\alpha(n), \tilde{P}_\alpha (n), \tilde{Q}_\alpha (n)$ as follows.
\begin{align*}
\tilde{P}_\alpha (n) &= \sum_{l=0}^n \binom{2n}{n-l} \overline{(2l)} \; (2l)^{2\alpha}, \\
\tilde{p}_\alpha(n) &= \sum_{l=0}^n \binom{2n}{n-l} (2l)^{2\alpha + 1} \\
\tilde{Q}_\alpha (n) = \tilde{q}_\alpha (n) &= \sum_{l=0}^n \binom{2n+1}{n-l} \overline{(2l+1)} (2l+1)^{2 \alpha} = \sum_{l=0}^n \binom{2n+1}{n-l} (2l+1)^{2\alpha + 1}
\end{align*}
\end{defn}
Observe $\tilde{P}_\alpha (n)$ (resp. $\tilde{Q}_\alpha(n)$) gives the first sum in equation \eqref{eqn:sums_to_add} when $b$ is even, $b=2n$ (resp. odd, $b=2n+1$), and $\tilde{p}_\alpha(n)$ (resp. $\tilde{q}_\alpha (n)$) gives the second sum when $b$ is even, $b=2n$ (resp. odd, $b=2n+1$).

Clearly $\tilde{P}_\alpha (n)$ differs from $\tilde{p}_\alpha$ only in the $l=0$ term, and this only when $\alpha=0$; for all $\alpha \geq 1$,
\[
\tilde{P}_\alpha (n) = \tilde{p}_\alpha (n) + \binom{2n}{n} \delta_{\alpha,0}.
\]
Norbury--Scott show that $\tilde{p}_\alpha (n), \tilde{q}_\alpha (n)$ are closely related to the following polynomials $p_\alpha (n), q_\alpha(n)$.

\begin{defn}
For integers $\alpha \geq 0$, the integer polynomials $p_\alpha (n), q_\alpha(n)$ are defined recursively by
\begin{align*}
p_0 (n) = 1, \quad &p_{\alpha+1}(n) = 4n^2 \left( p_\alpha(n) - p_\alpha(n-1) \right) + 4n p_\alpha (n-1) \\
q_0 (n) = 1, \quad &q_{\alpha+1}(n) = 4n^2 \left( q_\alpha(n) - q_\alpha(n-1) \right) + (4n+1) q_\alpha (n)
\end{align*}
\end{defn}
(Equation (15) in \cite{Norbury-Scott14} appears to have a typo; the $(4n+1) q_\alpha(n-1)$ should be $(4n+1) q_\alpha(n)$.) 

\begin{prop}[Norbury--Scott~\cite{Norbury-Scott14}]
\label{prop:Ps_and_Qs}
Let $\alpha \geq 0$ be an integer. Then $p_\alpha, q_\alpha$ are integer polynomials of degree $\alpha$ with positive leading coefficients. Moreover,
\[
\tilde{p}_\alpha (n) = \binom{2n}{n} \; n \; p_\alpha(n)
\quad \text{and} \quad 
\tilde{q}_\alpha (n) = \binom{2n}{n} (2n+1) q_\alpha (n).
\]
Further, 
\[
\tilde{P}_\alpha (n) = \binom{2n}{n} P_\alpha (n)
\quad \text{and} \quad
\tilde{Q}_\alpha (n) = \binom{2n}{n} Q_\alpha (n),
\]
where $P_\alpha (n) = n p_\alpha (n) + \delta_{\alpha,0}$ and $Q_\alpha (n) = (2n+1) q_\alpha (n)$ are integer polynomials of degree $\alpha+1$ with positive leading coefficients.
\end{prop}

\begin{proof}
Norbury--Scott \cite{Norbury-Scott14} show that $\tilde{p}_\alpha$ and $\tilde{q}_\alpha$ are as claimed, and $p_\alpha, q_\alpha$ have degree $\alpha$. It is clear from the recurrence that the coefficients are integers and the leading coefficients are positive. The claims for $\tilde{P}_\alpha$ and $\tilde{Q}_\alpha$ then follow immediately from $\tilde{P}_\alpha (n) = \tilde{p}_\alpha (n) + \binom{2n}{n} \delta_{\alpha,0}$ and $\tilde{Q}_\alpha (n) = \tilde{q}_\alpha (n)$.
\end{proof}

We compute the first few of the sums $\tilde{P}_\alpha (n)$ and $\tilde{Q}_\alpha (n)$.
\begin{multicols}{2}
\begin{align*}
\tilde{P}_0 (n) &= \binom{2n}{n} (n+1)  \\
\tilde{P}_1 (n) &= \binom{2n}{n} n \; 4n \\
\tilde{P}_2 (n) &= \binom{2n}{n} n \; 16n(2n-1) \\
\tilde{P}_3 (n) &= \binom{2n}{n} n \; 64n (6n^2 - 8n + 3)
\end{align*}

\begin{align*}
\tilde{Q}_0 (n) &= \binom{2n}{n} (2n+1) \\
\tilde{Q}_1 (n) &= \binom{2n}{n} (2n+1) \; (4n+1) \\
\tilde{Q}_2 (n) &= \binom{2n}{n} (2n+1) \; (32n^2 + 8n + 1) \\
\tilde{Q}_3 (n) &= \binom{2n}{n} (2n+1) \; (384n^3 - 32n^2 + 12 n + 1)
\end{align*}
\end{multicols}

In any case, we have now computed the sums arising in $G_{0,3}(b_1, b_2, b_3)$ and we have the following.
\begin{align*}
G_{0,3} (b_1, b_2, b_3) &= \prod_{i=1}^3 \sum_{\substack{0 \leq a_i \leq b_i \\ a_i \equiv b_i  \!\!\!\! \pmod{2}}} \binom{b_i}{\frac{b_i - a_i}{2}} \; \bar{a}_i
= \prod_{i=1}^3 \left\{ \begin{array}{ll}
\tilde{P}_0 (m_i) & b_i \text{ even, } b_i = 2m_i \\
\tilde{Q}_0 (m_i) & b_i \text{ odd, } b_i = 2m_i + 1
\end{array} \right\} \\
&= \prod_{i=1}^3 \left\{ \begin{array}{ll}
\binom{2m_i}{m_i} (m_i + 1) & b_i \text{ even, } b_i = 2m_i \\
\binom{2m_i}{m_i} (2m_i + 1) & b_i \text{ odd, } b_i = 2m_i + 1
\end{array} \right\} 
\end{align*}
This immediately gives the formulae in theorem \ref{thm:formulas}(\ref{eqn:G03eee})-(\ref{eqn:G03ooe}).

\section{Topological recursion}
\label{sec:top_recursion}

\subsection{For curve counts}
\label{sec:recursion_for_curve_counts}

We now show that the numbers $G_{g,n}(b_1, \ldots, b_n)$ obey the recursion of theorem \ref{thm:Ggn_recursion}. More specifically, we prove the following.
\begin{thm}
\label{thm:G_recursion}
For integers $g \geq 0$, $n \geq 1$ and $b_1, \ldots, b_n$ such that $b_1 > 0$ and $b_2, \ldots, b_n \geq 0$,
\begin{align*}
G_{g,n}(b_1, \ldots, b_n) &= \sum_{\substack{i,j \geq 0 \\ i+j = b_1 - 2}} G_{g-1,n+1} (i,j,b_2, \ldots, b_n) \\
& \quad + \sum_{k=2}^n b_k G_{g,n-1}(b_1 + b_k - 2, b_2, \ldots, \widehat{b}_k, \ldots, b_n) \\
& \quad + \sum_{\substack{i,j \geq 0 \\ i+j = b_1 - 2}} \; \sum_{\substack{g_1, g_2 \geq 0 \\ g_1 + g_2 = g}} \; \sum_{I \sqcup J = \{2, \ldots, n\}} G_{g_1, |I|+1} (i, b_I) G_{g_2, |J|+1} (j, b_J).
\end{align*}
In particular, any $G_{g,n}({\bf b})$ can be computed using this recursion and the initial conditions $G_{g,n}({\bf 0}) = 1$, .
\end{thm}
Here the first term is a sum is over integers $i,j \geq 0$ summing to $b_1 - 2$; if $b_1 = 1$ this sum is empty. In the second term, the notation $\widehat{b}_k$ means that $b_k$ is omitted from the list $b_2, \ldots, b_n$. In the third term, the sum over $I, J$ is a sum over all pairs of (possibly empty) disjoint sets $(I, J)$ whose union is $\{2, \ldots, n\}$. The notation $b_I$ is shorthand for the set of all $b_k$ where $k \in I$; and similarly for $b_J$. As $G_{g,n}(b_1, \ldots, b_n)$ is a symmetric function of the $b_i$, it is sufficient to give $b_I$ as a set rather than a sequence.

Note that when $b_1 + \cdots + b_n$ is odd, all terms are zero; in each term the inputs to each $G_{g,n}$ have the same parity sum.

This recursion expresses $G_{g,n}$ in terms of curve counts on ``simpler" surfaces. We regard the complexity of a surface as given by $-\chi$, where $\chi = 2-2g-n$ is the Euler characteristic. All terms on the right involve surfaces with complexity $\leq -\chi(S_{g,n}) = 2g+n-2$. The first two terms involve surfaces with complexity strictly less than $2g+n-2$, but the third term may involve surfaces homeomorphic to $S$, for instance when $g_1 = g$ and $I = \{2, \ldots, n\}$; however in this case the number of marked points $b_1 + \cdots + b_n$ decreases. Thus, repeatedly applying the recursion, (and permuting the $b_i$ if necessary, to avoid $b_1 = 0$, as $G_{g,n}$ is a symmetric function), one eventually arrives at terms of the form $G_{g,n}({\bf 0}) = 1$.

In particular, theorem \ref{thm:G_recursion} implies that all $G_{g,n}({\bf b})$ are finite!

\begin{proof}
Fix an orientation on $S_{g,n}$, label the boundary components $B_1, \ldots, B_n$, and label the marked points on $B_j$ by $1, \ldots, b_j$ following the induced orientation around $B_j$. Let $p$ be the first marked point on $B_1$. Such a $p$ exists because $b_1 > 0$. 

Given an arc diagram $C$ on $(S_{g,n}, F({\bf b}))$, consider the arc $\gamma$ with an endpoint at $p$. Cutting along $\gamma$ yields a surface $S'$ with an arc diagram $C'$ and one less arc; $\gamma$ becomes two arcs on $\partial S'$. We consider the various cases for $\gamma$ and show how, in each case, we can give a standard numbering on the boundary components and points, so that the arc diagrams so obtained are counted in $G_{g',n'}({\bf b'})$ for ``simpler" $g',n',{\bf b'}$. In each case, the location of $\gamma$ on $\partial S'$ after cutting can be determined, so that $C$ can be uniquely reconstructed from $S'$ and $C'$ by gluing these two boundary arcs together. In this way we obtain a bijection between $\mathcal{G}_{g,n} ({\bf b})$ and various sets involving simpler $\mathcal{G}_{g',n'}({\bf b'})$ and establish the desired recursion.

We deal with the cases as follows.
\begin{enumerate}
\item  \emph{$\gamma$ has both endpoints on $B_1$ and is nonseparating}. In this case cutting along $\gamma$ gives $S'$ of genus $g-1$ with $n+1$ boundary components. Let the endpoints of $\gamma$ have labels $1$ and $i+2$, for some integer $i$ with $2 \leq i+2 \leq b_1$, i.e. $0 \leq i \leq b_1 - 2$. We name the boundary components of $S'$ as $B'_1, \ldots, B'_{n+1}$, and their marked points,  as follows. The original boundary component $B_1$ splits into $B'_1$ and $B'_2$ so that $B'_1$ contains the points originally labelled $2, \ldots, i+1$; we now number these points $1, \ldots i$. The other boundary component contains points originally labelled $i+3, \ldots, b_1$. Letting $j = b_1 - i-2$, we now number them $1, \ldots, j$. We obtain an element of $\mathcal{G}_{g-1,n+1}(i,j,b_2, \ldots, b_n)$ where $i,j \geq 0$ satisfy $i+j = b_1 - 2$. And given such $i,j$ and an element of $\mathcal{G}_{g-1,n+1} (i,j,b_2, \ldots, b_n)$, we can reconstruct the original arc diagram in $\mathcal{G}_{g,n}({\bf b})$, giving a bijective correspondence between arc diagrams in $\mathcal{G}_{g,n}({\bf b})$ of this type, and elements of $\mathcal{G}_{g-1,n+1}(i,j,b_2, \ldots, b_n)$ for a choice of $i,j \geq 0$ with $i+j=b_1 - 2$.

\item \emph{$\gamma$ has endpoints on distinct boundary components}. In this case, cutting along $\gamma$ gives $S'$ of genus $g$ with $n-1$ boundary components. Let the endpoints of $\gamma$ lie on boundary components $B_1$ and $B_k$. We name the boundary components of $S'$ as $B'_1, \ldots, B'_{n-1}$ where $B'_1$ is a union of $B_1, B_k$ and the two copies of $\gamma$, and then number $B'_2, \ldots, B'_{n-1}$ in order as $B_2, \ldots, \widehat{B}_k, \ldots, B_n$. The marked points, in order around $B'_1$, consist of $b_1 - 1$ points of $B_1$, followed by $b_k - 1$ points of $B_k$, so that $B'_1$ has $b_1 + b_k - 2$ boundary points. Numbering marked points on other boundary components as on $S$, we obtain an element of $\mathcal{G}_{g,n-1}(b_1 + b_k - 2, b_2, \ldots, \widehat{b}_k, \ldots, b_n)$; and we also keep track of which of the $b_k$ points on $B_k$ was an endpoint of $\gamma$. Conversely, given an arc diagram of genus $g$ with $n-1$ boundary components, with one of the points not marked $1$ on $B_1$ marked, we can reconstruct an arc diagram in $\mathcal{G}_{n-1}(b_1 + b_k - 2, b_2, \ldots, \widehat{b}_k, \ldots, b_n)$ by gluing two arcs on the boundary together. This gives a bijective correspondence between arc diagrams in $\mathcal{G}_{g,n}({\bf b})$ of this type with elements of $\mathcal{G}_{g,n-1}(b_1 + b_k - 2, b_2, \ldots, \widehat{b}_k, \ldots, b_n)$ with a special marked point on $B_k$. 

\item \emph{$\gamma$ has both endpoints on $B_1$ and is separating}. In this case $\gamma$ cuts $S$ into two surfaces, $S'$ and $S''$; for definiteness we say that, as we proceed along $\gamma$ from the point marked $1$, $S'$ is on the left and $S''$ is on the right. Let $S'$ have genus $g_1$ and $S''$ have genus $g_2$. The boundary components of $S_1$ are then $B'_1$, which contains some of $B_1$ and $\gamma$, as well as other boundary components $B_k$ for $k \in I \subset \{2, \ldots, n\}$. Similarly, the boundary components of $S_2$ are $B''_1$, which contains some of $B_1$ and $\gamma$, as well as other boundary components $B_k$ for $k \in J \subset \{2, \ldots, n\}$. Here $I$ and $J$ are disjoint and $I \sqcup J = \{2, \ldots, n\}$. Let $B'_1$ and $B''_1$ contain $i$ and $j$ marked points respectively; then $i,j \geq 0$ and $i+j = b_1 - 2$. As in the previous cases, we obtain a bijection between arc diagrams in $\mathcal{G}_{g,n}({\bf b})$ of this type, and elements of $\mathcal{G}_{g_1, |I| + 1} (i, b_I) \times \mathcal{G}_{g_2, |J| + 1} (j, b_J)$ over the various possible $i,j,g_1, g_2, I, J$. Since $B_1$ is split into two boundary components, we can number the marked points on the pair of smaller surfaces so as to indicate how they can be glued back together. \qedhere
\end{enumerate}
\end{proof}

\subsection{For non-boundary-parallel curve counts}
\label{sec:non-parallel_recursion}

The $N_{g,n}({\bf b})$ also obey a recursion; it is slightly more complicated than the $G_{g,n}({\bf b})$ case.
\begin{thm}
\label{thm:Ngn_recursion}
For $(g,n) \neq (0,1), (0,2), (0,3)$ and integers $b_1, \ldots, b_n$ such that $b_1 >0$, $b_2, \ldots, b_n \geq 0$,
\begin{align*}
N_{g,n}({\bf b}) &= \sum_{\substack{ i,j,m \geq 0 \\ i+j+m = b_1 \\ m \text{ even}}} \frac{m}{2} \; N_{g-1,n+1} (i,j,b_2, \ldots, b_n) \\
& + \sum_{j=2}^n \left( \sum_{\substack{i,m \geq 0 \\ i+m = b_1 + b_j \\ m \text{ even}}} \frac{m}{2} \; \bar{b}_j \; N_{g,n-1} (i,b_2, \ldots, \widehat{b}_j, \ldots, b_n) + \widetilde{\sum_{\substack{i,m \geq 0 \\ i+m = b_1 - b_j \\ m \text{ even}}}} \frac{m}{2} \; \bar{b}_j \; N_{g,n-1} (i, b_2, \ldots, \widehat{b}_j, \ldots, b_n) \right) \\
& + \sum_{\substack{g_1 + g_2 = g \\ I \sqcup J = \{2, \ldots, n\} \\ \text{No discs or annuli}}} \sum_{\substack{i,j,m \geq 0 \\ i+j+m = b_1 \\ m \text{ even}}} \frac{m}{2} \; N_{g_1, |I|+1} (i, b_I) \; N_{g_2, |J|+1} (j, b_J)
\end{align*}
\end{thm}
We explain the notation. The tilde over the second summation in brackets is interpreted as follows. If $b_1 - b_j \geq 0$ then read the sum as is: it is a sum over non-negative integers $i,m$ such that $i+m = b_1 - b_j$ and $m$ is even. If $b_1 - b_j \leq 0$ then replace $b_1 - b_j$ with $b_j - b_1$ and make the sum negative: i.e. the term becomes
\[
- \sum_{\substack{i,m \geq 0 \\ i+m = b_j - b_1 \\ m \text{ even}}} \frac{m}{2} \; \bar{b}_j \; N_{g,n-1} (i, b_2, \ldots, \widehat{b_j}, \ldots, b_n).
\]
This idea of splitting the sum this way appears in \cite{Norbury10_counting_lattice_points}; indeed the recursion on $N_{g,n}$ is very similar to the recursion on the $N_{g,n}$ in that paper. The final summation is over decompositions of the non-negative integer $g$ into non-negative genera $g=g_1 + g_2$, and also over partitions of $\{2, \ldots, n\}$ into disjoint subsets; the ``no discs or annuli" condition means that we exclude terms in which $(g_1, |I|+1)$ or $(g_2, |J|+1)$ are equal to $(0,1)$ or $(0,2)$.

Again, the terms are only nonzero when $b_1 + \cdots + b_n$ is even; but the result holds even when this sum is odd, as all terms are then zero.

We have excluded the cases $(g,n) = (0,1), (0,2), (0,3)$. It is natural to exclude discs as they only contain boundary-parallel arcs. 
The reason for excluding annuli and pants is more subtle. The proof is based on a similar analysis as the topological recursion for $G_{g,n}({\bf b})$; but there are significantly more subtleties arising from the lack of boundary-parallel arcs, and the argument fails for annuli and pants.

To illustrate some of the difficulties, let $C$ be a non-boundary-parallel arc diagram on $(S_{g,n}, F({\bf b}))$, and let $\gamma$ be an arc of $C$ starting at the base point on $B_1$. After cutting along $\gamma$, we obtain a less complex surface $S'$ (possibly disconnected), with a simpler arc diagram $C'$. However, the arc diagram $C'$ may contain boundary-parallel arcs: it is possible that arcs of $C$, which were not boundary-parallel, might become boundary-parallel after cutting along $\gamma$.

For instance, suppose we have an arc $\delta$ which is parallel to $\gamma$. After cutting along $\gamma$, $\delta$ becomes boundary-parallel: see figure \ref{fig:y2}. For another example, suppose $\gamma$ connects two distinct boundary components $B_1$ and $B_2$, and $\delta$ is an arc which runs from $B_1$, around $B_2$, back to $B_1$: see figure \ref{fig:y3}. Again $\delta$ is not boundary-parallel, but after cutting along $\gamma$, $\delta$ becomes boundary-parallel.

\begin{figure}
\begin{center}
\begin{tikzpicture}
\def\xlen{80mm}
\draw[fill=lightgray!10] (0,0) ellipse (20mm and 16mm);
\draw[fill=white] (8mm,0) circle (4mm);
\draw[fill=white] (-8mm,0) circle (4mm);
\draw[thick,red] (8mm,4mm) -- node[left] {$\delta$} (8mm,14.6mm);
\draw[thick,red] (12mm,0mm) -- node[below] {$\gamma$} (20mm,0mm);

\draw[->] (20mm+12mm,0) -- node[above] {cut along $\gamma$} (\xlen-16mm-12mm,0);

\draw[fill=lightgray!10] (\xlen,0) circle (16mm);
\draw[fill=white] (\xlen,0) circle (4mm);
\draw[red,thick] ($(\xlen,0mm)+(90:16mm)$) to[out=270,in=210] node[below] {$\delta$} ($(\xlen,0mm)+(30:16mm)$);
\end{tikzpicture}
\end{center}
\caption{Arc $\delta$ parallel to $\gamma$ becomes boundary-parallel after cutting along $\gamma$.}
\label{fig:y2}
\end{figure}
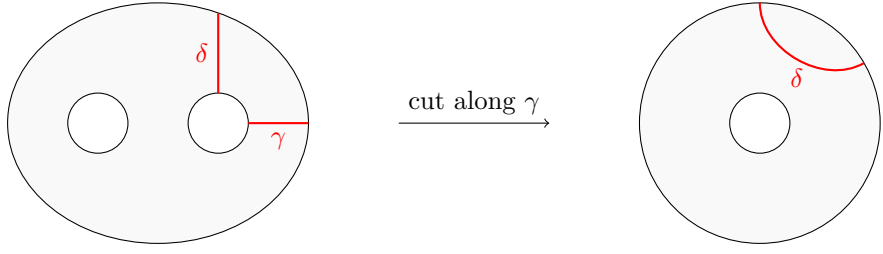

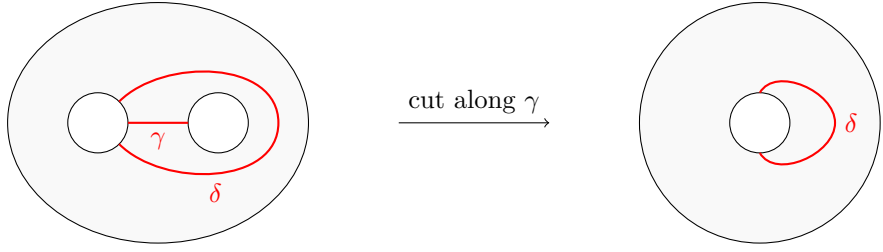
\begin{figure}
\begin{center}
\begin{tikzpicture}
\def\xlen{80mm}
\draw[fill=lightgray!10] (0,0) ellipse (20mm and 16mm);
\draw[fill=white] (8mm,0) circle (4mm);
\draw[fill=white] (-8mm,0) circle (4mm);
\draw[thick,red] (-4mm,0mm) -- node[below] {$\gamma$} (4mm,0mm);
\draw[thick,red] ($(-8mm,0) + (45:4mm)$) to[out=45,in=90] (16mm,0) to[out=-90,in=-45] node[below] {$\delta$} ($(-8mm,0) + (-45:4mm)$);

\draw[->] (20mm+12mm,0) -- node[above] {cut along $\gamma$} (\xlen-16mm-12mm,0);

\draw[fill=lightgray!10] (\xlen,0) circle (16mm);
\draw[fill=white] (\xlen,0) circle (4mm);
\draw[red,thick] ($(\xlen,0mm)+(90:4mm)$) to[out=60,in=90] (\xlen+10mm,0mm) node[right] {$\delta$} to[out=270,in=300] ($(\xlen,0mm)+(270:4mm)$);
\end{tikzpicture}
\end{center}
\caption{Arc $\delta$ running from $B_1$ around $B_2$ becomes boundary-parallel after cutting along $\gamma$ from $B_1$ to $B_2$.}
\label{fig:y3}
\end{figure}

In the following lemma, we establish precisely which arcs can become boundary-parallel; essentially, only the cases mentioned above.
\begin{lem}
\label{lem:become_boundary_parallel}
Let $C$ be an arc diagram on $(S,F) = (S_{g,n}, F({\bf b}))$ without boundary-parallel arcs. Let $\gamma$ be an arc of $C$ and let the result of cutting along $\gamma$ be the arc diagram $C'$ on $(S', F')$. If $\delta$ is an arc of $C$ which is boundary-parallel in $S'$, then exactly one of the following cases occurs:
\begin{enumerate}
\item
$\gamma$ has endpoints on two distinct boundary components $B_i, B_j$ of $S$, and
\begin{enumerate}
\item $\delta$ is parallel to $\gamma$ (as in figure \ref{fig:y2})
\item $\delta$ has both endpoints on $B_i$, and runs around $B_j$ as in figure \ref{fig:y3};
\end{enumerate}
\item
$\gamma$ is nonseparating with both endpoints on the same boundary component $B_i$ of $S$, and $\delta$ is parallel to $\gamma$;
\item $\gamma$ is separating, and $\delta$ is parallel to $\gamma$.
\end{enumerate}
\end{lem}

\begin{proof}
Any arc $\gamma$ of $C$ falls into precisely one of the cases (i), (ii) or (iii). Clearly any arc $\delta$ of one of the types listed becomes boundary-parallel after cutting along $\gamma$.

Now suppose an arc $\delta$ becomes boundary-parallel after cutting along $\gamma$. Let the endpoints of $\gamma$ lie on boundary components $B_i$ and $B_j$ (possibly $i=j$). Then $\delta$ must be homotopic, relative to endpoints, to an arc lying along $B_i \cup B_j \cup \gamma$. So $\delta$ is parallel to $\gamma$, or runs around $B_j$ as claimed.
\end{proof}

For the purposes of the proof of theorem \ref{thm:Ngn_recursion}, we group together the cases for $\gamma$ a little differently from the proof of theorem \ref{thm:G_recursion}. In particular, we will gather together arcs $\gamma$ which connect $B_1$ to a distinct boundary component, with those $\gamma$ which cut off an annulus. An arc $\delta$ which cuts off an annulus must run from some boundary component $B_i$, loop around another boundary component $B_j$, and return to $B_i$. This grouping corresponds to cases (i), (ii), (iii) of the above lemma. 

\begin{proof}[Proof of theorem \ref{thm:Ngn_recursion}]
Let $p$ be the first marked point on $B_1$, and let $\gamma$ be the arc of $C$ with an endpoint at $p$. We consider the various cases for $\gamma$; in each case, we cut along $\gamma$, and remove any arcs which become parallel (i.e. those described in lemma \ref{lem:become_boundary_parallel}) to obtain a simpler boundary-parallel arc diagram on a simpler surface. We can then construct bijections between equivalence classes of arc diagrams on $(S,F)$ with various sets of arc diagrams on simpler surfaces.

We consider the following three cases for $\gamma$.
\begin{enumerate}
\item \emph{$\gamma$ has both endpoints on $B_1$, and is nonseparating.} Orient $\gamma$ so that $p$ is the start point.  Cutting along $\gamma$ produces $S' = S_{g-1,n+1}$. So $B_1$ is split into two boundary components $B'_1, B'_2$. Let the number of arcs parallel to $\gamma$, including $\gamma$, be $\frac{m}{2}$, so $m \geq 2$ is even. The $\frac{m}{2}-1$ arcs parallel to $\gamma$, other than $\gamma$, are precisely the ones that become boundary-parallel in $S'$. Let the number of points remaining on $B'_1$ and $B'_2$ after removing these boundary-parallel arcs be $i$ and $j$ respectively. Together $\gamma$ and the arcs parallel to $\gamma$ have $m$ endpoints, all originally on $B_1$, so we have $i+j+m=b_1$. Labelling boundary points in a standard fashion, we obtain an equivalence class of arc diagram $C'$ in $\mathcal{N}_{g-1,n+1}(i,j,b_2, \ldots, b_n)$. For any $i,j,m \geq 0$ such that $i+j+m=b_1$ and $m$ is even, such arc diagrams $C'$ exist. Moreover, from the data of $C'$ and $m$, the original arc diagram $C$ can be reconstructed: the boundary labelling on $C'$ indicates which boundary segments are to be glued back together, and $m$ parallel arcs are drawn there.

However, there may be several arc diagrams on $(S,F)$ which lead to the same arc diagram $C'$ and the same number $m$. In particular, this occurs if we take the same arc diagram $C$ but shift the basepoint $p$ on $B_1$ so that $\gamma$ becomes another one of the $m/2$ arcs parallel to the original $\gamma$. All the arc diagrams on $S$ which lead to $C'$ and $m$ are of this form. As $\gamma$ starts at $p$, there are $m/2$ such possibilities for $p$. Hence the number of arc diagrams in $\mathcal{N}_{G,n}(b_1, \ldots, b_n)$ for which $\gamma$ has both endpoints on $B_1$ and is nonseparating is given by the first summation in the recursion.

\item \emph{$\gamma$ has endpoints on distinct boundary components $B_1$ and $B_j$, \emph{or} is separating (hence has both endpoints on $B_1$) and cuts an annulus off $S$.} (Note that, as $(g,n) \neq (0,3)$, $\gamma$ cannot cut $S$ into \emph{two} annuli; if $\gamma$ cuts off an annulus, then only one annulus appears. The possibility of two annuli causes the recursion to fail in the case $(g,n) = (0,3)$.) In the latter case, let the boundary component around which $\gamma$ loops be $B_j$, so that $B_j$ is a boundary component of the annulus cut off by $\gamma$. 

Let $m/2$ be the number of arcs ``parallel" to $\gamma$, in the following sense. If $\gamma$ runs from $B_1$ to $B_j$, we take the arcs parallel to $\gamma$, including $\gamma$; and also those which run from $B_1$ around $B_j$ and back to $B_1$; and also those which run from $B_j$ around $B_1$ and back to $B_j$; these curves become boundary-parallel in $S' = S_{g,n-1}$. If $\gamma$ cuts off an annulus around $B_j$, we take the arcs parallel to $\gamma$, including $\gamma$, and also those which run from $B_1$ to $B_j$. These $m/2$ arcs consist precisely of $\gamma$ and the arcs which become boundary-parallel in $S'$. (Note that there cannot both be loops from $B_1$ around $B_j$, and loops from $B_j$ around $B_1$. The former can only occur if $b_1 > b_j$, and the latter can only occur if $b_j > b_1$.)

For those arcs with endpoints on $B_1$ and $B_j$, orient them from $B_1$ to $B_j$. For those arcs with endpoints on $B_1$ which loop around $B_j$ (i.e. those which cut off annuli), orient them as shown in figure \ref{fig:y4}(a) so that they run anticlockwise around $B_j$. For those arcs with endpoints on $B_j$ which loop around $B_1$, orient them as shown in figure \ref{fig:y4}(b) so that they run anticlockwise around $B_1$.

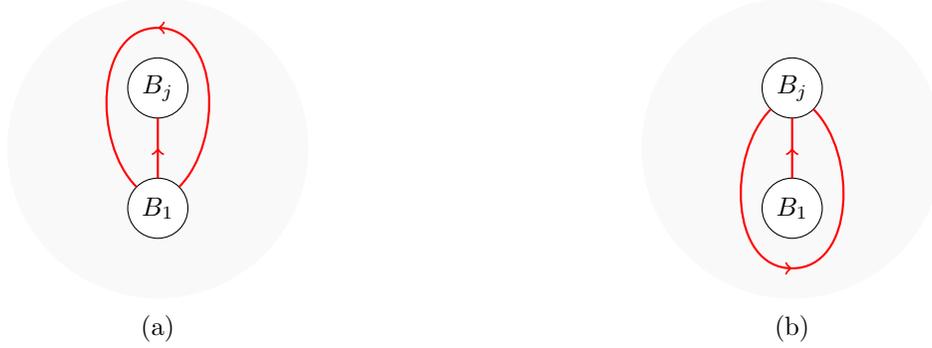
\begin{figure}
\begin{multicols}{2}
\begin{center}
\begin{tikzpicture}
\def\xlen{80mm}

\fill[fill=lightgray!10] (0,0) circle (20mm);
\draw[fill=white] (0,8mm) circle (4mm);
\draw[fill=white] (0,-8mm) circle (4mm);

\begin{scope}[red,thick,decoration={markings,mark=at position 0.5 with {\arrow{>}}}] 
	\draw[postaction={decorate}] ($(0,-8mm) + (45:4mm)$) to[out=45,in=0] (0,16mm) to[out=180,in=135] ($(0,-8mm) + (135:4mm)$);
	\draw[postaction={decorate}] (0,-4mm) -- (0,4mm);
\end{scope}

\node at (0,-8mm) {$B_1$};
\node at (0,8mm) {$B_j$};

\node at (0,-24mm) {(a)};
\end{tikzpicture}

\begin{tikzpicture}
\def\xlen{80mm}

\fill[fill=lightgray!10] (0,0) circle (20mm);
\draw[fill=white] (0,8mm) circle (4mm);
\draw[fill=white] (0,-8mm) circle (4mm);

\begin{scope}[red,thick,decoration={markings,mark=at position 0.5 with {\arrow{>}}}] 
	\draw[postaction={decorate}] ($(0,8mm) + (-135:4mm)$) to[out=-135,in=180] (0,-16mm) to[out=0,in=-45] ($(0,8mm) + (-45:4mm)$);
	\draw[postaction={decorate}] (0,-4mm) -- (0,4mm);
\end{scope}

\node at (0,-8mm) {$B_1$};
\node at (0,8mm) {$B_j$};

\node at (0,-24mm) {(b)};
\end{tikzpicture}
\end{center}
\end{multicols}
\caption{Orientation on arcs running (a) from $B_1$ around $B_j$, and (b) from $B_j$ around $B_1$.}
\label{fig:y4}
\end{figure}

After cutting along $\gamma$ and removing all the ``parallel" arcs described above --- which are precisely the arcs that become boundary-parallel in $S'$ --- we obtain an arc diagram on $S'$. Boundary components $B_1$ and $B_j$ are combined into a boundary component $B'_1$ of $S'$. Let $i$ be the number of marked points on $B'_1$. Now $\gamma$ and all the arcs ``parallel" to it have $m$ endpoints, so $i+m = b_1 + b_j$, and $m$ is even. Labelling boundary points in a standard fashion (starting near $p$, say, and proceeding around the boundary numbering points consecutively), we obtain an arc diagram $C'$ in $\mathcal{N}_{g,n-1}(i, b_2, \ldots, \widehat{b_j}, \ldots, b_n)$. For any integers $2 \leq j \leq n$ and $i,m \geq 0$ such that $i+m = b_1 + b_j$ and $m$ is even, the arc diagram $C$ can be reconstructed from $C'$ and $m$: again, the labelling on $C'$ indicates which boundary segments of $C'$ to glue to obtain two boundary components $B_1, B_j$ with $b_1, b_j$ marked points; and we draw $m$ parallel arcs there, possibly including loops from $B_1$ around $B_j$ or loops from $B_j$ around $B_1$.

However, there may be several arc diagrams on $S$ which lead to the same arc diagram $C'$ on $S'$ and the same $m$. For one thing, if we adjust the basepoint $p$ on $B_1$ so that $\gamma$ is replaced by any of the arcs ``parallel" to the original $\gamma$, cutting and removing boundary-parallel arcs leads to the same $C'$ and $m$. For another, the arcs from $B_1$ to $B_j$ can be adjusted so as to meet $B_j$ at different points; there are $\bar{b}_j$ ways to adjust any such diagram. (The effect is like a ``fractional Dehn twist" about $B_j$. Equivalently, the labelling on $B_j$ can be adjusted so as to place the basepoint at any position.) As $S$ is not an annulus, these two types of adjustment are independent. (When $S$ is an annulus, these two types of adjustment have the same result, and the claimed recursion fails.)

Suppose for now that $b_1 \geq b_j$. Then the $m/2$ arcs ``parallel" to $\gamma$ consist of arcs from $B_1$ to $B_j$, and possibly arcs from $B_1$ looping around $B_j$. (But there cannot be any arcs from $B_j$ looping around $B_1$.) There are precisely $m/2$ positions for $p$ on $B_1$ so that $p$ is the \emph{start} point of one of these arcs; and for each such choice, the points at which arcs meet $B_j$ can be adjusted in $\bar{b}_j$ ways. Hence the number of (equivalence classes of) arc diagrams in $\mathcal{N}_{g,n}(b_1, \ldots, b_n)$ for which $\gamma$ runs from $B_1$ to $B_j$, or runs from $B_1$ around $B_j$, and is oriented so that $p$ is the start point of $\gamma$, is given by
\[
\sum_{\substack{i,m \geq 0 \\ i+m = b_1 + b_j \\ m \text{ even}}} \frac{m}{2} \; \bar{b}_j \; N_{g,n-1} (i, b_2, \ldots, \widehat{b_j}, \ldots, b_n).
\]
However, there is also the possibility that $p$ is the \emph{endpoint} of $\gamma$. Such a situation only arises when $\gamma$ is an arc from $B_1$ which loops around $B_j$; in this case, as $b_j \leq b_1$, there must be $b_j$ arcs connecting $B_1$ to $B_j$. Redefine $m/2$ to be the number of arcs from $B_1$ looping around $B_j$ (i.e. cutting off an annulus with $B_j$ as a boundary component). Still letting $i$ denote the number of marked points of $C'$ on $B'_1$, these $i$ points together with the $m$ endpoints of the arcs looping around $B_j$ and the $b_j$ arcs from $B_1$ to $B_j$ together make up all the marked points on $B_1$, so $i+m+b_j = b_1$. Again $C$ can be reconstructed from $C'$ and $m$. Again there are several arc diagrams on $(S,F)$ with $p$ as the endpoint of $\gamma$ which lead to the same $C'$ and $m$: we may rotate the arcs to meet $B_j$ at different points in $\bar{b}_j$ ways; and we may adjust the basepoint $p$ to be any of the $m/2$ points on $B_1$ at which an arc looping around $B_j$ ends. Thus, the number of arc diagrams in $\mathcal{N}_{g,n}(b_1, \ldots, b_n)$ for which $\gamma$ runs from $B_1$ around $B_j$, and is oriented so that  $p$ is the end point of $\gamma$, is given by
\[
\sum_{\substack{i,m \geq 0 \\ i+m = b_1 - b_j \\ m \text{ even}}} \frac{m}{2} \; \bar{b}_j \; N_{g,n-1} (i, b_2, \ldots, \widehat{b_j}, \ldots, b_n).
\]
This covers all possibilities in the case $b_1 \geq b_j$.

Now suppose $b_1 \leq b_j$. Note that in this case the arcs ``parallel" to $\gamma$ consist of arcs from $B_1$ to $B_j$, and possibly arcs from $B_j$ looping around $B_1$. Let $m/2$ denote the number of these ``parallel" arcs. As in the case $b_1 \geq b_j$, the term
\[
\sum_{\substack{i,m \geq 0 \\ i+m = b_1 + b_j \\ m \text{ even}}} \frac{m}{2} \; \bar{b}_j \; N_{g,n-1} (i, b_2, \ldots, \widehat{b_j}, \ldots, b_n)
\]
gives the number of arc diagrams in $\mathcal{N}_{g,n}(b_1, \ldots, b_n)$ for which $p$ is the start point of the arc $\gamma$ along which we cut. However it counts these diagrams regardless of whether $p$ lies on $B_1$ or $B_j$! We must therefore subtract off the diagrams for which $p$ lies on $B_j$.

In such cases, $\gamma$ is an arc based at $B_j$ looping around $B_1$, and as $b_1 \leq b_j$, there must be $b_1$ arcs connecting $B_1$ to $B_j$. Redefine $m/2$ to be the number of arcs from $B_j$ looping around $B_1$, so $i+m+b_1 = b_j$. By a similar argument as above, the number of arc diagrams in $\mathcal{N}_{g,n}(b_1, \ldots, b_n)$ for which $\gamma$ runs from $B_j$ around $B_1$, and is oriented so that $p$ is the start point, is given by
\[
\sum_{\substack{i,m \geq 0 \\ i+m = b_j - b_1 \\ m \text{ even}}} \frac{m}{2} \; \bar{b}_j \; N_{g,n-1} (i, b_2, \ldots, \widehat{b_j}, \ldots, b_n).
\]

Hence the number of arc diagrams in $\mathcal{N}_{g,n}(b_1, \ldots, b_n)$ for which $\gamma$ has endpoints on distinct boundary components, or is separating and cuts off an annulus, is given by the summations in the second line of the recursion.

\item \emph{$\gamma$ is separating but does not cut off an annulus.} As $C$ has no boundary-parallel arcs, $\gamma$ cannot cut off a disc either. Thus it remains to consider separating $\gamma$ where no discs or annuli arise. If we orient $\gamma$ to start at $p$, as $S$ is oriented, then cutting along $\gamma$ there is a surface $S_1$ to the left of $\gamma$ and a surface $S_2$ to its right. Let $S_1$ have genus $g_1$ and $S_2$ have genus $g_2$, so $g_1, g_2 \geq 0$ and $g_1 + g_2 = g$. After cutting along $\gamma$, boundary component $B_1$ contributes a boundary component $B'_1$ to $S_1$ and $B''_1$ to $S_2$; the remaining boundary components of $S_1$ and $S_2$ come from the original boundary components $B_2, \ldots, B_n$ of $S$. Let $S_1$ contain boundary components whose numbers consist of $I \subset \{2, \ldots, n\}$, and let $S_2$ contain boundary components $J \subset \{2, \ldots, n\}$, so $I \sqcup J = \{2, \ldots, n\}$. There may be arcs which become boundary-parallel after cutting along $C$: such arcs will be parallel to $\gamma$; let there be $\frac{m}{2} - 1$ of them, so that $\gamma$ and its parallel arcs together contain $m$ endpoints. Let $B'_1, B''_1$ contain $i,j$ marked points respectively, so $i+j+m = b_1$ and $m$ is even. Labelling boundary points in a standard fashion, we obtain an arc diagram $C_1$ in $\mathcal{N}_{g_1, |I|+1}(i, b_I)$ and an arc diagram $C_2 \in \mathcal{N}_{g_2, |J|+1} (j, b_J)$. For any $g_1, g_2, i,j,m \geq 0$ and $I,J$ such that $g_1 + g_2 = g$, $I \sqcup J = \{2, \ldots, n\}$, $i+j+m = b_1$ and $m$ is even, such arc diagrams $C_1$ and $C_2$ exist, and conversely, from $C_1, C_2$ and $m$, the original $C$ can be reconstructed.

However, several arc diagrams on $(S,F)$ could lead to the same $C'$ and $m$: this occurs if we shift $p$ so that $\gamma$ is another one of the $m/2$ arcs parallel to the original $\gamma$. Since $\gamma$ starts at $p$, there are $m/2$ such possibilities for $p$. Hence the number of arc diagrams in $\mathcal{N}_{g,n}(b_1, \ldots, b_n)$ for which $\gamma$ is separating, but does not cut off any discs or annuli, is given by the third line of the recursion.
\end{enumerate}
Putting these cases together, the number of arc diagrams in $\mathcal{N}_{g,n}(b_1, \ldots, b_n)$ is as claimed.
\end{proof}

Dividing through by $\bar{b}_2 \cdots \bar{b}_n$ (and since $b_1 > 0$, so $\bar{b}_1 = b_1$) we obtain a recursion on $\widehat{N}_{g,n}$.
\begin{cor}
\label{cor:HatNgn_recursion}
For $(g,n) \neq (0,1), (0,2), (0,3)$ and $b_1 > 0$,
\begin{align*}
b_1 \widehat{N}_{g,n}({\bf b}) &= \sum_{\substack{ i,j,m \geq 0 \\ i+j+m = b_1 \\ m \text{ even}}} \frac{1}{2} \; \bar{i} \; \bar{j} \; m \; \widehat{N}_{g-1,n+1} (i,j,b_2, \ldots, b_n) \\
& \quad + \sum_{j=2}^n \frac{1}{2} \left( \sum_{\substack{i,m \geq 0 \\ i+m = b_1 + b_j \\ m \text{ even}}} \bar{i} \; m \; \widehat{N}_{g,n-1} (i,b_2, \ldots, \widehat{b_j}, \ldots, b_n) + \widetilde{\sum_{\substack{i,m \geq 0 \\ i+m = b_1 - b_j \\ m \text{ even}}}} \bar{i} \; m \widehat{N}_{g,n-1} (i, b_2, \ldots, \widehat{b_j}, \ldots, b_n) \right) \\
&\quad + \sum_{\substack{g_1 + g_2 = g \\ I \sqcup J = \{2, \ldots, n\} \\ \text{No discs or annuli}}} \sum_{\substack{i,j,m \geq 0 \\ i+j+m = b_1 \\ m \text{ even}}} \frac{1}{2} \; \bar{i} \; \bar{j} \; m \; \widehat{N}_{g_1, |I|+1} (i, b_I) \; \widehat{N}_{g_2, |J|+1} (j, b_J)
\end{align*}
\qed
\end{cor}
This recursion is very similar to the recursion in equation (5) of Norbury \cite{Norbury10_counting_lattice_points}. In fact, if we drop the bars over $i$'s and $j$'s, it should be identical. (Norbury does not explicitly specify the parity requirements, but they are implicit. His equation (5) also has a typographical error, since there should be factors of $1/2$ in each term.)

\subsection{Applying the recursion in small cases}
\label{sec:applying_small_recursion}

We can now compute $\widehat{N}_{1,1}(b)$ directly, using the computation of $N_{0,2}$ in lemma \ref{lem:Ngn_small_cases}:
\begin{align*}
\bar{b} N_{1,1}(b) &=
\sum_{\substack{i,j,m \geq 0 \\ i+j+m = b \\ m \text{ even}}} \frac{1}{2} \bar{i} \; \bar{j} \; m \; N_{0,2} (i,j) 
= \sum_{\substack{i,j,m \geq 0 \\ i+j+m = b \\ m \text{ even}}} \frac{1}{2} \bar{i} \; m \delta_{i,j} 
= \sum_{\substack{i,m \geq 0 \\ 2i+m = b \\ m \text{ even}}} \frac{1}{2} \bar{i} \; m.
\end{align*}
When $b$ is odd, there are no terms in the sum; when $b$ is even, we obtain
\begin{equation}
\label{eqn:N11_almost}
\widehat{N}_{1,1}(b) 
= \frac{1}{2 \bar{b}} \sum_{\substack{p,q \geq 0 \\ p+q = b \\ q \text{ even}}} \overline{(p/2)} \;  q
= \frac{1}{4 \bar{b}} \sum_{\substack{p,q \geq 0 \\ p+q = b \\ q \text{ even}}} \bar{p} \; q + \frac{1}{4}.
\end{equation}
In the last step we used the fact that $\overline{p/2} = \bar{p}/2$, except when $p=0$, in which case $\overline{p/2} = \bar{p}/2 + \frac{1}{2}$.

In the next section we compute this and more general sums: we will compute sums of the form
\[
\sum_{\substack{p,q \geq 0 \\ p+q = k \\ q \text{ even}}} \bar{p} \; p^{2m} q,
\]
for any integer $m \geq 0$, and  more.

\section{Polynomiality results}
\label{sec:polynomiality}

So far we have proved (lemma \ref{lem:Ngn_small_cases} and proposition \ref{prop:N03})
\begin{align*}
N_{0,1}(b_1) &= \delta_{b_1, 0} \\
N_{0,2}(b_1, b_2) &= \bar{b}_1 \delta_{b_1, b_2} \\
N_{0,3}(b_1, b_2, b_3) &= \left\{ \begin{array}{ll} \bar{b}_1 \bar{b}_2 \bar{b}_3 & \text{if $b_1 + b_2 + b_3$ even} \\ 0 & \text{if $b_1 + b_2 + b_3$ odd} \end{array} \right.
\end{align*}
so that
\begin{align*}
\widehat{N}_{0,1}(b_1) &= \delta_{b_1, 0} \\
\widehat{N}_{0,2}(b_1, b_2) &= \frac{\delta_{b_1, b_2}}{\bar{b}_1} \\
\widehat{N}_{0,3}(b_1, b_2, b_3) &= \left\{ \begin{array}{ll} 1 & \text{if $b_1 + b_2 + b_3$ even} \\ 0 & \text{if $b_1 + b_2 + b_3$ odd} \end{array} \right.
\end{align*}
We aim to show quasi-polynomiality of $\widehat{N}_{g,n}(b_1, \ldots, b_n)$ for $(g,n) \neq (0,1), (0,2)$. For this it will be useful first to compute certain summations.

\subsection{Some useful sums}
\label{sec:useful_sums}

We prove some algebraic lemmas, following the techniques of Norbury in \cite{Norbury10_counting_lattice_points}. Several of the following definitions appear in that paper, but we will need a few more.
\begin{defn}
For integers $m \geq 0$, define the functions $A_m, S_m: \N_0 \To \N_0$ by the following sums:
\[
A_m (k) = \sum_{\substack{p,q \geq 0 \\ p+q = k \\ q \text{ even}}} \bar{p} p^{2m} q, 
\quad \quad \quad
S_m (k) = \sum_{\substack{p,q \geq 0 \\ p+q = k \\ q \text{ even}}} p^{2m+1} q.
\]
\end{defn}
These sums are over non-negative integers $p,q$ such that $p+q=k$ and $q$ is even. (Note that once the parity of $k$ is given, the sum is over $p$ and $q$ of fixed parity: $q$ is even, and $p$ has the same parity as $k$.) 

The functions $A_m$ and $S_m$ are clearly very similar; they only differ in their $p=0$ terms, and then only when $m=0$.

\begin{defn}
For integers $m,n \geq 0$, define the functions $B_{m,n}, B_{m,n}^0, B_{m,n}^1, R_{m,n}, R_{m,n}^0, R_{m,n}^1: \N_0 \To \N_0$ by the following sums.
\newpage 
\begin{multicols}{2}
\begin{align*}
B_{m,n} (k) &= \sum_{\substack{p,q,r \geq 0 \\ p+q+r = k \\ r \text{ even}}} \bar{p} \; \bar{q} \; p^{2m} q^{2n} r\\
B_{m,n}^0 (k) &= \sum_{\substack{p,q,r \geq 0 \\ p+q+r=k \\ p \text{ even}, \; r \ \text{ even}}} \bar{p} \; \bar{q} \; p^{2m} q^{2n} r \\
B_{m,n}^1 (k) &= \sum_{\substack{p,q,r \geq 0 \\ p+q+r=k \\ p \text{ odd}, \; r \text{ even}}} \bar{p} \; \bar{q} \; p^{2m} q^{2n} r
\end{align*}

\begin{align*}
R_{m,n}(k) &= \sum_{\substack{p,q,r \geq 0 \\ p+q+r = k \\ r \text{ even}}} p^{2m+1} q^{2n+1} r \\
R_{m,n}^0 (k) &= \sum_{\substack{p,q,r \geq 0 \\ p+q+r=k \\ p \text{ even}, \; r \text{ even}}} p^{2m+1} q^{2n+1} r \\
R_{m,n}^1 (k) &= \sum_{\substack{p,q,r \geq 0 \\ p+q+r=k \\ p \text{ odd}, \; r \text{ even}}} p^{2m+1} q^{2n+1} r
\end{align*}
\end{multicols}
\end{defn}

The summations in $B_{m,n}, R_{m,n}$ are over integers $p,q,r \geq 0$ such that $p+q+r = k$ and $r$ is even. If the parity of $k$ is given, then the parities of $p$ and $q$ in the sum are not fixed. For instance, if $k$ is even, then the sum will be over triples $(p,q,r)$ where $(p,q,r) \equiv (0,0,0)$ and $(1,1,0) \pmod{2}$. When we split these sums into those terms for which $p$ is even or odd, we obtain $B_{m,n}^0$ and $B_{m,n}^1$ respectively, so $B_{m,n} = B_{m,n}^0 + B_{m,n}^1$. Similarly, $R_{m,n} = R_{m,n}^0 + R_{m,n}^1$.

Clearly each $B$ sum is very similar to the corresponding $R$ sum; they differ only in terms where $p=0$ or $q=0$, and then only when $m=0$ or $n=0$.

The sums $S_m$, $R_{m,n}$ were defined by Norbury in \cite{Norbury10_counting_lattice_points}. He showed that $S_m(k)$ is an odd quasi-polynomial function of $k$ of degree $2m+3$, depending on the parity of $k$. That is, $S_m(k)$ is given by two polynomials of degree $2m+3$ in $k$, one which applies for $k$ even and the other for $k$ odd. Similarly, Norbury shows that each $R_{m,n}$ is an odd quasi-polynomial function of $k$, depending on the parity of $k$, of degree $2m+2n+5$. It is clear from his argument that the coefficients are rational. We will show similar results for $A_m, B_{m,n}, B_{m,n}^0, B_{m,n}^1, R_{m,n}^0$ and $R_{m,n}^1$.

Our proof follows the methods of Norbury, which in turn rely on a result of Brion--Vergne \cite{Brion-Vergne97} generalising Ehrhart's theorem. By a \emph{convex lattice polytope} in $\R^n$, we mean a polytope $P$ in $\R^n$, with all vertices in the lattice $\Z^n$, i.e. the convex hull of a finite subset of $\Z^n$. As a subset of $\R^n$, we may speak of the \emph{interior} $P^0$ of $P$ and the \emph{boundary} $\partial P$; if the interior is nonempty then $P$ must be $n$-dimensional. For any non-negative integer $k$, the set $kP = \{kx : x \in P\}$ is again a convex lattice polytope. Given a function $\phi: \R^n \To \R$, we may sum it on the lattice points of $P$, $P^0$ or $\partial P$. We may in fact sum $\phi$ over the lattice points of $kP$ or $kP^0$ or $\partial P$ and see how this sum varies with $k$. Thus we define
\[
N_P (\phi, k) = \sum_{x \in \Z^n \cap kP} \phi(x),
\quad
N_{P^0} (\phi, k) = \sum_{x \in \Z^n \cap kP^0} \phi(x),
\quad \text{and} \quad
N_{\partial P}  (\phi, k) = \sum_{x \in \Z^n \cap k \partial P} \phi(x).
\]
Since $\partial P = P \setminus P^0$, and similarly $k\partial P = kP \setminus kP^0$, we have immediately
\[
N_{\partial P}(\phi, k) = N_P (\phi, k) - N_{P^0} (\phi, k).
\]
The result of Brion--Vergne says that, under certain circumstances, these are \emph{polynomials} obeying a surprising property.
\begin{prop}[Brion--Vergne \cite{Brion-Vergne97}, prop. 4.1]
\label{prop:Brion_Vergne}
Let $P$ be a convex lattice polytope in $\R^n$ with nonempty interior $P^0$. Let $\phi(x_1, \ldots, x_n)$ be a homogeneous rational polynomial of degree $d$. Then $N_P (\phi, k)$ and $N_{P^0} (\phi, k)$ are rational polynomials in $k$ of degree $n+d$. Moreover,
\[
N_{P^0}(\phi, k) = (-1)^{n+d} N_P (\phi, -k).
\]
\qed
\end{prop}
Note that while $N_P (\phi, -k)$ does not appear to be defined, when $-k$ is a negative integer, the notation means to substitute $-k$ for $k$ in the polynomial function $N_P (\phi,k)$. Thus the two polynomials are obtained from each other by replacing $k$ with $-k$ and adjusting the overall sign.

\begin{lem}[Cf. Norbury \cite{Norbury10_counting_lattice_points}, lemma 1]
\label{lem:As_and_Bs} \
\begin{enumerate}
\item
Let $m \geq 0$ be an integer. Then $A_m (k)$ and $S_m (k)$ are rational odd quasi-polynomials of degree $2m+3$, depending on the parity of $k$, which differ by a lower-degree quasi-polynomial.
\item
Let $m,n \geq 0$ be integers. Then $B_{m,n}(k), B_{m,n}^0(k)$, $B_{m,n}^1(k)$, $R_{m,n}(k)$, $R_{m,n}^0(k)$, $R_{m,n}^1(k)$ are all rational odd quasi-polynomials of degree $2m+2n+5$, depending on the parity of $k$. Each of $B_{m,n}, B_{m,n}^0, B_{m,n}^1$ differs from the respective $R_{m,n}, R_{m,n}^0, R_{m,n}^1$ by a lower-degree quasi-polynomial.
\end{enumerate}
In all cases, the leading coefficients are positive.
\end{lem}

\begin{proof}
For part (i), Norbury \cite[lemma 1]{Norbury10_counting_lattice_points} proved that $S_m (k)$ is an odd quasi-polynomial of degree $2m+3$, depending on the parity of $k$; it follows from the proof that the coefficients are rational. 

Using the equality $\bar{p} = p + \delta_{p,0}$ for integers $p \geq 0$, where $\delta$ is the Kronecker delta, we obtain
\begin{align*}
A_m (k) 
&= \sum_{\substack{p,q \geq 0 \\ p+q = k \\ q \text{ even}}} \bar{p} p^{2m} q
= \sum_{\substack{p,q \geq 0 \\ p+q = k \\ q \text{ even}}} (p + \delta_{p,0}) p^{2m} q 
= \sum_{\substack{p,q \geq 0 \\ p+q = k \\ q \text{ even}}} p^{2m+1} q  + \sum_{\substack{q \geq 0 \\ p = 0, \; q = k \\ q \text{ even}}} p^{2m} q \\
&= S_m (k) + \delta_{m,0} \sum_{\substack{q=k \\ q \text{ even}}} q
\end{align*}
The second term here is zero if $m \neq 0$; and if $m=0$, then it is a sum consisting of at most one term. If $k$ is odd the sum is empty. If $k$ is even then there is one term, and it is $k$. 

Thus, $A_0 (k) = S_0 (k) + k$ for $k$ even, and $A_0 (k) = S_0 (k)$ for $k$ odd. Since $S_0 (k)$ has degree $3$, $A_0 (k)$ is a rational odd quasi-polynomial of degree $3$, depending on the parity of $k$. When $m>0$ we have $A_m (k) = S_m (k)$ for all $k$. So for all $m$, $A_m (k)$ is a rational odd quasi-polynomial of degree $2m+3$, given by $S_m (k)$, plus a lower-degree quasi-polynomial. 

We now turn to the various $R$ functions. Consider the following convex lattice polytope in $\R^3$:
\[
P = \{(x,y,z) \in \R^3 : x,y,z \geq 0, \; 2x+y+2z \leq 2\},
\]
which is the convex hull of $\{(0,0,0), (2,0,0), (0,1,0), (0,0,1)\}$. This $P$ is a 3-simplex, with three of its four (2-dimensional) faces right-angled triangles in the $xy$, $yz$ and $zx$ planes, and the fourth face cut out by the plane $2x+y+2z=2$ in the positive octant.

Consider the polynomial function
\[
\phi(x,y,z) = x^{2m+1} y^{2n+1} z,
\]
fixing $m$ and $n$ throughout this discussion, so $\deg \phi = 2m+2n+3$. Applying proposition \ref{prop:Brion_Vergne} to $P$ and $\phi$, $N_{P} (\phi,k)$ and $N_{P^0} (\phi, k)$ are rational polynomials in $k$ of degree $2m+2n+6$, and $N_{P^0}(\phi,k) = N_{P} (\phi, -k)$. Hence
\[
N_{\partial P}(\phi, k) = N_{P} (\phi, k) - N_{P^0}(\phi,k) = N_{P} (\phi,k) - N_{P} (\phi,-k),
\]
is an odd rational polynomial function of $k$ of degree $\leq 2m+2n+5$.

Now $\phi(0,y,z) = \phi(x,0,z) = \phi(x,y,0) = 0$, so $\phi$ vanishes in the $xy$, $yz$ and $zx$ planes, hence on three sides of $P$. We thus have
\[
N_{\partial P} (\phi, k) = \sum_{\substack{x,y,z \geq 0 \\ 2x+y+2z=2k}} x^{2m+1} y^{2n+1} z.
\]
Let $p=2x$, $q=y$, $r=2z$. So $x,y,z$ are non-negative integers such that $2x+y+2z = 2k$ if and only if $p,q,r$ are non-negative integers such that $p+q+r = 2k$ with $p,r$ are even (hence $q$ is also even). Hence
\[
N_{\partial P} (\phi, k) = \sum_{\substack{p,q,r \geq 0 \\ p+q+r = 2k \\ p \text{ even}, \; r \text{ even}}} \left( \frac{p}{2} \right)^{2m+1} q^{2n+1} \left( \frac{r}{2} \right)
= \frac{1}{2^{2m+2}} R_{m,n}^0 (2k).
\]
Thus, for even $k$, $R_{m,n}^0 (k)$ is an odd polynomial of degree $\leq 2m+2n+5$. An elementary estimate gives us that the degree is exactly $2m+2n+5$. For instance, for any positive integer $k$ and positive integers $u,v,w$ with $u+v+w=k$, we have $(k+u,2k+2v,k+w) \in 8k \; \partial P$. For given $k$ there are $\binom{k-1}{2}$ such points, and for each we have
\[
\phi(k+u, 2k+2v, k+w) = (k+u)^{2m+1} (2k+2v)^{2n+1} (k+2) > k^{2m+2n+3},
\]
so that
\[
N_{\partial P} (\phi, 8k) = \sum_{v \in \Z^3 \cap 8k \partial P} \phi(v)
\geq \sum_{u,v,w} \phi(k+u,2k+2v,k+w)
> \binom{k-1}{2} k^{2m+2n+3}.
\]
Hence $\deg N_{\partial P} (\phi, k) \geq 2m+2n+5$ and thus the degree must be exactly $2m+2n+5$. We have proved that, for even $k$, $R_{m,n}^0 (k)$ is an odd polynomial of degree $2m+2n+5$.

Now we consider $R_{m,n}^0 (k)$ for odd $k$. So let $k = 2\kappa+1$ and consider
(following Norbury)
\[
N_{P^0} ( \phi, \kappa+1 ) - N_{P} (\phi, \kappa)
= N_{P^0} \left( \phi, \frac{k+1}{2} \right) - N_{P} \left( \phi, \frac{k-1}{2} \right).
\]
The first sum is a sum of $\phi(x,y,z)$ over lattice points $(x,y,z)$ in the interior of $(\kappa+1) P$, hence over all integers $x,y,z > 0$ such that $2x+y+2z < 2(\kappa + 1) = k+1$. The second sum is a sum of $\phi(x,y,z)$ over lattice points $(x,y,z)$ in $\kappa P$, hence over all integers $x,y,z \geq 0$ such that $2x+y+2z \leq 2\kappa = k-1$. After subtracting (and recalling that $\phi$ vanishes when any of $x,y,z$ are zero), we are only left with $x,y,z>0$ such that $2x+y+2z = k$. Thus
\begin{align*}
N_{P^0} \left( \phi, \frac{k+1}{2} \right) - N_{P} \left( \phi, \frac{k-1}{2} \right)
&= \sum_{\substack{ x,y,z > 0 \\ 2x+y+2z = k}} x^{2m+1} y^{2n+1} z \\
&= \sum_{\substack{ p,q,r \geq 0 \\ p+q+r=k \\ p \text{ even}, \; r \text{ even}}} \left( \frac{p}{2} \right)^{2m+1} y^{2n+1} \left( \frac{r}{2} \right)
= \frac{1}{2^{2m+2}} R_{m,n}^0 (k).
\end{align*}
Applying proposition \ref{prop:Brion_Vergne}, $N_{P} (\phi,t)$ and $N_{P^0}(\phi,t)$ are rational polynomials in $t$ of degree $2m+2n+6$, and $N_{P^0}(\phi,t) = N_{P}(\phi,-t)$. Thus, still taking $k=2\kappa+1$ to be odd,
\[
\frac{1}{2^{2m+2}} R_{m,n}^0(k)
= N_{P^0} \left( \phi, \frac{k+1}{2} \right) - N_{P} \left( \phi, \frac{k-1}{2} \right)
= N_{P} \left( \phi, \frac{-k-1}{2} \right) - N_{P} \left( \phi, \frac{k-1}{2} \right).
\]
This is evidently an odd function of $k$, and it is a polynomial in $k$ of degree $\leq 2m+2n+5$. A similar estimate as above in the case $k$ even shows that the degree  is exactly $2m+2n+5$. 

We have now shown that $R_{m,n}^0 (k)$ is a rational odd quasi-polynomial of degree $2m+2n+5$. Norbury in \cite[lemma 1]{Norbury10_counting_lattice_points} showed that $R_{m,n} (k)$ has the same property. Thus $R_{m,n}^1 = R_{m,n} - R_{m,n}^0$ is a rational odd quasi-polynomial of degree $\leq 2m+2n+5$, depending on the parity of $k$. An estimate of the sort used above shows that $R_{m,n}^1$ has degree exactly $2m+2n+5$.

Now consider the various $B$ functions. For $B_{m,n}(k)$ we compute, 
\begin{align*}
B_{m,n}(k)
&= \sum_{\substack{p,q,r \geq 0 \\ p+q+r = k \\ r \text{ even}}} \bar{p} \; \bar{q} \; p^{2m} q^{2m} r
= \sum_{\substack{p,q,r \geq 0 \\ p+q+r = k \\ r \text{ even}}} (p + \delta_{p,0}) (q + \delta_{q,0}) p^{2m} q^{2n} r \\
&= R_{m,n} (k) + \delta_{n,0} S_m (k) + \delta_{m,0} S_n (k) + \delta_{m,0} \delta_{n,0} \sum_{\substack{r=k \\ r \text{ even}}} r
\end{align*}
The last sum is $k$, if $k$ is even, and $0$ if $k$ is odd. 

When $m=n=0$ are zero, $B_{0,0}(k)$ is given by $R_{0,0} (k) + 2S_0 (k) + k$ for $k$ even and $R_{0,0} (k) + 2 S_0 (k)$ for $k$ odd, where $\deg S_0 = 3 < 5 = \deg R_{0,0}$. When $m=0$ and $n \neq 0$ we have $B_{0,n} (k) = R_{0,n}(k) + S_n (k)$, where $\deg S_n = 2n+3 < 2n+5 = \deg R_{0,n}$. The case $m \neq 0$ and $n=0$ is similar. If $m,n$ are both nonzero, then $B_{m,n} (k) = R_{m,n} (k)$.

In all cases then $B_{m,n}(k)$ is given by $R_{m,n}(k)$, plus a lower-degree odd rational quasi-polynomial (possibly zero) depending on the parity of $k$. Hence $B_{m,n}(k)$ is a rational odd quasi-polynomial of degree $2m+2n+5$.

We can perform a similar computation for $B_{m,n}^0 (k)$, expressing it as $R_{m,n}^0(k)$ plus lower degree terms; and similarly again for $B_{m,n}^1(k)$. We conclude that both are also odd rational quasi-polynomials of degree $2m+2n+5$ depending on the parity of $k$.

As all the functions are defined by summing positive polynomials on positive integers, all highest degree coefficients must be positive.
\end{proof}

The first few polynomials among the $A_m$ and $B_{m,n}$ are
\begin{align*}
A_0 (k) &= \left\{ \begin{array}{ll}
\frac{1}{12} k^3 + \frac{2}{3} k & k \text{ even} \\
\frac{1}{12} k^3 - \frac{1}{12} k & k \text{ odd} \end{array} \right. \\
A_1 (k) &= \left\{ \begin{array}{ll}
\frac{1}{40} k^5 - \frac{1}{6} k^3 + \frac{4}{15} k & k \text{ even} \\
\frac{1}{40} k^5 - \frac{1}{6} k^3 + \frac{17}{120} k & k \text{ odd} \end{array} \right. \\
A_2 (k) &= \left\{ \begin{array}{ll}
\frac{1}{84} k^7 - \frac{1}{6} k^5 + \frac{2}{3} k^3 - \frac{16}{21} k & k \text{ even} \\
\frac{1}{84}  k^7 - \frac{1}{6} k^5 + \frac{2}{3} k^3 - \frac{43}{84} k & k \text{ odd} \end{array} \right. \\
A_3 (k) &= \left\{ \begin{array}{ll}
\frac{1}{144} k^9 - \frac{1}{6} k^7 + \frac{7}{5} k^5 - \frac{40}{9} k^3 + \frac{64}{15} k & k \text{ even} \\
\frac{1}{144} k^9 - \frac{1}{6} k^7 + \frac{7}{5} k^5 - \frac{40}{9} k^3 + \frac{769}{240} k & k \text{ odd} \end{array} \right.
\end{align*}

\begin{align*}
B_{0,0}(k) &= \left\{ \begin{array}{ll}
\frac{1}{240} k^5 + \frac{1}{8} k^3 + \frac{13}{30} k & k \text{ even} \\
\frac{1}{240} k^5 + \frac{1}{8} k^3 - \frac{31}{240} k & k \text{ odd} \end{array} \right. \\
B_{0,1}(k) &= \left\{ \begin{array}{ll}
\frac{1}{1680} k^7 + \frac{7}{480} k^5 - \frac{7}{60} k^3 + \frac{41}{210} k & k \text{ even} \\
\frac{1}{1680} k^7 + \frac{7}{480} k^5 - \frac{7}{60} k^3 + \frac{341}{3360} k & k \text{ odd} \end{array} \right. \\
B_{0,2}(k) &= \left\{ \begin{array}{ll}
\frac{1}{6048} k^9 + \frac{1}{144} k^7 - \frac{169}{1440} k^5 + \frac{185}{378} k^3 - \frac{17}{30} k & k \text{ even} \\
\frac{1}{6048} k^9 + \frac{1}{144} k^7 - \frac{169}{1440} k^5 + \frac{185}{378} k^3 - \frac{91}{240} k & k \text{ odd} \end{array} \right. \\
B_{1,1}(k) &= \left\{ \begin{array}{ll}
\frac{1}{20160} k^9 - \frac{1}{840} k^7 + \frac{1}{96} k^5 - \frac{23}{630} k^3 + \frac{3}{70} k & k \text{ even} \\
\frac{1}{20160} k^9 - \frac{1}{840} k^7 + \frac{1}{96} k^5 - \frac{23}{630} k^3 + \frac{61}{2240} k & k \text{ odd} \end{array} \right. \\
\end{align*}

Recall $A_m(k)$ was originally defined as a function $\N_0 \To \N_0$. We have now showed it is a quasi-polynomial. Hence it naturally extends to a function $\Z \To \Z$. We can similarly extend all the $B, R, S$ functions.
\begin{lem}
\label{lem:Am_tildesum}
For any non-negative integers $m,n$ and any integer $k$,
\[
A_m (k) = \widetilde{\sum_{\substack{p,q \geq 0 \\ p+q = k \\ q \text{ even}}}} \bar{p} \; p^{2m} q
\quad \text{and} \quad
B_{m,n}(k) = \widetilde{\sum_{\substack{p,q,r \geq 0 \\ p+q+r=k \\ r \text{ even}}}} \bar{p} \; \bar{q} \; p^{2m} q^{2m} r.
\]
\end{lem}
Recall from section \ref{sec:non-parallel_recursion} that the tilde over the summation means that if $k \geq 0$, interpret the sum as is; if $k < 0$, then take the negative of the sum over $p+q = -k$ or $p+q+r = -k$.

\begin{proof}
We prove the result for $A_m (k)$; $B_{m,n}(k)$ is similar.
When $k>0$, this is true by definition. When $k=0$ we know $A_m(k) = 0$ as $A_m$ is odd, and the summation consists only of the term with $p=q=0$, also giving zero.

When $k<0$ we have, since $A_m$ is odd,
\[
\widetilde{\sum_{\substack{p,q \geq 0 \\ p+q=k \\ q \text{ even}}}} \bar{p} \; p^{2m} q = - \sum_{\substack{p,q \geq 0 \\ p + q = -k \\ q \text{ even}}} \bar{p} \; p^{2m} q
= - A_m(-k) = A_m (k)
\]
as desired.
\end{proof}

\begin{lem}
\label{lem:Am_odd_even}
For any integer $m \geq 0$,
\[
A_m (x+y) + A_m (x-y)
\]
is a quasi-polynomial function of $x$ and $y$, depending on the parity of $x$ and $y$, odd in $x$ and even in $y$, of degree $2m+3$, with all coefficients of highest total degree being positive.
\end{lem}

\begin{proof}
Evenness and oddness follow immediately from oddness of $A_m$. Since $A_m(k)$ is quasi-polynomial depending on the parity of $k$, the given function is quasi-polynomial depending on the parity of $x$ and $y$. Since $A_m$ has degree $2m+3$, with positive coefficients in highest degree, and since terms of highest total degree in $(x+y)^{2m+3} + (x-y)^{2m+3}$ also have degree $2m+3$, we conclude that $A_m (x+y) + A_m (x-y)$ has the desired properties.
\end{proof}

\subsection{Polynomiality for non-boundary-parallel counts}
\label{sec:Ngn_polynomiality}

We now have enough to prove polynomiality of $\widehat{N}_{g,n}$. Essentially, lemma \ref{lem:As_and_Bs} above allows us to compute the summations in the recursion of corollary \ref{cor:HatNgn_recursion} for $\widehat{N}_{g,n}$, and we have computed enough initial values.

For instance, in section \ref{sec:applying_small_recursion} we found (equation \eqref{eqn:N11_almost}) an expression for $\widehat{N}_{1,1}(b)$ for $b$ even, which we now recognise as
\[
\widehat{N}_{1,1}(b) = \frac{1}{4 \bar{b}} A_0 (b) + \frac{1}{4}.
\]
Since $A_0 (b) = \frac{1}{12} b^3 + \frac{2}{3} b$ for even $b$ we immediately obtain a closed expression
\[
\widehat{N}_{1,1}(b) = \frac{1}{48} b^2 + \frac{5}{12} \quad \text{for $b \neq 0$ even},
\]
proving equation (\ref{eqn:N11}) in theorem \ref{thm:N_formulas}.

To compute further closed expressions for $\widehat{N}_{g,n}$, one can proceed in a similar fashion. In fact, for subsequent calculations there are no Kronecker deltas. For instance, we can compute $\widehat{N}_{0,4}$ as follows. From the recursion (corollary \ref{cor:HatNgn_recursion}) for $\widehat{N}_{g,n}$, we have 
\begin{align*}
b_1 \widehat{N}_{0,4}(b_1, b_2, b_3, b_4)
= \sum_{j=2}^4 \frac{1}{2} &\left( \sum_{\substack{i,m \geq 0 \\ i+m = b_1 + b_j \\ m \text{ even}}} \bar{i} \; m \; \widehat{N}_{0,3}(i, b_2, \ldots, \widehat{b_j}, \ldots, b_4) \right. \\
&\left. + \widetilde{\sum_{\substack{i,m \geq 0 \\ i+m = b_1 - b_j \\ m \text{ even}}}} \bar{i} \; m \; \widehat{N}_{0,3}(i, b_2, \ldots, \widehat{b_j}, \ldots, b_4) \right)
\end{align*}
Now each sum on the right hand side contains $\widehat{N}_{0,3}$, which is $1$ when its arguments have even sum and $0$ otherwise. Each summation is over non-negative $i,m$ such that $i+m = b_1 \pm b_j$ and $m$ is even; if these conditions are satisfied, then modulo 2 we have $0 \equiv m \equiv b_1 + b_j + i \equiv i + (b_2 + \cdots + \widehat{b_j} + \cdots + b_4)$ so all $\widehat{N}_{0,3}$ occurring in the sums are always 1. We then obtain (using lemma \ref{lem:Am_tildesum})
\begin{align*}
2 b_1 \widehat{N}_{0,4} ({\bf b})
&= \sum_{j=2}^4 \left( \sum_{\substack{i,m \geq 0 \\ i+m = b_1 + b_j \\ m \text{ even}}} \bar{i} \; m + \widetilde{\sum_{\substack{i,m \geq 0 \\ i+m = b_1 - b_j \\ m \text{ even}}}} \bar{i} \; m \right) \\
&= A_0 (b_1 + b_2) + A_0 (b_1 - b_2) + A_0 (b_1 + b_3) + A_0 (b_1 - b_3) + A_0 (b_1 + b_4) + A_0 (b_1 - b_4).
\end{align*}

We will compute $\widehat{N}_{0,4}(b_1, b_2, b_3, b_4)$ assuming that $b_1, b_2$ are even and $b_3, b_4$ are odd; when $b_1, b_2, b_3, b_4$ have different parities a similar method will work. Recall $A_0 (k) = \frac{1}{12} k^3 + \frac{2}{3} k$ when $k$ is even and $\frac{1}{12} k^3 - \frac{1}{12} k$ when $k$ is odd. By lemma \ref{lem:Am_odd_even}, each $A_0(b_1 + b_i) + A_0(b_1 - b_i)$ is odd in $b_1$ and even in $b_i$; explicitly
\[
A_0 (b_1 + b_2) + A_0 (b_1 - b_2) 
= \frac{1}{6} b_1^3 + \frac{1}{2} b_1 b_2^2 + \frac{4}{3} b_1,
\]
and
\[
A_0 (b_1 + b_3) + A_0 (b_1 - b_3) 
= \frac{1}{6} b_1^3 + \frac{1}{2} b_1 b_3^2 - \frac{1}{6} b_1
\]
with a similar calculation for $A_0(b_1 \pm b_4)$. Putting these together, we obtain
\[
\widehat{N}_{0,4} = \frac{1}{4}(b_1^2 + b_2^2 + b_3^2 + b_4^2) + \frac{1}{2}.
\]
Completing the calculation for other parities of $(b_1, b_2, b_3, b_4)$ we obtain equation \eqref{eqn:N04} in theorem \ref{thm:N_formulas}.

A similar method can be used in general to prove polynomiality of $\widehat{N}_{g,n}$.

\begin{thm}
\label{thm:N_polynomiality}
For any integers $g \geq 0$ and $n \geq 1$ except $(g,n) = (0,1)$ or $(0,2)$, and ${\bf b} \neq {\bf 0}$, $\widehat{N}_{g,n}(b_1, \ldots, b_n)$ is a symmetric quasi-polynomial over $\Q$ in $b_1^2, \ldots, b_n^2$ of degree $3g-3+n$, depending on the parity of $b_1, \ldots, b_n$, with all highest-degree coefficients positive.
\end{thm}
At this point, we only assert that, for each such $(g,n)$, each polynomial in the quasi-polynomial $\widehat{N}_{g,n}$ has the same degree and positive highest-degree coefficients. But in fact, we will see in section \ref{sec:volume_moduli} that, for given $(g,n)$, all the nonzero polynomials in $\widehat{N}_{g,n}$ have identical terms of highest total degree.

\begin{proof} 
We prove the result by induction on the complexity $-\chi = 2g+n-2$ of the surface. For $-\chi = 1$, i.e. $(g,n) = (1,1)$ and $(0,3)$, we have proved the result directly; we consider $\widehat{N}_{g,n}$ with complexity at least $2$, supposing the result holds for all surfaces of smaller positive complexity. 

Fix a parity for each $b_1, \ldots, b_n$; we must show that $\widehat{N}_{g,n}({\bf b})$ is given by a polynomial in $b_1, \ldots, b_n$ which is even in every variable and with total degree $3g-3+n$ in $b_1^2, \ldots, b_n^2$, with all top-degree coefficients positive.

Consider the recursion for $\widehat{N}_{g,n}({\bf b})$ in corollary \ref{cor:HatNgn_recursion}. Each $\widehat{N}$ in the recursion is of the form $\widehat{N}_{g-1,n+1}$, $\widehat{N}_{g,n-1}$, or $\widehat{N}_{g',n'}$ where $g' \leq g$ and $n' \leq n-2$, but $(g',n') \neq (0,1)$ or $(0,2)$. So each term corresponds to a surface with strictly smaller complexity. Further, neither $\widehat{N}_{0,1}$ nor $\widehat{N}_{0,2}$ ever appears; every term appearing has positive complexity. So we know that every $\widehat{N}$ appearing in the recursion is given by a quasi-polynomial with all the claimed properties.

After expanding out the sum $\sum_{j=2}^n$ in the second line of the recursion, and the sum over $g_1 + g_2 = g$ and $I \sqcup J = \{2, \ldots, n\}$ in the third line, we can express $b_1 \widehat{N}_{g,n}({\bf b})$ as a finite collection of sums, where each sum is either over $i,m \geq 0$ such that $i+m = b_1 \pm b_j$ (for some $j$) and $m$ is even, or over $i,j,m \geq 0$ such that $i+j+m=b_1$ and $m$ is even. In the first case, each sum is, up to a constant, of the form $\bar{i} m \widehat{N}_{g',n'}(i, b_I)$; and in the second case, of the form $\bar{i} \; \bar{j} \; m \widehat{N}_{g',n'}(i, j, b_I)$ or $\bar{i} \; \bar{j} \; m \widehat{N}_{g',n'} (i, b_I) \widehat{N}_{g'',n''} (j, b_J)$. In both cases, $(g',n')$ or $(g'',n'')$ has positive complexity that is lower than $(g,n)$, and $b_I$ denotes some subset of $b_2, \ldots, b_n$. Either way, the $\widehat{N}$ factors are all even functions of all their variables. Also, all sums are equal to the sums we obtain by writing a tilde over them; and all top-degree coefficients are positive. Thus each summation is of one of the following two types, for some function $p$ that is a positive constant times an $\widehat{N}$ or a product of $\widehat{N}$'s:
\[
\text{Type 1: } \widetilde{\sum_{\substack{i,m \geq 0 \\ i+m = b_1 \pm b_j \\ m \text{ even}}}} \bar{i} \; m \; p(i, b_I),
\quad
\text{Type 2: } \widetilde{\sum_{\substack{i,j,m \geq 0 \\ i+j+m = b_1 \\ m \text{ even}}}} \bar{i} \; \bar{j} \; m \; p(i,j,b_I).
\]
Having fixed the parity of $b_1, \ldots, b_n$, we now consider the possible parity of $i$ and $j$ occurring in the sums. In a summation of type 1, the sum is over $i,m$ such that $i+m = b_1 \pm b_j$ and $m$ is even. Here the parity of $b_1 \pm b_j$ is fixed. Thus the parity of $i$ is fixed. In a summation of type 2, the sum is over all $i,j,m$ such that $i+j+m = b_1$ and $m$ is even; again $b_1$ has fixed parity. So the parity of $i+j$ is fixed; hence there are two possibilities for the parity of $i$ and $j$, and we can split the summation into two separate summations where the parity of all variables is fixed.

In any case, we are able to express $b_1 \widehat{N}_{g,n}({\bf b})$ as a finite sum of terms, where each term is a summation of type 1 or 2, with the parities of each variable fixed. In each summation $p(i, b_I)$ or $p(i,j,b_I)$ is a polynomial with top-degree coefficients positive and even in all its variables. Taking each term of each polynomial separately, and factoring out variables not involved in the summation, each term of type 1 becomes a finite collection of sums of one of the forms
\begin{align*}
q(b_I) \widetilde{\sum_{\substack{i,m \geq 0 \\ i+m = k \\ i \text{ even}, m \text{ even}}}} \bar{i} \; i^{2a} m &= \left\{ \begin{array}{ll} q(b_I) A_a (k) & k \text{ even} \\ 0 & k \text{ odd,} \end{array} \right. \\
q(b_I) \widetilde{\sum_{\substack{i,m \geq 0 \\ i+m = k \\ i \text{ odd}, m \text{ even}}}} \bar{i} \; i^{2a} m &= \left\{ \begin{array}{ll} 0 & k \text{ even} \\ q(b_I) A_a (k) & k \text{ odd,} \end{array} \right.
\end{align*}
where $q(b_I)$ is a constant (for terms of highest degree, a positive constant) times a product of even powers of $b_i$'s. We determined in lemma \ref{lem:As_and_Bs} that $A_a (k)$ is odd in $k$; and since the parity of $k$ is fixed, $A_a (k)$ is a polynomial in $k$. Every time we see $A_a$ arising, it is from the second line of the recursion, hence appears in the form $A_a (b_1 \pm b_j)$, both terms appearing together; by lemma \ref{lem:Am_odd_even}, the result is odd in $b_1$, even in $b_j$ (and indeed all other $b_i$), with top-degree coefficients positive.

Each term of type 2 becomes a finite collection of sums of one of the forms
\begin{align*}
q(b_I) \widetilde{\sum_{\substack{i,j,m \geq 0 \\ i+j+m = k \\ i,j,m \text{ even}}}} \bar{i} \; \bar{j} \; i^{2a} j^{2b} m = \left\{ \begin{array}{ll} q(b_I) B_{a,b}^0(k) & k \text{ even} \\ 0 & k \text{ odd}, \end{array} \right. \\
q(b_I) \widetilde{\sum_{\substack{i,j,m \geq 0 \\ i+j+m = k \\ j \text{ odd}, \; i,m \text{ even}}}} \bar{i} \; \bar{j} \; i^{2a} j^{2b} m = \left\{ \begin{array}{ll} 0 & k \text{ even}, \\ q(b_I) B_{a,b}^0(k) & k \text{ odd} \end{array} \right. \\
q(b_I) \widetilde{\sum_{\substack{i,j,m \geq 0 \\ i+j+m = k \\ i \text{ odd}, \; j,m \text{ even}}}} \bar{i} \; \bar{j} \; i^{2a} j^{2b} m = \left\{ \begin{array}{ll} 0 & k \text{ even}, \\ q(b_I) B_{a,b}^1(k) & k \text{ odd} \end{array} \right. \\
q(b_I) \widetilde{\sum_{\substack{i,j,m \geq 0 \\ i+j+m = k \\ i,j \text{ odd}, \; m \text{ even}}}} \bar{i} \; \bar{j} \; i^{2a} j^{2b} m = \left\{ \begin{array}{ll} q(b_I) B_{a,b}^1(k) & k \text{ even} \\ 0 & k \text{ odd}. \end{array} \right.
\end{align*}
Here again $q(b_I)$ is a constant (positive for highest degree terms) times a product of even powers of $b_i$'s. From lemma \ref{lem:As_and_Bs}, each $B_{a,b}(k)$ is odd in $k$, and since the parity of $k$ is fixed, $B_{a,b}(k)$ is a polynomial in $k$. Every time we see $B_{a,b}$ arising, it appears in the form $B_{a,b}(b_1)$, hence the result is odd in $b_1$ and even in all other $b_i$.

After collecting terms and simplifying all $A_a$'s and $B_{a,b}$'s, the result for $b_1 \widehat{N}_{g,n}({\bf b})$ is divisible by $b_1$, odd in $b_1$, and even in all the other variables. Hence $\widehat{N}_{g,n}({\bf b})$ is an even polynomial in all the variables as desired.

We can also keep track of degrees. Let us keep track of the degrees of the variables rather than their squares, so we will show $\widehat{N}_{g,n}$ has degree $6g-6+2n$. In the recursion, the first term has $\widehat{N}_{g-1,n+1}$, which has degree $6g+2n-10$: it is multiplied by $\bar{i} \; \bar{j} m$ and all summations are of $B_{a,b}$'s, leading to a total degree of $6g-5+2n$. The terms in the second line have $\widehat{N}_{g,n-1}$, which has degree $6g-8+2n$: it is multiplied by $\bar{i} \; m$ and the summations give $A_a$ polynomials, leading to a total degree of $6g-5+2n$; the summation over $j$ does not alter the degree. The terms in the third line have $\widehat{N}_{g_1, |I|+1} \widehat{N}_{g_2, |J|+1}$ which has degree $6(g_1 + g_2) - 12 + 2|I| + 2|J| + 4 = 6g-10+2n$; we then multiply by $\bar{i} \; \bar{j} \; m$ and sum, obtaining $B_{a,b}$ polynomials and a total degree of $6g-5+2n$. As all top-degree terms positive, there can be no cancellation of terms and the right hand side of the recursion is of degree $6g-5+2n$, with all highest-degree coefficients positive. Dividing by $b_1$ then gives the degree of $\widehat{N}_{g,n}$ as $6g-6+2n$.
\end{proof}

We have now proved theorem \ref{thm:intro_N_polynomiality}.

\subsection{Comparison with lattice count polynomials}
\label{sec:comparison_with_lattice}

Norbury in \cite{Norbury10_counting_lattice_points} derives a recursion for counts of lattice points in the moduli space of curves, which correspond to ribbon graphs without degree 1 vertices. We denote the number of such ribbon graphs with prescribed genus, number of boundary components, and boundary lengths, by $\mathfrak{N}_{g,n}(b_1, \ldots, b_n)$. Writing equation (5) of \cite{Norbury10_counting_lattice_points} in our notation, 
these lattice point counts satisfy the recursion
\begin{align*}
b_1 &\mathfrak{N}_{g,n}(b_1, \ldots, b_n) = \sum_{\substack{i,j,m \geq 0 \\ i+j+m = b_1 \\ m \text{ even}}} \frac{1}{2} \; i \; j \; m \; \mathfrak{N}_{g-1,n+1} (i,j,b_2, \ldots, b_n) \\
& \quad + \sum_{j=2}^n \frac{1}{2} \left( \sum_{\substack{i,j \geq 0 \\  i+m = b_1 + b_j \\ m \text{ even}}} i \; m \; \mathfrak{N}_{g,n-1}(i,b_2, \ldots, \widehat{b_j}, \ldots, b_n) + \widetilde{\sum_{\substack{i,m \geq 0 \\ i+m = b_1 - b_j \\ m \text{ even}}}} i \; m \; \mathfrak{N}_{g,n-1} (i, b_2, \ldots, \widehat{b_j}, \ldots, b_n) \right) \\
&\quad + \sum_{\substack{g_1 + g_2 = g \\ I \sqcup J = \{2, \ldots, n\} \\ \text{No discs or annuli}}} \sum_{\substack{i,j,m \geq 0 \\ i+j+m = b_1 \\ m \text{ even}}} \frac{1}{2} \; i \; j \; m \; \mathfrak{N}_{g_1, |I|+1} (i, b_I) \; \mathfrak{N}_{g_2, |J|+1} (j, b_J).
\end{align*}
This recursion is identical to the recursion on $\widehat{N}_{g,n}$ (corollary \ref{cor:HatNgn_recursion}), with the bars dropped from $i$'s and $j$'s.

The initial conditions for the recursions on $\mathfrak{N}_{g,n}$ and $\widehat{N}_{g,n}$ are
\[
\begin{array}{cc}
\mathfrak{N}_{0,3}(b_1, b_2, b_3) = 1 & \widehat{N}_{0,3}(b_1, b_2, b_3) = 1 \\
\mathfrak{N}_{1,1}(b_1) = \frac{1}{48} b_1^2 - \frac{1}{12} & \widehat{N}_{1,1}(b_1) = \frac{1}{48} b_1^2 + \frac{5}{12}
\end{array}
\]
(Both $(g,n)=(1,1)$ expressions are for even $b_1$; they are both zero when $b_1$ is odd.) Norbury's proof that each $\mathfrak{N}_{g,n}(b_1, \ldots, b_n)$ is a quasi-polynomial, depending on the parity of $b_1, \ldots, b_n$, of degree $3g-3+n$ in $b_1^2, \ldots, b_n^2$, is analogous to our $\widehat{N}_{g,n}(b_1, \ldots, b_n)$; indeed we adapted his proof above. Thus $\mathfrak{N}$ and $\widehat{N}$ agree in initial cases in highest-degree terms. As their recursions are also similar, it is now not too surprising that they should have the same highest degree terms.

\begin{prop}
\label{prop:agreement_in_highest_degree}
Let $(g,n) \neq (0,1)$ or $(0,2)$ and fix the parity of $b_1, \ldots, b_n$. Then the corresponding polynomials in the quasi-polynomials $\mathfrak{N}_{g,n}(b_1, \ldots, b_n)$ and $\widehat{N}_{g,n}(b_1, \ldots, b_n)$ have identical terms of highest total degree.
\end{prop}

\begin{proof}
We proceed by induction on the complexity $-\chi = 2g+n-2$.

Consider the proofs of quasi-polynomiality of $\widehat{N}_{g,n}$ and $\mathfrak{N}_{g,n}$.

We saw that, having fixed the parity of each $b_1, \ldots, b_n$, the expression for $b_1 \widehat{N}_{g,n}(b_1, \ldots, b_n)$ in the recursion can be written as a sum of terms, where each term is a positive constant, multiplied by a product of even powers of $b_i$'s, multiplied by some $A_a (b_1 \pm b_j)$ or $B_{a,b}^0(b_1)$ or $B_{a,b}^1(b_1)$. Each $A_a$ term occurs in a pair where we can factor out $A_a (b_1 + b_j) + A_a (b_1 - b_j)$; these terms are then collected together, and we obtain the desired polynomial.

Similarly, the expression for $b_1 \mathfrak{N}_{g,n}(b_1, \ldots, b_n)$ can be written as a sum of terms, where each term is a constant, multiplied by a product of even powers of $b_i$'s, multiplied by some $S_a (b_1 \pm b_j)$ or $R_{a,b}^0(b_1)$ or $R_{a,b}^1(b_1)$. Each $S_a$ occurs in a pair where we can factor out $S_a (b_1 + b_j) + S_1 (b_1 - b_j)$; collecting terms, we obtain the desired polynomial.

Now, suppose by induction that all $\widehat{N}$ and $\mathfrak{N}$ of lower complexity have polynomials with identical terms of highest degree. Then, from the similarity of the recursions for $\widehat{N}$ and $\mathfrak{N}$, the expression obtained for $b_1 \widehat{N}_{g,n}(b_1, \ldots, b_n)$ in terms of $A_a (b_1 \pm b_j)$ and $B_{a,b}(b_1)$ agrees, in its highest total degree terms, with the expression for $b_1 \mathfrak{N}_{g,n}(b_1, \ldots, b_n)$, upon replacing each $A_a (b_1 \pm b_j)$ with $S_a (b_1 \pm b_j)$, each $B_{a,b}^0(b_1)$ with $R_{a,b}^0(b_1)$ and each $B_{a,b}^1(b_1)$ with $R_{a,b}^1(b_1)$.

From lemma \ref{lem:As_and_Bs}, we know that $A_a (k)$ and $S_a (k)$ agree in their leading terms; and similarly $B_{a,b}0(k), B_{a,b}^1(k)$ and $R_{a,b}^0(k),R_{a,b}^1$ respectively agree in their leading terms. Hence after all terms are expanded out and the dust settles, the polynomials obtained for $\mathfrak{N}_{g,n}(b_1, \ldots, b_n)$ and $\widehat{N}_{g,n}(b_1, \ldots, b_n)$ agree in top degree.
\end{proof}

\subsection{Volume of moduli space and intersection numbers}
\label{sec:volume_moduli}

We have shown that the quasi-polynomials $\widehat{N}_{g,n}(b_1, \ldots, b_n)$ and $\mathfrak{N}_{g,n}(b_1, \ldots, b_n)$ agree in highest degree terms. Now we apply \cite[theorem 3]{Norbury10_counting_lattice_points}, where Norbury shows that 
\[
\mathfrak{N}_{g,n}(b_1, \ldots, b_n) = \frac{1}{2} V_{g,n}(b_1, \ldots, b_n) + \text{lower order terms}.
\] 
Here $V_{g,n}(b_1, \ldots, b_n)$ is the volume polynomial of Kontsevich \cite{Kontsevich_Intersection}. In fact, the Kontsevich volumes also agree with the highest order terms in the Weil--Petersson volume polynomials of Mirzakhani up to a simple normalisation~\cite{mirzakhani}.
Moreover, the coefficients of $V_{g,n}$ are, up to a combinatorial factor, the intersection numbers on the moduli space of curves.

Note $V_{g,n}$ is a \emph{polynomial}, not quasi-polynomial. It immediately follows that the polynomials defining each quasi-polynomial $\widehat{N}_{g,n}(b_1, \ldots, b_n)$ all agree in highest degree. We then have the following.
\begin{thm}
\label{thm:volume_polynomials}
For $(g,n) \neq (0,1), (0,2)$, the polynomials defining $\widehat{N}_{g,n}(b_1, \ldots, b_n)$ all agree in their terms of highest degree, and these agree with the highest degree terms of $V_{g,n}(b_1, \ldots, b_n)$. Thus
\[
\widehat{N}_{g,n}(b_1, \ldots, b_n) = \frac{1}{2} V_{g,n}(b_1, \ldots, b_n) + \text{lower order terms}.
\] 
Moreover, for $d_1 + \cdots + d_n = 3g-3+n$, the coefficient $c_{d_1, \ldots, d_n}$ of $b_1^{2d_1} \cdots b_n^{2d_n}$ in any polynomial of the quasi-polynomial $\widehat{N}_{g,n}(b_1, \ldots, b_n)$ is given by
\[
c_{d_1, \ldots, d_n} = \frac{1}{2^{5g-6+2n} d_1! \cdots d_n!} \langle \psi_1^{d_1} \cdots \psi_n^{d_n}, \overline{\mathcal{M}}_{g,n} \rangle.
\]
\qed
\end{thm}

We have now proved theorem \ref{thm:intro_intersection_numbers}.

\subsection{Polynomiality for general curve counts}
\label{sec:polynomiality_general}

We now use the polynomiality of $\widehat{N}_{g,n}$ to prove polynomiality for $G_{g,n}$. It is now no more difficult than our computation of $G_{0,3}$ in section \ref{sec:general_pants}; in fact we developed all we need there.

Recall (theorem \ref{thm:G_and_N}) that $G_{g,n}(b_1, \ldots, b_n)$ can be written in terms of $\widehat{N}_{g,n}$:
\begin{align}
G_{g,n}(b_1, \ldots, b_n) 
&= \sum_{\substack{0 \leq a_i \leq b_i \\ a_i \equiv b_i  \!\!\!\! \pmod{2}}}
\binom{b_1}{\frac{b_1 - a_1}{2}} \cdots \binom{b_n}{\frac{b_n - a_n}{2}} \; \bar{a}_1 \cdots \bar{a}_n \; \widehat{N}_{g,n}(a_1, \ldots, a_n). \label{eqn:add_ns_up}
\end{align}
Recall from definition \ref{defn:ps_and_qs} that, for integers $\alpha \geq 0$,
\[
\tilde{P}_\alpha (n) = \sum_{l=0}^n \binom{2n}{n-l} \overline{(2l)} \; (2l)^{2\alpha}, 
\quad \quad \quad
\tilde{Q}_\alpha (n) = \sum_{l=0}^n \binom{2n+1}{n-l} \overline{(2l+1)} (2l+1)^{2 \alpha}.
\]
We also defined $\tilde{p}_\alpha(n)$ and $\tilde{q}_\alpha (n)$ ``without the bars''. We showed (proposition \ref{prop:Ps_and_Qs}) that $\tilde{P}_\alpha(n) = \binom{2n}{n} P_\alpha (n)$, $\tilde{Q}_\alpha (n) = \binom{2n}{n} Q_\alpha (n)$, where $P_\alpha, Q_\alpha$ are integer polynomials of degree $\alpha + 1$; and similarly for $\tilde{p}_\alpha (n)$ and $\tilde{q}_\alpha (n)$.

In evaluating the summations in equation \eqref{eqn:add_ns_up}, we can write the even polynomial $\widehat{N}_{g,n}(b_1, \ldots, b_n)$ as a sum of even monomials, and factorise each term into sums of the form
\[
\sum_{\substack{0 \leq a \leq b \\ a \equiv b \!\!\!\! \pmod{2}}} \binom{b}{\frac{b-a}{2}} \bar{a} \; a^{2\alpha}.
\]
When $b$ is even, $b=2n$, we only sum over even $a$, so with $a=2l$ and the sum is $\tilde{P}_\alpha (n)$. When $b$ is odd, $b=2n+1$, we sum over odd $a = 2l+1$ and the sum is $\tilde{Q}_\alpha (n)$. When all $a_i$ are set to zero however, $\widehat{N}_{g,n}({\bf 0}) = 1$, to which the quasi-polynomial for $\widehat{N}_{g,n}$ does not apply; separating out this term, we have a $\tilde{p}_\alpha (n)$ rather than a $\tilde{P}_\alpha (n)$.

\begin{proof}[Proof of theorem \ref{thm:G_polynomiality}]
We may evaluate $G_{g,n}(b_1, \ldots, b_n)$ by simply replacing sums of the above type with functions $\tilde{P}_\alpha$, $\tilde{p}_\alpha$ and $\tilde{Q}_\alpha$. More precisely, each monomial in $\bar{a}_1 \cdots \bar{a}_n \widehat{N}_{g,n}(a_1, \ldots, a_n)$ is of the form $\bar{a}_1 \cdots \bar{a}_n a_1^{2\alpha_1} \cdots a_n^{2\alpha_n}$, and we replace each factor $\bar{a}_i a_i^{2\alpha_i}$ with $\tilde{P}_\alpha (m_i) = \tilde{P}_\alpha (b_i/2)$ or $\tilde{p}_\alpha (m_i)$, when $b_i=2m_i$ is even, and with $\tilde{Q}_\alpha (m_i) = \tilde{Q}_\alpha(\frac{b_i - 1}{2})$ when $b_i = 2m_i + 1$ is odd. Each such substitution replaces a factor of degree $2\alpha_i + 1$ with an expression $\binom{2m_i}{m_i}$ multiplied by a polynomial of degree $\alpha_i +1$ in $b_i$.

After performing this substitution over all $a_i$, each monomial becomes an expression of the form $\binom{2m_1}{m_1} \cdots \binom{2m_n}{m_n}$ multiplied by a product of $P_\alpha(m)$ and $Q_\alpha(m)$, which is a polynomial in $b_1, \ldots, b_n$. Since each monomial has $\sum 2 \alpha_i = 6g-6+2n$, we end up with a polynomial of degree $\sum (\alpha_i + 1) = 3g-3+2n$.

Furthermore, we determined in theorem \ref{thm:N_polynomiality} that each polynomial that appears in the quasi-polynomial $\widehat{N}_{g,n}(b_1, \ldots, b_n)$ had positive highest-degree coefficients. After making the substitutions described above, we still have positive leading coefficients. When we collect terms, the result then is of the form $\binom{2m_1}{m_1} \cdots \binom{2m_n}{m_n} P_{g,n}(b_1, \ldots, b_n)$ where $P_{g,n}$ has positive highest-order coefficients and degree $3g-3+2n$.
\end{proof}

We illustrate the technique with an example, computing $G_{1,1}(b)$; clearly we need only consider $b$ even, $b=2m$. We computed in section \ref{sec:Ngn_polynomiality} equation \eqref{eqn:N11} of theorem \ref{thm:N_formulas},
\[
\widehat{N}_{1,1} (b) = \frac{1}{48} b^2 + \frac{5}{12} \text{ for $b \neq 0$ even,}
\quad \quad
\widehat{N}_{1,1}(0) = 1.
\]
Hence
\begin{align*}
G_{1,1} (b) 
&= \sum_{\substack{0 \leq a \leq b \\ a \equiv b  \!\!\!\! \pmod{2}}} \bar{a} \; \widehat{N}_{1,1}(a)
= \binom{b}{b/2} \widehat{N}_{1,1} (0) + \sum_{\substack{0 < a \leq b \\ a \equiv b  \!\!\!\! \pmod{2}}}
\binom{b}{\frac{b-a}{2}} \; \bar{a} \left( \frac{1}{48} a^2 + \frac{5}{12} \right) \\
&= \binom{2m}{m} + \frac{1}{48} \tilde{p}_1 (m) + \frac{5}{12} \tilde{p}_0 (m) 
= \binom{2m}{m} \left( \frac{1}{12} m^2 + \frac{5}{12} m + 1\right)
\end{align*}
This gives equation \eqref{eqn:G11} in theorem \ref{thm:formulas}.

\section{Differentials and generating functions}
\label{sec:differential_forms_functions}

\subsection{Definitions}
\label{sec:defns}

We now take the curve counts $N_{g,n}(b_1, \ldots, b_n)$ and $G_{g,n}(b_1, \ldots, b_n)$ and string them together into generating functions and differentials.

\begin{defn}[Generating functions]
For any $g \geq 0$ and $n \geq 1$ let
\begin{align*}
f^G_{g,n}(x_1, \ldots, x_n) &= \sum_{\mu_1, \ldots, \mu_n \geq 0} G_{g,n}(\mu_1, \ldots, \mu_n) x_1^{-\mu_1 - 1} \cdots x_n^{-\mu_n - 1} \\
f^N_{g,n}(z_1, \ldots, z_n) &= \sum_{\nu_1, \ldots, \nu_n \geq 0} N_{g,n}(\nu_1, \ldots, \nu_n) z_1^{\nu_1 - 1} \cdots z_n^{\nu_n - 1}.
\end{align*}
\end{defn}

\begin{defn}[Differential forms]
For $g \geq 0$ and $n \geq 1$ let
\begin{align*}
\omega^G_{g,n}(x_1, \ldots, x_n) &=
f^G_{g,n}(x_1, \ldots, x_n) \; dx_1 \otimes \cdots \otimes dx_n \\ 
\omega^N_{g,n}(z_1, \ldots, z_n) &= 
f^N_{g,n}(z_1, \ldots, z_n) \; dz_1 \otimes \cdots \otimes dz_n. 
\end{align*}
\end{defn}
Here $x_1, \ldots, x_n$ are coordinates on $\CP^1$, as are $z_1, \ldots, z_n$. For now these are formal Laurent series, providing a convenient device to arrange the $N_{g,n}$ and $G_{g,n}$. In section \ref{sec:meromorphicity} we show that they are all meromorphic functions and forms.

The differential forms can be regarded as sections of the product bundle
\[
(T^* \CP^1)^{\boxtimes n} = \pi_1^* (T^* \CP^1) \otimes \pi_2^* (T^* \CP^1) \otimes \cdots \otimes (\pi_n^* T^* \CP^1).
\]
This is a vector bundle over $(\CP^1)^n$, where $\pi_i: (\CP^1)^n \To \CP^1$ is projection onto the $i$'th coordinate. For convenience, we write $dz_1 \cdots dz_n$ rather than $dz_1 \otimes \cdots \otimes dz_n$. The $\omega_{g,n}$ are multidifferentials, but we will refer to them simply as differential forms.

We will often regard the coordinates $z$ and $x$ as related by the equation $x = z + \frac{1}{z}$; indeed, as we will see in section \ref{sec:change_of_coords}, $\omega^G_{g,n}$ and $\omega^N_{g,n}$ are essentially \emph{equal}, with this change of coordinates.

\subsection{Small cases}
\label{sec:small_gen_fns}

We can compute the generating functions $f^G_{g,n}, f^N_{g,n}$ and differential forms $\omega^G_{g,n}, \omega^N_{g,n}$ directly in the cases $(g,n) = (0,1)$ or $(0,2)$.

For $(g,n) = (0,1)$, we know $N_{0,1}(0) = 1$ and all other $N_{0,1}(\nu) = 0$. Thus
\[
f^N_{0,1}(z_1) = z_1^{-1}
\quad \text{and} \quad
\omega^N_{0,1} (z_1) = z_1^{-1} \; dz_1.
\]
We have $G_{0,1}(2m) = \frac{1}{m+1} \binom{2m}{m}$ and $G_{0,1}(\mu) = 0$ for odd $\mu$, so $f^G_{0,1}(x_1)$ is a generating function for the Catalan numbers:
\begin{align*}
f_{0,1}^G (x_1) 
&= \sum_{m=0}^\infty G_{0,1}(2m) x_1^{-2m-1} 
= \sum_{m=0}^\infty \frac{1}{m+1} \binom{2m}{m} x_1^{-2m-1}  
= \frac{x_1 - \sqrt{x_1^2 - 4}}{2}.  
\end{align*}
If $z_1$ and $x_1$ are related by $x_1 = z_1 + \frac{1}{z_1}$, then we can write $\frac{x_1 - \sqrt{x_1^2 - 4}}{2} = z_1$, so that $f_{0,1}^G(x_1) = z_1$ and
\[
\omega_{0,1}^G (x_1) = \frac{x_1 - \sqrt{x_1^2 - 4}}{2}  \; dx_1 = z_1 \; dx_1
\]

Turning to $(g,n) = (0,2)$, recall $N_{0,2}(\nu_1, \nu_2) = \delta_{\nu_1, \nu_2} \overline{\nu_1}$ (lemma \ref{lem:Ngn_small_cases}). 
Noting that
\[
\sum_{\nu = 0}^\infty \nu z^{\nu - 1} = \frac{d}{dz} \sum_{\nu = 0}^\infty z^{\nu}
= \frac{d}{dz} \frac{1}{1-z} = \frac{1}{(1-z)^2},
\]
we compute
\begin{align*}
f_{0,2}^N (z_1, z_2)
&= \sum_{\nu_1, \nu_2 \geq 0}
N_{0,2}(\nu_1, \nu_2) \; z_1^{\nu_1 - 1} z_2^{\nu_2 - 1} 
= \sum_{\nu = 0}^\infty \overline{\nu} \; (z_1 z_2)^{\nu - 1}  \\
&= z_1^{-1} z_2^{-1} + \sum_{\nu = 0}^\infty \nu (z_1 z_2)^{\nu - 1}  
= z_1^{-1} z_2^{-1} + \frac{1}{(1-z_1 z_2)^2}.
\end{align*}
Thus
\[
\omega_{0,2}^N (z_1, z_2) = \left( \frac{1}{z_1 z_2} + \frac{1}{(1-z_1 z_2)^2} \right) \; dz_1 dz_2.
\]

Turning to $(g,n)=(0,3)$, we have the following.
\begin{lem}
\label{lem:omegaN03}
\[
f^N_{0,3}(z_1, z_2, z_3) = 
\frac{1 + z_1^4 z_2^4 z_3^4 + 
\sum_{\text{cyc}} (z_1^4 + z_1 z_2 + z_1^3 z_2^3 + z_1^4 z_2^4)
+ \sum_{\text{sym}} (z_1^3 z_2 + z_1^4 z_2^3 z_3 + z_1^4 z_2 z_3)}
{z_1 z_2 z_3 (1-z_1^2)^2 (1-z_2^2)^2 (1-z_3^2)^2}.
\]
\end{lem}
Here the notation $\sum_\text{cyc}$ means to sum over cyclic permutations of $(z_1, z_2, z_3)$ (i.e. $(1,2,3) \mapsto (2,3,1), (3,1,2)$), and $\sum_\text{sym}$ means to sum over all permutations of $(z_1, z_2, z_3)$. Thus terms under $\sum_\text{cyc}$ expand to three terms, while terms under $\sum_\text{sym}$ expand to six terms. 

\begin{proof}
From proposition \ref{prop:N03} we have $N_{0,3}(b_1, b_2, b_3) = \bar{b}_1 \bar{b}_2 \bar{b}_3$ if $b_1 + b_2 + b_3$ is even, and $0$ otherwise. Thus
\begin{align*}
f^N_{0,3}(z_1, z_2, z_3)
&= \sum_{\substack{\nu_1, \nu_2, \nu_3 \geq 0 \\ \nu_1 + \nu_2 + \nu_3 \text{ even}}} 
\overline{\nu_1} \; \overline{\nu_2} \; \overline{\nu_3} \; 
z_1^{\nu_1 - 1} z_2^{\nu_2 - 1} z_3^{\nu_3 - 1} \\
&= \left( \sum_{\nu_1,\nu_2,\nu_3 \text{ even}} + 
\sum_{\substack{\nu_1 \text{ even} \\ \nu_2, \nu_3 \text{ odd}}} +
\sum_{\substack{\nu_2 \text{ even} \\ \nu_3, \nu_1 \text{ odd}}} +
\sum_{\substack{\nu_3 \text{ even} \\ \nu_1, \nu_2 \text{ odd}}} \right)
\overline{\nu_1} \; \overline{\nu_2} \; \overline{\nu_3} \; 
z_1^{\nu_1 - 1} z_2^{\nu_2 - 1} z_3^{\nu_3 - 1}.
\end{align*}
Each of the sums in the last line fixes the parity of all $\nu_i$. We will consider these sums separately. In fact they now factorise and we obtain
\[
f^N_{0,3}(z_1, z_2, z_3) = 
\rho(z_1) \rho(z_2) \rho(z_3)
+ \rho(z_1) \sigma(z_2) \sigma(z_3)
+ \rho(z_2) \sigma(z_3) \sigma(z_1)
+ \rho(z_3) \sigma(z_1) \sigma(z_2)
\]
where $\rho, \sigma$ are sums over even and odd $\nu$ respectively:
\[
\rho(z) = \sum_{\substack{\nu \geq 0 \\ \nu \text{ even}}} \overline{\nu} \; z^{\nu - 1},
\quad
\sigma(z) = \sum_{\substack{\nu \geq 0 \\ \nu \text{ odd}}} \overline{\nu} \; z^{\nu - 1}.
\]
We can then compute directly (say via differentiating the geometric series $\frac{1}{1-z^2} = \sum_{m \geq 0} z^{2m}$) that
\[
\rho(z) 
= \left( z^{-1} + \frac{2z}{(1-z^2)^2} \right)
\quad \text{ and } \quad
\sigma(z) 
= \frac{1+z^2}{(1-z^2)^2}.
\]
Writing out $f^N_{0,3}$ in terms of $z_1, z_2, z_3$, we obtain the claimed expression.
\end{proof}
We have now proved theorem \ref{thm:small_omegas}.

Note that we have only computed $f_{0,2}^N$ and $f_{0,3}^N$, not $f_{0,2}^G$ or $f_{0,3}^G$. In section \ref{sec:change_of_coords} we will give $f_{0,2}^G$ explicitly, and a general method for converting $\omega_{g,n}^N$ to $\omega_{g,n}^G$.

Observe that all the functions and forms computed above are meromorphic; we next show this is true in general.

\subsection{Meromorphicity}
\label{sec:meromorphicity}

\begin{prop}
\label{prop:omega1_meromorphic}
For all $g \geq 0$ and $n \geq 1$, $f^N_{g,n}(z_1, \ldots, z_n)$ is a meromorphic function and $\omega_{g,n}^N (z_1, \ldots, z_n)$ is a meromorphic differential form.
\end{prop}

\begin{proof}
In section \ref{sec:small_gen_fns} above we computed $\omega_{0,1}^N (z_1)$ and $\omega_{0,2}^N (z_1, z_2)$, seeing directly that they are meromorphic. (Indeed we also computed $\omega_{0,3}^N$ in lemma \ref{lem:omegaN03}.) Clearly it suffices to show $f^N_{g,n}(z_1, \ldots, z_n)$ is a meromorphic function, as $\omega^N_{g,n} = f^N_{g,n} \; dz_1 \cdots dz_n$.

Having established that each $N_{g,n}(\nu_1, \ldots, \nu_n)$ is $\bar{\nu}_1 \bar{\nu}_2 \cdots \bar{\nu}_n$ times a quasi-polynomial function of $\nu_1, \ldots, \nu_n$, fix the parity of each $\nu_i$, by setting $\nu_i \equiv \epsilon_i \pmod{2}$, where $\epsilon_i \in \{0,1\}$. We thus split the sum for $f^N_{g,n}$ into $2^n$ sums of the form
\[
\sum_{\substack{\nu_1 \geq 0 \\ \nu_1 \equiv \epsilon_1  \!\!\!\! \pmod{2}}} 
\cdots
\sum_{\substack{\nu_n \geq 0 \\ \nu_n \equiv \epsilon_n  \!\!\!\! \pmod{2}}}
\overline{\nu_1} \cdots \overline{\nu_n} \; 
P(\nu_1, \ldots, \nu_n) z_1^{\nu_1 - 1} \cdots z_n^{\nu_n - 1},
\]
where $P(\nu_1, \ldots, \nu_n)$ is a polynomial. Splitting each such polynomial into monomials, we can write $f_{g,n}^N$ as a finite sum of terms of the form of a constant times
\[
\sum_{\substack{\nu_1 \geq 0 \\ \nu_1 \equiv \epsilon_1  \!\!\!\! \pmod{2}}} 
\cdots
\sum_{\substack{\nu_n \geq 0 \\ \nu_n \equiv \epsilon_n  \!\!\!\! \pmod{2}}}
\overline{\nu_1} \cdots \overline{\nu_n} \; 
\nu_1^{a_1} \cdots \nu_n^{a_n} z_1^{\nu_1 - 1} \cdots z_n^{\nu_n - 1}
=
\prod_{i=1}^n 
\left( 
\sum_{\substack{\nu_i \geq 0 \\ \nu_i \equiv \epsilon_i  \!\!\!\! \pmod{2}}}
\overline{\nu_i} \; \nu_i^{a_i} z_i^{\nu_i - 1} 
\right),
\]
where $a_1, \ldots, a_n$ are non-negative integers. Thus it suffices to show that, for $a \geq 0$ and $\epsilon \in \{0,1\}$,
\[
\sum_{\substack{\nu \geq 0 \\ \nu \equiv \epsilon  \!\!\!\! \pmod{2}}}
\overline{\nu} \; \nu^a \; z^{\nu-1} 
= 
\delta_{a,0} z^{-1} + 
\sum_{\substack{\nu \geq 0 \\ \nu \equiv \epsilon  \!\!\!\! \pmod{2}}}
\nu^{a+1} z^{\nu - 1} 
\]
is meromorphic. 
Now we have
\[
\sum_{\substack{\nu \geq 0 \\ \nu \equiv \epsilon  \!\!\!\! \pmod{2}}}
\nu^a z^\nu 
= \left( z \frac{d}{dz} \right)^a \sum_{\substack{\nu \geq 0 \\ \nu \equiv \epsilon  \!\!\!\! \pmod{2}}}
z^\nu,
\]
so it suffices to show that $\sum_{\nu \equiv \epsilon  \!\! \pmod{2}} z^\nu$ is meromorphic.
Accordingly as $\epsilon = 0$ or $1$, we have
\[
\sum_{\substack{\nu \geq 0 \\ \nu \text{ even}}}
z^\nu 
= \sum_{m \geq 0} z^{2m} = \frac{1}{1-z^2},
\quad \text{or} \quad
\sum_{\substack{\nu \geq 0 \\ \nu \text{ odd}}}
z^\nu 
= \sum_{m \geq 0} z^{2m+1} = \frac{z }{1-z^2},
\]
both of which are clearly meromorphic.
\end{proof}

In fact, since the operator $z \frac{d}{dz}$ appearing in the above proof introduces no new poles, we have the following result.

\begin{cor}
For all $g \geq 0$ and $n \geq 1$, $f^N_{g,n}(z_1, \ldots, z_n)$ and $\omega_{g,n}^N (z_1, \ldots, z_n)$ have poles only at $z_i = -1, 0, 1$.
\qed
\end{cor}

\subsection{Change of coordinates between non-boundary-parallel and general curve counts}
\label{sec:change_of_coords}

We have defined $\omega^N_{g,n}$ to keep track of the non-boundary-parallel curve counts $N_{g,n}$ on $S_{g,n}$, and $\omega^G_{g,n}$ to keep track of general curve counts $G_{g,n}$. It turns out that, after a change of variable, these two formal differential forms are \emph{equal}.

The change of variable required is $x = z + \frac{1}{z}$. So we can define $\phi: \CP^1 \To \CP^1$ by $\phi(z) = z + \frac{1}{z} = x$, and consider pulling back $\omega_{g,n}^G (x_1, \ldots, x_n)$ under $\phi$.
\begin{thm}
\label{thm:omega1_equals_omega2}
For any $g \geq 0$ and $n \geq 1$ other than $(g,n) = (0,1)$,
\[
\phi^* \omega_{g,n}^G (x_1, \ldots, x_n) = \omega_{g,n}^N (z_1, \ldots, z_n).
\]
\end{thm}
That is, upon substituting $x_i = z_i + \frac{1}{z_i}$ into $\omega_{g,n}^G (x_1, \ldots, x_n)$ (which involves substituting $dx_i = (1 - z_i^{-2}) dz_i$), we have $\omega^N_{g,n}$. If we regard $x$ and $z$ as alternative coordinates on $\CP^1$ and $\phi$ as a change of coordinate, then $\omega_{g,n}^G$ and $\omega^N_{g,n}$ give the same differential form, which we simply denote $\omega_{g,n}$.

We can express this formula in shorthand as
\[
\omega^N_{g,n} (z) = \sum_{{\bf \nu} = {\bf 0}}^\infty N_{g,n} ({\bf \nu}) z^{\nu - 1} dz = 
\sum_{{\bf \mu} = {\bf 0}}^\infty G_{g,n} ({\bf \mu}) x^{-\mu - 1} dx = \omega^G_{g,n} (x),
\]
and we can simply write $\omega_{g,n}$ for the associated differential forms.

Our explicit computations of $\omega_{0,1}^G$ and $\omega_{0,1}^N$ show that the theorem fails for $(g,n) = (0,1)$. 

The proof is by a residue argument, following ideas of Do--Norbury in \cite{Do-Norbury13}.
\begin{proof}
Let $\mu = (\mu_1, \ldots, \mu_n)$ and $\nu = (\nu_1, \ldots, \nu_n)$.
We begin with the expression for $G_{g,n}$ in terms of $N_{g,n}$ (the stronger version of proposition \ref{prop:stronger_G_N}.) For $(g,n) \neq (0,1)$ and any integer vector $\mu = (\mu_1, \ldots, \mu_n)$ we then have
\[
G_{g,n}(\mu) 
= \sum_{\substack{0 \leq \nu_j \leq \mu_j \\ j=1,\ldots, n}} N_{g,n}(\nu)
\prod_{i=1}^n \binom{\mu_i}{\frac{\mu_i - \nu_i}{2}}
= \sum_{\substack{\nu_j \geq 0 \\ j = 1, \ldots, n}} N_{g,n}(\nu)
\prod_{i=1}^n \binom{\mu_i}{\frac{\mu_i - \nu_i}{2}}.
\]

Now we note that, for any integers $\mu,\nu$ (even if negative, even if $\nu>\mu$),
\begin{align*}
\binom{\mu}{\frac{\mu-\nu}{2}} 
&= \Res_{z=0} \; z^{\nu - \mu - 1} \sum_{m=0}^\infty \binom{\mu}{m} z^{2m} \; dz 
= \Res_{z=0} \; z^{\nu - \mu - 1} (1+z^2)^\mu \; dz \\
&= \Res_{z=0} \; z^{\nu-1} \left( z + \frac{1}{z} \right)^{\mu} \; dz 
= \Res_{z=0} \; z^{\nu-1} \; dz \; x^{\mu}.
\end{align*}
Here we have used the Taylor series of $(1+z^2)^\mu$ at $z=0$ (even when $\mu$ is negative). Hence we obtain
\begin{align*}
G_{g,n}(\mu)
= \Res_{(z_1, \ldots, z_n)=(0, \ldots, 0)} \;
\omega_{g,n}^N (z_1, \ldots, z_n) \; \prod_{i=1}^n x_i^{\mu_i}.
\end{align*}
Since the infinite sum over $\nu_j$ defines the meromorphic form $\omega_{g,n}^N$ (and indeed can be rewritten to involve only finitely many polynomial terms, as in the proof of \ref{prop:omega1_meromorphic}), we may exchange the residue and the sum.

Now suppose we rewrite $\omega_{g,n}^N (z_1, \ldots, z_n)$ in terms of $x_1, \ldots, x_n$; as $\omega_{g,n}^N$ is meromorphic this form is determined by its Laurent series. Let $a_{g,n}(\lambda_1, \ldots, \lambda_n)$ be the coefficient of $x_1^{-\lambda_1-1} \cdots x_n^{-\lambda_n -1} \; dx_1 \cdots dx_n$, so that
\[
\omega_{g,n}^N = \sum_{\lambda_1, \ldots, \lambda_n} a_{g,n}(\lambda_1, \ldots, \lambda_n) x_1^{-\lambda_1 -1} \cdots x_n^{-\lambda_n -1} \; dx_1 \cdots dx_n.
\]
The residue at $(z_1, \ldots, z_n)$ corresponds to the residue at $(x_1, \ldots, x_n) = (\infty, \ldots, \infty)$; if we substitute $y_i = x_i^{-1}$, this corresponds to the residue at $(y_1, \ldots, y_n) = (0, \ldots, 0)$. Since $dx_i = -y_i^{-2} \; dy_i$ and $x_i^{-\lambda_i - 1} dx_i = -y_i^{\lambda_i - 1} \; dy_i$, we have
\begin{align*}
G_{g,n}(\mu) &= 
(-1)^n \Res_{(x_1, \ldots, x_n) = (\infty, \ldots, \infty)} \;
\sum_{\lambda_1, \ldots, \lambda_n} 
a_{g,n}(\lambda_1, \ldots, \lambda_n) 
x_1^{-\lambda_1 - 1} \cdots x_n^{-\lambda_n - 1} \; dx_1 \cdots dx_n \; \prod_{i=1}^n x_i^{\mu_i} \\
&= (-1)^{2n} \Res_{(y_1, \ldots, y_n) = (0, \ldots, 0)} \;
\sum_{\lambda_1, \ldots, \lambda_n} 
a_{g,n}(\lambda_1, \ldots, \lambda_n) 
y_1^{\lambda_1 - 1 - \mu_1} \cdots y_n^{\lambda_n - 1 - \mu_n} \;
dy_1 \cdots dy_n \\
&= a_{g,n}(\mu_1, \ldots, \mu_n).
\end{align*}
Hence $G_{g,n}(\mu) = a_{g,n}(\mu)$, and $\omega_{g,n}^N$, expressed in terms of the $x_i$, is actually a generating function for the $G_{g,n}(\mu)$, as desired:
\[
\omega_{g,n}^N = \sum_{\lambda_1, \ldots, \lambda_n} G_{g,n}(\lambda_1, \ldots, \lambda_n) x_1^{-\lambda_1 - 1} \cdots x_n^{-\lambda_n - 1} \; dx_1 \cdots dx_n = \omega_{g,n}^G. \qedhere
\]
\end{proof}
We have now proved theorem \ref{thm:change_of_coords}.

As an illustration of the usefulness of this theorem, we calculate $f_{0,2}^G (x_1, x_2)$, the generating function for all $G_{0,2}(\mu_1, \mu_2)$, given by the complicated formulae in equations \ref{eqn:G02ee} and \ref{eqn:G02oo}.
\begin{lem}
The generating function $f^G_{0,2} (x_1, x_2)$ is given by
\[
f^G_{0,2}(x_1, x_2) = 
\frac{1}{2(x_1 - x_2)^2} \left( 1 + \frac{2x_1^2 - 3 x_1 x_2 + 2 x_2^2 - 4}{\sqrt{(x_1^2 - 4)(x_2^2 - 4)}} \right).
\]
\end{lem}

\begin{proof}
In section \ref{sec:small_gen_fns} we computed $\omega_{0,2} = \left( \frac{1}{z_1 z_2} + \frac{1}{(1-z_1 z_2)^2} \right) dz_1 dz_2$. Substituting $z_i = \frac{x_i - \sqrt{x_i^2 - 4}}{2}$ gives the desired expression.
\end{proof}

\subsection{Free energies}
\label{sec:free_energies}

Each $\omega_{g,n}(z_1, \ldots, z_n)$ is a meromorphic section of the vector bundle $(T^* \CP^1)^{\boxtimes n}$ over $(\CP^1)^n$. A form of this type may be obtained by taking a function $F: (\CP^1)^n \To \CP^1$ and applying an exterior differential operator in each coordinate. We write $d_{z_i} F = \frac{\partial F}{\partial z_i} dz_i$, so that $d_{z_1} d_{z_2} \cdots d_{z_n} F$ is a section of $(T^* \CP^1)^{\boxtimes n}$.

\begin{defn}
\label{defn:free_energy}
Let $(g,n) \neq (0,1)$. A function $F_{g,n} : (\CP^1)^n \To \CP^1$ such that 
\[
d_{z_1} \cdots d_{z_n} F_{g,n} (z_1, \ldots, z_n) = \omega_{g,n} (z_1, \ldots, z_n)
\]
is called a \emph{free energy} function.
\end{defn}
Given $\omega_{g,n}$, there are many free energies: $F_{g,n} = \int^{z_1} \!\cdots\! \int^{z_n} \omega_{g,n}$, and each integral introduces a constant of integration. Note we also have
\[
f_{g,n}^G (x_1, \ldots, x_n) = \frac{\partial^n F_{g,n}}{\partial x_1 \; \partial x_2 \; \cdots \; \partial x_n}
\quad \text{and} \quad
f_{g,n}^N (z_1, \ldots, z_n) = \frac{\partial^n F_{g,n}}{\partial z_1 \; \partial z_2 \; \cdots \; \partial z_n}.
\]
In the case $(g,n) = (0,1)$, we can still integrate $\omega_{0,1}^G$ and $\omega_{0,1}^N$; in this case we can obtain two distinct free energy functions $F_{0,1}^G (z_1)$ and $F_{0,1}^N (x_1)$.

By integrating $\omega_{0,1}^G$, $\omega_{0,1}^N$, $\omega_{0,2}$ we readily obtain free energy functions
\[
F_{0,1}^N (z_1) = \log z_1,
\quad
F_{0,1}^G (x_1) = \frac{1}{2} z_1^2 - \log z_1, 
\quad
F_{0,2} (z_1, z_2) = \log z_1 \log z_2 - \log(1-z_1 z_2).
\]
With a little more effort, one can obtain the free energy $F_{g,n} (z_1, z_2, z_3)$ given in theorem \ref{thm:free_energy_examples}; differentiating it one obtains $\omega_{0,3} (z_1, z_2, z_3)$, verifying the theorem.

\subsection{Recursion and generating functions}
\label{sec:recursion_generating_functions}

We now make a first attempt to turn the recursion on $G_{g,n}$ into a recursion on generating functions $f_{g,n}^G$. Throughout this section we write $f_{g,n}$ rather than $f_{g,n}^G$, for convenience. (No $f_{g,n}^N$'s arise, so there is no possible ambiguity.)

The recursion on $G_{g,n}$ (theorem \ref{thm:G_recursion}), as noted in section \ref{sec:intro_curve_counts}, is identical to the recursion obeyed by the ``generalised Catalan numbers", but has different initial conditions. Since generating functions for the ``generalised Catalan numbers" obey a recursive differential equation \cite{Dumitrescu-Mulase15, Mulase13_laplace, Mulase_Sulkowski12}, we might expect the $f_{g,n}$ to obey a similar differential equation. However the different initial conditions lead to some difficulties.

Recall from theorem \ref{thm:G_recursion} that the recursion applies only when $b_1 > 0$. When $b_1 = 0$, the recursion fails. 
The generalised Catalan numbers avoid this issue, as $b_1 = 0$ implies that the corresponding generalised Catalan number is zero; but we cannot.

For now however we postpone the issue, and consider only the case $b_1 > 0$. To obtain a relation between generating functions, we take the recursion on $G_{g,n}$ (theorem \ref{thm:G_recursion}), multiply by $x_1^{-b_1 - 1} \cdots x_n^{-b_n - 1}$, and sum over all $b_1 \geq 1$ and $b_2, \ldots, b_n \geq 0$. We obtain, on the right hand side, the three terms
\begin{align*}
I &= \sum_{\substack{b_1 \geq 1 \\ b_2, \ldots, b_n \geq 0}} \;
\sum_{\substack{i,j \geq 0 \\ i+j = b_1 - 2}} G_{g-1,n+1} (i,j,b_2, \ldots, b_n)  \;
x_1^{-b_1 - 1} \cdots x_n^{-b_n - 1} \\
II &= \sum_{\substack{b_1 \geq 1 \\ b_2, \ldots, b_n \geq 0}} \;
\sum_{k=2}^n b_k G_{g,n-1} (b_1 + b_k - 2, b_2, \ldots, \widehat{b}_k, \ldots, b_n ) \;
x_1^{-b_1 - 1} \cdots x_n^{-b_n - 1} \\
III &= \sum_{\substack{b_1 \geq 1 \\ b_2, \ldots, b_n \geq 0}} \;
\sum_{\substack{g_1 + g_2 = g \\ I_1 \sqcup I_2 = \{2, \ldots, n\}}} \;
\sum_{\substack{i,j \geq 0 \\ i+j = b_1 - 2}} G_{g_1, |I_1|+1} (i, b_{I_1}) G_{g_2, |I_2|+1} (j, b_{I_2}) \;
x_1^{-b_1 - 1} \cdots x_n^{-b_n - 1}
\end{align*}
Each of $I, II, III$ can be written in terms of the generating functions $f_{g,n}$. 

Considering $I$, we note that $i+j=b_1 - 2$ implies $x_1^{-b_1 - 1} = x_1^{-i-1} x_1^{-j-1} x_1^{-1}$; and we note that a sum over $b_1 \geq 1$, followed by a sum over $i,j \geq 0$ with $i+j = b_1 - 2$, is simply a sum over $i,j \geq 0$. Thus
\begin{align*}
I &= x_1^{-1} \sum_{b_2, \ldots, b_n \geq 0} \; \sum_{i,j \geq 0}
G_{g-1,n+1} (i,j,b_2, \ldots, b_n) \; x_1^{-i-1} x_1^{-j-1} x_2^{-b_2 - 1} \cdots x_n^{-b_n - 1} \\
&= x_1^{-1} f_{g-1,n+1}^G (x_1, x_1, x_2, \ldots, x_n)
\end{align*}

To simplify $II$, let $m = b_1 + b_k - 2$, and replace the sum over $b_1$ and $b_k$ with a sum over $m$, followed by a sum over $b_1 \geq 1$, $b_k \geq 0$ satisfying $b_1 + b_k - 2 = m$.
\begin{align*}
II &= \sum_{k=2}^n 
\sum_{m \geq 0} \sum_{\substack{b_1 \geq 1, \; b_k \geq 0 \\ b_1 + b_k - 2 = m}} 
\sum_{b_2, \ldots, \widehat{b}_k, \ldots, b_n \geq 0}
b_k G_{g,n-1}(m, b_2, \ldots, \widehat{b}_k, \ldots, b_n) \;
x_1^{-b_1 - 1} \cdots x_n^{-b_n - 1} \\
&= 
\sum_{k=2}^n 
\sum_{b_2, \ldots, \widehat{b}_k, \ldots, b_n, m \geq 0}
G_{g,n-1}(m, b_2, \ldots, \widehat{b}_k, \ldots, b_n) 
x_2^{-b_2 - 1} \cdots \widehat{x_k^{-b_k - 1}} \cdots x_n^{-b_n - 1}
\sum_{\substack{b_1 \geq 1, \; b_k \geq 0 \\ b_1 + b_k - 2 = m}} 
b_k \; x_1^{-b_1 - 1} x_k^{-b_k - 1} 
\end{align*}
We thus consider the sum, for fixed $m$ and $k$, over $b_1$ and $b_k$:
\begin{align*}
\sum_{\substack{b_1 \geq 1, \; b_k \geq 0 \\ b_1 + b_k -2 = m}} & b_k \; x_1^{-b_1 - 1} x_k^{-b_k - 1}
=
(m+1) x_1^{-2} x_k^{-m-2} + m x_1^{-3} x_1^{-3} x_k^{-m-1} 
+ \cdots
+ 2 x_1^{-m-1} x_k^{-3} + x_1^{-m-2} x_k^{-2} \\
&= 
- \frac{\partial}{\partial x_k} \left( x_1^{-2} x_k^{-m-1} +x_1^{-3} x_k^{-m} + \cdots + x_1^{-m-2} x_k^{-1} \right) 
= - \frac{\partial}{\partial x_k} \left( x_1^{-1} \frac{x_1^{-m-1} - x_k^{-m-1}}{x_k - x_1} \right) \\
&= \frac{-x_1^{-1}}{x_k - x_1} (m+1) x_k^{-m-2} + \frac{x_1^{-1}}{(x_k-x_1)^2} \left( x_1^{-m-1}- x_k^{-m-1}  \right).
\end{align*}
Hence we obtain an expression for $II$, which is a sum over $k$ and the parameters $b_2, \ldots, \widehat{b}_k, \ldots, b_n, m$ appearing in the $G_{g,n}$:
\begin{align*}
II &= 
\sum_{k=2}^n 
\sum_{b_2, \ldots, \widehat{b}_k, \ldots, b_n, m \geq 0}
G_{g,n-1}(m, b_2, \ldots, \widehat{b}_k, \ldots, b_n) 
x_2^{-b_2 - 1} \cdots \widehat{x_k^{-b_k - 1}} \cdots x_n^{-b_n - 1} \\
&\quad \quad \left[ 
\frac{-x_1^{-1}}{x_k - x_1} (m+1) x_k^{-m-2} + \frac{x_1^{-1}}{(x_k-x_1)^2} \left( x_1^{-m-1}- x_k^{-m-1}  \right) \right] \\
&= 
\sum_{k=2}^n 
\sum_{b_2, \ldots, \widehat{b}_k, \ldots, b_n, m \geq 0}
\frac{-x_1^{-1}}{x_k - x_1}
G_{g,n-1}(m, b_2, \ldots, \widehat{b}_k, \ldots, b_n) 
x_2^{-b_2 - 1} \cdots \widehat{x_k^{-b_k - 1}} \cdots x_n^{-b_n - 1} \;
(m+1) x_k^{-m-2} \\
& \quad +
\frac{x_1^{-1}}{(x_k-x_1)^2}
G_{g,n-1}(m, b_2, \ldots, \widehat{b}_k, \ldots, b_n) 
x_2^{-b_2 - 1} \cdots \widehat{x_k^{-b_k - 1}} \cdots x_n^{-b_n - 1}
\left( x_1^{-m-1} - x_k^{-m-1}  \right)  \\
&=
\sum_{k=2}^n 
\frac{x_1^{-1}}{x_k - x_1}
\frac{\partial}{\partial x_k} 
f_{g,n-1}(x_2, \ldots, x_n)
+ 
\frac{x_1^{-1}}{(x_k-x_1)^2}
\left( f_{g,n-1} (x_1, x_2, \ldots, \widehat{x}_k, \ldots, x_n) - f_{g,n-1} (x_2, \ldots, x_n) \right) \\
&=
x_1^{-1} 
\sum_{k=2}^n 
\frac{\partial}{\partial x_k}
\frac{1}{x_k - x_1}
\left(
f_{g,n-1}(x_2, \ldots, x_n) - f_{g,n-1}(x_1, x_2, \ldots, \widehat{x}_k, \ldots, x_n)
\right).
\end{align*}

Finally we turn to $III$. As with $I$, a sum over $b_1 \geq 1$ and then over $i,j \geq 0$ with $i+j = b_1 - 2$ is equivalent to a sum over $i,j \geq 0$, so we obtain
\begin{align*}
III &= 
\sum_{\substack{b_2, \ldots, b_n, i, j \geq 0}}
\sum_{\substack{g_1 + g_2 = g \\ I_1 \sqcup I_2 = \{2, \ldots, n\}}}
G_{g_1, |I_1|+1} (i, b_{I_1}) G_{g_2, |I_2|+1} (j, b_{I_2}) \;
x_1^{-i-j-3} \cdots x_n^{-b_n - 1} \\
&=
x_1^{-1}
\sum_{\substack{g_1 + g_2 = g \\ I_1 \sqcup I_2 = \{2, \ldots, n\}}}
\sum_{i, b_I \geq 0}
G_{g_1, |I_1|+1} (i, b_{I_1}) 
x_1^{-i-1} \; x_{I_1}^{-b_{I_1} - 1} \;
\sum_{j, b_J \geq 0}
G_{g_2, |I_2|+1} (j, b_{I_2}) 
x_1^{-j-1} \; x_{I_2}^{-b_{I_2}-1} \\
&=
x_1^{-1}
\sum_{\substack{g_1 + g_2 = g \\ I_1 \sqcup I_2 = \{2, \ldots, n\}}}
f_{g_1, |I_1|+1}(x_1, x_{I_1}) \; f_{g_2, |I_2|+1} (x_1, x_{I_2}).
\end{align*}

Having now expressed $I$, $II$, $III$ in terms of the generating functions $f_{g,n}$, we can express the recursion, for $b_1 >0$, in terms of the generating functions, and obtain the following lemma.
\begin{lem}
\label{lem:intermediate_step_to_diff_eqn}
For any $g \geq 0$ and $n \geq 1$, we have
\begin{align*}
\sum_{\substack{b_1 \geq 1 \\ b_2, \ldots, b_n \geq 0}} 
G_{g,n}&(b_1, \ldots, b_n) 
x_1^{-b_1 - 1} \cdots x_n^{-b_n - 1}
=
x_1^{-1} f_{g-1,n+1} (x_1, x_1, x_2, \ldots, x_n) \\
& +
x_1^{-1} 
\sum_{k=2}^n 
\frac{\partial}{\partial x_k}
\frac{1}{x_k - x_1}
\left(
f_{g,n-1}(x_2, \ldots, x_n) - f_{g,n-1}(x_1, x_2, \ldots, \widehat{x}_k, \ldots, x_n)
\right) \\
& +
x_1^{-1}
\sum_{\substack{g_1 + g_2 = g \\ I_1 \sqcup I_2 = \{2, \ldots, n\}}}
f_{g_1, |I_1|+1}(x_1, x_{I_1}) \; f_{g_2, |I_2|+1} (x_1, x_{I_2}).
\end{align*}
\qed
\end{lem}

Thus we have an equation relating ``most of $f_{g,n}$" to $f_{g',n'}$ for $(g',n')$ of lower complexity.

\subsection{Differential equation on generating functions}
\label{sec:diff_eqn_on_generating_fns}

We found above that if we ignored the terms with $b_1 = 0$ we could use the recursion on $G_{g,n}$ to express the remaining terms of $f_{g,n}$ in terms of $f_{g',n'}$ of smaller complexity.

An elementary way to obtain a differential equation for $f_{g,n}$ is to find an operation which \emph{deletes} all terms with $b_1 = 0$. If we arrange such terms to be the constant terms,  we remove them by differentiation with respect to $x_1$. Hence we consider
\[
\frac{\partial}{\partial x_1} \left( x_1 f_{g,n}(x_1, \ldots, x_n) \right)
=
\sum_{b_1, \ldots, b_n \geq 0} G_{g,n}(b_1, \ldots, b_n) (-b_1) x_1^{-b_1 - 1} \cdots x_n^{-b_n - 1},
\]
in which all terms with $b_1 = 0$ are zero.
Hence we have 
\[
\frac{\partial}{\partial x_1} \left( x_1 f_{g,n} \right)
=
\frac{\partial}{\partial x_1} x_1 I
+ \frac{\partial}{\partial x_1} x_1 II
+ \frac{\partial}{\partial x_1} x_1 III.
\]
and obtain the following differential equation.

\begin{thm}
\label{thm:first_diff_eq}
For any $g \geq 0$ and $n \geq 1$,
\begin{align*}
\frac{\partial}{\partial x_1} &\left( x_1 f_{g,n}(x_1, \ldots, x_n) \right)
=
\frac{\partial}{\partial x_1} f_{g-1,n+1} (x_1, x_1, x_2, \ldots, x_n) \\
&\quad +
\frac{\partial}{\partial x_1}
\sum_{k=2}^n 
\frac{\partial}{\partial x_k}
\frac{1}{x_k - x_1}
\left(
f_{g,n-1}(x_2, \ldots, x_n) - f_{g,n-1}(x_1, x_2, \ldots, \widehat{x}_k, \ldots, x_n)
\right) \\
&\quad +
\frac{\partial}{\partial x_1} \sum_{\substack{g_1 + g_2 = g \\ I_1 \sqcup I_2 = \{2, \ldots, n\}}}
f_{g_1, |I_1| + 1} (x_1, x_{I_1}) f_{g_2, |I_2| + 1} (x_1, x_{I_2}).
\end{align*}
\qed
\end{thm}
We return to the search for a differential equation in section \ref{sec:refined_diff_eqns}.

An alternative method to obtain a differential equation is to find a simple way to compute $G_{g,n}(0, b_2, \ldots, b_n)$. There is a straightforward way to do this, if we keep track of the number of \emph{complementary regions} in the arc diagram. This is the subject of the next section.

\section{Keeping track of regions}
\label{sec:regions}

As it turns out, many of the results already proved about $G_{g,n}$ and $N_{g,n}$, can be refined by keeping track of the number of complementary regions (definition \ref{def:comp_region}) in these arc diagrams.

\subsection{Refining curve counts}
\label{sec:refined_counts}

We begin by refining our definitions of $G_{g,n}$ and $N_{g,n}$.
\begin{defn}\
\begin{enumerate}
\item
The set of equivalence classes of arc diagrams on $(S_{g,n}, F({\bf b}))$ with $r$ complementary regions is denoted $\mathcal{G}_{g,n,r}({\bf b})$. The number of such equivalence classes is denoted $G_{g,n,r}({\bf b})$.
\item
The subset of $\mathcal{G}_{g,n,r}({\bf b})$ without boundary-parallel arcs is denoted $\mathcal{N}_{g,n,r}({\bf b})$. The number of such equivalence classes is denoted $N_{g,n,r}({\bf b})$. We also define
\[
\widehat{N}_{g,n,r}(b_1, \ldots, b_n) = \frac{ N_{g,n,r}(b_1, \ldots, b_n) }{ \overline{b_1} \; \overline{b_2} \; \cdots \; \overline{b_n} }.
\]
\end{enumerate}
\end{defn}

It will turn out to be very useful to use a parameter $t$ rather than $r$, defined as follows.
\begin{defn}
For $g \geq 0$ and $n,r \geq 1$ and $b_1, \ldots, b_n \geq 0$, define
\[
t = r - (2 - 2g - n) - \frac{1}{2} \sum_{i=1}^n b_i = r - \chi - \frac{1}{2} \sum_{i=1}^n b_i.
\]
\end{defn}
Here $\chi$ is the Euler characteristic. We can then refine all of the counts $G_{g,n}, N_{g,n}, \widehat{N}_{g,n}$ by $t$ as well as $r$.
\begin{defn}
For $g \geq 0$, $n \geq 1$, and $b_1, \ldots, b_n \geq 0$, define
\begin{align*}
G_{g,n}^t ({\bf b}) &= G_{g,n,r}({\bf b}) = G_{g,n, t + (2-2g-n) + \frac{1}{2}  \sum_{i=1}^n b_i} ({\bf b}), \\ 
N_{g,n}^t ({\bf b}) &= N_{g,n,r}({\bf b}) = N_{g,n, t + (2-2g-n) + \frac{1}{2}  \sum_{i=1}^n b_i} ({\bf b}), \\ 
\widehat{N}_{g,n}^t ({\bf b}) &= \widehat{N}_{g,n,r}({\bf b}) = \widehat{N}_{g,n, t + (2-2g-n) + \frac{1}{2}  \sum_{i=1}^n b_i} ({\bf b}). 
\end{align*}
\end{defn}

We can also write $\mathcal{G}_{g,n}^t({\bf b}) = \mathcal{G}_{g,n,r}({\bf b})$ 
and $\mathcal{N}_{g,n}^t ({\bf b}) = \mathcal{N}_{g,n,r} ({\bf b})$. 
Clearly $\mathcal{G}_{g,n}({\bf b}) = \sqcup_{r \geq 0} \mathcal{G}_{g,n,r}({\bf b}) = \sqcup_{t} \mathcal{G}_{g,n}^t({\bf b})$ and $G_{g,n}({\bf b}) = \sum_{r \geq 0} G_{g,n,r}({\bf b})$, so we immediately obtain the following.
\begin{prop}
\label{prop:sum_over_rt}
For any $g \geq 0$, $n \geq 1$ and $b_1, \ldots, b_n \geq 0$,
\begin{align*}
G_{g,n}({\bf b}) &= \sum_{r \geq 0} G_{g,n,r}({\bf b}) = \sum_t G_{g,n}^t ({\bf b}), \\
N_{g,n}({\bf b}) &= \sum_{r \geq 0} N_{g,n,r}({\bf b}) = \sum_t N_{g,n}^t ({\bf b}), \\
\widehat{N}_{g,n}({\bf b}) &= \sum_{r \geq 0} \widehat{N}_{g,n,r}({\bf b}) = \sum_t \widehat{N}_{g,n}^t ({\bf b}).
\end{align*}
\qed
\end{prop}
We will discuss how many nonzero terms are in these sums, i.e. bounds on $r$ and $t$, in section \ref{sec:inequalities_on_regions}.

\subsection{Counting arc diagrams with punctures}
\label{sec:punctures}

When $b_1 = 0$, the first boundary component of $S_{g,n}$ has no points marked on it; we may regard the boundary component as a \emph{puncture} in $S_{g,n-1}$. Filling in the puncture gives arc diagrams on $S_{g,n-1}$.
We now show precisely how keeping track of regions allows us to compute $G_{g,n,r}(0, b_2, \ldots, b_n)$.
\begin{prop}
\label{prop:b1_equals_zero}
For any $g \geq 0$, $n \geq 2$ and $b_2, \ldots, b_n \geq 0$,
\[
G_{g,n,r}(0, b_2, \ldots, b_n) = r \; G_{g,n-1,r}(b_2, \ldots, b_n).
\]
\end{prop}

In the case of enumerating lattice points in moduli spaces of curves~\cite{Norbury13_string}, the evaluation $b_1 = 0$ is related to the dilaton equation that appears in the general theory of the topological recursion~\cite{EynardOrantin09}. Thus, the equation above can be regarded as a kind of dilaton equation for curve counts.

\begin{proof}
Filling in the first boundary component with a disc gives a well-defined map
\[
\mathcal{G}_{g,n,r}(0, b_2, \ldots, b_n) \To \mathcal{G}_{g,n-1,r}(b_2, \ldots, b_n).
\]
Conversely, one can remove a disc from any complementary region of an arc diagram (with equivalence class) in $\mathcal{G}_{g,n-1,r}(b_2, \ldots, b_n)$, to obtain an (equivalence class of) arc diagram in $\mathcal{G}_{g,n,r}(0, b_2, \ldots, b_n)$; filling in the disc again recovers the original diagram. Two arc diagrams obtained on $(S_{g,n}, F(0, b_2, \ldots, b_n))$ by removing discs from a given arc diagram $C$ on $(S_{g,n-1}, F(b_2, \ldots, b_n))$ are equivalent if and only if the discs were removed from the same complementary region. Thus the map above is surjective and $r$-to-1, giving the claimed equality.
\end{proof}

We essentially saw this argument already in the case $(g,n) = (0,2)$ in proposition \ref{prop:Prz_Catalan}.

\subsection{Refined counts on discs and annuli}
\label{sec:refined_discs_annuli}

We now compute the refined counts $G_{g,n,r}$ and $N_{g,n,r}$ in the cases $(g,n) = (0,1), (0,2)$.

First consider $(g,n) = (0,1)$. On the disc $(S_{0,1}, F(b_1))$, an arc diagram requires $b_1$ to be even, $b_1 = 2m$, and then has $m$ arcs dividing the disc into $m+1$ complementary regions. Thus $G_{0,1,r}(2m) = G_{0,1}(2m)$ if $r=m+1$, and is zero otherwise. This value of $r$ corresponds to $t = 0$.
\begin{lem}
\label{lem:G01r}
For any integer $m \geq 0$,
\[
G_{0,1,r}(2m) = \frac{\delta_{r,m+1}}{m+1} \binom{2m}{m}
\quad \text{ and, equivalently, } \quad
G_{0,1}^t (2m) = \frac{\delta_{t,0}}{m+1} \binom{2m}{m}.
\]
\qed
\end{lem}

Now consider annuli $(S_{0,2}, F(b_1, b_2))$. For an arc diagram to exist we need $b_1 + b_2 \equiv 0 \pmod{2}$. From lemma \ref{lem:traversing_complementary_regions}, a traversing arc diagram has $r = \frac{1}{2}(b_1 + b_2)$, and an insular diagram has $r = \frac{1}{2}(b_1 + b_2) + 1$; these correspond to $t=0$ and $t=1$ respectively.  Propositions \ref{prop:insular_count} and \ref{prop:traversing_count} give the number of insular and traversing diagrams. We obtain the following.

\begin{lem}
\label{lem:G02r}
For integers $m_1, m_2 \geq 0$,
\begin{align*}
G_{0,2}^0 (2m_1, 2m_2) &= G_{0,2,m_1 + m_2} (2m_1, 2m_2) = \frac{m_1 m_2}{m_1 + m_2} \binom{2m_1}{m_1} \binom{2m_2}{m_2} \\
G_{0,2}^0 (2m_1 + 1, 2m_2 + 1) &= G_{0,2, m_1 + m_2} (2m_1 + 1, 2m_2 + 1) = \frac{(2m_1 + 1)(2m_2 + 1)}{m_1 + m_2 + 1} \binom{2m_1}{m_1} \binom{2m_2}{m_2} \\
G_{0,2}^1 (2m_1, 2m_2) &= G_{0,2,m_1 + m_2 + 1} (2m_1, 2m_2) = \binom{2m_1}{m_1} \binom{2m_2}{m_2} 
\end{align*}
All other $G_{0,2}^t (b_1, b_2)$ and $G_{0,2,r}(b_1, b_2)$ are zero. 
\qed
\end{lem}

As for diagrams without boundary-parallel arcs, as discussed in section \ref{sec:non_boundary_parallel_small_cases}, there are none on a disc except for the empty diagram, for which $r=1$ and $t=0$. On an annulus, such a diagram must consist entirely of parallel traversing arcs, so $b_1 = b_2=b$; there are $\bar{b}$ such diagrams, which have $\bar{b}$ complementary regions, so $r = \bar{b}$ and $t = \bar{b} - b = \delta_{b,0}$. Hence we have the following.
\begin{lem}\
\label{lem:N01t_N02t_computations}
\begin{enumerate}
\item
For $b \geq 0$, $N_{0,1,1}(0) = N_{0,1}^0 (0) = 1$, and all other $N_{0,1,r}(b_1)$ and $N_{0,1}^t (b_1)$ are zero.
\item
For $b \geq 0$, $N_{0,2,b}(b, b) = \bar{b}$, and all other $N_{0,2,r} (b_1, b_2)$ are zero.

Equivalently, for $b>0$, $N_{0,2}^0 (b,b) = \bar{b}$, $N_{0,2}^1 (0,0) = 1$, and all other $N_{0,2}^t (b_1, b_2)$ are zero.
\end{enumerate}
\qed
\end{lem}

\subsection{Refining local decomposition}
\label{sec:refining_local}

Recall the local decomposition of an arc diagram discussed in section \ref{sec:decomposing_arc_diagrams} (definition \ref{defn:local_decomposition}). Let $C$ be an arc diagram on $(S_{g,n}, F(b_1, \ldots, b_n))$. The local decomposition of $C$ gives a $B_i$-local arc diagram on an annulus neighbourhood  $A_i$ of each boundary component $B_i$, and a diagram $C'$ without boundary-parallel arcs on the core $S'$. Suppose the $B_i$-local arc diagram lies in $L(b_i, a_i)$, so has $a_i$ traversing arcs and $(b_i - a_i)/2$ insular arcs.

Let $C$ and $C'$ have $r,r'$ complementary regions respectively, and corresponding parameters $t,t'$. Now $C'$ can be obtained by successively removing from $C$ outermost boundary-parallel arcs, at each stage cutting off a disc complementary region. There are $\sum_{i=1}^n (b_i - a_i)/2$ such boundary-parallel arcs, so 
\[
r' = r - \frac{1}{2} \sum_{i=1}^n (b_i - a_i).
\]
Since $S$ and $S'$ have the same Euler characteristic $\chi$, we have
\[
t' = r' - \chi - \frac{1}{2} \sum_{i=1}^n a_i
= r- \chi - \frac{1}{2} \sum_{i=1}^n b_i
= t.
\]
In other words, the arc diagram $C'$ of the local decomposition has the same $t$-parameter as $C$.

The map which glues local decompositions into a general arc diagram is a map
\[
L(b_1, a_1) \times \cdots \times L(b_n, a_n) \times \mathcal{N}_{g,n}(a_1, \ldots, a_n) \To \mathcal{G}_{g,n}(b_1, \ldots, b_n),
\]
and it thus refines into maps for each value of $t$
\[
L(b_1, a_1) \times \cdots \times L(b_n, a_n) \times \mathcal{N}_{g,n}^t (a_1, \ldots, a_n) \To \mathcal{G}_{g,n}^t (b_1, \ldots, b_n).
\]
Taking the quotient by the action of $\Z_{\bar{a}_1} \times \cdots \times \Z_{\bar{a}_n}$ by rotations, and a union over $a_i$ as in section \ref{sec:counting_local_decomposition}, we obtain a bijection
\[
\Delta : \mathcal{G}_{g,n}^t (b_1, \ldots, b_n) \To \bigsqcup_{\substack{0 \leq a_i \leq b_i \\ a_i \equiv b_i  \!\!\!\! \pmod{2}}} 
\frac{ L(b_1, a_1) \times \cdots \times L(b_n, a_n) \times \mathcal{N}_{g,n}^t (a_1, \ldots, a_n)}{ \Z_{\bar{a}_1} \times \cdots \times \Z_{\bar{a}_n} }.
\]
In proposition \ref{prop:Lbb}  we computed that $|L(b,a)| = \binom{b}{\frac{1}{2}(b-a)} \bar{a}$, and so we can express $G_{g,n}^t$ in terms of $N_{g,n}^t$, giving a refinement of proposition \ref{prop:stronger_G_N}. As discussed there, the result holds for any integers $a_1, \ldots, a_n$, $b_1, \ldots, b_n$, not just non-negative integers.
\begin{prop}
\label{prop:G_and_N_refined}
For any integers $b_1, \ldots, b_n$ we have
\begin{align*}
G_{g,n}^t(b_1, \ldots, b_n) 
&= \sum_{\substack{a_i \in \Z \\ i=1, \ldots, n}}
\binom{b_1}{\frac{b_1 - a_1}{2}} \cdots \binom{b_n}{\frac{b_n - a_n}{2}}
N_{g,n}^t (a_1, \ldots, a_n).
\end{align*}
\qed
\end{prop}
This proposition indicates that for many purposes, it is much more convenient to use a refinement by $t$ than by $r$.

\subsection{Refined curve counts on pants}
\label{sec:refined_pants}

In this section we compute refined versions of $G_{0,3}$ and $N_{0,3}$.

Recall from section \ref{sec:pants_approach} the definitions of insular, prodigal and traversing arcs. We defined $p_i$ to be the number of prodigal arcs with endpoints on $B_i$, and $t_{ij}$ to be the number of traversing arcs with endpoints on $B_i$ and $B_j$.

We begin with $N_{0,3}$, so consider an arc diagram on a pair of pants $(S_{0,3}, F(b_1, b_2, b_3))$ without boundary-parallel arcs, where $b_1 + b_2 + b_3$ is even. In section \ref{sec:non_boundary_parallel_pants} we showed that $b_1, b_2, b_3$ determine the $p_i$ and $t_{ij}$ uniquely, so that there is a unique arc diagram in $\mathcal{N}_{0,3}(b_1, b_2, b_3)$, up to rotations around boundary components. 
We can also show that $b_1, b_2, b_3$ determine $r$, the number of complementary regions, and hence $t$. In fact, $t$ is determined simply by how many of the $b_i$ are zero.
\begin{lem}
\label{lem:pants_regions}
Let $b_1, b_2, b_3 \geq 0$ be integers such that $b_1 + b_2 + b_3 \equiv 0 \pmod{2}$. Then for any arc diagram without boundary-parallel curves on $(S_{0,3},F(b_1, b_2, b_3))$, $r$ and $t$ are given by
\[
\begin{array}{lll}
r=1, & t=2 &\text{if all } b_i = 0 \\
r = \frac{1}{2}(b_1 + b_2 + b_3) + 1, & t = 2 & \text{if two $b_i$ are zero and one is nonzero} \\
r = \frac{1}{2}(b_1 + b_2 + b_3), & t = 1 & \text{if one $b_i$ is zero and two are nonzero} \\
r = \frac{1}{2}(b_1 + b_2 + b_3) - 1, & t = 0 & \text{if all $b_i$ are nonzero}
\end{array}
\]
\end{lem}

\begin{proof}
We repeatedly apply proposition \ref{prop:pants_numbers_of_curves}, which gives the number of curves of each type. We compute values of $r$; the claimed values of $t$ then follow immediately from $r = 1 - \chi - \frac{1}{2} \sum b_i$. Without loss of generality suppose $b_1 \geq b_2 \geq b_3$.

If all $b_i = 0$ then there are no arcs, hence one complementary region. 

If $b_2 = b_3 = 0$ and $b_1 \neq 0$ then $p_1 = b_1 / 2$ and all other $p_i, t_{ij}$ are zero. These $b_1 / 2$ parallel prodigal arcs cut the pants into $\frac{b_1}{2} + 1 = \frac{1}{2}(b_1 + b_2 + b_3) + 1$ regions.

If $b_3 = 0$ and $b_1, b_2 \neq 0$ then we consider two cases. If $b_1 = b_2$ then the only nonzero parameter is $t_{12} = b_1$. These $b_1$ traversing arcs split the pants into $b_1 = \frac{1}{2}(b_1 + b_2 + b_3)$ regions. If $b_1 > b_2$ then the nonzero parameters are $p_1 = \frac{1}{2}(b_1 - b_2)$ and $t_{12} = b_2$. Cutting along the first traversing arc leaves a connected surface; cutting along every subsequent arc increases the number of components by $1$. So $r = \frac{b_1 - b_2}{2} + b_{2} = \frac{1}{2}(b_1 + b_2 + b_3)$.

If all $b_i$ are nonzero then we have two cases. If $b_2 + b_3 \geq b_1$ then there are no prodigal arcs, but $t_{12} = \frac{1}{2}(b_1 + b_2 - b_3)$, $t_{23} = \frac{1}{2}(b_2 + b_3 - b_1)$, $t_{31} = \frac{1}{2}(b_3 + b_1 - b_2)$. Cut along one of the traversing arcs of each type; this cuts the pants into two discs. Cutting along each subsequent arc increases the number of components by $1$, so $r = t_{12} + t_{23} + t_{31} - 1 = \frac{1}{2}(b_1 + b_2 + b_3) - 1$. If $b_2 + b_3 < b_1$ then the nonzero parameters are $p_1 = \frac{1}{2}(b_1 - b_2 - b_3)$, $t_{12} = b_2$ and $t_{31} = b_3$. Cutting along one of the traversing arcs of each type cuts the pants into a disc; cutting along each subsequent arc increases the number of components by $1$. Thus $r = p_1 + t_{12} + t_{31} - 1 = \frac{1}{2}(b_1 + b_2 + b_3) - 1$.
\end{proof}

Thus, all arc diagrams in $\mathcal{N}_{g,n}(b_1, b_2, b_3)$ have a fixed $r$ and $t$, determined as above, and for this value of $r$ and $t$, $\mathcal{N}_{g,n,r}(b_1, b_2, b_3) = \mathcal{N}_{g,n}^t (b_1, b_2, b_3) = \mathcal{N}_{g,n} (b_1, b_2, b_3)$. Since $t$ has a simpler expression than $r$, we proceed to refine $\widehat{N}_{0,3}$ by $t$. We obtain the following.
\begin{prop}
\label{prop:Nhat03t}
For integers $b_1, b_2, b_3 > 0$,
\begin{align*}
\widehat{N}_{0,3}^0 (b_1, b_2, b_3) &= 1 \quad \text{provided } b_1 + b_2 + b_3 \equiv 0 \!\!\!\! \pmod{2}, \\
\widehat{N}_{0,3}^1 (b_1, b_2, 0) &= 1 \quad \text{provided } b_1 + b_2 \equiv 0  \!\!\!\! \pmod{2}, \\
\widehat{N}_{0,3}^2 (b_1, 0, 0) &= 1 \quad \text{provided } b_1 \equiv 0  \!\!\!\! \pmod{2}, \\
\widehat{N}_{0,3}^2 (0,0,0) &= 1.
\end{align*}
All other $\widehat{N}_{0,3}^t (b_1, b_2, b_3)$ are zero.
\qed
\end{prop}

Letting $k$ denote the number of $b_i$ equal to zero, we can tabulate the various $\widehat{N}_{0,3}^t$ over possible values of $k$ and $t$.
\begin{center}
\begin{tabular}{x{0.5cm}| c c c}
\diag{.1em}{.5cm}{$k$}{$t$} & $0$ & $1$ & $2$  \\ \hline
$0$ & $1$ &  \\
$1$ & & $1$ \\
$2$ & & & $1$ \\
$3$ & & & $1$
\end{tabular}
\end{center}

Now we proceed to consider general arc diagrams. After propositions \ref{prop:G_and_N_refined} and \ref{prop:Nhat03t}, we refine by $t$ rather than $r$. And once we have $\widehat{N}_{g,n}^t$, we may obtain $G_{g,n}^t$ by the same method used to obtain $G_{g,n}$ from $\widehat{N}_{g,n}$, as discussed in section \ref{sec:polynomiality_general}. 
\begin{prop}
For integers $m_1, m_2, m_3 \geq 0$,
\begin{align*}
G_{0,3}^0 (2m_1, 2m_2, 2m_3) &= 
\binom{2m_1}{m_1}
\binom{2m_2}{m_2}
\binom{2m_3}{m_3} 
m_1 m_2 m_3 \\
G_{0,3}^{0} (2m_1 +1, 2m_2 +1, 2m_3) &= 
\binom{2m_1}{m_1}
\binom{2m_2}{m_2}
\binom{2m_3}{m_3} 
(2m_1 + 1)(2m_2 + 1) m_3 \\
G_{0,3}^1 (2m_1, 2m_2, 2m_3) &=
\binom{2m_1}{m_1}
\binom{2m_2}{m_2}
\binom{2m_3}{m_3} 
(m_1 m_2 + m_2 m_3 + m_3 m_1) \\
G_{0,3}^{1} (2m_1 +1, 2m_2 +1, 2m_3) &= 
\binom{2m_1}{m_1}
\binom{2m_2}{m_2}
\binom{2m_3}{m_3} 
(2m_1 + 1)(2m_2 + 1) \\
G_{0,3}^{2} (2m_1, 2m_2, 2m_3) &=
\binom{2m_1}{m_1}
\binom{2m_2}{m_2}
\binom{2m_3}{m_3} 
(m_1 + m_2 + m_3 + 1).
\end{align*}
For any other $t$ and $b_1, b_2, b_3$ not covered by these cases, $G_{0,3}^t(b_1, b_2, b_3) = 0$.
\end{prop}
It is easily verified that these expressions for $G_{0,3}^t$, when summed over $t$, give the expressions for $G_{0,3}$ found earlier in section \ref{sec:general_pants}.

\begin{proof}
Propositions \ref{prop:G_and_N_refined} and \ref{prop:Nhat03t} give
\begin{align*}
G_{0,3}^t(b_1, b_2, b_3)
&=
\sum_{a_i \in \Z} 
\binom{b_1}{\frac{b_1 - a_1}{2}}
\binom{b_2}{\frac{b_2 - a_2}{2}}
\binom{b_3}{\frac{b_3 - a_3}{2}}
\; \bar{a}_1 \; \bar{a}_2 \; \bar{a}_3 \widehat{N}_{0,3}^t (a_1, \ldots, a_n).
\end{align*}
Since $N_{0,3}^t$ is nonzero only for $t=0,1,2$, the same is true for $G_{0,3}^t$. We consider each value of $t$ separately.

If $t=0$ we have
\begin{align*}
G_{0,3}^0 (b_1, b_2, b_3) &= 
\sum_{a_i > 0} 
\binom{b_1}{\frac{b_1 - a_1}{2}}
\binom{b_2}{\frac{b_2 - a_2}{2}}
\binom{b_3}{\frac{b_3 - a_3}{2}}
\bar{a}_1 \bar{a}_2 \bar{a}_3
= 
\prod_{i=1}^3 \sum_{a_i > 0} \binom{b_i}{\frac{b_i - a_i}{2}} \bar{a}_i.
\end{align*}
We now consider parities. Write $b_i = 2m_i$ if $b_i$ is even and $b_i = 2m_i + 1$ if $b_i$ is odd. If the $b_i$ are all even, $b_i = 2m_i$, then all the $a_i$ must also be even and the above expression is $\tilde{p}_0 (m_1) \tilde{p}_0 (m_2) \tilde{p}_0 (m_3)$. If two $b_i$ are odd and one is even, say $b_1, b_2$, and $b_3$ are even, then the expression is $\tilde{q}_0 (m_1) \tilde{q}_0 (m_2) \tilde{p}_0 (m_3)$. Now $\tilde{p}_0 (m) = \binom{2m}{m} m$ and $\tilde{q}_0 (m) = \binom{2m}{m} (2m+1)$ (see section \ref{sec:general_pants}), giving $G_{0,3}^0$ as claimed.

Now suppose $t=1$. For $\widehat{N}_{0,3}^1 (a_1, a_2, a_3)$ to be nonzero, we require exactly one of the $a_i$ to be zero. 
\begin{align*}
G_{0,3}^1 (b_1, b_2, b_3) &=
\left( \sum_{\substack{a_1 = 0 \\ a_2, a_3 > 0}} + \sum_{\substack{a_2 = 0 \\ a_3, a_1 > 0}} + \sum_{\substack{a_3 = 0 \\ a_1, a_2 > 0}} \right)
\binom{b_1}{\frac{b_1 - a_1}{2}}
\binom{b_2}{\frac{b_2 - a_2}{2}}
\binom{b_3}{\frac{b_3 - a_3}{2}}
\bar{a}_1 \bar{a}_2 \bar{a}_3 \\
\end{align*}
As $a_i \equiv b_i \pmod{2}$, the corresponding $b_i$ must be even. If all $b_i = 2m_i$ are even, we have
\begin{align*}
G_{0,3}^1 (2m_1, 2m_2, 2m_3) &=
\binom{2m_1}{m_1} \tilde{p}_0 (m_2) \tilde{p}_0 (m_3)
+ \binom{2m_2}{m_2} \tilde{p}_0 (m_3) \tilde{p}_0 (m_1)
+ \binom{2m_3}{m_3} \tilde{p}_0 (m_1) \tilde{p}_0 (m_2),
\end{align*}
and if $b_1, b_2$ are odd and $b_3$ even, we have
\begin{align*}
G_{0,3}^1 (2m_1 + 1, 2m_2 + 1, 2m_3) &=
\binom{2m_3}{m_3} \tilde{q}_0 (m_1) \tilde{q}_0 (m_2).
\end{align*}
Thus $G_{0,3}^1$ is as claimed.

Finally, let $t=2$. Now for $\widehat{N}_{0,3}^2 (a_1, a_2, a_3)$ to be nonzero we require at least two of the $a_i$ to be zero; hence for $G_{0,3}^2 (b_1, b_2, b_3)$ to be nonzero at least two of the $b_i$ must be even; since their sum is even then all the $b_i$ must be even, $b_i = 2m_i$. We then have
\begin{align*}
G_{0,3}^2 (b_1, b_2, b_3) &=
\left( \sum_{\substack{a_1 = a_2 = 0 \\ a_3 > 0}} + \sum_{\substack{a_2 = a_3 = 0 \\ a_1 > 0}} + \sum_{\substack{a_3 = a_1 = 0 \\ a_2 > 0}} + \sum_{a_1 = a_2 = a_3 = 0}\right)
\binom{b_1}{\frac{b_1 - a_1}{2}}
\binom{b_2}{\frac{b_2 - a_2}{2}}
\binom{b_3}{\frac{b_3 - a_3}{2}}
\bar{a}_1 \bar{a}_2 \bar{a}_3 \\
&= \binom{2m_1}{m_1} \binom{2m_2}{m_2} \tilde{p}_0 (m_3)
+ \binom{2m_2}{m_2} \binom{2m_3}{m_3} \tilde{p}_0 (m_1)
\\
& \qquad+ \binom{2m_3}{m_3} \binom{2m_1}{m_1} \tilde{p}_0 (m_2)
+ \binom{2m_1}{m_1} \binom{2m_2}{m_2} \binom{2m_3}{m_3},
\end{align*}
which gives the claimed expression for $G_{0,3}^2$.
\end{proof}

\subsection{Inequalities on regions}
\label{sec:inequalities_on_regions}

We now establish various bounds on the number of complementary regions $r$ of an arc diagram, in terms of the various parameters $g,n$ and $b_1, \ldots, b_n$. Clearly, if $g,n$ and $b_1, \ldots, b_n$ are fixed, the number of regions $r$ is bounded.

\begin{lem}
\label{lem:upper_bound_1_on_r}
Suppose an arc diagram on $(S_{g,n}, F(b_1, \ldots, b_n))$ has $r$ complementary regions. Then
\[
r \leq 1 + \frac{1}{2} \sum_{i=1}^n b_i.
\]
\end{lem}

\begin{proof}
Consider such an arc diagram. It has precisely $\frac{1}{2} \sum_{i=1}^n b_i$ arcs. Cutting along each arc of $C$ can increase the number of components by at most $1$.
\end{proof}

\begin{lem}
\label{lem:bounds_on_r}
\label{lem:lower_bound_on_r}
Suppose an arc diagram on $(S_{g,n}, F(b_1, \ldots, b_n))$ has $r$ complementary regions. Then
\[
r \geq 2-2g-n + \frac{1}{2} \sum_{i=1}^n b_i
= \chi(S) + \frac{1}{2} \sum_{i=1}^n b_i.
\]
More generally, if $k$ elements of $\{b_1, \ldots, b_n\}$ are zero, where $0 \leq k \leq n-1$, then
\[
r \geq 2 - 2g - n + k + \frac{1}{2} \sum_{i=1}^n b_i 
= \chi(S) + k + \frac{1}{2} \sum_{i=1}^n b_i.
\]
If all $b_i$ are zero, i.e. $k=n$, then $r=1$.
\end{lem}
Here $\chi(S) = 2 - 2g - n$ is the Euler characteristic. 

\begin{proof}
The case $k=n$ is clear; with no arcs, $S$ is the single complementary region.

Suppose $k=0$. Each time we cut along an arc we obtain a new surface (disconnected in general), and $\chi$ increases by $1$ with each cut. Let $S'$ be the surface obtained after cutting along all the arcs. So $S'$ has $r$ connected components and Euler characteristic $\chi(S) + \frac{1}{2} \sum_{i=1}^n b_i$. Each component of $S'$ has nonempty boundary, hence has Euler characteristic $\leq 1$, and so $\chi(S) + \frac{1}{2} \sum_{i=1}^n b_i \leq r$.

Now suppose $1 \leq k \leq n-1$. Fill in the $k$ boundary components of $S$ with $b_i = 0$ by gluing discs. This yields a surface $\tilde{S}$ with Euler characteristic $\chi(S) + k$. Cutting along the curves of $C$ (which do not intersect the filled-in boundary components), by the $k=0$ case, we obtain $r$ components, where $r \geq \chi(\tilde{S}) + \frac{1}{2} \sum_{i=1}^n b_i = \chi(S) + k + \frac{1}{2} \sum_{i=1}^n b_i$. Removing the $k$ discs glued to the boundary components does not change the number of connected components, and we obtain the result.
\end{proof}

The converse of the above result is not true. There are many $g,n,r,b_1, \ldots, b_n$ (and with all $b_i > 0$) which satisfy the above inequality but for which $G_{g,n,r}(b_1, \ldots, b_n) = 0$: for instance, $G_{0,2,2}(1,1) = 0$.

We have one further result giving an upper bound on $r$, when there are no boundary-parallel curves. As we will see, it only provides additional information to the upper bound in lemma \ref{lem:upper_bound_1_on_r} in some fairly specific cases.
\begin{lem}
\label{lem:upper_bound_2_on_r}
Suppose $(g,n) \neq (0,1), (0,2)$, and an arc diagram on $(S_{g,n}, F(b_1, \ldots, b_n))$ has no boundary-parallel arcs and $r$ complementary regions. Suppose that precisely $k$ elements of $\{b_1, \ldots, b_n\}$ are zero, where $0 \leq k \leq n-1$. Then
\[
r \leq g+k - 1 + \frac{1}{2} \sum_{i=1}^n b_i.
\]
\end{lem}

Note that $g+k \geq 2$ is equivalent to $g+k-1 + \frac{1}{2} \sum_{i=1}^n b_i \geq 1 + \frac{1}{2} \sum_{i=1}^n b_i$. So if $g+k \geq 2$ then lemma \ref{lem:upper_bound_1_on_r} immediately implies this result; this upper bound thus only gives new information when $g+k \leq 1$.

\begin{proof}
As discussed, we can assume $g +k \leq 1$.

First suppose $g=0$ and $0 \leq k \leq 1$. We prove $r \leq k-1 + \frac{1}{2} \sum b_i$ for all $n \geq 3$ by induction on $n$.

If $n = 3$ then we draw upon our discussion of arc diagrams without boundary-parallel arcs on pants. Lemma \ref{lem:pants_regions} says that if $k=0$ then $r = -1 + \frac{1}{2} \sum_{i=1}^n b_i$, and if $k=1$ then $r = \frac{1}{2} \sum_{i=1}^n b_i$. So in these cases $r = k-1 + \frac{1}{2} \sum_{i=1}^n b_i$, and the desired inequality holds (in fact it is an equality).

Now consider general $n$, and an arc $\gamma$ in the arc diagram. If $\gamma$ connects two distinct boundary components, then cutting along $\gamma$ gives an arc diagram with the same $r$, with $k$ increased by at most $1$, and $\frac{1}{2} \sum_{i=1}^n b_i - 1$ arcs, on an $(n-1)$-punctured sphere, so by induction we have $r \leq (k+1) - 1 + (\frac{1}{2} \sum_{i=1}^n b_i - 1)$, from which we obtain the desired result for general $n$. If $\gamma$ has both endpoints on the same boundary component $B$, as there are no boundary-parallel arcs and $g=0$, it cuts the surface into two pieces. Let the number of arcs parallel to $\gamma$ (including $\gamma$) be $p$. Cutting along $\gamma$, and removing these parallel arcs, yields an $n'$-holed sphere $S'$ and an $n''$-holed sphere $S''$, with numbers of marked points given by $b'_i$ and $b''_i$ respectively. Here $n',n'' \geq 2$ satisfy $n' + n'' = n+1$, and $\frac{1}{2} \sum b'_i + \frac{1}{2} \sum b''_i + p = \frac{1}{2} \sum b_i$. Let $k', k''$ of the $b'_i, b''_i$ respectively be zero. Now $k' + k''$ cannot be much greater than $k$; the only way we can obtain extra boundary components with zero marked points is from the boundary components of $S', S''$ arising from $B$; thus $k' + k'' \leq k+2$.

Let $S', S''$ have $r', r''$ complementary regions respectively, so $r' + r'' + p-1 = r$. Note $n' + n'' = n+1$ and $n', n'' \geq 2$, so $2 \leq n', n'' \leq n-1$. If the inequality holds for both $S'$ and $S''$, then we have
\[
r = r' + r'' + p-1 \leq k' + k'' + p - 3 + \frac{1}{2} \sum b'_i + \frac{1}{2} \sum b''_i \leq k -1 + \frac{1}{2} \sum_{i=1}^n b_i
\]
as desired. By induction, the inequality holds for both $S'$ and $S''$ if neither is an annulus, and in this case we are done. So now suppose we obtain an annulus. As $n \geq 4$, $S'$ and $S''$ cannot both be annuli. So we may assume $S'$ is an annulus, and $S''$ is not. Then the inequality holds for $S''$. If the annulus $S'$ has no arcs, then actually the inequality holds for $S'$ too ($r=1$ and $g+k-1+ \frac{1}{2} \sum b_i = 0+2-1+0 = 1$), so we are done. However if the annulus $S'$ has nonempty arc diagram, then the inequality fails. In this case we have $r' = \frac{1}{2} \sum b'_i$ and $k' = 0$; and since we do not obtain any extra boundary components with zero marked points on $S'$, we must have $k'' = k' + k'' \leq k + 1$. Then we obtain
\[
r = r' + r'' + p-1 
\leq \frac{1}{2} \sum b'_i + \frac{1}{2} \sum b''_i + p + k'' - 2
\leq k - 1 +  \frac{1}{2} \sum_{i=1}^n b_i
\]
and the inequality holds. This completes the proof in the case $g=0$.

Now suppose $g = 1$ and $k=0$, and we prove $r \leq \frac{1}{2} \sum_{i=1}^n b_i$. We proceed by induction on $n \geq 1$. If $n = 1$ then take an arc $\gamma$ in the arc diagram; as $\gamma$ is not boundary-parallel, it cuts $S$ into an annulus. Suppose there are $p$ arcs parallel to $\gamma$ (including $\gamma$), so cutting along $\gamma$ and removing these parallel arcs gives a diagram on the annulus without boundary-parallel curves, with $r-p+1$ regions and $\frac{1}{2} b_1 - p$ arcs. If there are no arcs on the annulus then we have one region, so $r-p+1 = 1$ and $\frac{1}{2} b_1 - p = 0$, and hence $r = p = \frac{1}{2} b_1$. If there are arcs on the annulus then the number of arcs and regions are equal, so $r = -1 + \frac{1}{2} b_1$. Either way we have $r \leq \frac{1}{2} b_1 = \frac{1}{2} \sum_{i=1}^n b_i$.

Now take a general $n \geq 2$. Take an arc $\gamma$ on $S$ with $p$ parallel copies (including $\gamma$). If $\gamma$ is non-separating, then cutting along $\gamma$ and removing its parallel arcs gives a surface $S'$ of genus $g'$, with $n'$ boundary components, with $b'_i$ marked points on boundary components, $k'$ of which are zero, and an arc diagram with $r'=r-p+1$ complementary regions. The number of arcs is $\frac{1}{2} \sum b'_i = \frac{1}{2} \sum b_i - p$. Since we originally had all $b_i > 0$, after cutting along $\gamma$ and removing parallel copies, we can make at most one $b'_i = 0$, so $k' \leq 1$. Now $S'$ either has genus zero and $n' = n+1 \geq 3$ boundary components, in which case the result holds by our previous arguments; or $S'$ has genus 1 and $n' = n-1$ boundary components, in which case the result holds by inductive assumption (if $k'=0$) or previous argument (if $k' = 1$). Either way $r' \leq g' + k' - 1 + \frac{1}{2} \sum_{i=1}^{n'} b'_i$ and $g' + k' \leq 2$, and hence
\[
r = r' + p-1 \leq p + g' + k' - 2 + \frac{1}{2} \sum_{i=1}^{n'} b'_i 
= g' + k' - 2 + \frac{1}{2} \sum_{i=1}^n b_i
\leq \frac{1}{2} \sum_{i=1}^n b_i.
\]

On the other hand, if $\gamma$ is separating, with $p$ parallel copies, then cutting along $\gamma$ and removing parallel arcs gives two surfaces $S', S''$, with genera $g'+g''=1$; say $g' = 0$ and $g'' = 1$. Let them have $n', n''$ boundary components, with $b'_i, b''_i$ marked points, of which $k', k''$ are zero, and arc diagrams with $r', r''$ complementary regions. So we have $n' + n'' = n+1$; as there are no boundary-parallel curves, $n', n'' \geq 2$ and hence $n', n'' \leq n-1$. We also have $r'+r'' + p-1 = r$ and $\frac{1}{2} \sum b'_i + \frac{1}{2} \sum b''_i + p = \frac{1}{2} \sum b_i$. The only way to have $b'_i = 0$ or $b''_i = 0$ is from the boundary components involving $\gamma$, so $k', k'' \leq 1$. The inductive assumption certainly applies to $S''$, and we obtain $r'' \leq k'' + \frac{1}{2} \sum b''_i \leq 1 + \frac{1}{2} \sum b''_i$. If $S'$ is an annulus then as $k' \leq 1$, the arc diagram is nonempty and $r' = \frac{1}{2} \sum {b'_i}$. If $S'$ is not an annulus, then the inductive hypothesis (if $k' = 0$) or the above argument applies (if $k' = 1$), so the inequality holds for $S'$ giving $r' \leq k' - 1 + \frac{1}{2} \sum b'_i \leq \frac{1}{2} \sum b'_i$. Either way, $r' \leq \frac{1}{2} \sum b'_i$. Putting this together yields
\[
r = r' + r'' + p - 1
\leq \frac{1}{2} \sum b'_i + \frac{1}{2} \sum b''_i + p
= \frac{1}{2} \sum_{i=1}^n b_i.
\]
This completes the proof.
\end{proof}

Putting together lemmas \ref{lem:upper_bound_1_on_r}, \ref{lem:lower_bound_on_r} and \ref{lem:upper_bound_2_on_r}, and translating bounds on $r$ into $t$ via $t = r - \chi - \frac{1}{2} \sum b_i$ gives the following result.
\begin{prop}
\label{prop:bounds_on_r_and_t}
Suppose an arc diagram on $(S_{g,n}, F(b_1, \ldots, b_n))$ has no boundary-parallel arcs, $r$ complementary regions and $k$ boundary components without marked points. If $0 \leq k \leq n-1$ then
\[
\max \left( 1, \quad k + 2-2g-n + \frac{1}{2} \sum_{i=1}^n b_i \right)
\leq r \leq
\min \left( 1 + \frac{1}{2} \sum_{i=1}^n b_i, \quad g+k-1 + \frac{1}{2} \sum_{i=1}^n b_i \right)
\]
and
\[
\max \left( k, \quad 2g+n-1 - \frac{1}{2} \sum_{i=1}^n b_i \right)
\leq t \leq
\min \left( 2g+n-1, \quad k+3g-3+n \right).
\]
If $k=n$ then $r=1$ and $t = 2g+n-1$.
\qed
\end{prop}

The above inequalities are necessary for the existence of an arc diagram without boundary-parallel arcs. 
However, they are not sufficient. For instance, $\widehat{N}_{3,2}^7 (2n+1,1) = N_{3,2,n+2}(2n+1,1) = 0$, but $\max(0, 6-n) = \max(k, 2g+n-1 - \frac{1}{2} \sum b_i) \leq 7=t \leq \min(3g-3+n, 2g+n-1) = \min(8,7)$. To see why, suppose there were such an arc diagram; then there must be a (necessarily non-separating) arc connecting the two boundary components. Cutting along this arc gives an arc diagram in $\mathcal{N}_{3,1,n+2}(2n)$, hence with $n$ arcs. But cutting along $n$ arcs can only create $n+1$ regions, not the required $n+2$.

In the particular case $t=k$ we can give necessary and sufficient conditions in the next section.

When $k=0$, so that all boundary components have marked points, we have $0 \leq t \leq 1-\chi$. So $t$ is roughly a measure of how ``separating" an arc diagram is: when $t=0$ it is as non-separating as possible, and as $t$ increases, it is more and more separating.

\subsection{Existence of certain arc diagrams}

\label{sec:existence}

We just gave various conditions which must be satisfied by $g,n,k,t,r$ and $b_1, \ldots, b_n$ in order for an arc diagram to exist. We will now give some results guaranteeing the existence of arc diagrams in certain circumstances.

\begin{lem}
\label{lem:nonzero_Ggnk}
Suppose $g \geq 0$ and $n \geq 1$, and $0 \leq k \leq n-1$. Let $b_1, \ldots, b_{n-k} > 0$ be positive integers such that $b_1 + \cdots + b_{n-k}$ is even, and suppose $b_{n-k+1} = \cdots = b_n = 0$.
\begin{enumerate}
\item
If $\frac{1}{2} \sum_{i=1}^n b_i < 2g+n-1-k$, then $G_{g,n}^k (b_1, \ldots, b_{n-k}, 0, \ldots, 0) = 0$.
\item
If $\frac{1}{2} \sum_{i=1}^n b_i \geq 2g+n-1-k$, then $G_{g,n}^k (b_1, \ldots, b_{n-k}, 0, \ldots, 0) > 0$.
\end{enumerate}
\end{lem}
(Here the notation $G_{g,n}^k$ means that we set $t=k$.)

\begin{proof}
If $t = k$ then $r = k + \chi + \frac{1}{2} \sum_{i=1}^n b_i = k + 2-2g-n + \frac{1}{2} \sum_{i=1}^n b_i$. If $\frac{1}{2} \sum_{i=1}^n b_i < 2g+n-1-k$ then $r < 1$, so no arc diagram exists, proving (i). 

We prove (ii) first assuming $k=0$. So suppose all $b_1, \ldots, b_n > 0$, with even sum, and $\frac{1}{2} \sum_{i=1}^n b_i \geq 2g+n-1 = 1 - \chi$. As all $b_i$ are positive, we may draw $1-\chi$ arcs which cut the surface into a disc. For instance, we may successively draw curves joining distinct boundary components and cut along them, in order to reduce the number of boundary components to $1$. (Provided at each stage we do not join two boundary components each with one boundary component, we retain at least one point on each boundary component. And this is certainly possible since $\sum_{i=1}^n b_i \geq 4g+2n-2 \geq 2n-2$.) We then have a genus $g$ surface with one boundary component, and an even number of marked points; we then cut the handles to form a disc.

We then have a disc, with some number of points on the boundary, and $\frac{1}{2} \sum_{i=1}^n b_i - 1 + \chi$ remaining arcs to draw. Drawing them arbitrarily and successively cutting (for instance, by always drawing outermost arcs), we obtain $r = \frac{1}{2} \sum_{i=1}^n b_i + \chi$ connected components. This corresponds to an arc diagram on the original surface with $t=0$.

Now consider general $k$. We can fill in the $k$ boundary components with no marked points with discs, to obtain a surface $S'$ with genus $g$ and $n-k$ boundary components, all with a positive number of marked points. Then we apply the above argument for $k=0$ to obtain an arc diagram on $S'$, with $\chi(S') + \frac{1}{2} \sum_{i=1}^n b_i = 2-2g-n+k + \frac{1}{2} \sum_{i=1}^n b_i$ complementary regions. Now removing the $k$ discs, we have an arc diagram on the original surface $S$, with the same number of regions, hence with $t = k$.
\end{proof}

We now show a similar result in the non-boundary-parallel case.

\begin{prop}
\label{prop:nonzero_Ngnk}
Suppose $(g,n) \neq (0,1), (0,2)$, and $0 \leq k \leq n-1$. Let $b_1, \ldots, b_{n-k} > 0$ be positive integers such that $b_1 + \cdots + b_{n-k}$ is even, and suppose $b_{n-k+1} = \cdots = b_n = 0$.
\begin{enumerate}
\item
If $\frac{1}{2} \sum_{i=1}^n b_i < 2g+n-1-k$, then $N_{g,n}^k (b_1, \ldots, b_{n-k}, 0, \ldots, 0) = 0$.
\item
If $\frac{1}{2} \sum_{i=1}^n b_i \geq 2g+n-1-k$, then $N_{g,n}^k (b_1, \ldots, b_{n-k}, 0, \ldots, 0) > 0$.
\end{enumerate}
\end{prop}

\begin{proof}
If $\frac{1}{2} \sum_{i=1}^n b_i < 2g+n-1-k$, then we have $G_{g,n}^k (b_1, \ldots, b_{n-k}, 0, \ldots, 0) = 0$ from above, so $N_{g,n}^k = 0$ also.

It remains to prove (ii); we first prove it under the assumption $k=0$. So suppose all $b_i > 0$, $\frac{1}{2} \sum_{i=1}^n b_i \geq 2g+n-1 = 1 - \chi$, and we will construct an arc diagram with the desired parameters. We consider two cases.

First, suppose $g=0$, so $n \geq 3$. Then, as in the proof of lemma \ref{lem:nonzero_Ggnk}, since $\sum_{i=1}^n b_i \geq 4g+2n-2 = 2n-2$, we may draw arcs connecting boundary components and cut along them, always maintaining at least one marked point on each boundary component. We proceed until we have a pair of pants, with a nonzero number of points on each boundary component. Since each cut increases Euler characteristic by $1$, at this stage we have drawn and cut along $-1-\chi$ arcs; so we have $\frac{1}{2} \sum_{i=1}^n b_i +1+\chi$ remaining arcs to draw. From proposition \ref{prop:N03} above, there is an arc diagram on the pants,  with the required number of points on each boundary component, without boundary-parallel curves, and from lemma \ref{lem:pants_regions}, the number of regions into which they cut the pants is one less than the number of arcs drawn. So, drawing these arcs and cutting, we obtain $\frac{1}{2} \sum_{i=1}^n b_i + \chi$ components. This corresponds to an arc diagram on the original surface without boundary-parallel arcs, and with $r = \frac{1}{2} \sum_{i=1}^n b_i + \chi$ complementary regions, hence with $t=0$.

Now suppose $g \geq 1$. Use a similar method to join boundary components until we obtain a single boundary component, with an even number of points  on it. At this stage we have a genus $g$ surface with a single boundary component, hence Euler characteristic has increased from $\chi$ to $1-2g$, so we have drawn and cut along $1-2g-\chi$ non-boundary-parallel arcs. There are $\frac{1}{2} \sum_{i=1}^n b_i + \chi + 2g-1 = \frac{1}{2} \sum_{i=1}^n b_i -n+1 \geq 2g$ remaining arcs to draw.

Now we can draw $2g$ curves to cut the genus $g$ surface into a disc. We draw these curves, along with some parallel copies of them, so that there are $\frac{1}{2} \sum_{i=1}^n b_i + \chi + 2g - 1$ arcs drawn in total, none of them boundary-parallel. Cutting along all these curves, including the parallel copies splits the surface into $\frac{1}{2} \sum_{i=1}^n b_i + \chi$ components. This corresponds to a diagram on the original surface, without boundary-parallel arcs, and with $r = \frac{1}{2} \sum_{i=1}^n b_i + \chi$ complementary regions, so $t=0$. 

This proves the result in the case $k=0$. We now consider general $k$. We fill in the $k$ boundary components with no marked points with $k$ discs, to obtain a surface of genus $g$ with $n-k$ boundary components. Provided we do not end up with a disc or annulus, the $k=0$ argument applies, and we obtain an arc diagram with $\frac{1}{2} \sum_{i=1}^n b_i + 2-2g-n+k$ regions, with no boundary-parallel arcs. Then removing the $k$ discs gives an arc diagram on the original surface, still with no boundary-parallel arcs, and with the same number of complementary regions, hence with $t = k$.

If this argument fails, ending up with a disc or annulus, then we must have $g=0$, $n \geq 3$, and $k \geq n-2$. In this case we fill in $n-3$ of the $k$ boundary components without marked points, to obtain a pair of pants, on which $k' = k-n+3$ boundary components have no marked points. Note $1 \leq k' \leq 2$. Using proposition \ref{prop:N03} again, there is an arc diagram on the pants with no boundary-parallel arcs, and with the required number of points on each boundary component. Using lemma \ref{lem:pants_regions}, the number of complementary regions of this arc diagram on the pants is $\frac{1}{2} \sum b_i + k' - 1 = \frac{1}{2} \sum b_i + 2-n+k$. Removing the $n-3$ discs gives an arc diagram on the original surface with no boundary-parallel arcs and with the same number of regions, hence with $t=k$.
\end{proof}

So, fixing $g,n$ and setting $t=k$ (and hence fixing $r - \frac{1}{2} \sum_{i=1}^n b_i$), provided we have sufficiently many marked points, we can find an arc diagram with these parameters --- and, provided $(g,n) \neq (0,1)$ or $(0,2)$, one without any boundary-parallel arcs.

\subsection{Refining recursion}

\label{sec:refining_recursion}

Now we may refine the recursions in theorems \ref{thm:G_recursion} and \ref{thm:Ngn_recursion}, on the $G_{g,n}$ and $N_{g,n}$ respectively.

As with theorem \ref{thm:G_recursion}, this recursion on $G_{g,n,r}$ only applies when $b_1 > 0$.
\begin{thm}
\label{thm:G_refined_recursion}
For any integers $g \geq 0$, $n \geq 1$, $r \geq 1$, $b_1 > 0$ and $b_2, \ldots, b_n \geq 0$,
\begin{align*}
G_{g,n,r}(b_1, \ldots, b_n)
&=
\sum_{\substack{i,j \geq 0 \\ i+j = b_1 - 2}} G_{g-1,n+1,r} (i,j,b_2, \ldots, b_n) \\
&\quad + \sum_{k=2}^n b_k G_{g,n-1,r} (b_1 + b_k - 2, b_2, \ldots, \widehat{b}_k, \ldots, b_n) \\
&\quad + \sum_{\substack{g_1 + g_2 = g \\ I_1 \sqcup I_2 = \{2, \ldots, n\} }}
\sum_{\substack{i,j \geq 0 \\ i+j = b_1 - 2}} \sum_{\substack{r_1, r_2 \geq 1 \\ r_1 + r_2 = r}}
G_{g_1, |I_1|+1, r_1} (i, b_{I_1}) G_{g_2, |I_2| + 1, r_2} (j, b_{I_2}).
\end{align*}
\end{thm}
Summing this recursion over $r$ gives the recursion of theorem \ref{thm:G_recursion}.
\begin{proof}
The proof is essentially the same as that of theorem \ref{thm:G_recursion}. Given an arc diagram $C$ in $\mathcal{G}_{g,n,r}({\bf b})$, take the first marked point $p$ (as $b_1 > 0$) on the first boundary component $B_1$; let the arc with endpoint at $p$ be $\gamma$. Cutting along $\gamma$ gives a surface $S'$ with an arc diagram $C'$. The various cases for the topology of $\gamma$ and $S'$ are the same as in the proof of theorem \ref{thm:G_recursion}. In each case, $C$ can be reconstructed by gluing together two boundary arcs on $S'$ in a specified way. 

The key fact we need here is that cutting along $\gamma$ does not change the number of complementary regions, so all the arc diagrams considered have $r$ complementary regions. Hence, enumerating the various cases, the (equivalence classes of) arc diagrams in $\mathcal{G}_{g,n,r}({\bf b})$ are in bijection with the various (equivalence classes of) arc diagrams enumerated on the right hand side of the equation.
\end{proof}
Theorem \ref{thm:G_refined_recursion_intro} is now proved. Turning to the $N_{g,n}$, we obtain the following.
\begin{thm}
\label{thm:Ngnr_recursion}
For $(g,n) \neq (0,1), (0,2), (0,3)$, $r \geq 1$ and $b_1, \ldots, b_n$ such that $b_1 >0$, $b_2, \ldots, b_n \geq 0$,
\begin{align*}
N_{g,n,r}({\bf b}) &= \sum_{\substack{ i,j,m \geq 0 \\ i+j+m = b_1 \\ m \text{ even}}} \frac{m}{2} N_{g-1,n+1, r - \frac{m}{2}+1} (i,j,b_2, \ldots, b_n) \\
& + \sum_{j=2}^n \Bigg( \sum_{\substack{i,m \geq 0 \\ i+m = b_1 + b_j \\ m \text{ even}}} \frac{m}{2} \bar{b}_j N_{g,n-1,r - \frac{m}{2} + 1 - \delta_{b_j, 0}} (i,b_2, \ldots, \widehat{b_j}, \ldots, b_n) \\
& + \widetilde{\sum_{\substack{i,m \geq 0 \\ i+m = b_1 - b_j \\ m \text{ even}}}} \frac{m}{2} \bar{b}_j N_{g,n-1, r - \frac{m}{2} - \overline{\min(b_1, b_j)} +1 } (i, b_2, \ldots, \widehat{b_j}, \ldots, b_n) \Bigg) \\
& + \sum_{\substack{g_1 + g_2 = g \\ I \sqcup J = \{2, \ldots, n\} \\ \text{No discs or annuli}}} \; \sum_{\substack{i,j,m \geq 0 \\ i+j+m = b_1 \\ m \text{ even}}} \;
\sum_{\substack{r_1, r_2 \geq 0 \\ r_1 + r_2 = r - \frac{m}{2} + 1}}
 \frac{m}{2} N_{g_1, |I|+1, r_1} (i, b_I) N_{g_2, |J|+1, r_2} (j, b_J)
\end{align*}
\end{thm}

\begin{proof}
We proceed similarly to the proof of theorem \ref{thm:Ngn_recursion}. Let $C$ be a non-boundary-parallel arc diagram on $(S_{g,n}, F({\bf b}))$. Let $p$ be the first marked point on the first boundary component $B_1$, and let $\gamma$ be the arc of $C$ with an endpoint there. We consider the same three cases as in the proof of \ref{thm:Ngn_recursion}.

The first case is when $\gamma$ has both endpoints on $B_1$ and is nonseparating. There are $\frac{m}{2}$ arcs (including $\gamma$) parallel to $\gamma$. Between the $\frac{m}{2}$ parallel arcs there are $\frac{m}{2} - 1$ complementary regions. Cutting along $\gamma$ and removing arcs which become boundary-parallel produces an $S_{g-1,n+1}$ with $\frac{m}{2} - 1$ fewer complementary regions. So all diagrams considered in this case have $r-\frac{m}{2}-1$ regions, and following the argument in the proof of theorem \ref{thm:Ngn_recursion}, the number of arc diagrams so obtained is given by the first term in the recursion.

The second case is when $\gamma$ has endpoints on distinct boundary components $B_1$ and $B_j$, or is separating and cuts off an annulus with $B_j$ as a boundary component. This corresponds to the second and third lines above. 

Let $m/2$ be the number of arcs which are ``parallel" to $\gamma$, in the extended sense of the argument of \ref{thm:Ngn_recursion}: if $\gamma$ runs from $B_1$ to $B_j$, then we take as ``parallel" all arcs parallel to $\gamma$, and those which run from from $B_1$ around $B_j$ back to $B_1$, and those which run from $B_j$ around $B_1$ back to $B_j$; while if $\gamma$ cuts off an annulus around $B_j$, we take as ``parallel" all arcs parallel to $\gamma$, and those which run from $B_1$ to $B_j$. These $m/2$ arcs consist precisely of $\gamma$ and those arcs which become boundary-parallel after cutting along $\gamma$.

Assuming that $b_j > 0$, these $m/2$ arcs, running from $B_1$ to $B_j$, or from one of these boundary components around the other and back to itself, enclose $\frac{m}{2} - 1$ regions within an annular region which is effectively removed from $S$: see figure \ref{fig:y5}. If $b_j = 0$ then we only have $m/2$ arcs running around $B_j$, which enclose precisely $m/2$ regions. Thus the number of regions effectively removed from $S$ is $\frac{m}{2} - 1 + \delta_{b_j, 0}$. We again orient these arcs so that they run from $B_1$ to $B_j$, or run anticlockwise around $B_1$ or $B_j$. Hence, as discussed in the proof of \ref{thm:Ngn_recursion}, the number of arc diagrams for which $\gamma$ runs from $B_1$ to $B_j$, or runs from $B_1$ around $B_j$, and is oriented so that $p$ is the start point of $\gamma$, is given by the summation in the second line above: all diagrams obtained by cutting along such $\gamma$ and removing boundary-parallel arcs have $r - \frac{m}{2} + 1 - \delta_{b_j, 0}$ complementary regions.

\begin{figure}
\begin{center}
\begin{tikzpicture}
\def\xlen{80mm}

\fill[fill=lightgray!10] (0,0) circle (24mm);
\draw[fill=white] (0,8mm) circle (4mm);
\draw[fill=white] (0,-8mm) circle (4mm);

\draw[red,thick] ($(0,-8mm) + (45:4mm)$) to[out=45,in=0] (0,16mm) to[out=180,in=135] ($(0,-8mm) + (135:4mm)$);
\draw[red,thick] ($(0,-8mm) + (20:4mm)$) to[out=45,in=0] (0,18mm) to[out=180,in=135] ($(0,-8mm) + (155:4mm)$);

\draw[red,thick] (0,-4mm) -- (0,4mm);
\draw[red,thick] ($(0,-8mm) + (70:4mm)$) to[out=70,in=-70] ($(0,8mm) + (-70:4mm)$);
\draw[red,thick] ($(0,-8mm) + (110:4mm)$) to[out=110,in=-110] ($(0,8mm) + (-110:4mm)$);

\node at (0,-8mm) {$B_1$};
\node at (0,8mm) {$B_j$};
\end{tikzpicture}
\end{center}\caption{The $m/2$ ``parallel" arcs enclose $m/2 - 1$ regions.}
\label{fig:y5}
\end{figure}
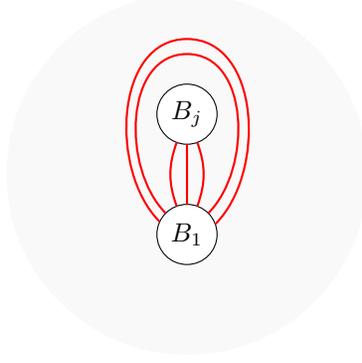

Next suppose $b_1 \geq b_j$. Following the proof of \ref{thm:Ngn_recursion}, we need to count arc diagrams where $p$ is the endpoint of $\gamma$. We redefine $m$ so that the number of arcs from $B_1$ looping around $B_j$ is $m/2$. If $b_j = 0$ then these $m/2$ arcs looping around $B_j$ enclose $m/2$ regions within an annular region which is effectively removed from $S$, so resulting diagrams have $r-\frac{m}{2}$ complementary regions. If $b_j > 0$ then the $m/2$ arcs looping around $B_j$ also enclose $b_j$ arcs running from $B_1$ to $B_j$ and the annular region has $\frac{m}{2} - b_j + 1$ regions. Thus the resulting diagrams have $r - \frac{m}{2} + b_j - 1$ complementary regions. Either way, the resulting diagrams have $r - \frac{m}{2} + \bar{b}_j - 1$ regions and, following the proof of \ref{thm:Ngn_recursion} (and noting $\overline{\min(b_1, b_j)} = \bar{b}_j$, we obtain the summation in the third line.

If $b_1 \leq b_j$ then we have overcounted, and as in the proof of \ref{thm:Ngn_recursion} need to subtract off diagrams where $p$ lies on $B_j$. We redefine $m$ so that the number of arcs from $B_j$ looping around $B_1$ is $m/2$. By a similar argument to the previous paragraph, these arcs enclose an annular region with $\frac{m}{2} - \bar{b}_1 + 1$ complementary regions, so that diagrams obtained after removing this annulus have $r - \frac{m}{2} + \bar{b}_1 - 1$ complementary regions. (Note $b_1 > 0$, so that $\bar{b}_1 = b_1$ in any case; but we write $\bar{b}_1$ for consistency.) Since we have $\overline{\min(b_1, b_j)} = \bar{b}_1$, we obtain the summation in the third line again.

The third and final case is when $\gamma$ is separating but does not cut off an annulus. There are $m/2$ arcs (including $\gamma$) parallel to $\gamma$. As in the first case, there are $\frac{m}{2} - 1$ complementary regions between the $\frac{m}{2}$ parallel arcs. Cutting along $\gamma$ and removing arcs which become boundary-parallel, we obtain a surface with $r - \frac{m}{2} + 1$ complementary regions. This surface is disconnected, with two components $S_1, S_2$ with numbers of complementary regions $r_1, r_2$ satisfying $r_1 + r_2 = r - \frac{m}{2} + 1$. Thus, again following the previous proof, the number of arc diagrams in this case is given by the final line in the recursion.
\end{proof}

Rewriting the recursion in terms of $t$ rather than $r$ achieves a simplification. Let $t$ be the parameter for the left hand side, and $t'$ for the term in the first line of the right hand side. Then
\[
t = r - (2-2g-n) - \frac{1}{2} \sum_{i=1}^n b_i
= r - \frac{m}{2} + 1 - (2 - 2(g-1)-(n+1)) - \frac{1}{2} (i + j + \sum_{i=2}^n b_i )
= t'
\]
where we used $i+j+m=b_1$. If we write $t''$ for the parameter for the term in the second line, we similarly obtain $t'' = t - \delta_{b_j,0}$. In the third line, if $b_1 \geq b_j$ then $\min(b_1, b_j) = b_j$, so writing $t'''$ for the parameter, we have
\begin{align*}
t''' 
&= r - \frac{m}{2} - \overline{b_j} + 1 - (2-2g-(n-1)) - \frac{1}{2}(i+b_2 + \cdots + \widehat{b_j} + \cdots + b_n) \\
&= r - (2-2g-n) - \frac{1}{2} \sum_{i=1}^n b_i + \frac{1}{2}(b_1 - b_j - i - m) + b_j - \bar{b}_j = t - \delta_{b_j,0}.
\end{align*}
Here we used the fact that $i+m = b_1 - b_j$ in the summation. If alternatively $b_1 \leq b_j$, then we obtain $t''' = t - \delta_{b_1, 0}$. Either way, we have $t''' = t - \delta_{\min(b_1, b_j),0}$. But we assume $b_1 > 0$, and if $b_j = 0$ then clearly $\min(b_1, b_j) = b_j = 0$, so $\delta_{\min(b_1, b_j),0} = \delta_{b_j,0}$. In the final term, if the two factors have parameters $t_1, t_2$, the condition $r_1 + r_2 = r - \frac{m}{2} + 1$ translates to $t_1 + t_2 = t$. We now have the following.
\begin{cor}
\label{cor:Ngnt_recursion}
For $(g,n) \neq (0,1), (0,2), (0,3)$ and integers $b_1, \ldots, b_n$ such that $b_1 >0$, $b_2, \ldots, b_n \geq 0$,
\begin{align*}
N_{g,n}^t({\bf b}) &= \sum_{\substack{ i,j,m \geq 0 \\ i+j+m = b_1 \\ m \text{ even}}} \frac{m}{2} N_{g-1,n+1}^t (i,j,b_2, \ldots, b_n) \\
& + \sum_{j=2}^n \Bigg( \sum_{\substack{i,m \geq 0 \\ i+m = b_1 + b_j \\ m \text{ even}}} \frac{m}{2} \bar{b}_j N_{g,n-1}^{t - \delta_{b_j, 0}} (i,b_2, \ldots, \widehat{b_j}, \ldots, b_n)
+ \widetilde{\sum_{\substack{i,m \geq 0 \\ i+m = b_1 - b_j \\ m \text{ even}}}} \frac{m}{2} \bar{b}_j N_{g,n-1}^{t - \delta_{b_j,0}} (i, b_2, \ldots, \widehat{b_j}, \ldots, b_n) \Bigg) \\
& + \sum_{\substack{g_1 + g_2 = g \\ I \sqcup J = \{2, \ldots, n\} \\ \text{No discs or annuli}}} \sum_{\substack{i,j,m \geq 0 \\ i+j+m = b_1 \\ m \text{ even}}}
\sum_{\substack{t_1, t_2 \geq 0 \\ t_1 + t_2 = t}}
 \frac{m}{2} N_{g_1, |I|+1}^{t_1} (i, b_I) N_{g_2, |J|+1}^{t_2} (j, b_J).
\end{align*}
\qed
\end{cor}

Dividing through by $\bar{b}_2 \cdots \bar{b}_n$ immediately gives a recursion on $\widehat{N}_{g,n,r}$.
\begin{cor}
\label{cor:Nhatgnt_recursion}
For $(g,n) \neq (0,1), (0,2), (0,3)$, $b_1 > 0$ and $b_2, \ldots, b_n \geq 0$,
\begin{align*}
b_1 \widehat{N}_{g,n}^t ({\bf b})
&=
\sum_{\substack{i,j,m \geq 0 \\ i+j+m = b_1 \\ m \text{ even}}}
\frac{1}{2} \; \bar{i} \; \bar{j} \; m \; \widehat{N}_{g-1,n+1}^t (i,j, b_2, \ldots, b_n) \\
& \quad +
\sum_{j=2}^n \frac{1}{2} 
\Bigg( 
\sum_{\substack{i,m \geq 0 \\ i+m = b_1 + b_j \\ m \text{ even}}}
\bar{i} \; m \; \widehat{N}_{g,n-1}^{t - \delta_{b_j,0}} (i, b_2, \ldots, \widehat{b_j}, \ldots, b_n) 
+ \widetilde{\sum_{\substack{i,m \geq 0 \\ i+m = b_1 - b_j \\ m \text{ even}}}} \bar{i} \; m \; \widehat{N}_{g,n-1}^{t-\delta_{b_j,0}} (i, b_2, \ldots, \widehat{b_j}, \ldots, b_n) \Bigg) \\
& \quad +
\sum_{\substack{g_1 + g_2 = g \\ I \sqcup J = \{2, \ldots, n\} \\ \text{No discs or annuli}}}
\sum_{\substack{i,j,m \geq 0 \\ i+j+m = b_1 \\ m \text{ even}}}
\sum_{\substack{t_1, t_2 \geq 0 \\ t_1 + t_2 = t}}
\frac{1}{2} \; \bar{i} \; \bar{j} \; m \; \widehat{N}_{g_1, |I|+1}^{t_1} (i, b_I) \; \widehat{N}_{g_2, |J|+1}^{t_2} (j, b_J).
\end{align*}
\qed
\end{cor}

\subsection{Polynomiality in small cases}
\label{sec:polynomiality_small_cases}

We can now use the recursion of corollary \ref{cor:Nhatgnt_recursion} to find $\widehat{N}_{g,n}^t$ for $(g,n) = (1,1)$ and $(0,4)$. 

Consider $\widehat{N}_{1,1}^t (b_1)$. We assume $b_1$ is even. We have, for $b_1 > 0$, 
\[
b_1 \widehat{N}_{1,1}^t (b_1)
=
\sum_{\substack{i,j,m \geq 0 \\ i+j+m = b_1 \\ m \text{ even}}}
\frac{1}{2} \bar{i} \; \bar{j} \; m \; \widehat{N}_{0,2}^t (i,j).
\]
Now lemma \ref{lem:N01t_N02t_computations} we have computed $\widehat{N}_{0,2}^t$: we found $\widehat{N}_{0,2}^0 (b, b) = \frac{1}{\bar{b}}$ for $b > 0$, $\widehat{N}_{0,2}^1 (0,0) = 1$, and all other $\widehat{N}_{0,2}^t(b_1, b_2) = 0$.

Thus we only need consider the cases $t=0,1$. In the $t=0$ case we obtain
\[
b_1 \widehat{N}_{1,1}^0 (b_1) = \sum_{\substack{i > 0, \; m \geq 0 \\ 2i+m = b_1 \\ m \text{ even}}} \frac{1}{2} i^2 \; m \; \frac{1}{i}
= \sum_{\substack{i>0, \; m \geq 0 \\ 2i+m = b_1 \\ m \text{ even}}} \frac{1}{2} i \; m
= \frac{1}{4} \sum_{\substack{\iota, m \geq 0 \\ \iota +m = b_1 \\ m \text{ even}}}  \iota \; m
= \frac{1}{4} S_0 (b_1)
= \frac{1}{48} b_1^3 - \frac{1}{12} b_1
\]
Here we let $2i = \iota$, and $S_0$ is the sum studied in section \ref{sec:useful_sums}.

For $t=1$, we have a nonzero term only when $i=j=0$:
\[
b_1 \widehat{N}_{1,1}^1 (b_1)
=
\sum_{\substack{i,j,m \geq 0 \\ i+j+m = b_1 \\ m \text{ even}}}
\frac{1}{2} \bar{i} \; \bar{j} \; m \; \widehat{N}_{0,2}^0 (i,j)
= 
\frac{1}{2} b_1.
\]

The above assumes that $b_1 > 0$. When $b_1 = 0$, the only nonzero count is $N_{1,1,1}(0) = N_{1,1}^{2} (0) = 1$. We have now computed all $\widehat{N}_{1,1}^t$.

\begin{prop}
For $b_1$ even and nonzero,
\begin{align*}
\widehat{N}_{1,1}^0 (b_1) &= \frac{1}{48} b_1^2 - \frac{1}{12} \\
\widehat{N}_{1,1}^1 (b_1) &= \frac{1}{2} \\
\widehat{N}_{1,1}^2 (0) &= 1.
\end{align*}
All other $\widehat{N}_{1,1}^t(b_1)$ are zero.
\qed
\end{prop}

We can summarise the polynomials for $\widehat{N}_{1,1}^t$ in a table of $k$ and $t$.
\begin{center}
\begin{tabular}{x{0.5cm}| c c c}
\diag{.1em}{.5cm}{$k$}{$t$} & $0$ & $1$ & $2$  \\ \hline
$0$ & $\frac{1}{48}b_1^2 - \frac{1}{12}$ & $\frac{1}{2}$ \\
$1$ & & & $1$
\end{tabular}
\end{center}

We can also consider the case $(g,n) = (0,4)$. Corollary \ref{cor:Nhatgnt_recursion} gives, for $b_1 > 0$,
\begin{align}
\label{eqn:N04t}
b_1 \widehat{N}_{0,4}^t({\bf b}) 
=
\sum_{j=2}^4 \frac{1}{2} \Bigg(
&\sum_{\substack{i,m \geq 0 \\ i+m = b_1 + b_j \\ m \text{ even}}}
\bar{i} \; m \; \widehat{N}_{0,3}^{t - \delta_{b_j,0}} (i, b_2, \ldots, \widehat{b_j}, \ldots, b_n) \notag
\\
 +
&\widetilde{\sum_{\substack{i,m \geq 0 \\ i+m = b_1 - b_j \\ m \text{ even}}}} \bar{i} \; m \; \widehat{N}_{0,3}^{t-\delta_{b_j,0}} (i, b_2, \ldots, \widehat{b_j}, \ldots, b_n) \Bigg).
\end{align}
Proposition \ref{prop:bounds_on_r_and_t} gives us bounds on $k$ and $t$. Either $0 \leq k \leq 3$ and $\max ( k, 3 - \frac{1}{2} \sum_{i=1}^n b_i ) \leq t \leq \min ( 3, 1+k )$, or $k=4$ and $t=3$. Since $b_i$ may become large, we first consider $0 \leq k \leq 3$ and $k \leq t \leq \min(k+1,3)$. This gives 8 cases to consider: $(k,t) = (0,0), (0,1), (1,1), (1,2), (2,2), (2,3), (3,3), (4,3)$.

So, first take $t=0$ and $k=0$. Then equation \eqref{eqn:N04t} expresses $\widehat{N}_{0,4}^0(b_1, b_2, b_3, b_4)$ in terms of $\widehat{N}_{0,3}^0$. From proposition \ref{prop:Nhat03t}, we see that $\widehat{N}_{0,3}^0(b_1, b_2, b_3) = 1$ provided $b_1 + b_2 + b_3$ is even, and all $b_i$ are nonzero. Hence every $\widehat{N}_{0,3}^0(i, b_2, \ldots, \widehat{b_j}, \ldots, b_n) = 1$, except when $i=0$. We see sums $S_0 (b_1 \pm b_j)$, and obtain
\[
2 b_1 \widehat{N}_{0,4}^0 (b_1, b_2, b_3, b_4)
= S_0 (b_1 + b_2) + S_0 (b_1 - b_2) + S_0 (b_1 + b_3) + S_0 (b_3 - b_3) + S_0 (b_1 + b_4) + S_0 (b_1 - b_4).
\]
We have $S_0 (k) = \frac{k^3}{12} - \frac{k}{3}$ when $k$ is even, and $\frac{k^3}{12} - \frac{k}{12}$ when $k$ is odd. Thus, we obtain
\[
\widehat{N}_{0,4}^0 (b_1, b_2, b_3, b_4) = 
\left\{ \begin{array}{ll}
\frac{1}{4}(b_1^2 + b_2^2 + b_3^2 + b_4^2) - 1 & \text{all $b_i$ even,} \\
\frac{1}{4}(b_1^2 + b_2^2 + b_3^2 + b_4^2) - \frac{1}{2} & \text{two $b_i$ even, two odd,} \\
\frac{1}{4}(b_1^2 + b_2^2 + b_3^2 + b_4^2) - 1 & \text{all $b_i$ odd.}
\end{array} \right.
\]
Next we take $k=0$, $t=1$. In this case \ref{eqn:N04t} expresses $\widehat{N}_{0,4}^1({\bf b})$ in terms of $\widehat{N}_{0,3}^1$. Proposition \ref{prop:Nhat03t} says that $\widehat{N}_{0,3}^1 (b_1, b_2, b_3) = 1$ provided that precisely one of the $b_i$ is zero, and $b_1 + b_2 + b_3$ is even. As $k=0$, all $b_i > 0$ so only setting $i=0$ (hence $m = b_1 \pm b_j$) can provide the zero. But $i \equiv b_1 \pm b_j \pmod{2}$, so only those $j$ for which $b_j \equiv b_1$ provide a nonzero term. If all $b_i$ are even, or all $b_i$ are odd, then all $j$ provide a nonzero term, and we obtain
\[
2 b_1 \widehat{N}_{0,4}^1 (b_1, b_2, b_3, b_4) 
=
(b_1 + b_2) + (b_1 - b_2) + (b_1 + b_3) + (b_1 - b_3) + (b_1 + b_4) + (b_1 - b_4)
= 6b_1
\]
and hence $\widehat{N}_{0,4}^1 (b_1, b_2, b_3, b_4) = 3$. But if two of the $b_i$ are even and two of the $b_i$ are odd, then we obtain $2 b_1 \widehat{N}_{0,4} (b_1, b_2, b_3, b_4) = 2b_1$, so $\widehat{N}_{0,4}^1(b_1, b_2, b_3, b_4) = 1$.

Consider next $k=1$, $t=1$; set $b_4 = 0$ and assume $b_1, b_2, b_3 > 0$. Equation \eqref{eqn:N04t} again expresses $\widehat{N}_{0,4}^1$ in terms of $\widehat{N}_{0,3}$; but since $b_4 = 0$ we now obtain terms $\widehat{N}_{0,3}^1 (i, b_3, 0)$, $\widehat{N}_{0,3}^1 (i, b_2, 0)$ and $\widehat{N}_{0,3}^0 (i, b_2, b_3)$. In each case we obtain $1$ when $i > 0$ ($i$ is always of the appropriate parity) and and zero otherwise. Thus we obtain
\[
2 b_1 \widehat{N}_{0,4}^1 (b_1, b_2, b_3, 0)
=
S_0 (b_1 + b_2) + S_0 (b_1 - b_2) + S_0 (b_1 + b_3) + S_0 (b_1 - b_3) + 2 S_0 (b_1).
\]
We then obtain $\widehat{N}_{0,4}^1$ depending on the parity of the nonzero $b_i$
\begin{align*}
\widehat{N}_{0,4}^1 (b_1, b_2, b_3, 0) =
\begin{cases}
\frac{1}{4} (b_1^2 + b+2^2 + b_3^2) - 1, & \text{all $b_i$ even,} \\
\frac{1}{4} (b_1^2 + b_2^2 + b_3^2) - \frac{1}{2}, & \text{two $b_i$ odd, one even.}
\end{cases}
\end{align*}

Proceeding in a similar fashion through the rest of the cases, we end up with the following result.

\begin{prop}
For the various possible values of $t,k$, with $b_1, \ldots, b_{n-k} > 0$ and $b_{n-k+1} = \cdots = b_n = 0$, $\widehat{N}_{0,4}^t(b_1, \ldots, b_{n-k}, 0, \ldots, 0)$ is given by the following tables.
\begin{enumerate}
\item
If all $b_i$ are even:

\begin{center}
\begin{tabular}{x{0.5cm}| c c c c}
\diag{.1em}{.5cm}{$k$}{$t$} & $0$ & $1$ & $2$ & $3$ \\ \hline
$0$ & $\frac{1}{4}(b_1^2 + b_2^2 + b_3^2 + b_4^2) - 1$ & $3$ \\
$1$ & & $\frac{1}{4}(b_1^2 + b_2^2 + b_3^2) - 1$ & $3$ \\
$2$ & & & $\frac{1}{4} (b_1^2 + b_2^2)$ & $2$\\
$3$ & & & & $\frac{1}{4} b_1^2 + 2$ \\
$4$ & & & & $1$
\end{tabular}
\end{center}

\item
If two $b_i$ are odd:

\begin{center}
\begin{tabular}{x{0.5cm}| c c c c}
\diag{.1em}{.5cm}{$k$}{$t$} & $0$ & $1$ & $2$ & $3$ \\ \hline
$0$ & $\frac{1}{4}(b_1^2 + b_2^2 + b_3^2 + b_4^2) - \frac{1}{2}$ & $1$ \\
$1$ & & $\frac{1}{4}(b_1^2 + b_2^2 + b_3^2) - \frac{1}{2} $ & $1$ \\
$2$ & & & $\frac{1}{4} (b_1^2 + b_2^2) + \frac{1}{2} $ & $0$\\
$3$ & & & & $0$ \\
$4$ & & & & $0$
\end{tabular}
\end{center}

\item
If four $b_i$ are odd:

\begin{center}
\begin{tabular}{x{0.5cm}| c c c c}
\diag{.1em}{.5cm}{$k$}{$t$} & $0$ & $1$ & $2$ & $3$ \\ \hline
$0$ & $\frac{1}{4}(b_1^2 + b_2^2 + b_3^2 + b_4^2) - 1$ & $3$ \\
$1$ & & $0$ & $0$ \\
$2$ & & & $0$ & $0$\\
$3$ & & & & $0$ \\
$4$ & & & & $0$
\end{tabular}
\end{center}

\end{enumerate}
\qed
\end{prop}

In these examples, within the range $0 \leq k \leq n-1$ and $k \leq t \leq \min(2g+n-1, k+3g-3+n)$ the degrees of the polynomials decrease as $t$ increases, and increase as $k$ increases. When all $b_i$ are even, these polynomials are all nonzero and their degrees in the $b_i^2$ precisely decrease by 1 at each step. However, when the $b_i$ are not all even, sometimes the polynomials drop abruptly to zero. Sometimes this is forced: for instance if $k$ of the $b_i$ are zero, then we can have at most $n-k$ of the $b_i$ being odd. But even when the value of $k$ does not force $\widehat{N}_{g,n}^t (b_1, \ldots, b_{n-k}, 0, \ldots, 0)$ to be zero for parity reasons, the polynomial may drop to zero anyway, as seen above for $\widehat{N}_{0,4}^3 (b_1, b_2, 0, 0)$ with $b_1, b_2$ odd.

We will prove that such behaviour always occurs in the next section.

\subsection{Polynomiality of refined non-boundary-parallel counts}

We will prove the following theorem.

\begin{thm}
\label{thm:Nhatgnt_polynomiality}
Suppose that $(g,n) \neq (0,1), (0,2)$. Let $k,t$ be non-negative integers and $b_1, \ldots, b_{n-k}$ be positive integers.
\begin{enumerate}
\item
If $0 \leq k \leq n-1$ and $k \leq t \leq \min(2g+n-1, k+3g-3+n)$, then $\widehat{N}_{g,n}^t (b_1, \ldots, b_{n-k}, 0, \ldots, 0)$ is a symmetric quasi-polynomial over $\Q$ in $b_1^2, \ldots, b_{n-k}^2$, depending on the parity of $b_1, \ldots, b_{n-k}$. 
\item
If $k=n$ and $t = 2g+n-1$, then $\widehat{N}_{g,n}^t (0, \ldots, 0) = 1$.
\item
For any other values of $k$ and $t$, $\widehat{N}_{g,n}^t (b_1, \ldots, b_{n-k}, 0, \ldots, 0) = 0$.
\end{enumerate}
\end{thm}

When $k=t=0$, this theorem reduces to theorem \ref{thm:N_polynomiality_k_t_both_0} (apart from the statement about degree).

The proof is essentially a refinement of the proof of theorem \ref{thm:N_polynomiality}. The computations above have established the theorem for $(g,n) = (0,3), (0,4)$ and $(1,1)$. However, because of the inequalities on $g,k,n,t$, establishing that various terms are nonzero is a more technical exercise.

\begin{proof}
We can dispose of parts (ii) and (iii) quickly. When $k = n$, we have all $b_i = 0$, so the only possible arc diagram is the empty one, which has $t = 2g+n-1$, so $\widehat{N}_{g,n}^t (0, \ldots, 0) = 1$ as claimed, proving (ii).

Suppose $k,t$ are not covered by parts (i) or (ii). As $k$ is the number of zero boundary components, $0 \leq k \leq n$. If $0 \leq k \leq n-1$ then we must have $t<k$ or $t > 2g+n-1$ or $t > k+3g-3+n$; and if $k = n$, then we must have $t \neq 1 - \chi$. In any of these cases, the conditions of proposition \ref{prop:bounds_on_r_and_t} are violated, so $\widehat{N}_{g,n}^t = 0$ as claimed in (iii).

It remains to prove (i). The proof is by induction on the complexity $-\chi = 2g+n-2$; we have computed the $-\chi = 1$ cases $(g,n) = (0,3)$ and $(1,1)$ explicitly. We now take $(g,n)$ with complexity $\geq 2$, assuming the theorem holds for any smaller complexity. We also take $k,t$ such that $0 \leq k \leq n-1$ and $k \leq t \leq \min(2g+n-1, k+3g-3+n)$. Take $k$ of the $b_i$ to be zero; without loss of generality assume $b_1, \ldots, b_{n-k} > 0$ and $b_{n-k+1} = \cdots = b_n = 0$. Further, fix the parity of $b_1, \ldots, b_{n-k}$; we must show that $\widehat{N}_{g,n}^t (b_1, \ldots, b_{n-k}, 0, \ldots, 0)$ is a polynomial with the required properties.

The recursion in corollary \ref{cor:Nhatgnt_recursion}
expresses $\widehat{N}_{g,n}^t ({\bf b})$ in terms of $\widehat{N}_{g',n'}^{t'}$ where $(g',n')$ is of smaller complexity (but neither $(g',n') = (0,1)$ nor $(0,2)$ are ever seen), hence for which the result holds. Explicitly, the following $\widehat{N}$s occur:
\begin{itemize}
\item $\widehat{N}_{g-1,n+1}^t (i,j,b_2, \ldots, b_n)$ where $i,j \geq 0$, $i+j \leq b_1$ and $i+j \equiv b_1 \pmod{2}$;
\item $\widehat{N}_{g,n-1}^{t - \delta_{b_j,0}} (i, b_2, \ldots, \widehat{b_j}, \ldots, b_n)$, where $i \geq 0$, $i \leq b_1 \pm b_j$ and $i \equiv b_1 \pm b_j \pmod{2}$;
\item $\widehat{N}_{g_1, |I|+1 }^{t_1} (i, b_I) \; \widehat{N}_{g_2, |J|+1}^{t_2} (j, b_J)$ where $g_1, g_2, i, j, t_1, t_2 \geq 0$, $g_1 + g_2 = g$, $i+j \leq b_1$, $i+j \equiv b_1 \pmod{2}$, $t_1 + t_2 = t$, $|I|, |J| \geq 2$, and $|I| + |J| = n-1$.
\end{itemize}

Expanding out the $\sum_{j=2}^n$ sum in the second line, and the sums over $g_1 + g_2 = g$, $I \sqcup J = \{2, \ldots, n\}$, $t_1 + t_2 = t$ in the third line, we express $b_1 \widehat{N}_{g,n}^t (b_1, \ldots, b_{n-k}, 0, \ldots, 0)$ as a finite collection of sums of the types
\[
\text{Type 1: } \widetilde{\sum_{\substack{i,m \geq 0 \\ i+m = b_1 \pm b_j \\ m \text{ even}}}} \bar{i} \; m \; \cdots
\quad \text{or} \quad
\text{Type 2: } \widetilde{\sum_{\substack{i,j,m \geq 0 \\ i+j+m = b_1 \\ m \text{ even}}}} \bar{i} \; \bar{j} \; m \; \cdots.
\]
Here the $\cdots$ represents some constant times an $\widehat{N}_{\cdot,\cdot}^\cdot(b_I, 0, \ldots, 0)$, or a product of two such terms. As discussed in the proof of \ref{cor:Nhatgnt_recursion}, having fixed the parity of $b_1, \ldots, b_{n-k}$, the parity of $i$ in a sum of type 1 is determined, but in a sum of type 2 only the parity of $i+j$ is fixed; so there are two possibilities for $(i,j) \pmod{2}$. Fixing the parity of all variables, every $\widehat{N}$ occurring has inputs which are all fixed in parity. We further need to distinguish between zero and nonzero inputs to each $\widehat{N}$. So we split sums of type 1 into the $i=0$ term and the sum over $i>0$ terms. And we split sums of type 2 into the $i=0, j=0$ term, a sum over $i=0, j>0$ terms, a sum over $i>0, j=0$ terms, and a sum over $i>0, j>0$ terms.

Each term of type 1 becomes a finite collection of monomial terms, or sums, of one of the forms
\begin{align*}
q(b_I) (b_1 \pm b_j), \\
q(b_I) \widetilde{\sum_{\substack{i > 0, \; m \geq 0 \\ i+m = b_1 \pm b_j \\ i \text{ even}, \;  m \text{ even}}}} \bar{i} \; i^{2a} m 
&= \left\{ \begin{array}{ll} 
q(b_I) S_a (b_1 \pm b_j) & b_1 \pm b_j \text{ even} \\ 
0 & b_1 \pm b_j \text{ odd,} \end{array} \right. \\
q(b_I) \widetilde{\sum_{\substack{i>0, \; m \geq 0 \\ i+m = b_1 \pm b_j \\ i \text{ odd}, \; m \text{ even}}}} 
\bar{i} \; i^{2a} m &= 
\left\{ \begin{array}{ll} 
0 & b_1 \pm b_j \text{ even} \\ 
q(b_I) S_a (b_1 \pm b_j) & b_1 \pm b_j \text{ odd.} 
\end{array} \right.
\end{align*}
Here each $q(b_I)$ is a constant multiplied by a monomial in the $b_i^2$ other than $b_1^2$ and $b_j^2$. We have seen (lemma \ref{lem:As_and_Bs}) that $S_a (k)$, with fixed parity of $k$, is an odd polynomial. Every time we see an $S_a$, it appears in a pair $S_a (b_1 + b_j) + S_a (b_1 - b_j)$, which is odd in $b_1$ and even in $b_j$.

Each term of type 2, similarly, becomes a finite collection of sums of one of the forms
\begin{align*}
q(b_I) b_1, \\
q(b_I) \widetilde{\sum_{\substack{i>0, m \geq 0 \\ i+m = b_1 \\ i, m \text{ even}}}} \bar{i} \; i^{2a} \; m
&= \left\{ \begin{array}{ll}
q(b_I) S_a (b_1) & b_1 \text{ even}, \\
0 & b_1 \text{ odd},
\end{array} \right.
\\
q(b_I) \widetilde{\sum_{\substack{i>0, \; m \geq 0 \\ i+m = b_1 \\ i \text{ odd}, \; m \text{ even}}}} \bar{i} \; i^{2a} \; m
&= \left\{ \begin{array}{ll}
0 & b_1 \text{ even}, \\
q(b_I) S_a (b_1) & b_1 \text{ odd},
\end{array} \right.
\end{align*}
\begin{align*}
q(b_I) \widetilde{\sum_{\substack{i,j>0, \; m \geq 0 \\ i+j+m = b_1 \\ i,j,m \text{ even}}}} \bar{i} \; \bar{j} \; i^{2a} j^{2b} m
&= \left\{ \begin{array}{ll}
q(b_I) R_{a,b}^0 (b_1) & b_1 \text{ even,} \\
0 & b_1 \text{ odd,}
\end{array} \right. 
\\
q(b_I) \widetilde{\sum_{\substack{i,j>0, \; m \geq 0 \\ i+j+m = b_1 \\ i \text{ odd}, \; j, m \text{ even}}}} \bar{i} \; \bar{j} \; i^{2a} j^{2b} m
&= \left\{ \begin{array}{ll}
0 & b_1 \text{ even}, \\
q(b_I) R_{a,b}^1  (b_1) & b_1 \text{ odd},
\end{array} \right.
\\
q(b_I) \widetilde{\sum_{\substack{i,j > 0, \; m \geq 0 \\ i+j+m = b_1 \\ i,j \text{ odd}, \; m \text{ even}}}} \bar{i} \; \bar{j} \; i^{2a} j^{2b} m
&= \left\{ \begin{array}{ll}
q(b_I) R_{a,b}^1 (b_1) & b_1 \text{ even}, \\
0 & b_1 \text{ odd}.
\end{array} \right.
\end{align*}
Each $S(b_1)$ and $R_{a,b}(b_1)$ is an odd polynomial in $b_1$ (lemma \ref{lem:As_and_Bs}).

Collecting all these terms together, we obtain on the right hand side a polynomial which is odd in $b_1$ and even in all other variables. Hence $\widehat{N}_{g,n}^t(b_1, \ldots, b_{n-k}, 0, \ldots, 0)$ is an even polynomial in all variables.
\end{proof}

We can, further, say something about the degrees of the polynomials involved.
\begin{thm}
\label{thm:Nhatgnt_degree_upper_bound}
Suppose that $(g,n) \neq (0,1), (0,2)$. Let $k,t$ be non-negative integers satisfying $0 \leq k \leq n-1$ and $k \leq t \leq \min(2g+n-1, k+3g-3+n)$. Let $b_1, \ldots, b_{n-k}$ be positive integers. Fixing the parity of $b_1, \ldots, b_{n-k}$, the degree of the polynomial $\widehat{N}_{g,n}^t (b_1, \ldots, b_{n-k}, 0, \ldots, 0)$ in the $b_i^2$ is at most $3g-3+n-t+k$. 
\end{thm}

We will see in theorem \ref{thm:Ngntk_Ngn_agreement} that when $0 \leq k \leq n-1$ and $t=k$, the degree is in fact exactly $3g-3+n-t+k$. Note that the bounds $k \leq t \leq k+3g-3+n$ provide ``just enough room" in $t$ for the degrees of the polynomials $\widehat{N}_{g,n}^t$ to decrease from $3g-3+n$ (when $t=k$) to $0$ (when $t=k+3g-3+n$). However, as we have seen, it is possible for the polynomials obtained to have degree less than $k+3g-3+n$.

\begin{proof}
From the previous theorem such an $\widehat{N}_{g,n}^t(b_1, \ldots, b_{n-k}, 0, \ldots, 0)$ is a quasi-polynomial of the claimed type; we only need to check its maximum degree. To do this we consider each term of the recursion separately, and consider the possible $\widehat{N}_{g',n'}^{t'}(b_1, \ldots, b_{n-k'}, 0, \ldots, 0)$ which can occur, keeping track of the possible genera $g'$, numbers of boundary components $n'$, complementary region parameter $t'$, and number of boundary components with no marked points $k'$.

In the first line of the recursion (case 1), we see terms involving $\widehat{N}_{g-1,n+1}^t (i,j,b_2, \ldots, b_{n-k}, 0, \ldots, 0)$. So $g' = g-1$, $n'=n+1$ and $t'=t$. The variables $i$ and $j$ can be zero or nonzero, hence $k' = k, k+1$ or $k+2$. We refer to these cases as 1a, 1b, 1c respectively.

In the second line of the recursion (case 2), we have $\widehat{N}_{g,n-1}^{t-\delta_{b_j,0}} (i, b_2, \ldots, \widehat{b_j}, \ldots, b_n)$. Here $i$ and $b_j$ can be zero or nonzero. We refer to the cases $(\Sgn i, \Sgn b_j) = (0,0), (0,1), (1,0), (1,1)$ as 2a, 2b, 2c, 2d respectively. 

In the third line of the recursion (case 3), we have $\widehat{N}_{g_1, |I|+1}^{t_1} (i, b_I) \widehat{N}_{g_2, |J|+1}^{t_2} (j, b_J)$. Let $k_1, k_2$ be the number of zeroes in $(i, b_I)$ and $(j, b_J)$ respectively. We deal with the two $\widehat{N}$ terms separately. There are many possibilities for $g_1, g_2, |I|, |J|, t_1, t_2, k_1 k_2$, subject to the constraints in the summations. There are also the further possibilities that $i,j$ may be zero or nonzero. We refer to the cases $(\Sgn i, \Sgn j) = (0,0), (0,1), (1,0), (1,1)$ as 3a, 3b, 3c, 3d respectively. 

In cases 1a-2d, in order to compute the possible degrees, we construct a table of the possible $g',n',t',k'$, together with the maximum degree $3g'-3+n'-t'+k'$ of the corresponding quasi-polynomials $\widehat{N}_{g',n'}^{t'}$. 
In all cases we assume $g \geq 0$, $n \geq 1$, and $0 \leq k \leq n-1$ (we assume that $b_1 > 0$ in corollary \ref{cor:Nhatgnt_recursion}).

\begin{center}
\begin{tabular}{l|ccccc}
Case & $g'$ & $n'$ & $t'$ & $k'$ & $3g' - 3 + n' - t' + k'$ \\ \hline
1a & $g-1$ & $n+1$ & $t$ & $k$ & $3g-5+n-t+k$ \\
1b & $g-1$ & $n+1$ & $t$ & $k+1$ & $3g-4+n-t+k$ \\
1c & $g-1$ & $n+1$ & $t$ & $k+2$ & $3g-3+n-t+k$ \\
2a & $g$ & $n-1$ & $t-1$ & $k$ & $3g-3+n-t+k$ \\
2b & $g$ & $n-1$ & $t$ & $k+1$ & $3g-3+n-t+k$ \\
2c & $g$ & $n-1$ & $t-1$ & $k-1$ & $3g-4+n-t+k$ \\
2d & $g$ & $n-1$ & $t$ & $k$ & $3g-4+n-t+k$
\end{tabular}
\end{center}

Recall, we are assuming that $g \geq 0$, $n \geq 1$, $(g,n) \neq (0,1), (0,2)$, $0 \leq k \leq n-1$ and $k \leq t \leq \min(2g+n-1, k+3g-3+n)$.

In case 1a, we have $k'=k$, so we sum $\frac{1}{2} i j \widehat{N}_{g',n'}^{t'}$ over $i,j > 0$, subject to $i+j+m = b_1$ where $m \geq 0$ is even. By induction, after fixing the parity of all entries, $\widehat{N}_{g',n'}^{t'}(i,j,b_2, \ldots, b_{n-k}, 0, \ldots, 0)$ has degree at most $6g-10+2n-2t+2k$ in $i, j$ and the $b_i$. (If $i+j$ and $b_1$ have distinct parity then the polynomial is zero, but the degree condition is still satisfied.) After multiplying by $ijm$ and performing the summation, obtaining a $R_{a,b}^{0 \text{ or } 1}(b_1)$ in the process, we have a polynomial of degree at most $6g-5+2n-2t+2k$ which is odd in $b_1$ and even in all other variables; dividing by $b_1$ we obtain a polynomial of degree at most $6g-6+2n-2t+2k$, hence of degree $\leq 3g-3+n-t+k$ in the $b_i^2$.

In case 1b, we have $k' = k+1$, so one of $i$ or $j$ is set to zero and the other is positive; without loss of generality suppose $j=0$ and $i>0$. Then we sum $\frac{1}{2} i \widehat{N}_{g',n'}^{t'} (i,0,b_2, \ldots, b_{n-k}, 0, \ldots, 0)$ with $k'=k+1$ over $i+m = b_1$ where $m \geq 0$ is even. Fixing the parity of all entries again, this $\widehat{N}_{g',n'}^{t'}$ is a polynomial (possibly zero) of degree at most 
$6g-8+2n-2t+2k$. Multiplying by $im$ and summing, obtaining an $S_a (b_1)$ in the process, yields a polynomial of degree $\leq 6g-5+2n-2t+2k$, odd in $b_1$ and even in all other $b_i$. Dividing by $b_1$ we again have a polynomial of degree $\leq 3g-3+n-t+k$ in the $b_i^2$. 

In case 1c, with $k' = k+2$, both of $i$ and $j$ are set to zero. Then our sum reduces to a single term $\frac{1}{2} b_1 \widehat{N}_{g',n'}^{t'}(0,0,b_2, \ldots, b_{n-k}, 0, \ldots, 0)$ (possibly zero if $i+j$ and $b_1$ have distinct parity). Fixing parity and dividing out by $b_1$ as usual, we have a polynomial of degree $\leq 
3g-3+n-t+k$ in the $b_i^2$. 

In case 2, we consider the sum of $\frac{1}{2} \bar{i} m \widehat{N}_{g,n-1}^{t'} (i, b_2, \ldots, \widehat{b_j}, \ldots, b_{n-k}, 0, \ldots, 0)$, over $i$ and $m \geq 0$ satisfying $i+m = b_1 \pm b_j$ with $m$ even. There are two summations, one with $b_1 + b_j$ and one with $b_1 - b_j$, and we add them. In case 2a, we have $i=0$ and $b_j = 0$, and the sums both reduce to the same single term $\frac{1}{2} b_1 \widehat{N}_{g',n'}^{t'} (0, b_2, \ldots, b_{n-k}, 0, \ldots, 0)$ with $t'=t-1$ and $k'=k$. (This term is zero if $b_1$ is odd.) Fixing the parity of the variables and dividing out by $b_1$, we have a polynomial of degree $\leq 
3g-3+n-t+k$ in the $b_i^2$. 

In case 2b we have $i=0$ again, so the sums reduce to single terms, but now $b_j \neq 0$, so the single terms are $\frac{1}{2} (b_1 \pm b_j) \widehat{N}_{g',n'}^{t'}(0, b_2, \ldots, \widehat{b_j}, \ldots, b_{n-k}, 0, \ldots, 0)$, where $t' = t$ and $k' = k+1$. These sum to $b_1 \widehat{N}_{g',n'}^{t'}(0, b_2, \ldots, \widehat{b_j}, \ldots, b_{n-k}, 0, \ldots, 0)$. (This is zero unless $i \equiv b_1 \pm b_j \pmod{2}$.) Fixing parity and dividing out by $b_1$, again we have a polynomial of degree $\leq 
3g-3+n-t+k$ in the $b_i^2$. 

In case 2c we sum over $i>0$, and $b_j = 0$, so we sum $\frac{1}{2} i m \widehat{N}_{g',n'}^{t'}(i, b_2, \ldots, b_{n-k}, 0, \ldots, 0)$, where $t' = t-1$ and $k' = k-1$. Fixing parity of variables, $\widehat{N}_{g',n'}^{t'}$ has degree $\leq 
6g-8+2n-2t+2k$ in its variables. Multiplying by $\frac{1}{2} i m$ and summing, each summation gives an $S_a (b_1 \pm b_j)$, and the result is a polynomial of degree $\leq 6g-5+2n-2t+2k$, odd in $b_1$ and even in all other $b_j$; dividing by $b_1$ gives a polynomial of degree at most $3g-3+n-t+k$ in the $b_i^2$. 

In case 2d we again sum over $i>0$, but now $b_j > 0$. We sum $\frac{1}{2} \widehat{N}_{g',n'}^{t'} (i, b_2, \ldots, \widehat{b_j}, \ldots, b_{n-k}, 0, \ldots, 0)$ where $t' = t$ and $k' = k$. Fixing parity, the $\widehat{N}_{g',n'}^{t'}$ has degree $\leq 
6g-8+2n-2t+2k$, multiplying and summing (obtaining $S_a (b_1 \pm b_j)$ along the way) yields a polynomial of degree $ \leq 6g-5+2n-2t+2k$; again, dividing by $b_1$ gives polynomial with the required properties.

We turn next to cases 3a-3d. In each case, 
each $\widehat{N}_{g_1, |I|+1}^{t_1} (i, b_I)$ and $\widehat{N}_{g_2, |J|+1}^{t_2}(j, b_J)$ by induction satisfies the conditions of the theorem; so once we fix parity of the nonzero variables, and recalling that $g_1 + g_2 = g$, $t_1 + t_2 = t$ and $|I|+|J|=n-1$, we obtain polynomials in the $b_i^2$, with degree
\begin{align*}
\deg \widehat{N}_{g_1, |I|+1}^{t_1} (i, b_I) \; \widehat{N}_{g_2, |J|+1}^{t_2} (j, b_J)
&\leq 
\left( 3g_1 - 3 + (|I| + 1) - t_1 + k_1 \right) + \left( 3g_2 - 3 + (|J| + 1) - t_2 + k_2 \right) \\
&=
3g - 5 + n - t + (k_1 + k_2).
\end{align*}

In case 3a we have $i=j=0$, so $k_1 + k_2 = k+2$ and the sum reduces to a single term $\frac{1}{2} b_1 \widehat{N}_{g_1, |I|+1}^{t_1}(0, b_I) \widehat{N}_{g_2, |J|+1}^{t_2} (0, b_J)$. (This term is zero if $b_1$ is odd, since the sum is over $i+j \equiv b_1 \pmod{2}$.) Dividing out by $b_1$ yields a polynomial of degree $\leq 3g-5+n-t+(k_1 + k_2) = 3g-3+n-t+k$. 

In cases 3b and 3c we have one of $i,j$ being zero and the other nonzero; without loss of generality suppose $j=0$ and $i>0$. Then the sum reduces to a sum over $i>0$ and $m \geq 0$ with $i+m = b_1$ and $m$ even. We have $k_1 + k_2 = k+1$, so fixing parities, $\widehat{N}_{g_1, |I|+1}^{t_1} (i, b_I) \widehat{N}_{g_2, |J|+1}^{t_2} (0, b_J)$ has degree $\leq 6g-8+2n-2t+2k$ in its variables. (If $i$ and $b_1$ have distinct parity, it is zero.) Multiplying by $\frac{1}{2} im$ and summing, we see an $S_a (b_1)$, and obtain a polynomial of degree $\leq 6g-5+2n-2t+2k$ which is odd in $b_1$ and even in all other $b_i$. Dividing by $b_1$, we obtain a polynomial of degree $\leq 3g-3+n-t+k$ in the $b_i^2$.

Finally, in case 3d we sum over $i,j > 0$. We have $k_1 + k_2 = k$, so fixing parities, the product $\widehat{N}_{g_1, |I|+1}^{t_1} (i, b_I) \widehat{N}_{g_2, |J|+1}^{t_2} (j, b_J)$ has degree $\leq 6g-10+2n-2t+2k$ (zero unless $i+j \equiv b_1 \pmod{2}$). Multiplying by $\frac{1}{2} ijm$ and summing, we see a $R_{a,b}^{0 \text{ or } 1} (b_1)$ and obtain a polynomial of degree $\leq 6g-5+2n-2t+2k$; 
dividing by $b_1$ gives a polynomial of degree $\leq 3g-3+n-t+k$ in the $b_i^2$.

We have now shown that, using the recursion, we can take $b_1 \widehat{N}_{g,n}^t (b_1, \ldots, b_{n-k}, 0, \ldots, 0)$, express it as a finite collection of sums, and, fixing the parity of $b_1, \ldots, b_{n-k}$, each nonzero sum yields a polynomial of degree $\leq 3g-3+n-t+k$ in the $b_i^2$ with positive coefficients of highest degree. Summing them, the result is a polynomial of degree at most $3g-3+n-t+k$.
\end{proof}

\subsection{Relations between polynomials, volumes and moduli spaces}

\label{sec:relations_refined}

It is clear (proposition \ref{prop:sum_over_rt}) that when we sum $\widehat{N}_{g,n}^t ({\bf b})$ over $t$, we must obtain $\widehat{N}_{g,n} ({\bf b})$. This is true regardless of whether some $b_i$ are zero. Thus, for any $g \geq 0$, $n \geq 1$ and $0 \leq k \leq n$,
\[
\widehat{N}_{g,n} (b_1, \ldots, b_{n-k}, 0, \ldots, 0)
=
\sum_t \widehat{N}_{g,n}^t (b_1, \ldots, b_{n-k}, 0, \ldots, 0).
\]
When $(g,n) \neq (0,1), (0,2)$, these functions are quasi-polynomials. The summation in $t$ runs over the range $k \leq t \leq \min(2g+n-1, k+3g-3+n)$. 

Obviously, $\widehat{N}_{g,n}(b_1, \ldots, b_{n-k}, 0, \ldots, 0)$ can be obtained from $\widehat{N}_{g,n}(b_1, \ldots, b_{n-k}, b_{n-k+1}, \ldots, b_n)$ by setting $b_{n-k+1} = \cdots = b_n = 0$. But it is not true that setting $b_{n-k+1} = \cdots = b_n = 0$ in $\widehat{N}_{g,n}^t (b_1, \ldots, b_{n-k}, b_{n-k+1}, \ldots, b_n)$ gives $\widehat{N}_{g,n}^t (b_1, \ldots, b_{n-k}, 0, \ldots, 0)$, since $t$ depends on the (sum of the) $b_i$. Indeed, as seen from the examples in section \ref{sec:polynomiality_small_cases}, for distinct values of $k$, the sequences (in $t$) of quasi-polynomials $\widehat{N}_{g,n}^t (b_1, \ldots, b_{n-k}, 0, \ldots, 0)$ may be quite distinct, and cannot be obtained from each other simply by setting some variables to zero (or even by setting variables of designated even parity equal to zero).

Nonetheless, fix $k$ in the range $0 \leq k \leq n-1$, and consider the sequence (in $t$) of quasi-polynomials $\widehat{N}_{g,n}^t (b_1, \ldots, b_{n-k}, 0, \ldots, 0)$. These quasi-polynomials can only be nonzero for $t$ in the range $k \leq t \leq \min(2g+n-1, k+3g-3+n)$, by theorem \ref{thm:Nhatgnt_polynomiality}. By theorem \ref{thm:Nhatgnt_degree_upper_bound}, for such $k$ and $t$, these quasi-polynomials have degree at most $3g-3+n-t+k$. And since $k \leq t$, we have $3g-3+n-t+k \leq 3g-3+n$. However, we know from theorem \ref{thm:N_polynomiality}  that the quasi-polynomials $\widehat{N}_{g,n} (b_1, \ldots, b_{n})$ have degree $3g-3+n$. This leads us to the following.
\begin{thm}
\label{thm:Ngntk_Ngn_agreement}
Let $g \geq 0$ and $n \geq 1$ satisfy $(g,n) \neq (0,1), (0,2)$. Let $0 \leq k \leq n-1$. Fix parities for $b_1, \ldots, b_{n-k}$, so that
\[
\widehat{N}_{g,n}^{k} (b_1, \ldots, b_{n-k}, 0, \ldots, 0),
\]
is given by a polynomial in $b_1^2, \ldots, b_{n-k}^2$. Then the degree $3g-3+n$ terms of this polynomial agree with the highest degree terms of the quasi-polynomial
\[
\widehat{N}_{g,n} (b_1, \ldots, b_{n-k}, 0, \ldots, 0).
\]
In particular, $\widehat{N}_{g,n}^k (b_1, \ldots, b_{n-k}, 0, \ldots, 0)$ is a polynomial of degree $3g-3+n$.
\end{thm}
When $k=0$, we obtain the degree statement in theorem \ref{thm:N_polynomiality_k_t_both_0}, and the first statement in theorem \ref{thm:intersection_numbers_refined}.

\begin{proof}
Fixing $k$ and the parities of $b_1, \ldots, b_{n-k}$, we have
\begin{equation}
\label{eqn:N_in_terms_of_Ngnt}
\widehat{N}_{g,n} (b_1, \ldots, b_{n-k}, 0, \ldots, 0) = \sum_t \widehat{N}_{g,n}^t (b_1, \ldots, b_{n-k}, 0, \ldots, 0).
\end{equation}
In this sum, the terms are all polynomials in the $b_i^2$.  We know $\deg \widehat{N}_{g,n} (b_1, \ldots, b_n) = 3g-3+n$, and by theorem \ref{thm:volume_polynomials}
, for any $d_1, \ldots, d_n \geq 0$ satisfying $d_1 + \cdots + d_n = 3g-3+n$, the coefficient of $b_1^{2d_1} \cdots b_n^{2d_n}$ is nonzero. In particular, setting $b_{n-k+1}, \ldots, b_n$ to zero, the degree remains $3g-3+n$.

Thus, equation \eqref{eqn:N_in_terms_of_Ngnt} expresses a polynomial of degree $3g-3+n$ in terms of polynomials of degree at most $3g-3+n-t+k$, where $t$ runs over the range $k \leq t \leq \min(2g+n-1, k+3g-3+n)$, so that $3g-3+n-t+k \leq 3g-3+n$, with equality if and only if $t=k$. The only way for the sum to yield a polynomial of degree exactly $3g-3+n$ then is if the polynomial with $t=k$ has degree exactly $3g-3+n$. Hence we conclude that $\deg \widehat{N}_{g,n}^k (b_1, \ldots, b_{n-k}, 0, \ldots, 0) = 3g-3+n$, and is the only term in the sum with this degree. It follows that the terms of degree $3g-3+n$ in $\widehat{N}_{g,n} (b_1, \ldots, b_{n-k}, 0, \ldots, 0)$ and $\widehat{N}_{g,n}^k (b_1, \ldots, b_{n-k}, 0, \ldots, 0)$ must agree.
\end{proof}

Proposition \ref{prop:agreement_in_highest_degree} and theorem \ref{thm:volume_polynomials} then immediately give the following results. Recall that the expression $\mathfrak{N}_{g,n}(b_1, \ldots, b_n)$ denotes the lattice count quasi-polynomials of \cite{Norbury10_counting_lattice_points}, and $V_{g,n}(b_1, \ldots, b_n)$ denotes the volume polynomials of \cite{Kontsevich_Intersection}.
\begin{prop}
Let $(g,n) \neq (0,1)$, $0 \leq k \leq n-1$ and fix the parity of $b_1, \ldots, b_{n-k}$. Then the corresponding polynomials in the quasi-polynomials $\mathfrak{N}_{g,n}(b_1, \ldots, b_{n-k}, 0, \ldots, 0)$ and $\widehat{N}_{g,n}^k (b_1, \ldots, b_{n-k}, 0, \ldots, 0)$ have identical terms of highest total degree.
\qed
\end{prop}

\begin{thm}
\label{thm:intersection_numbers_k}
The polynomials defining $\widehat{N}_{g,n}^k (b_1, \ldots, b_{n-k}, 0, \ldots, 0)$ all agree in their terms of highest degree, and these agree with the highest degree terms of $V_{g,n}(b_1, \ldots, b_n)$. Thus
\[
\widehat{N}_{g,n}^k (b_1, \ldots, b_{n-k}, 0, \ldots, 0) = V_{g,n} (b_1, \ldots, b_{n-k}, 0, \ldots, 0) + \text{lower order terms}.
\]
Moreover, for any integers $d_1, \ldots, d_{n-k}$ satisfying $d_1 + \ldots + d_{n-k} = 3g-3+n$, the coefficient $c_{d_1, \ldots, d_{n-k}}$ of $b_1^{2d_1} \cdots b_{n-k}^{2d_{n-k}}$ in any polynomial of the quasi-polynomial $\widehat{N}_{g,n}^k (b_1, \ldots, b_{n-k}, 0, \ldots, 0)$ is given by
\[
c_{d_1, \ldots, d_{n-k}} =
\frac{1}{2^{5g-6+2n} \; d_1 ! \cdots d_{n-k} !}
\langle \psi_1^{d_1} \cdots \psi_{n-k}^{d_{n-k}}, \overline{\mathcal{M}}_{g,n} \rangle.
\]
\qed
\end{thm}
When $k=0$, we obtain the second statement of theorem \ref{thm:intersection_numbers_refined}.

Thus, we can recover the full set of intersection numbers of $\psi$-classes on the moduli space of curves from $\widehat{N}_{g,n}^0 (b_1, \ldots, b_n)$, restricting the number of regions in arc diagrams by $k=t=0$. The constraints $k=t=0$ mean topologically that each boundary component has at least one arc endpoint, and that the arcs cut the surface into the minimum number of regions possible; these curves provide the same asymptotics.

When $k=t$, the quasi-polynomials $\widehat{N}_{g,n}^k (b_1, \ldots, b_{n-k}, 0, \ldots, 0)$, in addition to recovering intersection numbers on moduli spaces, have an interesting set of zeroes and positivity constraints. Indeed, proposition \ref{prop:nonzero_Ngnk} immediately implies that these quasi-polynomials must be zero, or positive, for certain values of $b_1, \ldots, b_{n-k}$, giving the following result.
\begin{thm}
Consider the quasi-polynomials $\widehat{N}_{g,n}^k (b_1, \ldots, b_{n-k}, 0, \ldots, 0)$, for $(g,n) \neq (0,1), (0,2)$, $0 \leq k \leq n-1$ and $b_1, \ldots, b_{n-k} > 0$.
\begin{enumerate}
\item
Any integer point ${\bf b} = (b_1, \ldots, b_{n-k}, 0, \ldots, 0)$ satisfying $\frac{1}{2}(b_1 + \cdots + b_{n-k}) < 2g+n-1-k$ is a zero of $\widehat{N}_{g,n}^k$, i.e. $\widehat{N}_{g,n}^k({\bf b}) = 0$.
\item
At any integer point ${\bf b} = (b_1, \ldots, b_{n-k}, 0, \ldots, 0)$ satisfying $\frac{1}{2} (b_1 + \cdots + b_{n-k}) \geq 2g+n-1-k$, $\widehat{N}_{g,n}^k$ is positive, i.e. $\widehat{N}_{g,n}^k ({\bf b}) > 0$.
\end{enumerate}
\qed
\end{thm}

\subsection{Polynomiality for general refined curve counts}

\label{sec:polynomiality_refined}

It is now not difficult to use a similar method as section \ref{sec:polynomiality_general} to show that the $G_{g,n}^t$ have a similar form to the $G_{g,n}$, and prove a similar polynomiality result. We recall proposition \ref{prop:G_and_N_refined}: for $(g,n) \neq (0,1)$ and any integers $b_1, \ldots, b_n$,
\[
G_{g,n}^t (b_1, \ldots, b_n) 
= \sum_{\substack{a_i \geq 0 \\ i=1, \ldots, n}}
\binom{b_1}{\frac{b_1 - a_1}{2}} \cdots \binom{b_n}{\frac{b_n - a_n}{2}}
N_{g,n}^t (a_1, \ldots, a_n).
\]
We have seen, however, that the $\widehat{N}_{g,n}^t(a_1, \ldots, a_n)$ are not quite quasi-polynomials over all integers $a_1, \ldots, a_n$; we need to take into account when some of these inputs are zero.

As in section \ref{sec:refined_pants}, we split the terms with $a_i = 0$ out of sums, and we use $\tilde{p}_\alpha(n)$ and $\tilde{q}_\alpha (n)$ rather than $\tilde{P}_\alpha(n)$ and $\tilde{Q}_\alpha (n)$. Recall section \ref{sec:general_pants} for their definitions.
Now $\tilde{p}_\alpha (n) = \binom{2n}{n} n p_\alpha (n)$ and $\tilde{q}_\alpha (n) = \binom{2n}{n} (2n+1) q_\alpha (n)$, where $p_\alpha, q_\alpha$ are polynomials of degree $\alpha$ (proposition \ref{prop:Ps_and_Qs}); a similar argument shows they have positive leading coefficients.

\begin{thm}
\label{thm:Ggnt_polynomiality}
Let $(g,n) \neq (0,1), (0,2)$, let $t$ be an integer satisfying $0 \leq t \leq \min(2g+n-1, 3g-3+n)$, and let $b_1, \ldots, b_n$ be non-negative integers. Then $G_{g,n}^t(b_1, \ldots, b_n)$ is given by a product of
\begin{enumerate}
\item
a combinatorial factor $\binom{2m_i}{m_i}$ for each $1 \leq i \leq n$, where $b_i = 2m_i$ if $b_i$ is even and $b_i = 2m_i + 1$ if $b_i$ is odd, and
\item
a symmetric rational quasi-polynomial $P_{g,n}^t (b_1, \ldots, b_n)$, depending on the parity of each $b_i$, of degree $\leq 3g-3+2n-t$.
\end{enumerate}
If we fix the parity of each $b_i$ so that at least $t$ of the $b_i$ are even, then the degree of the corresponding polynomial in $P_{g,n}^t(b_1, \ldots, b_n)$ is exactly $3g-3+2n-t$.
\end{thm}

\begin{proof}
Fix the parity of $b_1, \ldots, b_n$, and write $b_i = 2m_i$ or $b_i = 2m_i + 1$ accordingly as $b_i$ is even or odd. Using proposition \ref{prop:G_and_N_refined}, we express $G_{g,n}^t (b_1, \ldots, b_n)$ as a sum over $a_1, \ldots, a_n$, where $0 \leq a_i \leq b_i$ and $a_i \equiv b_i \pmod{2}$. For those $b_i$ which are even, we split the sum over $a_i$ into the $a_i = 0$ term and the $a_i > 0$ terms. This expresses $G_{g,n}$ as a sum of terms of the form
\[
\prod_{i \in K} \binom{2m_i}{m_i} 
\sum_{\substack{1 \leq a_j \leq b_j \\ a_j \equiv b_j  \!\!\!\! \pmod{2} \\ j \in J}}
\left( \prod_{j \in J} \binom{b_j}{\frac{b_j - a_j}{2}} a_j \right) \widehat{N}_{g,n}^t (a_1, \ldots, a_n).
\]
Here $K \sqcup J = \{1, \ldots, n\}$;  $K$ is the set of $i$ for which $a_i$ has been set to zero. In fact, $G_{g,n}(t)$ is the sum of all such expressions, over the subsets $I$ of the indices $i$ for which $b_i$ has been chosen to be even. We write $|K| = k$ for the number of zeroes among the $a_i$, as per our previous notation. 

Now in each such expression, each $a_i$ is fixed to be even or odd (accordingly as the corresponding $b_i$, and each of the even $a_i$ are fixed to be zero or nonzero. We have $a_i = 0$ for $i \in I$, and $a_j \neq 0$ for $j \in J$ Hence we can write $\widehat{N}_{g,n}^t (a_1, \ldots, a_n) = \widehat{N}_{g,n}^t (a_J, 0)$. By theorem \ref{thm:Nhatgnt_polynomiality}, when  $0 \leq k \leq n-1$, $\widehat{N}_{g,n}^t (a_J, 0)$ is either zero, or $t$ lies in the range specified in the theorem (in particular, $k \leq t$), and $\widehat{N}_{g,n}^t (a_J, 0)$ is a symmetric quasi-polynomial in the $a_j^2$, of degree $\leq 3g-3+n-t+k \leq 3g-3+n$. When $k=n$, we have $\widehat{N}_{g,n}^t (0, \ldots, 0) = \delta_{t,2g+n-1}$. Depending on $t$, this is also a symmetric quasi-polynomial of degree $0$, or is zero; it gives a term of the form $\binom{2m_1}{m_1} \cdots \binom{2m_n}{m_n}$ times a constant in $G_{g,n}^t$, when $t=2g+n-1$.

This leaves the terms with $0 \leq k \leq n-1$. Splitting up $\widehat{N}_{g,n}^t (a_J, 0)$ as a sum of monomials $c_{\bf \alpha} \prod_{j \in J} a_j^{2 \alpha_j}$, we can write $G_{g,n}$ as a finite sum of terms of the form
\[
\prod_{i \in K} \binom{2m_i}{m_i}
\sum_{{\bf \alpha}} c_{\bf \alpha} \prod_{j \in J} \sum_{\substack{1 \leq a_j \leq b_j \\ a_j \equiv b_j  \!\!\!\! \pmod{2}}} \binom{b_j}{\frac{b_j - a_j}{2}} a_j^{2 \alpha_j}.
\]
As $\widehat{N}_{g,n}^t (a_J, 0)$ has degree $\leq 3g-3+n-t+k$, we always have $\alpha_1 + \cdots + \alpha_n \leq 3g-3+n-t+k$. When $t=k$, by theorem \ref{thm:Ngntk_Ngn_agreement}, the degree of $\widehat{N}_{g,n}^t (a_J, 0)$ is exactly $3g-3+n-t+k=3g-3+n$, so in this case there are terms with $\alpha_1 + \cdots + \alpha_n = 3g-3+n$.

Each $\sum_{a_j} \binom{b_j}{\frac{b_j - a_j}{2}} a^{2\alpha_j}$ is either $\tilde{p}_{\alpha_j} (m_j) = \binom{2m_j}{m_j} m_j p_{\alpha_j} (m_j)$, if $a_j$ is even, or $\tilde{q}_{\alpha_j} (m_j) = \binom{2m_j}{m_j} (2m_j + 1) q_{\alpha_j} (m_j)$, if $a_j$ is odd.

Thus, $G_{g,n}(b_1, \ldots, b_n)$ can be expressed as a finite sum of terms, where each term is a constant multiplied by $\binom{2m_1}{m_1} \cdots \binom{2m_n}{m_n}$, multiplied by a polynomial in $m_1, \ldots, m_n$. This polynomial is either a constant (in the case $k=n$ and $t=2g+n-1$), or is a product of $m_j p_{\alpha_j} (m_j)$ and $(2m_j + 1) q_{\alpha_j} (m_j)$, over $j \in J$. Since $J \sqcup K = \{1, \ldots, n\}$ and $|K| = k$, we have $|J| = n-k$. In each term we thus have $(n-k)$ nonzero $\alpha_j$'s, and they have sum $\sum \alpha_j \leq 3g-3+n-t+k$. Each $m_j p_{\alpha_j} (m_j)$ and $(2m_j+1) q_{\alpha_j} (m_j)$ is a polynomial of degree $\alpha_j + 1$, so in each term, the degree of their product is $\sum_{j \in J} (\alpha_j + 1) = (\sum_{j \in J} \alpha_j) + n-k \leq 3g-3+2n-t$. Moreover, we know that when $t=k$ equality holds.

So, fixing the parity of $b_1, \ldots, b_n$, we obtain $G_{g,n}(b_1, \ldots, b_n)$ as a product of $\binom{2m_1}{m_1} \cdots \binom{2m_n}{m_n}$, multiplied by a finite sum of polynomials. Terms where the number $k$ of variables set to zero satisfies $0 \leq k \leq n-1$ contribute polynomials of degree $\leq 3g-3+2n-t$. When $k=n$, we must have $t=2g+n-1$, and the polynomial contribution is a constant; however when $t=2g+n-1$ we have $3g-3+2n-t = g+n-2 \geq 0$ (as $(g,n) \neq (0,1), (0,2)$). So in any case we obtain $G_{g,n}(b_1, \ldots, b_n)$ as $\binom{2m_1}{m_1} \cdots \binom{2m_n}{m_n}$ multiplied by a polynomial $P_{g,n}^t (b_1, \ldots, b_n)$ of degree $\leq 3g - 3 + 2n - t$.

If at least $t$ of the $b_i$ are even, then it is possible to set $t$ of the variables to zero, so there is a term with $k=t$ which contributes to the polynomial $P_{g,n}^t (b_1, \ldots, b_n)$. Hence, as discussed above, there are monomials appearing with $\sum_{j \in J} \alpha_j = 3g-3+n-t+k$, and when we perform the summations, we obtain products of $m_j p_{\alpha_j} (m_j)$ and $(2m_j + 1) q_{\alpha_j}$, contributing a polynomial of degree exactly $3g-3+2n-t$ to $P_{g,n}^t$. As all the polynomials involved have positive highest degree terms, the resulting polynomial $P_{g,n}^t (b_1, \ldots, b_n)$ must have degree exactly $3g-3+2n-t$.
\end{proof}

\section{Differential equations and partition functions}
\label{sec:dequations}

\subsection{Refining differential forms and generating functions}
\label{sec:refining_omega}

We may also refine the generating functions $f_{g,n}^G, f_{g,n}^N$ and differential forms $\omega_{g,n}$. We have developed two parameters $r$ and $t$ with which we can keep track of the number of regions; we can keep track of either of these two parameters.

\begin{defn}[Refined generating functions with respect to $r$]
For integers $g \geq 0$, $n \geq 1$, and $r \geq 1$, define the functions $f^G_{g,n,r} (x_1, \ldots, x_n)$and $f^N_{g,n,r} (z_1, \ldots, z_n)$ by 
\begin{align*}
f_{g,n,r}^{G} (x_1, \ldots, x_n) = \sum_{\mu_1, \ldots, \mu_n \geq 0} G_{g,n,r} (\mu_1, \ldots, \mu_n) x_1^{-\mu_1 - 1} \cdots x_n^{-\mu_n - 1} \\
f_{g,n,r}^N (z_1, \ldots, z_n) = \sum_{\nu_1, \ldots, \nu_n \geq 0} N_{g,n,r} (\nu_1, \ldots, \nu_n) z_1^{\nu_1 - 1} \cdots z_n^{\nu_n - 1}.
\end{align*}
\end{defn}

\begin{defn}[Refined generating functions with respect to $t$]
For integers $g \geq 0$, $n \geq 1$ and $t$, define the functions $f^{G,t}_{g,n} (x_1, \ldots, X_n)$ and $f^{N,t}_{g,n} (z_1, \ldots, z_n)$ by
\begin{align*}
f_{g,n}^{G,t} (x_1, \ldots, x_n) &= \sum_{\mu_1, \ldots, \mu_n \geq 0} G_{g,n}^t (\mu_1, \ldots, \mu_n) x_1^{-\mu_1 - 1} \cdots x_n^{-\mu_n - 1}, \\
f_{g,n}^{N,t} (z_1, \ldots, z_n) &= \sum_{\nu_1, \ldots, \nu_n \geq 0} N_{g,n}^t (\nu_1, \ldots, \nu_n) z_1^{\nu_1 - 1} \cdots z_n^{\nu_n - 1}.
\end{align*}
\end{defn}
From proposition \ref{prop:bounds_on_r_and_t}, $G_{g,n}^t$ can only be positive for $0 \leq t \leq 2g+n-1$; similarly $N_{g,n}^t$ can only be positive for $t$ in this range. So when $t > 2g+n-1$ we have $f_{g,n}^{G,t} (x_1, \ldots, x_n) = f_{g,n}^{N,t} (z_1, \ldots, z_n) = 0$.

We can define differential forms from each of these generating functions
\begin{defn}[Refined differential forms]
For integers $g \geq 0$, $n \geq 1$, $r \geq 1$ and $t$, let
\begin{align*}
\omega_{g,n,r}^G (x_1, \ldots, x_n) &= f_{g,n,r}^G (x_1, \ldots, x_n) \; dx_1 \cdots dx_n \\
\omega_{g,n,r}^N (z_1, \ldots, z_n) &= f_{g,n,r}^N (z_1, \ldots, z_n) \; dz_1 \cdots dz_n \\
\omega_{g,n}^{G,t} (x_1, \ldots, x_n) &= f_{g,n}^{G,t} (x_1, \ldots, x_n) \; dx_1 \cdots dx_n \\
\omega_{g,n}^{N,t} (z_1, \ldots, z_n) &= f_{g,n}^{N,t} (z_1, \ldots, z_n) \; dz_1 \cdots dz_n.
\end{align*}
\end{defn}

Since $G_{g,n}(\mu)= \sum_{r} G_{g,n,r} (\mu) = \sum_t G_{g,n}^t (\mu)$ and $N_{g,n}(\nu) = \sum_r N_{g,n,r}(\nu) = \sum_t N_{g,n}^t (\nu)$, we immediately have the following.
\begin{lem}
For any $g \geq 0$, $n \geq 1$,
\begin{align*}
f^G_{g,n} (x_1, \ldots, x_n) &= \sum_{r \geq 1} f_{g,n,r}^G (x_1, \ldots, x_n) 
= \sum_{t = 0}^{2g+n-1} f_{g,n}^{G,t} (x_1, \ldots, x_n) \\
f^N_{g,n} (z_1, \ldots, z_n) &= \sum_{r \geq 1} f_{g,n,r}^N (z_1, \ldots, z_n)
= \sum_{t=0}^{2g+n-1} f_{g,n}^{N,t} (z_1, \ldots, z_n).
\end{align*}
\qed
\end{lem}

\subsection{Small cases of refined generating functions and differential forms}
\label{sec:small_cases_refined}

We can compute these generating functions directly in small cases. We begin with $(g,n) = (0,1)$.
\begin{prop}
\label{prop:f01Nt}
For any $r$ and $t$, the four generating functions $f_{0,1,r}^G (x_1)$, $f_{0,1,r}^N (z_1)$, $f_{0,1}^{G,t} (x_1)$ and $f_{0,1}^{N,t} (z_1)$ are all meromorphic. These are given by
\begin{align*}
f_{0,1,r}^G (x_1) &= \frac{1}{r} \binom{2r-2}{r-1} x_1^{-2r+1}, \\
f_{0,1,r}^N (z_1) &= \left\{ \begin{array}{ll} 
z_1^{-1} & \text{if $r = 1$,} \\
0 & \text{otherwise,}
\end{array} \right. \\
f_{0,1}^{G,t} (x_1) &= 
\left\{ \begin{array}{ll}
\frac{x_1 - \sqrt{x_1^2 - 4}}{2} & \text{if $t=0$,} \\
0 & \text{otherwise,}
\end{array} \right. \\
f_{0,1}^{N,t} (z_1) &= 
\left\{ \begin{array}{ll}
z_1^{-1} & \text{if $t=0$,} \\
0 & \text{otherwise.}
\end{array} \right.
\end{align*} 
\end{prop}

\begin{proof}
The arc diagrams on the disc with $r$ complementary regions are precisely those with $r-1$ arcs, hence the only nonzero $G_{0,1,r}(\mu)$ is $G_{0,1,r}(2r-2) = \frac{1}{r} \binom{2r-2}{r-1}$, giving the sole contribution to $f_{0,1,r}^G (x_1)$. 

Any arc diagram on the disc without boundary-parallel arcs must be empty, so in the sum for $f_{0,1,r}^N$ we must have $r=1$ and $\nu_1 = 0$.

All arc diagrams on the disc have $t=0$, hence $f_{0,1}^{G,t} (x_1)$ is identical to the unrefined function $f_{0,1}^G (x_1)$ when $t=0$, and zero otherwise; similarly, $f_{0,1}^{N,t} (z_1)$ is identical to the unrefined function $f_{0,1}^N (z_1)$ when $t=0$, and is zero otherwise.
\end{proof}

We can also compute $f_{0,2}^{N,t}$ and $f_{0,3}^{N,t}$.
\begin{prop}
\label{prop:f02Nt}
The function $f_{0,2}^{N,t}$ is meromorphic and is given by
\[
f_{0,2}^{N,t} (z_1, z_2) = \left\{ \begin{array}{ll}
\frac{1}{(1 - z_1 z_2)^2} & \text{if $t=0$,} \\
\frac{1}{z_1 z_2} & \text{if $t = 1$,} \\
0 & \text{otherwise.}
\end{array} \right.
\]
\end{prop}
We calculated in section \ref{sec:small_gen_fns} that $f_{0,2}^N (z_1, z_2) = \frac{1}{z_1 z_2} + \frac{1}{(1-z_1 z_2)^2}$; the two terms in this sum correspond precisely to $t=0$ and $t=1$.

\begin{proof}
From lemma \ref{lem:N01t_N02t_computations}, we have $N_{0,2}^0 (b_1, b_2) = b_1$ for $b_1 = b_2 > 0$, and $N_{0,2}^0 = 0 $ otherwise. We have $N_{0,2}^1 (0,0) = 1$, and all other $N_{0,2}^t (b_1, b_2) = 0$.
Thus we have
\begin{align*}
f_{0,2}^{N,0}(z_1, z_2) &= \sum_{\nu=1}^\infty \nu (z_1 z_2)^{\nu - 1} = \frac{1}{(1-z_1 z_2)^2}, \\
f_{0,2}^{N,1}(z_1, z_2) &= z_1^{-1} z_2^{-1}. \qedhere
\end{align*}
\end{proof}

\begin{prop}
\label{prop:f03Nt}
The function $f_{0,3}^{N,t}$ is meromorphic and is given by
\begin{align*}
f_{0,3}^{N,0}(z_1, z_2, z_3) &= \frac{ 2(z_1 + z_2 + z_3 + z_1 z_2 z_3)(1 + z_1 z_2 + z_2 z_3 + z_3 z_1) }{(1 - z_1^2)^2 (1-z_2^2)^2 (1-z_3^2)^2 } \\
f_{0,3}^{N,1}(z_1, z_2, z_3) &=
\frac{1 + 4 z_1 z_2 + z_1^2 + z_2^2 + z_1^2 z_2^2}{(1-z_1^2)^2(1-z_2^2)^2 z_3}
+
\frac{1 + 4 z_2 z_3 + z_2^2 + z_3^2 + z_2^2 z_3^2}{(1-z_2^2)^2(1-z_3^2)^2 z_1}
+
\frac{1 + 4 z_3 z_1 + z_3^2 + z_1^2 + z_3^2 z_1^2}{(1-z_3^2)^2(1-z_1^2)^2 z_2} \\
f_{0,3}^{N,2} (z_1, z_2, z_3) &=
\frac{1 + 16 z_1^2 z_2^2 z_3^2 + z_1^4 z_2^4 z_3^4 + \sum_{\text{cyc}} (z_1^4 - 4 z_1^2 z_2^2 + z_1^4 z_2^4 - 4 z_1^4 z_2^2 z_3^2)}{z_1 z_2 z_3 (1-z_1^2)^2 (1-z_2^2)^2 (1-z_3^2)^2 },
\end{align*}
and $f_{0,3}^{N,t} (z_1, z_2, z_3) = 0$ otherwise.
\end{prop}
One can check that these three $f_{0,3}^{N,t}$ sum to the $f_{0,3}^N$ calculated in lemma \ref{lem:omegaN03}.

\begin{proof}
From proposition \ref{prop:Nhat03t} we have $N_{0,3}^0(b_1, b_2, b_3) = b_1 b_2 b_3$, for positive $b_i$ with even sum; $N_{0,3}^1 (b_1, b_2, 0) = b_1 b_2$ for positive $b_i$ with even sum; $N_{0,3}^2 (b_1, 0, 0) = b_1$ for positive even $b_1$; and $N_{0,3}^2 (0,0,0) = 1$. All other $N_{0,3}^t$ are zero. Thus, following a similar method to lemma \ref{lem:omegaN03} we have
\begin{align*}
f_{0,3}^{N,0} (z_1, z_2, z_3) &= \sum_{\nu_1, \nu_2, \nu_3 \geq 1} N_{0,3}^0 (\nu_1, \nu_2, \nu_3) z_1^{\nu_1 - 1} z_2^{\nu_2 - 1} z_3^{\nu_3 - 1} \\
&= \sum_{\substack{\nu_1, \nu_2, \nu_3 \geq 1 \\ \nu_1 + \nu_2 + \nu_3 \text{ even}}} 
\nu_1 \nu_2 \nu_3 \; z_1^{\nu_1 - 1} z_2^{\nu_2 - 1} z_3^{\nu_3 - 1} \\
&= \left( \sum_{\substack{\nu_1,\nu_2,\nu_3 \\ \text{all even}}} + 
\sum_{\substack{\nu_1 \text{ even} \\ \nu_2, \nu_3 \text{ odd}}} +
\sum_{\substack{\nu_2 \text{ even} \\ \nu_3, \nu_1 \text{ odd}}} +
\sum_{\substack{\nu_3 \text{ even} \\ \nu_1, \nu_2 \text{ odd}}} \right)
\nu_1 \nu_2 \nu_3 \; 
z_1^{\nu_1 - 1} z_2^{\nu_2 - 1} z_3^{\nu_3 - 1} \\
&= 
\rho(z_1) \rho(z_2) \rho(z_3)
+ \rho(z_1) \sigma(z_2) \sigma(z_3)
+ \rho(z_2) \sigma(z_3) \sigma(z_1)
+ \rho(z_3) \sigma(z_1) \sigma(z_2),
\end{align*}
where $\rho, \sigma$ are given by (note $\rho$ here is slightly different from lemma \ref{lem:omegaN03})
\[
\rho(z) = \sum_{\substack{\nu \geq 1 \\ \nu \text{ even}}} \nu \; z^{\nu - 1} = \frac{2z}{(1-z^2)^2},
\quad
\sigma(z) = \sum_{\substack{\nu \geq 1 \\ \nu \text{ odd}}} \nu \; z^{\nu - 1} = \frac{1+z^2}{(1-z^2)^2}.
\]
Expanding these out gives the claimed expression for $f_{0,3}^{N,0}$.

For $f_{0,3}^{N,1}$ we have one variable equal to zero, and the others positive. Thus
\begin{align*}
f_{0,3}^{N,1} (z_1, z_2, z_3) 
&= \!\! \sum_{\substack{\nu_1 = 0 \\ \nu_2, \nu_3 \geq 1 \\ \nu_2 + \nu_3 \text{ even}}} \nu_2 \nu_3 z_1^{-1} z_2^{\nu_2 - 1} z_3^{\nu_3 - 1}
+
\!\! \sum_{\substack{\nu_2 = 0 \\ \nu_3, \nu_1 \geq 1  \\ \nu_3 + \nu_1 \text{ even}}} \nu_3 \nu_1 z_2^{-1} z_3^{\nu_3 - 1} z_1^{\nu_1 - 1}
+
\!\! \sum_{\substack{\nu_3 = 0 \\ \nu_1, \nu_2 \geq 1  \\ \nu_1 + \nu_2 \text{ even}}} \nu_1 \nu_2 z_3^{-1} z_1^{\nu_1 - 1} z_2^{\nu_2 - 1} \\
&= 
\sum_{\text{cyc}} z_1^{-1} \left( \rho(z_2) \rho(z_3) + \sigma(z_2) \sigma(z_3) \right)
\end{align*} 
where $\rho, \sigma$ are as above; expanding this out gives the claimed expression for $f_{0,3}^{N,1}$.

Finally, for $f_{0,3}^{N,2}$ we have two or all variables equal to zero. Thus
\begin{align*}
f_{0,3}^{N,2} (z_1, z_2, z_3)
&= 
z_1^{-1} z_2^{-1} z_3^{-1} +
z_2^{-1} z_3^{-1} \sum_{\substack{\nu_1 \geq 1 \\ \nu_1 \text{ even}}} \nu_1 z_1^{\nu_1 - 1}
+ z_3^{-1} z_1^{-1} \sum_{\substack{\nu_2 \geq 1 \\ \nu_2 \text{ even}}} \nu_2 z_2^{\nu_2 - 1}
+ z_1^{-1} z_2^{-1} \sum_{\substack{\nu_3 \geq 1 \\ \nu_3 \text{ even}}} \nu_3 z_3^{\nu_3 - 1} \\
&= z_1^{-1} z_2^{-1} z_3^{-1} 
+ z_2^{-1} z_3^{-1} \rho(z_1) + z_3^{-1} z_1^{-1} \rho(z_2) + z_1^{-1} z_2^{-1} \rho(z_3).
\end{align*}
Again, $\rho$ is as above, and expanding out gives $f_{0,3}^{N,2}$.
\end{proof}

\subsection{Meromorphicity and change of coordinates}

A similar method to section \ref{sec:meromorphicity} shows that we have meromorphicity in many cases.

\begin{prop}
\label{prop:omegagnt_meromorphic}
For all integers $g \geq 0$, $n \geq 1$ and $t$, $f_{g,n}^{N,t} (z_1, \ldots, z_n)$ is a meromorphic function and $\omega_{g,n}^{N,t} (z_1, \ldots, z_n)$ is a meromorphic differential form.
\end{prop}

The proof follows a similar method to proposition \ref{prop:omega1_meromorphic}.
\begin{proof}
We just computed $f_{0,1}^{N,t} (z_1)$ (proposition \ref{prop:f01Nt}) and $f_{0,2}^{N,t} (z_1, z_2)$ (proposition  \ref{prop:f02Nt}), and showed that they are always meromorphic functions; and hence $\omega_{0,1}^{N,t}(z_1)$ and $\omega_{0,2}^{N,t}(z_1, z_2)$ are meromorphic forms.

For $(g,n) \neq (0,1), (0,2)$, we proved in theorem \ref{thm:Nhatgnt_polynomiality} that for $(g,n) \neq (0,1), (0,2)$, $\widehat{N}_{g,n}^t (\nu_1, \ldots, \nu_{n-k}, 0, \ldots, 0)$ is a rational symmetric quasi-polynomial in $\nu_1^2, \ldots, \nu_{n-k}^2$. Hence, if we fix each $\nu$ to be zero, positive odd, or positive even, then we obtain a polynomial. Let $\{1, 2, \ldots, n\} = I \sqcup J$, where $I$ is the set of $i$ for which $\nu_i$ is set to zero, and $J$ is the set of $j$ for which $\nu_j$ is positive. For those $\nu_j$ with $j \in J$, we can set $\nu_j \equiv \epsilon_j \pmod{2}$, where $\epsilon_j \in \{0,1\}$. Thus we can split the sum for $f_{g,n}^{N,t}$ into $3^n$ sums of the form
\[
\sum_{\substack{\nu_{j} \geq 1 \\ \nu_{j} \equiv \epsilon_{j}  \!\!\!\! \pmod{2} \\ j \in J}}
\left( \prod_{j \in J} \nu_{j} \right)
P(\nu_1, \ldots, \nu_n)|_{\nu_I = 0} \; 
z_1^{\nu_1  -1} \cdots z_n^{\nu_n - 1},
\] 
where $P(\nu_1, \ldots, \nu_n)$ is a polynomial, and $P(\nu_1, \ldots, \nu_n)|_{\nu_I = 0}$ means we set all $\nu_i = 0$ for $i \in I$. This is a polynomial in the $\nu_j$ for $j \in J$. Splitting each such polynomial into monomials, we can write $f_{g,n}^{N,t}$ as a finite sum of terms of the form of a constant times
\begin{equation}
\label{eqn:term_in_fgnNt}
\left( \prod_{i \in I} z_i^{-1} \right)
\sum_{\substack{\nu_j \geq 1 \\ \nu_j \equiv \epsilon_j  \!\!\!\! \pmod{2} \\ j \in J}}
\left( \prod_{j \in J} \nu_j^{a_j} z_j^{\nu_j - 1} \right).
\end{equation}
Now we know from the proof of proposition \ref{prop:omega1_meromorphic} that for any positive integer $a$ and $\epsilon \in \{0,1\}$,
\[
\sum_{\substack{\nu \geq 0 \\ \nu \equiv \epsilon  \!\!\!\! \pmod{2}}} \nu^a z^{\nu}
=
\sum_{\substack{\nu \geq 1 \\ \nu \equiv \epsilon  \!\!\!\! \pmod{2}}} \nu^a z^{\nu}
\]
is meromorphic. Hence each term as in equation \eqref{eqn:term_in_fgnNt} is meromorphic, and $f_{g,n}^{N,t}$ is a finite sum of such terms. So $f_{g,n}^{N,t} (z_1, \ldots, z_n)$ and $\omega_{g,n}^{N,t} (z_1, \ldots, z_n) = f_{g,n}^{N,t}(z_1, \ldots, z_n) \; dz_1 \cdots dz_n$ are meromorphic.
\end{proof}

\begin{prop}
\label{prop:omegagnr_meromorphic}
For all integers $g \geq 0$, $n \geq 1$ and $r \geq 1$, $f_{g,n,r}^{G} (x_1, \ldots, x_n)$ and $f_{g,n,r}^N (z_1, \ldots, z_n)$ are meromorphic functions, and $\omega_{g,n,r}^G (x_1, \ldots, x_n)$ and $\omega_{g,n,r}^{N} (z_1, \ldots, z_n)$ are meromorphic differential forms.
\end{prop}

\begin{proof}
Once $g,n$ and $r$ are given, lemma \ref{lem:lower_bound_on_r} says that if $G_{g,n,r} (\mu_1, \ldots, \mu_n) > 0$, then
\[
\frac{1}{2} \sum_{i=1}^n \mu_i \leq r+2g+n-2.
\]
Thus only finitely many $(\mu_1, \ldots, \mu_n)$ contribute to the sum for $f_{g,n,r}^G (x_1, \ldots, x_n)$, which therefore must be a Laurent polynomial in $x_1, \ldots, x_n$, hence meromorphic. The same inequality applies to $N_{g,n,r} (\nu_1, \ldots, \nu_n)$, so the sum for $f_{g,n,r}^N (z_1, \ldots, z_n)$ is also finite and we again obtain a meromorphic Laurent polynomial. We immediately then also obtain that $\omega_{g,n,r}^G (x_1, \ldots, x_n)$ and $\omega_{g,n,r}^N (z_1, \ldots, z_n)$ are meromorphic.
\end{proof}

We have now shown all the generating functions and differential forms are meromorphic, except for $f_{g,n}^{G,t} (x_1, \ldots, x_n)$ and $\omega_{g,n}^{G,t} (x_1, \ldots, x_n)$. This will follow from the next statement, which relates $\omega_{g,n}^{G,t}$ and $\omega_{g,n}^{N,t}$. These are related just as $\omega_{g,n}^G$ and $\omega_{g,n}^N$ are.
\begin{thm}
\label{thm:omegas_equal_refined}
For any $g \geq 0$, $n \geq 1$ other than $(g,n) = (0,1)$ and integer $t$,
\[
\phi^* \omega_{g,n}^{G,t} (x_1, \ldots, x_n) = \omega_{g,n}^{N,t} (z_1, \ldots, z_n)
\]
where $\phi(z_1, \ldots, z_n) = (z_1 + \frac{1}{z_1}, \ldots, z_n + \frac{1}{z_n})$.
\end{thm}

This proof is a straightforward refinement of the argument of theorem \ref{thm:omega1_equals_omega2}.
\begin{proof}
Just as theorem \ref{thm:G_in_terms_of_N} expresses $G_{g,n}(b_1, \ldots, b_n)$ in terms of $N_{g,n}(a_1, \ldots, a_n)$, proposition \ref{prop:G_and_N_refined} expresses $G_{g,n}^t (b_1, \ldots, b_n)$ in terms of $N_{g,n}^t (a_1, \ldots, a_n)$. The proof of theorem \ref{thm:omega1_equals_omega2} then applies verbatim, where we refine every $G_{g,n}(\mu)$ to $G_{g,n}^t (\mu)$ and every $N_{g,n}(\nu)$ to $N_{g,n}^t (\nu)$.
\end{proof}

As in the unrefined case, we can regard $x_i \leftrightarrow z_i$ as a change of coordinates and simply write $\omega_{g,n}^t$, rather than $\omega_{g,n}^{G,t}$ or $\omega_{g,n}^{N,t}$. In particular, $\omega_{g,n}^{G,t} (x_1, \ldots, x_n)$ is meromorphic.

While proposition \ref{prop:G_and_N_refined} gives a nice relationship between $G_{g,n}^t$ and $N_{g,n}^t$, there is no equally simple corresponding statement for $G_{g,n,r}$ and $N_{g,n,r}$. In particular, local decomposition preserves $t$ but does not preserve $r$. So the parameter $t$ is more natural than the parameter $r$, at least from the point of view of producing nice meromorphic refinements of differential forms.

In any case, we have shown that each meromorphic form $\omega_{g,n}$ can be decomposed into a finite sum of meromorphic forms $\omega_{g,n}^t$.

\subsection{Refined free energies}

Since we have refined meromorphic $\omega_{g,n}^t$, there is a well-defined notion of a refined free energy.
\begin{defn}
Let $g \geq 0$ and $n \geq 1$ be integers such that $(g,n)=(0,1)$. A function $F_{g,n}^t (z_1, \ldots, z_n)$ is a \emph{refined free energy} if
\[
d_{z_1} \cdots d_{z_n} F_{g,n} (z_1, \ldots, z_n) = \omega^t_{g,n} (z_1, \ldots, z_n).
\]
\end{defn}

Again, given $\omega_{g,n}^t$, there may be many free energies, differing by various constants of integration. Also,
\[
f_{g,n}^{G,t} (x_1, \ldots, x_n) = \frac{ \partial^n F_{g,n}^t }{ \partial x_1 \; \partial x_2 \; \cdots \; \partial x_n}
\quad \text{and} \quad
f_{g,n}^{N,t} (z_1, \ldots, z_n) = \frac{ \partial^n F_{g,n}^t }{ \partial z_1 \; \partial z_2 \; \cdots \; \partial z_n}.
\]

In the case $(g,n) = (0,1)$, we can integrate either $f_{0,1}^{G,t}$ or $f_{0,1}^{N,t}$ and obtain two possible free energies.

We give some free energies in simple cases.
\begin{prop}
The following functions are free energy functions.
\begin{align*}
F_{0,1}^{N, 0} (z_1) &= \log z_1 \\
F_{0,1}^{G,0} (x_1) &= \frac{1}{2} z_1^2 - \log z_1 \\
F_{0,2}^0 (z_1, z_2) &= -\log(1-z_1 z_2)  \\
F_{0,2}^1 (z_1, z_2) &= \log z_1 \log z_2 \\
F_{0,3}^0 (z_1, z_2, z_3) &= \frac{z_1 z_2 + z_2 z_3 + z_3 z_1 + 1}{(1-z_1^2)(1-z_2^2)(1-z_3^2)} \\
F_{0,3}^1 (z_1, z_2, z_3) &= \frac{(z_2 z_3  + 1) \log z_1}{(1-z_2^2)(1-z_3^2)}
+ \frac{(z_3 z_1 + 1) \log z_2}{(1-z_3^2)(1-z_1^2)}
+ \frac{(z_1 z_2 + 1) \log z_3}{(1-z_1^2)(1-z_2^2)} \\
F_{0,3}^2 (z_1, z_2, z_3) &= \log z_1 \log z_2 \log z_3 +
\frac{\log z_1 \log z_2}{1-z_3^2} + \frac{\log z_2 \log z_3}{1-z_1^2} + \frac{\log z_3 \log z_1}{1-z_2^2} \end{align*}
\end{prop}
We have now proved theorem \ref{thm:refined_free_energies}.

We can observe directly that these $F_{g,n}^t$ sum to the $F_{g,n}$ calculated previously. In the $(0,2)$ and $(0,3)$ cases especially, the terms of the rather complicated functions $F_{g,n}$ split up in a natural way.

\begin{proof}
In the $(g,n) = (0,1)$ cases, we saw in proposition \ref{prop:f01Nt} that $f_{0,1}^{G,0} = f_{0,1}^G$ and $f_{0,1}^{N,0} = f_{0,1}^N$, so the free energies must agree with those in theorem \ref{thm:free_energy_examples}.

Differentiating $F_{0,2}^0$ and $F_{0,2}^1$ with respect to $z_1, z_2$ gives the $f_{0,2}^{N,0}$ and $f_{0,2}^{N,1}$ from proposition \ref{prop:f02Nt}. Similarly, differentiating the $F_{0,3}^t$ with respect to $z_1, z_2, z_3$ gives the $f_{0,3}^{N,t}$ from proposition  \ref{prop:f03Nt}.
\end{proof}

\subsection{Putting the generating functions and differential forms together}
\label{sec:putting_together}

Having investigated generating functions, differential forms and free energies refined by the number of regions $r$, or the related parameter $t$, we can now put them together to obtain ``total" generating functions and differential forms; these will eventually be put together into partition functions.

We introduce variables $\alpha$ and $\beta$ to keep track of $r$ and $t$ respectively.
\begin{defn}
\label{defn:mathfrakf}
For integers $g \geq 0$ and $n \geq 1$, define the functions $\mathfrak{f}_{g,n}^{G}, \mathfrak{f}_{g,n}^N$ by
\begin{align*}
\mathfrak{f}_{g,n}^G (x_1, \ldots, x_n; \alpha) 
&= \sum_{r \geq 1} f_{g,n,r}^G (x_1, \ldots, x_n) \; \alpha^r, \\
\mathfrak{f}_{g,n}^N (z_1, \ldots, z_n; \alpha)
&= \sum_{r \geq 1} f_{g,n,r}^N (z_1, \ldots, z_n) \; \alpha^r.
\end{align*}
\end{defn}
Thus, $\mathfrak{f}_{g,n}^G$ sums the $G_{g,n,r}(\mu_1, \ldots, \mu_n)$ over all $\mu_1, \ldots, \mu_n$ and $r$, keeping track of the $\mu_i$ with a factor of $x_i^{-\mu_i - 1}$ and $r$ with a factor of $\alpha^r$. 
Similarly, $\mathfrak{f}_{g,n}^N$ sums the $N_{g,n,r}(\nu_1, \ldots, \nu_n)$ over all $\nu_1, \ldots, \nu_n$ and $r$, keeping track of the $\nu_i$ with a factor of $z_i^{\nu_i - 1}$, and $r$ with a factor of $\alpha^r$.

We can also define related differential forms.
\begin{defn}
For integers $g \geq 0$ and $n \geq 1$, let
\begin{align*}
\Omega_{g,n}^G (x_1, \ldots, x_n; \alpha) 
&= \mathfrak{f}_{g,n}^G (x_1, \ldots, x_n; \alpha) \; dx_1 \cdots dx_n 
= \sum_{r \geq 1} \omega_{g,n,r}^G (x_1, \ldots, x_n) \alpha^r
\\
\Omega_{g,n}^N (z_1, \ldots, z_n; \alpha) 
&= \mathfrak{f}_{g,n}^N  (z_1, \ldots, z_n; \alpha) \; dz_1 \cdots dz_n
= \sum_{r \geq 1} \omega_{g,n,r}^N (z_1, \ldots, z_n) \alpha^r.
\end{align*}
\end{defn} 
We can regard $\mathfrak{f}_{g,n}^G$ and $\mathfrak{f}_{g,n}^N$ as families of functions $(\CP^1)^n \To \CP^1$, parametrised by $\alpha \in \CP^1$. Similarly, we can regard $\Omega_{g,n}^G$ and $\Omega_{g,n}^N$ as families of sections of $(T^* \CP^1)^{\boxtimes n}$, parametrised by $\alpha$.

These generating functions and differential forms use the variable $\alpha$ and involve sums over $r$. While we know that each $f_{g,n,r}^G$, $f_{g,n,r}^N$, $\omega_{g,n,r}^G$ and $\omega_{g,n,r}^N$ is meromorphic, we now have an infinite sum of them, so we do not yet know that $\mathfrak{f}_{g,n}^G$, $\mathfrak{f}_{g,n}^N$, $\Omega_{g,n}^G$ or $\Omega_{g,n}^N$ are meromorphic. (We will see this later in proposition \ref{prop:mathfrakf_meromorphic}.)

Note that setting $\alpha = 1$ recovers the unrefined generating functions $f^G_{g,n}$, $f^N_{g,n}$, and differential forms $\omega^G_{g,n}$, $\omega^N_{g,n}$.

Now switching to the parameter $t$, we may take advantage of theorem \ref{thm:omegas_equal_refined}. 
\begin{defn}
For integers $g \geq 0$ and $n \geq 1$, let
\begin{align*}
{\bf f}_{g,n}^G (x_1, \ldots, x_n; \beta) &= 
\sum_t f_{g,n}^{G,t} (x_1, \ldots, x_n) \; \beta^t, \\
{\bf f}_{g,n}^N (z_1, \ldots, z_n; \beta) &=
\sum_t f_{g,n}^{N,t} (z_1, \ldots, z_n) \; \beta^t
\end{align*}
and for $(g,n) \neq (0,1)$,
\[
\Omega_{g,n} (\beta) = \sum_t \omega_{g,n}^t \; \beta^t.
\]
\end{defn}
Note that because of the bounds on $t$, namely $0 \leq t \leq 2g+n-1$, each sum above is a finite sum of meromorphic terms, immediately giving us the following.
\begin{prop}
\label{prop:bffgn_meromorphic}
Let $g \geq 0$ and $n \geq 1$. The functions ${\bf f}_{g,n}^G (x_1, \ldots, x_n; \beta)$ and ${\bf f}_{g,n}^N (z_1, \ldots, z_n; \beta)$ are meromorphic, and for each $\beta \in \C$, $\Omega_{g,n}(\beta)$ is a meromorphic form.
\qed
\end{prop}
Again, we can regard $\Omega_{g,n}$ as a family of meromorphic sections of $(T^* \CP^1)^{\boxtimes n}$, parametrised by $\beta$.

If we write $\Omega_{g,n}$ in terms of $x_1, \ldots, x_n$ or $z_1, \ldots, z_n$, we respectively obtain 
\begin{align*}
\Omega_{g,n} (x_1, \ldots, x_n; \beta) 
&= {\bf f}_{g,n}^G (x_1, \ldots, x_n; \beta) \; dx_1 \cdots dx_n 
= \sum_t f_{g,n}^{G,t} (x_1, \ldots, x_n) \; dx_1 \cdots dx_n \\
\Omega_{g,n} (z_1, \ldots, z_n; \beta)
&= {\bf f}_{g,n}^N (z_1, \ldots, z_n; \beta) \; dz_1 \cdots dz_n 
= \sum_t f_{g,n}^{N,t} (z_1, \ldots, z_n) \beta^t \; dz_1 \cdots dz_n.
\end{align*}
Note that setting $\beta = 1$ in the generating functions ${\bf f}_{g,n}^G (x_1, \ldots, x_n; \beta)$ and ${\bf f}_{g,n}^N (z_1, \ldots, z_n; \beta)$ recovers the functions $f_{g,n}^G (x_1, \ldots, x_n)$ and $f_{g,n}^N (z_1, \ldots, z_n)$. Similarly, setting $\beta = 1$ in $\Omega_{g,n} (\beta)$ recovers $\omega_{g,n}$.

Using our calculations of various $f_{0,1,r}^{G}$, $f_{0,1,r}^{N}$, $f_{0,1}^{G,t}$ and $f_{0,1}^{N,t}$ in proposition \ref{prop:f01Nt}, we obtain the following.
\begin{prop}
\label{prop:f01G_f01N_etc}
The functions $\mathfrak{f}_{0,1}^G$, $\mathfrak{f}_{0,1}^N$, ${\bf f}_{0,1}^G$, ${\bf f}_{0,1}^N$ and differential forms $\Omega_{0,1}^G$, $\Omega_{0,1}^N$ are given as follows.
\begin{align*}
\mathfrak{f}_{0,1}^G (x_1; \alpha) &= \frac{x_1 - \sqrt{x_1^2 - 4\alpha}}{2}
\quad \text{so} \quad
\Omega_{0,1}^G (x_1; \alpha) = \frac{x_1 - \sqrt{x_1^2 - 4\alpha}}{2} \; dx_1 \\
\mathfrak{f}_{0,1}^N (z_1; \alpha) &= z_1^{-1} \alpha 
\quad  \text{so} \quad
\Omega_{0,1}^N (z_1; \alpha) = z_1^{-1} \alpha \; dz_1 \\
{\bf f}_{0,1}^G (x_1; \beta) &= z_1 \\
{\bf f}_{0,1}^N (z_1; \beta) &= z_1^{-1}
\end{align*}
\end{prop}

\begin{proof}
All the claimed expressions except $\mathfrak{f}_{0,1}^G (x_1; \alpha)$ consist of sums with a single term, so are obtained immediately from proposition \ref{prop:f01Nt}. We compute $\mathfrak{f}_{0,1}(x_1; \alpha)$ below:
\begin{align*}
\mathfrak{f}_{0,1}^G (x_1; \alpha)
&= \sum_{m=0}^\infty G_{0,1,m+1}(2m) x_1^{-2m-1} \alpha^{m+1} 
= \alpha^{1/2} \sum_{m=0}^\infty G_{0,1}(2m) (x_1 \alpha^{-1/2} )^{-2m-1} 
= \alpha^{1/2} f_{0,1}^G (x_1 \alpha^{-1/2}), 
\end{align*}
which since $f_{0,1}^G (x) = \frac{x-\sqrt{x^2-4}}{2}$, gives the desired result.
\end{proof}

We also obtain ${\bf f}_{0,2}^N$, ${\bf f}_{0,3}^N$ immediately from propositions \ref{prop:f02Nt} and \ref{prop:f03Nt}. Multiplying by $dz_i$ then gives the corresponding differential forms $\Omega_{0,2}, \Omega_{0,3}$. 
\begin{prop}
\label{prop:bff0203}
The generating functions ${\bf f}_{0,2}^N$, ${\bf f}_{0,3}^N$ are given as follows:
\begin{align*}
{\bf f}_{0,2}^N (z_1, z_2; \beta) &=
\frac{1}{(1-z_1 z_2)^2} + \frac{t}{z_1 z_2} \\
{\bf f}_{0,3}^N (z_1, z_2, z_3; \beta) &=
\frac{ 2(z_1 + z_2 + z_3 + z_1 z_2 z_3)(1 + z_1 z_2 + z_2 z_3 + z_3 z_1) }{(1 - z_1^2)^2 (1-z_2^2)^2 (1-z_3^2)^2 } + 
t \left( \sum_{\text{cyc}}
\frac{1 + 4 z_1 z_2 + z_1^2 + z_2^2 + z_1^2 z_2^2}{(1-z_1^2)^2(1-z_2^2)^2 z_3} \right) \\
&\quad + 
t^2 \left( \frac{1 + 16 z_1^2 z_2^2 z_3^2 + z_1^4 z_2^4 z_3^4 + \sum_{\text{cyc}} z_1^4 - 4 z_1^2 z_2^2 + z_1^4 z_2^4 - 4 z_1^4 z_2^2 z_3^2}{z_1 z_2 z_3 (1-z_1^2)^2 (1-z_2^2)^2 (1-z_3^2)^2 } \right).
\end{align*}
\qed
\end{prop}

In the proof of proposition \ref{prop:f01G_f01N_etc}, we found an expression for $\mathfrak{f}_{0,1}^G (x_1; \alpha)$ by rewriting the sum as one involving $f_{0,1}^G (x_1 \alpha^{-1/2})$. We can use a similar trick in general to write each $\mathfrak{f}$ in terms of an ${\bf f}$.

\begin{prop}
\label{prop:relations_between_fs}
For any $g \geq 0$ and $n \geq 1$,
\begin{align*}
\mathfrak{f}_{g,n}^N (z_1, \ldots, z_n; \alpha) &=
\alpha^{2-2g-\frac{n}{2}} {\bf f}_{g,n}^N (z_1 \alpha^{1/2}, \ldots, z_n \alpha^{1/2}; \alpha) \\
\mathfrak{f}_{g,n}^G (x_1, \ldots, x_n; \alpha) &=
\alpha^{2-2g-\frac{3n}{2}} {\bf f}_{g,n}^G (x_1 \alpha^{-1/2}, \ldots, x_n \alpha^{-1/2} ; \alpha)
\end{align*}
\end{prop}
Note that the ``usual" inputs to ${\bf f}_{g,n}^N$ are $(z_1, \ldots, z_n; \beta)$; we are saying that if we substitute each $z_i$ with $z_i \alpha^{1/2}$, and $\beta$ with $\alpha$, then up to a factor of $\alpha^{2-2g-\frac{n}{2}}$ we recover $\mathfrak{f}_{g,n}^N (z_1, \ldots, z_n; \alpha)$. Similarly, if we substitute $z_i$ with $z_i \alpha^{-1/2}$ and $\beta$ with $\alpha$ in ${\bf f}_{g,n}^G$, then we can recover $\mathfrak{f}_{g,n}^G$.

Thus, generating functions with respect to the number of regions $r$, and the variable $\alpha$, can be recovered from generating functions with respect to the parameter $t$, and the variable $\beta$.

\begin{proof}
We compute
\begin{align*}
\alpha^{2-2g-\frac{n}{2}} &
{\bf f}_{g,n}^N (z_1 \alpha^{1/2}, \ldots, z_n \alpha^{1/2}; \alpha) =
\alpha^{2-2g-\frac{n}{2}} 
\sum_{t, \nu_1, \ldots, \nu_n} N_{g,n}^t (\nu_1, \ldots, \nu_n) (z_1 \alpha^{1/2})^{\nu_1 - 1} \cdots (z_n \alpha^{1/2})^{\nu_n - 1} \; \alpha^t \\
&= 
\sum_{t, \nu_1, \ldots, \nu_n} N_{g,n}^t (\nu_1, \ldots, \nu_n) z_1^{\nu_1 - 1} \cdots z_n^{\nu_n - 1} \; \alpha^{2-2g-\frac{n}{2} + \frac{1}{2} \sum_{i=1}^n (\nu_i - 1)} \\
&= \sum_{r, \nu_1, \ldots, \nu_n} N_{g,n,r} (\nu_1, \ldots, \nu_n) z_1^{\nu_1- 1} \cdots z_n^{\nu_n - 1} \; \alpha^r 
= \mathfrak{f}_{g,n}^N (z_1, \ldots, z_n; \alpha).
\end{align*}
Here we have used $r = t - (2-2g-n) - \frac{1}{2} \sum_{i=1}^n \nu_i$, and noted that with $r$ and $t$ related in this way, once $t, \nu_1, \ldots, \nu_n$ are fixed, $N_{g,n}^t (\nu_1, \ldots, \nu_n) = N_{g,n,r} (\nu_1, \ldots, \nu_n)$.

The computation for the second equality is similar.
\end{proof}

We know that each ${\bf f}_{g,n}^N (z_1, \ldots, z_n; \beta)$ and ${\bf f}_{g,n}^G (x_1, \ldots, x_n; \beta)$ is meromorphic; making the above substitution we immediately obtain the following.
\begin{prop}
\label{prop:mathfrakf_meromorphic}
Let $g \geq 0$ and $n \geq 1$. The functions $\mathfrak{f}_{g,n}^G (x_1, \ldots, x_n; \alpha)$ and $\mathfrak{f}_{g,n}^N (z_1, \ldots, z_n; \alpha)$ are locally meromorphic, and for each $\alpha \in \C$, $\Omega_{g,n}^G (x_1, \ldots, x_n; \alpha)$ and $\Omega_{g,n}^N (z_1, \ldots, z_n; \alpha)$ are locally meromorphic differential forms.
\qed
\end{prop}

\subsection{Refined differential equations}
\label{sec:refined_diff_eqns}

We can now finally return to the attempt to find differential equations satisfied by the generating functions $f_{g,n} (x_1, \ldots, x_n)$, which we left off in section \ref{sec:diff_eqn_on_generating_fns}.

Recall in section \ref{sec:recursion_generating_functions} that we took the recursion on $G_{g,n} (b_1, \ldots, b_n)$, multiplied by $x_1^{-b_1 - 1} \cdots x_n^{-b_n - 1}$, and summed over all $b_1 \geq 1$ and $b_2, \ldots, b_n \geq 0$. After suitable manipulation of the three terms $I$, $II$, $III$ on the right hand side, we arrived at lemma \ref{lem:intermediate_step_to_diff_eqn}: for any $(g,n)$,
\begin{align*}
\sum_{\substack{b_1 \geq 1 \\ b_2, \ldots, b_n \geq 0}} 
G_{g,n}&(b_1, \ldots, b_n) 
x_1^{-b_1 - 1} \cdots x_n^{-b_n - 1}
=
x_1^{-1} f^G_{g-1,n+1} (x_1, x_1, x_2, \ldots, x_n) \\
& +
x_1^{-1} 
\sum_{k=2}^n 
\frac{\partial}{\partial x_k}
\frac{1}{x_k - x_1}
\left(
f^G_{g,n-1}(x_2, \ldots, x_n) - f^G_{g,n-1}(x_1, x_2, \ldots, \widehat{x}_k, \ldots, x_n)
\right) \\
& +
x_1^{-1}
\sum_{\substack{g_1 + g_2 = g \\ I_1 \sqcup I_2 = \{2, \ldots, n\}}}
f^G_{g_1, |I_1|+1}(x_1, x_{I_1}) \; f^G_{g_2, |I_2|+1} (x_1, x_{I_2}).
\end{align*}
It remains to deal with the terms on the left hand side with $b_1 = 0$. To this end we can refine the process by number of regions.

We start again, not from the recursion on $G_{g,n}$, but the recursion on $G_{g,n,r}$ in theorem \ref{thm:G_refined_recursion}: for any $(g,n)$ and $b_1 >0$,
\begin{align*}
G_{g,n,r}(b_1, \ldots, b_n)
&=
\sum_{\substack{i,j \geq 0 \\ i+j = b_1 - 2}} G_{g-1,n+1,r} (i,j,b_2, \ldots, b_n) \\
&\quad + \sum_{k=2}^n b_k G_{g,n-1,r} (b_1 + b_k - 2, b_2, \ldots, \widehat{b}_k, \ldots, b_n) \\
&\quad + \sum_{\substack{g_1 + g_2 = g \\ I_1 \sqcup I_2 = \{2, \ldots, n\} }}
\sum_{\substack{i,j \geq 0 \\ i+j = b_1 - 2}} \sum_{\substack{r_1, r_2 \geq 1 \\ r_1 + r_2 = r}}
G_{g_1, |I_1|+1, r_1} (i, b_{I_1}) G_{g_2, |I_2| + 1, r_2} (j, b_{I_2}).
\end{align*}

Again we multiply by $x_1^{-b_1 - 1} \cdots x_n^{-b_n - 1}$; we also multiply by $\alpha^r$. We then sum over all $r>1$, $b_1 > 0$ and $b_2, \ldots, b_n \geq 0$. We obtain
\[
\sum_{\substack{b_1 \geq 1 \\ b_2, \ldots, b_n \geq 0 \\ r \geq 1}} 
G_{g,n} (b_1, \ldots, b_n) \; x_1^{-b_1 - 1} \cdots x_n^{-b_n - 1} \alpha^r
= I_\alpha + II_\alpha + III_\alpha,
\]
where the left hand side is ``almost" $\mathfrak{f}_{g,n}^G (x_1, \ldots, x_n; \alpha)$ (except for terms with $b_1 = 0$), and
\begin{align*}
I_\alpha &=
\sum_{\substack{b_1 \geq 1 \\ b_2, \ldots, b_n \geq 0 \\ r \geq 1}}
\sum_{\substack{i,j \geq 0 \\ i+j = b_1 - 2}}
G_{g-1,n+1,r} (i,j,b_2, \ldots, b_n)
x_1^{-b_1 - 1} \cdots x_n^{-b_n - 1} \alpha^r, \\
II_\alpha &= 
\sum_{\substack{b_1 \geq 1 \\ b_2, \ldots, b_n \geq 0 \\ r \geq 1}}
\sum_{k=2}^n
b_k G_{g,n-1} (b_1 + b_k - 2, b_2, \ldots, \widehat{b}_k, \ldots, b_n)
x_1^{-b_1 - 1} \cdots x_n^{-b_n - 1} \alpha^r, \\
III_\alpha &=
\sum_{\substack{b_1 \geq 1 \\ b_2, \ldots, b_n \geq 0 \\ r \geq 1}}
\sum_{\substack{g_1 + g_2 = g \\ I_1 \sqcup I_2 = \{2, \ldots, n\} }}
\sum_{\substack{i,j \geq 0 \\ i+j = b_1 - 2}}
\sum_{\substack{r_1, r_2 \geq 1 \\ r_1 + r_2 = r}}
G_{g_1, |I_1| + 1, r_1} (i, b_{I_1}) G_{g_2, |I_2| + 1, r_2} (j, b_{I_2})
x_1^{-b_1 - 1} \cdots x_n^{-b_n - 1} \alpha^r.
\end{align*}

The computations of section \ref{sec:recursion_generating_functions} work equally well for the terms $I_\alpha, II_\alpha III_\alpha$ here as for $I, II, III$ there. The only difference is that a factor of $\alpha^r$ is carried throughout; and in $III_\alpha$ we have $\alpha^r = \alpha^{r_1} \alpha^{r_2}$, so we obtain a similar factorisation. These computations yield
\begin{align*}
I_\alpha &= x_1^{-1} \mathfrak{f}_{g-1,n+1}^G (x_1, x_1, x_2, \ldots, x_n; \alpha) \\
II_\alpha &= 
x_1^{-1} \sum_{k=2}^n \frac{\partial}{\partial x_k} \frac{1}{x_k - x_1} 
\left( \mathfrak{f}_{g,n-1}^G (x_2, \ldots, x_n; \alpha) - \mathfrak{f}_{g,n-1}^G (x_1, x_2, \dots, \widehat{x_k}, \ldots, x_n; \alpha) \right) \\
III_\alpha &= 
x_1^{-1} \sum_{\substack{g_1+g_2 = g \\ I_1 \sqcup I_2 = \{2, \ldots, n\}}} 
\mathfrak{f}^G_{g_1, |I_1|+1} (x_1, x_{I_1}; \alpha) \; \mathfrak{f}^G_{g_2, |I_2|+1} (x_1, x_{I_2}; \alpha).
\end{align*}
For the rest of this section, we simply write $\mathfrak{f}_{g,n}$ rather than $\mathfrak{f}^G_{g,n}$ to avoid clutter; we will not be writing $\mathfrak{f}^N_{g,n}$, so there will be no ambiguity.

Now the left hand side we are looking for is
\begin{align*}
\mathfrak{f}_{g,n} (x_1, \ldots, x_n; \alpha) &=
\sum_{\substack{b_1, \ldots, b_n \geq 0 \\ r \geq 1}}
G_{g,n,r} (b_1, \ldots, b_n) \; x_1^{-b_1 - 1} \cdots x_n^{-b_n - 1} \; \alpha^r \\
&= I_\alpha + II_\alpha + III_\alpha + IV_\alpha
\end{align*}
where $IV_\alpha$ is the sum arising from terms with $b_1 = 0$:
\[
IV_\alpha = 
\sum_{\substack{b_2, \ldots, b_n \geq 0 \\ r \geq 1}}
G_{g,n,r} (0, b_2, \ldots, b_n) 
x_1^{-1} x_2^{-b_1 - 1} \cdots x_n^{-b_n - 1} \; \alpha^r.
\]
Applying proposition \ref{prop:b1_equals_zero} to $IV_\alpha$, we obtain
\begin{align*}
IV_\alpha &= 
\sum_{\substack{b_2, \ldots, b_n \geq 0 \\ r \geq 1}}
r \; G_{g,n-1,r} (b_2, \ldots, b_n) \; x_1^{-1} x_2^{-b_2 - 1} \cdots x_n^{-b_n - 1} \alpha^r \\
&= 
x_1^{-1}
\alpha \frac{d}{d\alpha}
\sum_{\substack{b_2, \ldots, b_n \geq 0 \\ r \geq 1}}
G_{g,n-1,r} (b_2, \ldots, b_n) \;  x_2^{-b_2 - 1} \cdots x_n^{-b_n - 1} \alpha^r \\
&= x_1^{-1} \alpha \frac{d}{d\alpha} \mathfrak{f}_{g,n-1} (x_2, \ldots, x_n; \alpha).
\end{align*}

Putting $I_\alpha$ through $IV_\alpha$ all together, we obtain the differential equation of theorem \ref{thm:diff_eqn_gen_fns}.

\subsection{Differential equation in free energies, and quantum curve?}
\label{sec:diff_eqns}

We can now integrate the differential equation to obtain a differential equation on something like a free energy. We obtained free energies by integrating the form $\omega_{g,n}^t$; it is this form, rather than $\omega_{g,n,r}$, which was natural. Similarly, it is the ${\bf f}_{g,n}^G$ and ${\bf f}_{g,n}^N$, which produces the natural differential form $\Omega_{g,n}$. However, as we have seen, a nice recursion can be obtained on $G_{g,n,r}$, and from it we have derived a differential equation for $\mathfrak{f}_{g,n}^G$. If we integrate \emph{this} function, we obtain another set of ``free energies" and we will now show that they obey a differential equation. In this, we follow the techniques of Mulase--Su{\l}kowski in \cite{Mulase_Sulkowski12}.

We therefore define $\mathfrak{F}_{g,n}(x_1, \ldots, x_n; \alpha)$ to be a \emph{free energy} if
\[
\frac{\partial^n \mathfrak{F}}{\partial x_1 \; \cdots \partial x_n} = \mathfrak{f}_{g,n}^G (x_1, \ldots, x_n; \alpha).
\]
We now consider integrating both sides of theorem \ref{thm:diff_eqn_gen_fns} with respect to $x_2, \ldots, x_n$. We obtain the following differential equation on free energies of theorem \ref{thm:free_energy_recursion_intro}.

\begin{thm}
\label{thm:free_energy_recursion}
There are free energies $\mathfrak{F}_{g,n} (x_1, \ldots, x_n; \alpha)$ such that
\begin{align*}
x_1 \frac{\partial}{\partial x_1} \mathfrak{F}_{g,n} (x_1, \ldots, x_n; \alpha)
&=
\frac{\partial^2}{\partial u \partial v} \mathfrak{F}_{g-1,n+1} (u,v,x_2, \ldots, x_n; \alpha) \Big|_{u=v=x_1} \\
&\quad + \sum_{k=2}^n \frac{1}{x_k - x_1} 
\left( 
\frac{\partial}{\partial x_k} \mathfrak{F}_{g,n-1} (x_2, \ldots, x_n; \alpha)
- \frac{\partial}{\partial x_1} \mathfrak{F}_{g,n-1} (x_1, \ldots, x_n; \alpha)
\right) \\
&\quad + \sum_{\substack{g_1 + g_2 = g \\ I_1 \sqcup I_2 = \{2, \ldots, n\}}}
\frac{\partial}{\partial x_1} \mathfrak{F}_{g_1, |I_1|+1} (x_1, x_{I_1}; \alpha) \;
\frac{\partial}{\partial x_1} \mathfrak{F}_{g_2, |I_2|+1} (x_1, x_{I_2}; \alpha) \\
&\quad + \alpha \frac{\partial}{\partial \alpha} \mathfrak{F}_{g,n-1} (x_2, \ldots, x_n; \alpha).
\end{align*}
\qed
\end{thm}

We now assemble the ingredients for a \emph{partition function}.
\begin{defn}
For integers $m \geq 0$, define
\[
S_m (x) = \sum_{2g+n-1 = m} \frac{1}{n!} \mathfrak{F}_{g,n} (x, \ldots, x).
\]
Further define
\[
{\bf F} = \sum_{m=0}^\infty \hbar^{m-1} S_m(x)
\]
and
\[
{\bf Z} = e^{{\bf F}}.
\]
\end{defn}
Here $\hbar$ is a formal parameter and we regard these as formal Laurent series.

\begin{lem}
\label{lem:eqn_for_Ss}
For each $m \geq 0$,
\[
x \frac{\partial}{\partial x} S_{m+1} 
=
\frac{\partial^2 S_m}{\partial x^2}
+ \sum_{a+b = m+1} \frac{\partial S_a}{\partial x} \frac{\partial S_b}{\partial x} + \alpha \frac{\partial S_m}{\partial \alpha}.
\]
\end{lem}

\begin{proof}
This proof follows the method of \cite[Appendix A]{Mulase_Sulkowski12} quite closely. We drop $\alpha$ from $\mathfrak{F}_{g,n}(x_1, \ldots, x_n; \alpha)$ to save space. We take the equation from theorem \ref{thm:free_energy_recursion}, set $x_1 = \cdots = x_n = x$, multiply by $\frac{1}{(n-1)!}$, and sum over all $g,n$ such that $2g+n-2=m$. Taking the terms of the equation separately, we first have
\[
\sum_{2g+n-2=m} \frac{1}{(n-1)!} x_1 \frac{\partial}{\partial x_1} \mathfrak{F}_{g,n}(x_1, \ldots, x_n) \Big|_{x_1, \ldots, x_n}
= x \frac{\partial}{\partial x} \sum_{2g+n-2=m} \frac{1}{n!} \mathfrak{F}_{g,n}(x, \ldots, x)
= x \frac{\partial}{\partial x} S_{m+1}.
\]
Here we used the general fact that 
\[
\frac{d}{dt} f(t, \ldots, t) = n \frac{\partial}{\partial u} f(u,t, \ldots, t) \Big|_{u=t},
\]
for a symmetric function $f$ of $n$ variables. 

Doing the same for the first term on the right hand side, we obtain
\begin{align*}
\sum_{2g+n-2=m} &\frac{1}{(n-1)!} \frac{\partial^2}{\partial u \; \partial v} \mathfrak{F}_{g-1,n+1} (u, v, x_2, \ldots, x_n) \Big|_{u=v=x_2 = \cdots = x_n} \\
&= \sum_{2g+n-2=m} \frac{1}{(n-1)!} \frac{\partial^2}{\partial u \; \partial v} \mathfrak{F}_{g-1,n+1} (u, v, x, \ldots, x) \Big|_{u=v=x}.
\end{align*}
Turning to the second term on the right hand side yields
\begin{align*}
\sum_{2g+n-2=m} & \frac{1}{(n-1)!} \sum_{k=2}^n \frac{1}{x_k - x_1} 
\left( \frac{\partial}{\partial x_k} \mathfrak{F}_{g,n-1} (x_2, \ldots, x_n)
- \frac{\partial}{\partial x_1} \mathfrak{F}_{g,n-1} (x_1, \ldots, \widehat{x_k}, \ldots, x_n) \right) \Big|_{x_2 = \cdots = x_n = x} \\
&=
\sum_{2g+n-2=m} \frac{1}{(n-1)!} \sum_{k=2}^n \frac{\partial^2}{\partial x^2} \mathfrak{F}_{g,n-1} (x, x_2, \ldots, \widehat{x_k}, \ldots, x_n ) \Big|_{x_2 = \cdots = \widehat{x_k} = \cdots = x_n = x} \\
&=
\sum_{2g+n-2=m} \frac{1}{(n-1)!} \sum_{k=2}^n \frac{\partial^2}{\partial u^2} \mathfrak{F}_{g,n-1} (u, x, \ldots, x) \Big|_{u=x} \\
&=
\sum_{2g+n-2=m} \frac{1}{(n-2)!} \frac{\partial^2}{\partial u^2} \mathfrak{F}_{g,n-1} (u, x, \ldots, x) \Big|_{u=x}.
\end{align*}
In the second line, we used the general fact that for functions $f$ and $g$,
\[
\frac{1}{x-y} \left( g(x) \frac{df(x)}{dx} - g(y) \frac{df(y)}{dy} \right) \Big|_{x=y}
= g'(x) f'(x) + g(x) f''(x).
\]
Now adding the first and second terms on the right hand side gives
\begin{align*}
\sum_{2g+n-2=m} 
&\frac{1}{(n-1)!} \frac{\partial^2}{\partial u \; \partial v} 
\mathfrak{F}_{g-1,n+1} (u, v, x, \ldots, x) \Big|_{u=v=x}
+ \frac{1}{(n-2)!} \frac{\partial^2}{\partial u^2} \mathfrak{F}_{g,n-1} (u, x, \ldots, x) \Big|_{u=x} \\
&=
\sum_{2g+n-2=m} \frac{1}{n!} \frac{\partial^2}{\partial x^2} \mathfrak{F}_{g,n-1} (x, \ldots, x)
= \frac{\partial^2 S_m}{\partial x^2}.
\end{align*}
Here we have used the general fact that, for a symmetric function $f$ of $n$ variables,
\[
\frac{d^2}{dt^2} f(t, \ldots, t) = n \frac{\partial^2}{\partial u^2} f(u, t, \ldots, t) \Big|_{u=t} + n(n-1) \frac{\partial^2}{\partial u_1 \; \partial u_2} f(u_1, u_2, t, \ldots, t) \Big|_{u_1 = u_2 = t}.
\]
For the final term, we find
\begin{align*}
\sum_{2g+n-2=m} &\frac{1}{(n-1)!} 
\sum_{\substack{g_1 + g_2 = g \\ I_1 \sqcup I_2 = \{2, \ldots, n\}}} \frac{\partial}{\partial x_1} \mathfrak{F}_{g_1, |I_1|+1} (x_1, x_{I_1}) \frac{\partial}{\partial x_1} \mathfrak{F}_{g_2, |I_2| + 1} (x_1, x_{I_2}) \Big|_{x_1 = \cdots = x_n = x} \\
&= \sum_{2g+n-2 = m} \frac{1}{(n-1)!} 
\sum_{\substack{g_1 + g_2 = g \\ n_1 + n_2 = n-1}} 
\binom{n-1}{n_1} 
\frac{\partial}{\partial x_1} \mathfrak{F}_{g_1, n_1 + 1} (x_1, x, \ldots, x)
\frac{\partial}{\partial x_1} \mathfrak{F}_{g_2, n_2 + 1} (x_1, x, \ldots, x) \Big|_{x_1 = x} \\
&= \sum_{2g+n-2 = m} \sum_{\substack{g_1 + g_2 = g \\ n_1 + n_2 = n-1}}
\frac{1}{n_1!} 
\frac{\partial}{\partial x_1} \mathfrak{F}_{g_1, n_1 + 1} (x_1, x, \ldots, x)
\frac{1}{n_2!}
\frac{\partial}{\partial x_1} \mathfrak{F}_{g_2, n_2 + 1} (x_1, x, \ldots, x) \Big|_{x_1 = x} \\
&= \sum_{a+b=m+1} 
\left( \sum_{2g_1 + n_1 - 2 = a-2} 
\frac{1}{(n_1 + 1)!} 
\frac{\partial}{\partial x} \mathfrak{F}_{g_1, n_1 + 1} (x, \ldots, x) \right)
\left( \sum_{2g_2 + m_2 = 2 = b-2}
\frac{1}{(n_2 + 1)!}
\frac{\partial}{\partial x} \mathfrak{F}_{g_2, n_2 + 1} (x, \ldots, x) \right) \\
&= \sum_{a+b = m+1} \frac{\partial S_a}{\partial x} \frac{\partial S_b}{\partial x}.
\end{align*}
Adding together all the terms then gives the desired result.
\end{proof}

Next, we find a differential equation satisfied by the master logarithmic partition function ${\bf F}$.
\begin{prop}
The function ${\bf F}$ satisfies
\[
\hbar^2 \left( \frac{\partial^2 {\bf F}}{\partial x^2} + \left( \frac{\partial {\bf F}}{\partial x} \right)^2 + \alpha \frac{\partial {\bf F}}{\partial \alpha} \right) - \hbar x \frac{\partial {\bf F}}{\partial x} + \alpha = 0.
\]
\end{prop}

\begin{proof}
We take the equation in lemma \ref{lem:eqn_for_Ss}, multiply by $\hbar^{m+1}$ and sum over $m \geq 0$. The left hand side becomes
\[
\sum_{m=0}^\infty x \frac{\partial S_{m+1}}{\partial x} \hbar^{m+1}
= x \hbar \frac{\partial {\bf F}}{\partial x} - x \frac{\partial S_0}{\partial x}.
\]
The first term on the right hand side becomes
\[
\sum_{m=0}^\infty \hbar^{m+1} \frac{\partial^2 S_m}{\partial x^2} = \hbar^2 \frac{\partial^2 {\bf F}}{\partial x^2}.
\]
The second term yields
\begin{align*}
\sum_{m=0}^\infty \sum_{a+b=m+1} &\hbar^{m+1} \frac{\partial S_a}{\partial x} \frac{\partial S_b}{\partial x}
= \sum_{a+b \geq 1} \hbar^a \frac{\partial S_a}{\partial x} \frac{\partial S_b}{\partial x} \\
&= \left( \sum_{a=0}^\infty \hbar^a \frac{\partial S_a}{\partial x} \right)
\left( \sum_{b=0}^\infty \hbar^b \frac{\partial S_b}{\partial x} \right)
- \left( \frac{\partial S_0}{\partial x} \right)^2 
= \hbar^2 \left( \frac{\partial {\bf F}}{\partial x} \right)^2 
- \left( \frac{\partial S_0}{\partial x} \right)^2.
\end{align*}
The final term gives
\[
\sum_{m=0}^\infty \hbar^{m+1} \alpha \frac{\partial S_m}{\partial \alpha}
= \hbar^2 \alpha \frac{\partial {\bf F}}{\partial \alpha}.
\]
Summing the terms and rearranging then gives
\[
\hbar^2 \frac{\partial^2 {\bf F}}{\partial x^2} + \hbar^2 \left( \frac{\partial {\bf F}}{\partial x} \right)^2 
+ \hbar^2 \alpha \frac{\partial {\bf F}}{\partial \alpha}
- x \hbar \frac{\partial {\bf F}}{\partial x} 
+ x \frac{\partial S_0}{\partial x}
- \left( \frac{\partial S_0}{\partial x} \right)^2
= 0.
\]
It remains to compute the $S_0$ terms.  Now from the definition, $S_0 (x) = \mathfrak{F}_{0,1}(x)$, which is obtained from integrating $\mathfrak{f}_{0,1}^G (x_1, \ldots, x_n; \alpha)$. Thus, using our computation in proposition \ref{prop:f01G_f01N_etc},
\[
\frac{\partial S_0}{\partial x} = \mathfrak{f}_{0,1}^G (x; \alpha)
= \frac{x - \sqrt{x^2 - 4\alpha}}{2},
\]
from which we compute
\[
x \frac{\partial S_0}{\partial x} - 
\left( \frac{\partial S_0}{\partial x} \right)^2 = \alpha,
\]
giving the desired result.
\end{proof}

Finally, we obtain a differential equation satisfied by the partition function ${\bf Z}$, of theorem \ref{thm:quantum_curve_intro}. This is reminiscent of the ``quantum curve" that appears in the general theory of the topological recursion~\cite{Mulase_Sulkowski12, Norbury15}.

\begin{thm}
\label{thm:quantum_curve}
\[
\left( \hbar^2 \frac{\partial}{\partial x^2} - \hbar x \frac{\partial}{\partial x} 
+ \hbar^2 \alpha \frac{\partial}{\partial \alpha} + \alpha \right) {\bf Z} = 0.
\]
\end{thm}

\begin{proof}
Since ${\bf Z} = e^{{\bf F}}$, we have $\frac{\partial {\bf Z}}{\partial x} = \frac{\partial {\bf F}}{\partial x} {\bf Z}$, so $\frac{\partial^2 {\bf Z}}{\partial x^2} = \left( \frac{\partial^2 {\bf F}}{\partial x^2} + \left( \frac{\partial {\bf F}}{\partial x} \right)^2 \right) {\bf Z}$. Also $\frac{\partial {\bf Z}}{\partial \alpha} = \frac{\partial {\bf F}}{\partial \alpha} {\bf Z}$. Using these we translate the previous statement into the claimed result.
\end{proof}

\addcontentsline{toc}{section}{References}

\small

\bibliography{counting_curves_on_surfaces_v2}
\bibliographystyle{amsplain}

\end{document}